\documentclass[10pt,twoside,a4paper]{article}

\usepackage{authblk}

\usepackage{geometry}
\usepackage[utf8]{inputenc}
\usepackage{hyperref}

\usepackage{todonotes}

\usepackage[toc,title,page]{appendix}

\usepackage{amsmath,amsfonts,amssymb,amsthm}
\usepackage{bbm}

\usepackage{enumitem}

\usepackage{algorithm}
\usepackage{algpseudocode}

\newtheorem{Theorem}{Theorem}[section]

\newtheorem{Definition}[Theorem]{Definition}
\newtheorem{Proposition}[Theorem]{Proposition}
\newtheorem{Lemma}[Theorem]{Lemma}
\newtheorem{Assumption}[]{Assumption}

\newtheorem{Remark}[Theorem]{Remark}

\usepackage[a-1b]{pdfx}

\usepackage{graphicx}
\usepackage{subcaption}

\usepackage[maxbibnames=50]{biblatex}
\definecolor{green(colorwheel)(x11green)}{rgb}{0.16, 0.5, 0.0}
\usepackage[normalem]{ulem}
\definecolor{electricultramarine}{rgb}{0.25, 0.0, 1.0}

\newcommand*{\fzcst}[1]{\relax\ifmmode\text{\textcolor{green(colorwheel)(x11green)}{\sout{\ensuremath{#1}}}}\else\textcolor{green(colorwheel)(x11green)}{\sout{#1}}\fi}

\newcommand*{\ykcst}[1]{\relax\ifmmode\text{\textcolor{blue}{\sout{\ensuremath{#1}}}}\else\textcolor{blue}{\sout{#1}}\fi}

\title{Numerical Methods for Dynamical Low-Rank Approximations of Stochastic Differential Equations \\
Part I: Time discretization}
\author[1]{Yoshihito Kazashi}
\author[2]{Fabio Nobile}
\author[2]{Fabio Zoccolan}

\affil[1]{Department of Mathematics,
	The University of Manchester, Oxford Road, Manchester, M13 9PL, UK. email: y.kazashi@manchester.ac.uk}
\affil[2]{Institut de Mathématiques, École Polytechnique Fédérale de Lausanne, 1015 Lausanne, Switzerland. email: fabio.nobile@epfl.ch, fabio.zoccolan@epfl.ch}
\date{}

\begin{document}

\maketitle

\begin{abstract}
	\noindent 
	In this work (Part I), we study three time-discretization schemes for the Dynamical Low-Rank Approximation (DLRA) of high-dimensional stochastic differential equations (SDEs). Specifically, we consider the Dynamically Orthogonal (DO) method for DLRA proposed and analyzed in \cite{kazashi2025dynamical}, which approximates the true solution by a linear combination of few products between deterministic orthonormal modes and stochastic modes, both time-dependent. 
	The first scheme considered consists in a forward discretization in time of both deterministic and stochastic components, in a Euler-Maruyama style.
	Its convergence is 
	proven subject 
	 to a time-step restriction dependent on the smallest singular value of the Gram matrix associated to the stochastic modes, which, on its turn, is shown to be always positive, provided that the SDE under study is driven by a non-degenerate noise.
	The second and the third schemes, on the other hand, are staggered ones, alternating updates of the deterministic and the stochastic modes in half steps, and have a projector splitting nature. We show stability of the second scheme and prove convergence with constants independent of the smallest singular value. The third scheme works better in practice, although our theoretical convergence bounds are worse than those for the second one. 
	Computational experiments support our theoretical results. 
	In this work we do not consider the discretization in probability, which will be the topic of Part II.
\end{abstract}	

\section{Introduction}
Many models used in engineering applications to model natural phenomena are described by stochastic differential equations (SDEs)
 \cite[Chapter 5]{allen2007modeling}, \cite[Chapter 7]{kloeden1992stochastic}, \cite[Chapter 5]{shreve2004stochastic}, \cite{majda1999models}.
When these are high dimensional, their numerical simulations become quickly computationally prohibitive. Examples of high dimensional SDEs include machine learning problems, such as simulations of particle dynamics for adaptive sampling of concentrated measures \cite{wen2024coupling}, climate modeling \cite{franzke2005low}, stochastic models in finance, for example considering high dimensional asset price structures described by an Heston model \cite{redmann2021low}, data assimilation and filtering procedures \cite{law2015data}, or, more generally, time integration of spatially-discretized stochastic partial differential equations \cite{lord2014introduction}.

A possible approach to lower the computational burden  
is 
to exploit model order reduction (MOR) techniques, which aim to build a surrogate system of low dimensionality that is numerically inexpensive to solve while keeping good agreement with the original one. MOR methods have witnessed great success when applied to high dimensional problems that feature some intrisic low-dimensional structures
\cite{franzke2005low,grafke2017non,hesthaven2016certified,redmann2021low,sapsis2013interaction,sapsis2013blending,tyranowski2024data}.
However, in certain cases, such low dimensional structure may evolve considerably over time, limiting the effectiveness of MOR techniques that project the dynamics on pre-defined fixed subspace, although carefully chosen.

To address these cases, in this work we consider a Dynamical Low-Rank Approximation (DLRA) framework for SDEs \cite{kazashi2025dynamical}. 
In contrast to standard reduced order models for uncertainty quantification, DLRA aims to construct a surrogate which is spanned by time-varying basis, in both the deterministic and stochastic domains. The first rigorous formalization of DLRA was proposed in \cite{koch2007dynamical} for deterministic matrix evolution equations, as a constrained dynamics onto the manifold of low (fixed) rank matrices.
DLRA has been intensively studied for ordinary differential equations \cite{carrel2025interpolatory,koch2007dynamical,lubich2014projector,kieri2016discretized,kieri2019projection}, deterministic \cite{bachmayr2021existence,einkemmer2018low,einkemmer2023robust,ramezanian2021fly}, and random partial differential equations (rPDEs) \cite{kazashi2021existence,kazashi2021stability,kusch2022dynamical}, while its application to SDEs is still in its infancy and its potential has not been fully exploited yet \cite{kazashi2025dynamical}. 

There are several ways to instantiate DLRA equations for SDEs. One possibility is through the Dynamically Orthogonal (DO) framework, which was first proposed in \cite{sapsis2009dynamically} for rPDEs and was used in \cite{kazashi2025dynamical} to obtain a well-posed mathematical setting in the SDEs framework.

The DO approach consists in building an approximation composed by two components, a deterministic basis and a stochastic one, both allowed to evolve in time and computed on the fly. More precisely, we seek a solution $X$ of the following form
\begin{equation}\label{eq: DO}
	X(t,\omega) = \sum_{i = 1}^{k} U^{i}(t)Y^{i}(t,\omega), \quad t \geq 0,
\end{equation}
where $\{U^{i}\}_{i=1,\dots, k}$ are the time-dependent deterministic modes, orthonormal according to a fixed scalar product, and $\{Y^{i}\}_{i=1,\dots, k}$ are the stochastic modes. In \cite{kazashi2025dynamical}, evolution equations have been derived for the deterministic and stochastic modes and well posedness of the resulting coupled system has been shown under suitable assumptions, which we assume to hold in this work as well. From the evolution of the modes one can reconstruct the DLRA via \eqref{eq: DO}.

The focus of this paper is the time discretization of DLRA for SDEs.
We consider three time-discretization algorithms and establish error bounds with respect to the true solution of the SDE involving a time discretization error and a low-rank approximation error. 
A critical issue in the analysis is the potential rank deficiency of the exact/numerical DLRA solution (i.e.\ some of the stochastic modes become linearly dependent). We discuss this issue and provide suitable conditions on the SDE and the time step for the discrete Gramian $C_{Y_n}$ to be invertible at any time step $t_n$, provided the true Gramian $C_{Y(t)}$ remains of rank $k$ for all $t$. This contrasts with deterministic equations and rPDEs, where both the continuous and the discretized DLRA may lose full rankness at final time.

The first numerical algorithm that we propose is the \textit{DLR Euler-Maruyama} scheme (DLR EM). It consists in discretizing the equation for the stochastic modes with an Euler-Maruyama scheme and the one for the deterministic modes with a Forward Euler method. To preserve the orthogonality of the deterministic basis, we perform a QR decomposition at each time iteration. The presence of this orthogonalization step makes the error analysis non standard.
Nevertheless, the DLR-Euler Maruyama turns out to be convergent in time with the same order of the standard Euler-Maruyama for SDEs. However, due to its explicit nature, a condition on the time step $\Delta t$, depending on the smallest singular value of $C_{Y_{n}}$, is needed to reach this goal.

The second and third algorithms proposed in this work adopt a staggered approach where first the stochastic modes are updated, using a Euler-Maruyama scheme as in the first algorithm, and then the deterministic modes are updated using the new stochastic modes. These two algorithms differ in the way the deterministic modes are updated.

More precisely. the second algorithm uses only the drift term to update the deterministic basis and is consistent with the DLRA equations derived in \cite{kazashi2025dynamical}. It can be interpreted as a projector splitting scheme, in the same spirit of the popular projector splitting methods for DLR approximation of matrix evolution equations \cite{lubich2014projector} where the drift term is projected onto the tangent space to rank-$k$ random vectors at an intermediate configuration obtained by combining the updated stochastic modes with the old deterministic ones, whereas the diffusion term is only projected on the old deterministic modes. This different treatment of the drift and diffusion terms was advocated in \cite{kazashi2025dynamical} because of the difficulty of rigorously defining a tangent space projection of the diffusion term on the manifold of rank-$k$ random vectors. We name this algorithm \textit{``DLR Projector Splitting for SDEs"} (DLR PS SDE). For this scheme we are able to obtain error bounds whose constants do not depend on the smallest singular value of $C_{Y_n}$ and do not require any restriction on the time step. We also analyze its asymptotically mean square stability for a linear SDE with multiplicative noise showing stability under the same conditions of the EM method applied to the original SDE, without any dependence on the smallest singular value.

The third algorithm uses, instead, both the drift and the diffusion to update the deterministic modes. It can be interpreted as a ``standard" projector splitting method, as presented e.g.\ in \cite{kazashi2021stability}, applied to the Euler-Maruyama approximation of the original SDE. In this respect, in the limit $\Delta t\to 0$, it is not consistent with the DLRA equations proposed in \cite{kazashi2025dynamical}, rather to those proposed in \cite{cao2018stochastic} where an extra diffusion-type term appears in the equation of the deterministic modes. We name this algorithm \textit{``DLR Projector Splitting for Euler-Maruyama"} (DLR PS EM). For this algorithm, we obtain analogous asymptotic mean square stability results as for the second algorithm. Concerning its accuracy, our numerical results show that it performs even better than the second algorithm. However, we were not able to obtain error bounds, with respect to the true solution of the SDE, that are independent of the smallest singular value, although this is not observed in our numerical experiments.

The outline of the paper is the following.
In Section \ref{sec: general setting}, we illustrate the general framework of DLRA for SDEs from \cite{kazashi2025dynamical}. In Section \ref{sec: discretization procedure}, the three discretization algorithms will be introduced. Sections \ref{sec: DLR EM}, \ref{sec: convergence PS SDE}, and \ref{sec: convergence PS EM} are centred at demonstrating properties of convergence in time of the three schemes; more specifically, mean-square stability, non-degeneracy of the Gramian under non-degenerate diffusion, and convergence in time. These results will be confirmed via numerical experiments in Section \ref{sec: numerical experiments}. Finally, we will draw some conclusions in Section \ref{sec: conclusion}.  For the sake of readability, several technical proofs and auxiliary lemmata have been postponed to the Appendix at the end of this document.

The discretization of randomness, such as the computation of expectations arising in the $L^2$-projection, will be treated in the second part of this paper \cite{kazashi2025dynamicalpartII}.

\section{General Setting of DLRA for SDEs}\label{sec: general setting}
In this section, we will briefly introduce the general theoretical framework for continuous DLRA for SDEs. More precisely, we will state the DO equations for DLRA of SDEs as derived in \cite{kazashi2025dynamical} and recall sufficient conditions that ensure the existence and uniqueness of a DO solution. Under these constraints, it is worth studying numerical time discretization schemes to find approximate solutions of the DO equations.

Let $\left( \Omega, \mathcal{F}, \mathbb{P}, (\mathcal{F}_t)_{t \geq 0} \right)$ be a filtered complete probability space with the usual conditions \cite[Remark 6.24]{schilling2021brownian}.
We denote by $W$ a real $m$-dimensional $(\mathcal{F}_t)$-Brownian motion,
i.e., $W(t) = \left(W_1(t), \ldots , W_m(t) \right)^{\top}.$
We want to approximate SDEs taking values in $\mathbb{R}^d$ of the following type
\begin{equation}\label{eq:SDE-diff}
	\mathrm{d}X^{\mathrm{true}}(t) = a(t,X^{\mathrm{true}}(t))\mathrm{d}t+b(t,X^{\mathrm{true}}(t))\mathrm{d}W_t, \quad \text{ for all } t \in [0, T], \quad X^{\mathrm{true}}(0) = X^{\mathrm{true}}_0,
\end{equation}
where $T>0$. The solution of \eqref{eq:SDE-diff} is a random vector $X^{\mathrm{true}}(t)=\left(X^{\mathrm{true}}_1(t), \ldots, X^{\mathrm{true}}_d(t)\right)^{\top}$. 
We work under the following assumptions.
\begin{Assumption}[{Lipschitz coefficients}]\label{lipschitz}
	The drift $a\colon[0,\infty)\times\mathbb{R}^{d}\to\mathbb{R}^{d}$ 
	and the diffusion $b\colon[0,\infty)\times\mathbb{R}^{d}\to\mathbb{R}^{d\times m}$
	are measurable between the Borel fields $\left([0,\infty)\times\mathbb{R}^{d};\mathcal{B}([0,\infty)\times\mathbb{R}^{d})\right)$ and $\left(\mathbb{R}^{d};\mathcal{B}(\mathbb{R}^{d})\right)$, $\left(\mathbb{R}^{d\times m};\mathcal{B}(\mathbb{R}^{d \times m})\right)$, respectively, and Lipschitz continuous with respect to the second variable, uniformly in the first one:
	\begin{equation}
	\begin{cases}
	\!\!\!\! & |a(s,x)-a(s,y)|\leq C_{\mathrm{Lip}}|x-y|, \quad \text{ for all } x,y \in \mathbb{R}^d, \ s \in [0,\infty),\\
	\!\!\!\! & \|b(s,x)-b(s,y)\|_{\mathrm{F}}\leq C_{\mathrm{Lip}}|x-y|, \quad \text{ for all } x,y \in \mathbb{R}^d, \ s \in [0,\infty),
	\end{cases}\label{eq:lip}
	\end{equation}
	for some constant $C_{\mathrm{Lip}}>0$, where $|\cdot|$ and $\|\cdot\|_{\mathrm{F}}$ are the usual Euclidean norm on $\mathbb{R}^{d}$ and the Frobenius one on $\mathbb{R}^{d \times m}$, respectively.
\end{Assumption}
\begin{Assumption}[{Linear growth bound}]\label{linear-growth-bound}
	The drift $a$
	and the diffusion $b$ fulfil the following linear-growth bound condition:
	\begin{equation}
	|a(s,x)|^{2}+\|b(s,x)\|_{\mathrm{F}}^{2}\leq C_{\mathrm{lgb}}(1+|x|^{2}), \quad \text{ for all } x \in \mathbb{R}^d, \ s \in [0,\infty)\label{eq:lin-growth}
	\end{equation}
	for some constant $C_{\mathrm{lgb}}>0$.
\end{Assumption}
Furthermore, we assume that the initial condition $X_{0}^{\mathrm{true}}$ in \eqref{eq:SDE-diff} satisfies the following:
\begin{Assumption}[{Square integrable initial condition}]\label{eq:initial value}
	\begin{equation}
	X^{\mathrm{true}}_0 \mbox{ is } \mathcal{F}_0\mbox{-measurable and satisfies } \mathbb{E}[|X^{\mathrm{true}}_0|^2] < +\infty.
	\end{equation}
\end{Assumption}
Under Assumptions 1-3, equation \eqref{eq:SDE-diff} has a unique strong solution; see for example \cite[Theorem 21.13]{schilling2021brownian}.  

In what follows, we say that a process $\{X(t)\}_{t \in [0,T]}$ is of rank $k$ if the matrix $\mathbb{E}[X(t)X(t)^{\top}]$ is of rank $k$ for all $t \in [0,T]$. In order to ensure that the DLRA provides an effective way of approximating the stochastic dynamics described by \eqref{eq:SDE-diff}, we suppose further that \eqref{eq:SDE-diff} exhibits a nearly low-rank structure at each time instant. 

	As discussed in \cite{kazashi2025dynamical}, approximating $X^{\mathrm{true}}$ with a rank-$k$ DO approximation $X$ consists in finding
	a pair $(U, Y)$ satisfying the following conditions. 
	The function $U : [0,T] \to \mathbb{R}^{k\times d}$ is a deterministic absolutely continuous matrix-valued function such that $U_0U_0^{\top} = I_{k \times k}$ and $U(t)\dot{U}(t)^{\top} = 0$ for all $t$ (which implies that $U(t)U(t)^{\top} = I_{k \times k}$ for all $t$), 
 and $Y: [0,T] \to L^{2}(\Omega,\mathbb{R}^{k})$ is an adapted stochastic process with almost surely continuous paths and $Y^1(t), \dots Y^k(t)$ linearly independent in $L^{2}(\Omega)$. The rank-$k$ DO approximation is then $X(t)= U(t)^{\top}Y(t)$ for all $t$. 
The $(U,Y)$ pair satisfies the following system, called the DO equations \cite{kazashi2025dynamical,sapsis2009dynamically},
\begin{equation}\label{DLR-conditions}
	\begin{aligned}
		C_{Y(t)}\dot{U}(t) & = \mathbb{E}\left[Y(t) a(t,U(t)^{\top}Y(t))^{\top}\right]\left(I_{d \times d} - P^{\mathrm{row} }_{U(t)} \right), \\
		Y(t)  & =Y_0 + \int_{0}^{t}U(s) a(s,U(s)^{\top}Y(s)) \mathrm{d}s + \int_{0}^{t}  U(s) b(s,U(s)^{\top}Y(s))\mathrm{d}W_s,
	\end{aligned}
\end{equation}
where $C_{Y(t)}:= \mathbb{E}[Y(t)Y(t)^{\top}]$ is the Gramian associated to $Y^{1},\dots,Y^{k}$, $P^{\mathrm{row}}_{U(t)}:= U(t)^{\top}U(t)$ is the projector-matrix onto the vector space $\operatorname{span}\{U^{1}(t), \ldots, U^{k}(t)\} \subset \mathbb{R}^{d}$, with $U^{i}(t)$ the $i$-th row of  $U(t)$, and $U_0^{\top}Y_0$ is a suitable rank-$k$ approximation of $X_0^{\text{true}}$. 
Under Assumptions 1--3, equation \eqref{DLR-conditions} admits a unique strong DO solution $(U,Y)$, at least for short times, if $X_0^{\text{true}}$ has rank greater than or equal to $k$ \cite{kazashi2025dynamical}. Furthermore, there exists a positive constant $K_2(T)$ which bounds the second moment of $Y(t)$ \cite[Lemma 2.7]{kazashi2025dynamical}. See \cite{kazashi2025dynamical} for more details on such DLR approximation.   

\subsection{General Notation}
For the sake of clarity, we list here the most used notation throughout the document.
Consider a vector $v \in \mathbb{R}^{m}$, then $|v|$ denotes the Euclidean norm of $v$. Consider a matrix $A\in \mathbb{R}^{m \times n}$, then $| A |$ and $\| A \|_{\mathrm{F}}$ denote its spectral and Frobenius norm, respectively.
Henceforth, the notation $A \succ B$ (respectively  $A \succeq B$) with $A,B$ square matrices means that $A-B$ is positive definite (respectively positive semi-definite). 
Moreover, $\{\sigma^i(A)\}_{i\geq 1}^{\min \{n,m\}}$ denote the singular values of the matrix $A$, $\{ \lambda^i(A)\}_{i\geq 1}$ denote the eigenvalues of $A$, whenever $A$ is a squared matrix, both listed in decreasing order, $\lambda_{\max}(A)$, $\sigma_{\max}(A)$ denote the eigenvalue with largest real part and the largest singular value, respectively.

Furthermore, consider a random variable $X \in L^{2}(\Omega,\mathbb{R}^{d})$ such that $X=U^{\top}Y$, where $U \in \mathbb{R}^{k \times d}$ has orthonormal rows and $Y \in L^{2}(\Omega,\mathbb{R}^{k})$ has linearly independent components (we call such  $X$ a ``rank-k" random vector). Then, we define the orthogonal projectors $P^{\mathrm{row} }_{U}[\ \cdot \ ] : \mathbb{R}^{d} \to \mathbb{R}^{d}$ and $P_{Y}[ \ \cdot \ ] : L^{2}(\Omega,\mathbb{R}^{d}) \to L^{2}(\Omega,\mathbb{R}^{d})$ as, respectively,
\begin{equation}\label{eq: P_U and P_y}
P^{\mathrm{row} }_{U}[\ \cdot \ ]= U^{\top}U[\ \cdot \ ] \cdot \text{ and } P_{Y}[ \ \cdot \ ] = \mathbb{E}\left[\ \cdot \ Y^{\top}\right]C^{-1}_{Y}Y,
\end{equation}
which project, respectively, onto the span of the rows of $U$, i.e.\ the orthogonal projector onto the range of $X$, and onto the span of the components of $Y$, i.e.\ the orthogonal projector onto the corange of $X$.
Then we can define the operator $P_{U^{\top}Y}: L^{2}(\Omega,\mathbb{R}^{d}) \to L^{2}(\Omega,\mathbb{R}^{d})$, i.e.\ the projector onto the tangent space of rank $k$ processes in $L^{2}(\Omega,\mathbb{R}^{d})$ computed in the point $X=U^{\top}Y$, as the following: 
\begin{equation}\label{eq: proj X}
	P_{U^{\top}Y}[\ \cdot \ ]= \left(I_{d \times d} - P^{\mathrm{row} }_{U} \right) P_{Y}[\ \cdot \ ] + P^{\mathrm{row} }_{U}[ \ \cdot \ ].
\end{equation}
It is easy to see that \eqref{eq: proj X} is a linear, idempotent, and symmetric operator, as such $P_{{U}^{\top}Y}$ is an orthogonal projector. For the sake of readability, from now on we will drop the superscript $\mathrm{row}$ and, hence, $P_{U} = P^{\mathrm{row} }_{U}$.

\section{Time integration schemes}
\label{sec: discretization procedure}
In this section, we introduce three time-integration schemes to compute the DLRA solution $X(t) = U^{\top}(t)Y(t)$.
Concretely, considering a partition $\Delta := \left\{t_n \ : \ 0= t_0 < t_1 < \ldots < t_{N-1} < t_{N} = T\right\}$ of $[0,T]$,
we seek suitable approximations $(U_n)_n$, $(Y_n)_n$ of $(U(t_n))_n$ and $(Y(t_n))_n$, respectively, and 
reconstruct afterwards 
an approximate DLRA solution $X_n = U_n^{\top}Y_n \approx X(t_n)$. 

In Section \ref{sec: DLR EM alg} we present the DLR Euler--Maruyama algorithm, which employs a forward discretization of both $U$ and $Y$, whereas in Section \ref{sec: PS SDE} and in Section \ref{sec: PS EM} we propose two extensions to the SDE setting of the DLR projector splitting scheme \cite{lubich2014projector,kazashi2021stability}, which updates $U$ and $Y$ in a staggered way.
\subsection{DLR Euler--Maruyama discretization}\label{sec: DLR EM alg}
The DO equations \eqref{DLR-conditions} are a system of stochastic differential equations. In this section, we consider the most basic discretization method for such equations: the Euler--Maruyama scheme, which reads: given $(U_n, Y_n)$,
\begin{align}
	 \text{find } \quad	C_{Y_n} \tilde{U}_{n+1}&  = C_{Y_n} U_n + \mathbb{E}\left[Y_n a(t_n,U_{n}^{\top}Y_n)^{\top}\right]\left(I_{d \times d} - P_{U_n} \right) \Delta t_n, \label{eq: DLR EM 1v}\\
		\tilde{Y}_{n+1} &= Y_n + U_n a(t_n,U_{n}^{\top}Y_n) \Delta t_n +  U_n b(t_n,U_{n}^{\top}Y_n) \Delta W_n, \label{eq: DLR EM 2v} \\
		\text{ compute }(U_{n+1},Y_{n+1}):& \quad \tilde{U}_{n+1}^{\top}\tilde{Y}_{n+1} = U_{n+1}^{\top}Y_{n+1} , \quad U_{n+1}U_{n+1}^{\top} = I_{d \times d}. \label{eq: DLR EM 3v}
\end{align}
with $C_{Y_n} := \mathbb{E}[Y_nY_n^{\top}]$, $\Delta t_n:= t_{n+1} -t_n$, $\Delta W_n: = W(t_{n+1})-W(t_{n}) \sim \mathcal{N}(0,\Delta t_n I_{m \times m})$, and $\mathbb{E}[\Delta W_n | \mathcal{F}_{t_n}]=0$ for all $n$.
In the continuous DO framework of equations \eqref{DLR-conditions} the deterministic modes $U(t)$ have orthogonal rows for all $t \in [0,T]$.
To be consistent with this continuous setting, in \eqref{eq: DLR EM 3v} we orthonormalize the computed deterministic modes $\tilde{U}_{n+1}$ through a \texttt{QR} decomposition. We summarize the whole discretization procedure in Algorithm~\ref{alg: DLR EM SDE algorithm}. We point out that in Algorithm~\ref{alg: DLR EM SDE algorithm}, as well as in the other algorithms of this paper, the QR decomposition works as follows: for all rectangular tall matrices $A \in \mathbb{R}^{n \times k}$ one has $(Q,R) =\texttt{QR}(A)$ with $Q \in \mathbb{R}^{n \times k}$ with orthogonal columns and $R \in \mathbb{R}^{k \times k}$. For the purpose of theoretical analysis, in Algorithm~\ref{alg: DLR EM SDE algorithm} we will always consider a QR decomposition based either on Cholesky factorization or on Gram-Schmidt orthonormalization \cite{golub2013matrix}, whereas in the other algorithms we do not impose such a restriction.

\begin{algorithm}
	\caption{DLR Euler--Maruyama approximation for SDEs}\label{alg: DLR EM SDE algorithm}
	
	\begin{flushleft}
		\textbf{Input}: initial data $U_0$, $Y_0$.
		
		\textbf{Output:} approximation $\{X_{n} = U_n^{\top}Y_n\}_{n=0,\ldots, N}$.
	\end{flushleft}
	
	\begin{algorithmic}[1]
		
		\ForAll {$n \in \{0, \ldots N-1\}$}
		
		\State Generate a Brownian increment $\Delta W_n \sim \mathcal{N}(0,\Delta t_n I_{m \times m})$ 
		
		\State Assemble $C_{Y_n} = \mathbb{E}[Y_nY_n^{\top}]$ 
		
		\State Compute $ \tilde{Y}_{n+1} = Y_n + U_n a(t_n,U_{n}^{\top}Y_n) \Delta t_n +  U_n b(t_n,U_{n}^{\top}Y_n) \Delta W_n$
		
		\State Compute $ C_{Y_n} \tilde{U}_{n+1} =  C_{Y_n} U_n +\mathbb{E}\left[Y_n a(t_n,U_{n}^{\top}Y_n)^{\top}\right]\left(I_{d \times d} - P_{U_n} \right) \Delta t_n$
		\State Reorthonormalize the deterministic modes: find  $(U_{n+1}, Y_{n+1})$ such that:
		\begin{equation*}
			U_{n+1}^{\top} Y_{n+1} = \tilde{U}_{n+1}^{\top} \tilde{Y}_{n+1}, \quad U_{n+1}U_{n+1}^{\top} = I_{d\times d}.
		\end{equation*}
		with $(U_{n+1}^{\top}, R) =\texttt{QR}(\tilde{U}_{n+1}^{\top})$ and $Y_{n+1} = R \tilde{Y}_{n+1}$.
		\EndFor
	\end{algorithmic}
\end{algorithm}

We assume hereafter that $C_{Y_n}$ remains of full rank for all $n$. Sufficient conditions for this assumption are discussed in Section \ref{sec: DLR EM}. Then, the scheme \eqref{eq: DLR EM 1v} can be rewritten as
\begin{equation}\label{Euler--Maruyama eq}
	\begin{aligned}
		\tilde{U}_{n+1} & = U_n +C_{Y_n}^{-1}\mathbb{E}\left[Y_n a(t_n,U_{n}^{\top}Y_n)^{\top}\right]\left(I_{d \times d} - P_{U_n} \right) \Delta t_n, \\
		\tilde{Y}_{n+1} &= Y_n + U_n a(t_n,U_{n}^{\top}Y_n) \Delta t_n +  U_n b(t_n,U_{n}^{\top}Y_n) \Delta W_n.
	\end{aligned}
\end{equation}

\begin{Remark}[Well posedness of equation \eqref{eq: DLR EM 1v}]\label{rem: wellposed U}
	Observe that the system \eqref{eq: DLR EM 1v} has a solution even if $C_{Y_n}$ is singular, since the right-hand side is always in the range of $C_{Y_n}$, which includes the $\mathrm{span}\{Y_n(\omega), \ \omega \in \Omega\} \subset \mathbb{R}^{k}$. 
	Indeed, the stochastic basis can be rewritten as $Y_n(\omega) = C_{Y_n}^{\frac{1}{2}}\check{Y}_n(\omega)$, with  $\mathbb{E}[\check{Y}_n\check{Y}_n^{\top}]=\mathrm{diag}(\check{\sigma}_n^i)$ with $\check{\sigma}_n^i=1$ or $0$, and, hence, $\mathrm{span}\{Y_n(\omega), \ \omega \in \Omega\}) \subseteq \mathrm{Ran}(C_{Y_n})$. 
\end{Remark}
In view of Remark \eqref{rem: wellposed U}, a numerical solution, for instance that of minimal norm, can be defined even 
if $C_{Y_n}$ is singular. However, our error analysis indicates that constants in the error estimate may deteriorate when the smallest eigenvalue of $C_{Y_n}$ tends to zero. For non-degenerate noise we will show that a stringent condition on the time step assures $C_{Y_n}$ to remain full-rank for all time steps, although, the constant in the error estimate may still be very large for small singular values. 
More detailed discussions are provided in Section~\ref{sec: DLR EM}.

The main advantages
of this method are the ease of implementation and the low computational cost. For instance, in the case of a linear deterministic drift $a(t,x) = A(t)x$, which appears in various practical applications including the Kalman-Bucy Filter \cite{law2015data} or the Heston model \cite{redmann2021low}, the algorithm significantly simplifies. 
Indeed, since in this case the inverse Gramian matrix $C_{Y_n}^{-1}$ can be simplified in \eqref{Euler--Maruyama eq}, the update of the deterministic modes becomes simply 
\begin{equation}\label{eq: linear drift}
	\tilde{U}_{n+1} = U_n + U_{n}^{\top}A^{\top}(t_n)\left(I_{d \times d} - P_{U_n} \right) \Delta t_n.
\end{equation}
This also means that a possible ill-conditioning of $C_{Y_n}$ does not affect the numerical scheme. 
Moreover, the deterministic modes can be evolved completely independently of the stochastic modes, leading to parallelization opportunities.

\subsection{DLR Projector Splitting for SDEs}\label{sec: PS SDE}
Here we present a numerical method for DLRA that is as cheap to compute as Algorithm~\ref{alg: DLR EM SDE algorithm}, yet whose accuracy and stability are less sensitive to the smallest singular value of the Gramian of $Y_n$ than Algorithm~\ref{alg: DLR EM SDE algorithm}. 
This algorithm (as well as Algorithm~\ref{alg: Eva Proj Splitt SDE algorithm}) can be interpreted as a projector splitting approximation, similar to the one proposed for rPDEs (see \cite[Section 4.4]{kazashi2021stability} for more details), which itself was inspired by the procedure presented in~\cite{lubich2014projector}.

Here we first update the stochastic modes using the deterministic modes $U_n$, and then we build the new deterministic modes using the newly computed stochastic modes. 
In this staggered scheme, the expectation appearing in the equation of the deterministic modes involves only the drift term consistently with relation \eqref{DLR-conditions}. In both steps, we still approximate the drift and diffusion in a forward-Euler way. We remark that in \cite{kazashi2021stability} other approximations of implicit or semi-implicit type are considered as well. 

This strategy applied to \eqref{DLR-conditions} leads to the scheme 
\begin{equation}
\begin{aligned}\label{eq: stoc proj}
	\tilde{Y}_{n+1} &= Y_n + U_n\, a(t_n,U_{n}^{\top}Y_n) \Delta t_n +  U_n\, b(t_n,U_{n}^{\top}Y_n) \Delta W_n  \\
	C_{\tilde{Y}_{n+1}} \tilde{U}_{n+1} &= C_{\tilde{Y}_{n+1}} U_n + \mathbb{E}\left[\tilde{Y}_{n+1}\, a(t_n,U_{n}^{\top}Y_{n})^{\top}\Delta t_n \right]\left(I_{d \times d} - P_{U_n}\right)  
\end{aligned}
\end{equation}
Then the final update $(U_{n+1}, Y_{n+1})$ at the $n$-th step is obtained by QR orthonormalization. 

Relation \eqref{eq: stoc proj} indeed yields a projected scheme. To see this, first let us recall that $P_{{U}_n}[ \ \cdot \ ]= U_n^{\top}U_n \cdot$ and $P_{\tilde{Y}_{n+1}}[ \ \cdot \ ] = \mathbb{E}\left[ \ \cdot \ \tilde{Y}_{n+1}^{\top}\right]C^{-1}_{\tilde{Y}_{n+1}}\tilde{Y}_{n+1}$ are orthogonal projectors,
	which project, respectively, onto the span of the rows of $U_n$ and onto the span of the components of $\tilde{Y}_{n+1}$. 
	Then, the operator $P_{{U}^{\top}_n\tilde{Y}_{n+1}}: L^{2}(\Omega,\mathbb{R}^{d}) \to L^{2}(\Omega,\mathbb{R}^{d})$, defined as in \eqref{eq: proj X} 
	\begin{equation}\label{eq: proj interm}
		P_{{U}^{\top}_n\tilde{Y}_{n+1}}[\ \cdot \ ]= \left(I_{d \times d} - P_{{U}_n} \right) P_{\tilde{Y}_{n+1}}[\ \cdot \ ] + P_{{U}_n}[ \ \cdot \ ]
	\end{equation}
	is also an orthogonal projector and \eqref{eq: stoc proj} can be interpreted as a projected method:
\begin{equation}
\begin{aligned}
		X_{n+1} =& {U}_{n+1}^{\top}{Y}_{n+1}=\tilde{U}_{n+1}^{\top}\tilde{Y}_{n+1} \\
		= & \left( U_n + C^{-1}_{\tilde{Y}_{n+1}}\mathbb{E}\left[\tilde{Y}_{n+1} (a(t_n,U_{n}^{\top}Y_n)^{\top}\Delta t_n) \right]\left(I_{d \times d} - P_{U_n} \right) \right)^{\top} \cdot \\
		&  \left( Y_n + U_n a(t_n,U_{n}^{\top}Y_n) \Delta t_n+ b(t_n,U_{n}^{\top}Y_n) \Delta W_n\right)\\
		= & X_n + \left(I_{d \times d} - P_{U_n} \right) \mathbb{E}\left[a(t_n,U_{n}^{\top}Y_n)  \tilde{Y}_{n+1} ^{\top}\right]C^{-1}_{\tilde{Y}_{n+1}}\tilde{Y}_{n+1} \Delta t_n \\
		& + U^{\top}_n U_n a(t_n,U_{n}^{\top}Y_n)  \Delta t_n + U^{\top}_n U_n b(t_n,U_{n}^{\top}Y_n) \Delta W_n \\
		=& X_n + \underbrace{\left(\left(I_{d \times d} - P_{{U}_n} \right) P_{\tilde{Y}_{n+1}} + P_{{U}_n}\right)}_{P_{{U}^{\top}_n\tilde{Y}_{n+1}}} [ a_n] \Delta t_n+ P_{{U}_n}[b_n] \Delta W_n\\
	  = & X_n + P_{{U}^{\top}_n\tilde{Y}_{n+1}}[ a_n] \Delta t_n+ P_{{U}_n}[b_n] \Delta W_n,\label{eq: stoc proj 2}
\end{aligned} 
\end{equation}
where we have used the shorthand notation $a_n = a(t_n,U_n^{\top}Y_n)$, $b_n = b(t_n,U_n^{\top}Y_n)$.
This form mirrors the continuous DLRA formulation for SDEs (see equation \eqref{DLR-conditions} and \cite[Equation (2.17)]{kazashi2025dynamical}), where, the diffusion term $b$ is projected only onto the subspace spanned by the deterministic modes.
Algorithm \ref{alg: Stoch DLR Proj algorithm} summarizes the whole procedure.
\begin{algorithm}
	\caption{DLR Projector Splitting for SDEs}\label{alg: Stoch DLR Proj algorithm}
	\begin{flushleft}
		\textbf{Input}: initial data $U_0$, $Y_0$.
		
		\textbf{Output:} approximation $\{X_{n}=U_n^{\top}Y_n\}_{n=0,\ldots, N}$. 
	\end{flushleft}
	\begin{algorithmic}[1]
		
		\ForAll {$n \in \{0, \ldots N-1\}$} 
		
		\State Generate a Brownian increment $\Delta W_n \sim \mathcal{N}(0,\Delta t_n I_{m \times m})$.
		
		\State Compute $ \tilde{Y}_{n+1} = Y_n + U_n a(t_n,U_{n}^{\top}Y_n) \Delta t_n +  U_n b(t_n,U_{n}^{\top}Y_n)  \Delta W_n$
		
		\State Assemble $C_{\tilde{Y}_{n+1}} = \mathbb{E}[\tilde{Y}_{n+1}\tilde{Y}_{n+1}^{\top}]$
		\State Compute $\tilde{U}_{n+1}$: 
		$$C_{\tilde{Y}_{n+1}} \tilde{U}_{n+1} =  C_{\tilde{Y}_{n+1}} U_n + \mathbb{E}\left[\tilde{Y}_{n+1} \left(a(t_n,U_{n}^{\top}Y_{n})^{\top}\right) \right]\left(I_{d \times d} - P_{U_n} \right) \Delta t_n$$
		\State Reorthonormalize the deterministic modes: find  $({U}_{n+1}, {Y}_{n+1})$ such that:
		\begin{equation*}
			U_{n+1}^{\top} Y_{n+1} = \tilde{U}_{n+1}^{\top} \tilde{Y}_{n+1}, \quad {U}_{n+1}{U}_{n+1}^{\top} = I_{d\times d}.
		\end{equation*}
		with $(U_{n+1}^{\top}, R) =\texttt{QR}(\tilde{U}_{n+1}^{\top})$ and $Y_{n+1} = R \tilde{Y}_{n+1}$.
		\EndFor
	\end{algorithmic}
\end{algorithm}

This nature of projected scheme will turn out to be favorable for numerical properties, as we will see in Section \ref{sec: convergence PS SDE}. 
More precisely, the operator $P_{{U}^{\top}_n\tilde{Y}_{n+1}}$ corresponds to the orthogonal projector onto the tangent space of the manifold of rank-$k$ random vectors at the intermediate point ${U}^{\top}_n\tilde{Y}_{n+1}$, see e.g. \cite{falco2019dirac,musharbash2015error} for more details on the topic. Thanks to its projected structure, the error estimates do not depend on the smallest singular values of the Gramian, unlike the DLR Euler-Maruyama.

\subsection{DLR Projector Splitting for Euler-Maruyama (EM)}\label{sec: PS EM}
We now present another staggered method that shows nice numerical properties. This procedure mirrors the construction of the projector splitting scheme in \cite{kazashi2021stability} for rPDEs, where the stochastic modes are first updated at step $n$ projecting the equations on the deterministic modes $U_n$, and then the deterministic modes are updated projecting the \emph{whole} equations on the newly computed stochastic modes. 

This procedure leads to the same update as in Algorithms~\ref{alg: DLR EM SDE algorithm} and \ref{alg: Stoch DLR Proj algorithm} for the stochastic modes  
\begin{equation}\label{eq: Y_tilde KNV}
	\begin{aligned}
		\tilde{Y}_{n+1} &= Y_n + U_n a(t_n,U_{n}^{\top}Y_n) \Delta t_n +  U_n b(t_n,U_{n}^{\top}Y_n) \Delta W_n,
	\end{aligned}
\end{equation}
while providing an alternative update for the deterministic modes $\tilde{U}_{n+1}$:	
\begin{equation}
	C_{\tilde{Y}_{n+1}} \tilde{U}_{n+1} = C_{\tilde{Y}_{n+1}} U_n + \mathbb{E}\left[\tilde{Y}_{n+1} \left(a(t_n,U_{n}^{\top}Y_{n})^{\top}\Delta t_n+ (\Delta W_n)^{\top}b(t_n,U_{n}^{\top}Y_{n})^{\top}\right) \right]\left(I_{d \times d} - P_{U_n} \right)\label{eq: U_tilde KNV v2}.
\end{equation}
Finally, $(U_{n+1}, Y_{n+1})$ at the $n$-th step is obtained by QR orthonormalization. 
Notice the presence of the diffusion term inside the expectation in \eqref{eq: U_tilde KNV v2} which does not appear in Algorithm~\ref{alg: Stoch DLR Proj algorithm}, nor in Algorithm~\ref{alg: DLR EM SDE algorithm} where the projection is done with respect to the old modes $Y_n$, causing that term to vanish. 
The final projector-splitting scheme that we consider is summarized in Algorithm \ref{alg: Eva Proj Splitt SDE algorithm}.
\begin{algorithm}
	\caption{DLR Projector Splitting for Euler-Maruyama approximation for SDEs}\label{alg: Eva Proj Splitt SDE algorithm}
	\begin{flushleft}
		\textbf{Input}: initial data $U_0$, $Y_0$.
		
		\textbf{Output:} approximation $\{X_{n}=U_n^{\top}Y_n\}_{n=0,\ldots, N}$. 
	\end{flushleft}
	\begin{algorithmic}[1]
		
		\ForAll {$n \in \{0, \ldots N-1\}$}
		
		\State Generate a Brownian increment $\Delta W_n \sim \mathcal{N}(0,\Delta t_n I_{m \times m})$.
		
		\State Compute $ \tilde{Y}_{n+1} = Y_n + U_n a(t_n,U_{n}^{\top}Y_n) \Delta t_n +  U_n b(t_n,U_{n}^{\top}Y_n)  \Delta W_n$
		
		\State Assemble $C_{\tilde{Y}_{n+1}} = \mathbb{E}[\tilde{Y}_{n+1}\tilde{Y}_{n+1}^{\top}]$
		\State Compute $\tilde{U}_{n+1}$: 
		$$C_{\tilde{Y}_{n+1}} \tilde{U}_{n+1} =  C_{\tilde{Y}_{n+1}} U_n + \mathbb{E}\left[\tilde{Y}_{n+1} \left(a(t_n,U_{n}^{\top}Y_{n})^{\top}\Delta t_n+ (\Delta W_n)^{\top}b(t_n,U_{n}^{\top}Y_{n})^{\top}\right) \right]\left(I_{d \times d} - P_{U_n} \right) $$
		\State Reorthonormalize the deterministic modes: find  $({U}_{n+1}, {Y}_{n+1})$ such that:
		\begin{equation*}
			U_{n+1}^{\top} Y_{n+1} = \tilde{U}_{n+1}^{\top} \tilde{Y}_{n+1}, \quad {U}_{n+1}{U}_{n+1}^{\top} = I_{d\times d}.
		\end{equation*}
		with $(U_{n+1}^{\top}, R) =\texttt{QR}(\tilde{U}_{n+1}^{\top})$ and $Y_{n+1} = R \tilde{Y}_{n+1}$.
		\EndFor
	\end{algorithmic}
\end{algorithm}

If we multiply \eqref{eq: U_tilde KNV v2} by \eqref{eq: Y_tilde KNV}, we obtain 
\begin{equation}
	\begin{aligned}
		X_{n+1} =& U_{n+1}^{\top}Y_{n+1} = \tilde{U}_{n+1}^{\top}\tilde{Y}_{n+1}\\
		= & X_n + \left(I_{d \times d} - P_{U_n} \right) \mathbb{E}\left[( a_n\Delta t_n + b_n \Delta W_n )\tilde{Y}_{n+1}^{\top}\right]C^{-1}_{\tilde{Y}_{n+1}}\tilde{Y}_{n+1} +P_{{U}_n} a_n\Delta t_n + P_{{U}_n} b_n \Delta W_n \\
		= & X_n + \left(\left(I_{d \times d} - P_{{U}_n} \right) P_{\tilde{Y}_{n+1}} + P_{{U}_n}\right) [ a_n \Delta t_n+ b_n \Delta W_n] \\
		=& X_n + P_{{U}^{\top}_n\tilde{Y}_{n+1}}[ a_n \Delta t_n+ b_n \Delta W_n].\label{eq:projector-DLREM}
	\end{aligned}
\end{equation}  
Now, consider the forward Euler approximation of the original SDE
\begin{equation}\label{eq: EM}
	X_{n+1}^{\textrm{FE}} = X_{n}^{\textrm{FE}} + a(t_n,X_{n}^{\textrm{FE}})\Delta t_n +b(t_n,X_{n}^{\textrm{FE}})\Delta W_n.
\end{equation}
Then, the update of the projector splitting method in Algorithm \ref{alg: Eva Proj Splitt SDE algorithm} is given by projecting the Euler-Maruyama increment $a(t_n,X_{n})\Delta t_n +b(t_n,X_{n})\Delta W_n$ of the studied SDE, i.e.\ equation \eqref{eq: EM}, using the projector $P_{{U}^{\top}_n\tilde{Y}_{n+1}}$ and adding it to the initial point $X_n$, resembling the rPDE scheme in \cite{kazashi2021stability}.

Unlike in Algorithm \ref{alg: Stoch DLR Proj algorithm}, where $b$ is projected only onto the subspace spanned by the deterministic modes, here the diffusion is projected in all the complements of the tangent space. The relation between these two projections in highlighted by the next identity:
\begin{equation}\label{eq: DLR proj b diff}
	\begin{aligned}
	P_{{U}^{\top}_n\tilde{Y}_{n+1}} [b_n \Delta W_n]= & P_{{U}_n}^{\perp}\mathbb{E}[b_n \Delta W_n \tilde{Y}_{n+1}^{\top}]C_{\tilde{Y}_{n+1}}^{-1}\tilde{Y}_{n+1}  + P_{{U}_n}[b_n] \Delta W_n\\
	= & P_{{U}_n}^{\perp}\mathbb{E}[b_n \Delta W_n (Y_n + U_n a_n \Delta t_n +  U_n b_n  \Delta W_n)^{\top}]C_{\tilde{Y}_{n+1}}^{-1}\tilde{Y}_{n+1}  + P_{{U}_n}[b_n] \Delta W_n\\
	= & P_{{U}_n}^{\perp}\mathbb{E}[b_n \Delta W_n (U_n b_n  \Delta W_n)^{\top}]C_{\tilde{Y}_{n+1}}^{-1}\tilde{Y}_{n+1}  + P_{{U}_n}[b_n] \Delta W_n\\
	= & P_{{U}_n}^{\perp}\mathbb{E}[b_n b_n^{\top}]U_n^{\top}C_{\tilde{Y}_{n+1}}^{-1}\tilde{Y}_{n+1} \Delta t + P_{{U}_n}[b_n] \Delta W_n,
	\end{aligned}
\end{equation}
where in the last lines we used the independence of the Brownian increments. The appearance of the term $P_{{U}_n}^{\perp}\mathbb{E}[b_n b_n^{\top} U_n]C_{\tilde{Y}_{n+1}}^{-1}\tilde{Y}_{n+1} \Delta t$ in \eqref{eq: DLR proj b diff} may suggest that Algorithm \ref{alg: Eva Proj Splitt SDE algorithm} is a good candidate approximation of the DLRA-type equations proposed in \cite{cao2018stochastic}. Therein the dynamically-orthogonal-type relations were derived through a measure-based strategy and the difference with respect to \eqref{DLR-conditions} lies in the addition of a diffusion-based term in the definition of the deterministic basis, i.e.
\begin{equation}\label{eq: DO cao-lu}
	\begin{aligned}
	\dot{U}(t) =& C_{Y(t)}^{-1}\mathbb{E}\left[Y(t) a(t,U(t)^{\top}Y(t))^{\top}\right]\left(I_{d \times d} - P^{\mathrm{row} }_{U(t)} \right) \\
	&+  C_{Y(t)}^{-1}U_t \mathbb{E}\left[b( t,U(t)^{\top}Y(t)) b(t,U(t)^{\top}Y(t))^{\top}\right]\left(I_{d \times d} - P^{\mathrm{row} }_{U(t)} \right)\\
	Y(t)  =&Y_0 + \int_{0}^{t}U(s) a(s,U(s)^{\top}Y(s)) \mathrm{d}s + \int_{0}^{t}  U(s) b(s,U(s)^{\top}Y(s))\mathrm{d}W_s.
	\end{aligned}
\end{equation}
Well-posedness of the system \eqref{eq: DO cao-lu} can be derived similarly to \cite{kazashi2025dynamical}.

Algorithm \ref{alg: Eva Proj Splitt SDE algorithm} shows nice numerical stability properties and we will provide a convergence result to the DLRA solution defined by \eqref{eq: DO cao-lu} (however, with constants that depend on the smallest singular value of the Gramian of the stochastic basis).

\begin{Remark}[Independence of the choice of solution for Projector Splitting schemes]
	As previously pointed out, the equation for the deterministic basis $\tilde{U}_{n+1}$ in equations \eqref{eq: DLR EM 2v}, \eqref{eq: stoc proj 2}, and \eqref{eq: U_tilde KNV v2} does always have a solution, even when the Gramian matrix is singular. In this case, the solution is not unique and we may choose the one of minimal norm. When considering the two Projector Splitting algorithms, this choice does not affect the ambient solution $X_n$. Indeed, in \eqref{eq: stoc proj 2} and \eqref{eq: U_tilde KNV v2}, any possible deterministic basis solution is of the type
	$$\widehat{\tilde{U}}_{n+1} = \tilde{U}_{n+1} + N,$$
	where $\tilde{U}_{n+1}$ is the minimal norm solution obtained, for the DLR PS SDE, through the equation
	$$\tilde{U}_{n+1} =  U_n + C_{\tilde{Y}_{n+1}}^{\dag}\mathbb{E}\left[\tilde{Y}_{n+1} \left(a(t_n,U_{n}^{\top}Y_{n})^{\top}\right) \right]\left(I_{d \times d} - P_{U_n} \right) \Delta t_n,$$
	where $ C_{\tilde{Y}_{n+1}}^{\dag}$ is the pseudoinverse of $C_{\tilde{Y}_{n+1}}$, and
	$N \in \mathbb{R}^{k \times d}$ is a matrix whose range belongs to the null space of $C_{\tilde{Y}_{n+1}}$. Since, $C_{\tilde{Y}_{n+1}}$ is symmetric, then one has $$\widehat{\tilde{U}}_{n+1}^{\top}\tilde{Y}_{n+1} = \tilde{U}_{n+1}^{\top}\tilde{Y}_{n+1} + N^{\top} \tilde{Y}_{n+1} = \tilde{U}_{n+1}^{\top}\tilde{Y}_{n+1} = X_{n+1} ,$$
	which implies that $X_{n+1}$ is independent of the choice of $N$. This is not the case for Algorithm \ref{alg: DLR EM SDE algorithm}, as any deterministic basis is of the type $\widehat{\tilde{U}}_{n+1} = \tilde{U}_{n+1} + N$, with $N \in \mathbb{R}^{k \times d}$ in the null space of $C_{{Y}_{n}}$, and,  
	$$\widehat{\tilde{U}}_{n+1}^{\top}\tilde{Y}_{n+1} = \tilde{U}_{n+1}^{\top}\tilde{Y}_{n+1} + N^{\top} \tilde{Y}_{n+1} \neq \tilde{U}_{n+1}^{\top}\tilde{Y}_{n+1}.$$
\end{Remark}
\section{Convergence of the DLR Euler-Maruyama}\label{sec: DLR EM}
We now derive error estimates. 
We begin with the DLR Euler-Maruyama scheme described in Algorithm \ref{alg: DLR EM SDE algorithm} and defer the analysis of the DLR Projector Splitting for SDEs (Algorithm \ref{alg: Stoch DLR Proj algorithm}) and of the DLR Projector Splitting for EM scheme (Algorithm \ref{alg: Eva Proj Splitt SDE algorithm}), to Section \ref{sec: convergence PS SDE} and \ref{sec: convergence PS EM}, respectively. 
The main results of this section are Theorems~\ref{thm: convergence of DLR Euler-Maruyama} and \ref{thm: Approx real sol}.

Our convergence analysis relies on the formulation \eqref{Euler--Maruyama eq} for the update of the deterministic modes, for which we need to ensure that the Gramian of the stochastic modes $C_{Y_n}$ is invertible at each time step. For this purpose, we first assume invertibility of the Gramian associated to the initial condition.
\begin{Assumption}[Invertibility of $C_{Y_0}$]\label{ass: C_Y_0}
	Given $Y_0 \in L^2(\Omega, \mathbb{R}^k)$, we have that $\sigma_{0} := \min\limits_{i=1,\dots,k} \sigma^i(C_{Y_0})>0$, where $C_{Y_0} = \mathbb{E}[Y_0 Y_0^{\top}]$.
	\end{Assumption} 
Assumption \eqref{ass: C_Y_0} will be considered to hold from now on. In this case, we trivially have $C_{Y_0}= \mathbb{E}[Y_0Y_0^{\top}] \succeq \sigma_{0} \cdot I_{k \times k} \succ 0$.

A first ingredient we need for convergence analysis is given by Proposition \ref{prop: 1 discrete covariance full-rank - EM}, which shows invertibility of $C_{Y_n}$ for all $n$ without any condition on the time step, provided the noise satisfies uniform ellipticity (see Assumption \ref{ass: diff} below). Such an assumption is often considered in literature and, for instance, it guarantees the existence of a transition density of the associated SDE \cite{pavliotis2014stochastic,schilling2021brownian}.

Another ingredient in proving convergence is a bound on the second moment of the DLR Euler–Maruyama solution.
We can establish such a bound, even though the algorithm includes a QR decomposition at the end of each time step $n$, provided that a certain condition on $\Delta t_n = t_{n+1}-t_n$ is satisfied; see Lemma \ref{lem: sup n yn}.  
This second-moment bound also yields a lower bound on $C_{Y_n}$, as shown in Proposition \ref{prop: 2 discrete covariance full-rank - EM}. This bound is identical in form to the one holding in the continuous case for uniformly elliptic noise, as shown in \cite[Proposition 4.5]{kazashi2025dynamical}.

\begin{Assumption}[Non-degenerate diffusion]\label{ass: diff}
	There exists a positive constant $\sigma_{B}$ such that $$b(t, x)b(t,x)^{\top}\succeq\sigma_{B} \cdot I_{d \times d} \succ 0,$$
	for all $t \in \mathbb{R}$ and for all $ x \in \mathbb{R}^d$. 
\end{Assumption}

\begin{Proposition} \label{prop: 1 discrete covariance full-rank - EM}
	If Assumptions \ref{ass: C_Y_0} and \ref{ass: diff} hold, then for any sequence $\{\Delta t_n\}_n$ with $\Delta t_n = t_{n+1}-t_n$, the Gramian of the DLR-EM solution $\{U_n, Y_n\}_n$ of Algorithm \ref{alg: DLR EM SDE algorithm} satisfies
	\begin{equation}\label{eq: cov 1 EM}
		C_{Y_{n+1}} \succeq \sigma_{B} \Delta t_n\cdot I_{k \times k}.
	\end{equation}
	\begin{proof}	
		We prove the proposition by induction on $n$.
		 For the sake of notation, hereafter we write $a_n :=a(t_n, U_n^{\top}Y_n)$ and $b_n := b(t_n, U_n^{\top}Y_n)$. 
		
		By hypothesis, $C_{Y_{0}}$ has full-rank and therefore the solution $(U_1, Y_1)$ is well (and uniquely) defined through \eqref{Euler--Maruyama eq}. 
		The calculations below show that $C_{Y_{1}}$ has full rank. We assume then that $C_{Y_{n}}$ is full-rank and $\{U_{\ell}, Y_{\ell}\}_{\ell=1}^n$ exist up to $n$, so that $\{U_{n+1}, Y_{n+1}\}$ is well-defined through \eqref{Euler--Maruyama eq}, and we derive a lower bound on the smallest singular eigenvalue of $C_{Y_{n+1}}$.
		Using \eqref{Euler--Maruyama eq}, we get
		\begin{equation}\label{eq:intermed covar}
			\begin{aligned}
				C_{Y_{n+1}}=\mathbb{E}[Y_{n+1}Y_{n+1}^{\top}]  = & R \mathbb{E}[\tilde{Y}_{n+1}\tilde{Y}_{n+1}^{\top}] R^{\top} \\
				= & R \mathbb{E}[Y_{n}Y_{n}^{\top}]R^{\top} + R\mathbb{E}[Y_{n}(U_n a_n\Delta t_n)^{\top}]R^{\top} +R \mathbb{E}[Y_{n}(U_n b_n\Delta W_n)^{\top}]R^{\top} \\
				& + R\mathbb{E}[(U_n a_n\Delta t_n)Y_{n}^{\top}] R^{\top} + R\mathbb{E}[(U_n b_n\Delta W_n)Y_{n}^{\top}]R^{\top} \\
				&+ R\mathbb{E}[(U_n a_n\Delta t_n)(U_n a_n\Delta t_n)^{\top}]R^{\top}  
				+ R\mathbb{E}[(U_n a_n\Delta t_n)(U_n b_n\Delta W_n)^{\top}]R^{\top} \\
				& +
				R\mathbb{E}[(U_n b_n\Delta W_n)(U_n a_n\Delta t_n)^{\top}]R^{\top}
				 + R\mathbb{E}[(U_n b_n\Delta W_n)(U_n b_n\Delta W_n)^{\top}]R^{\top} \\
				=& R\mathbb{E}[Y_{n}Y_{n}^{\top}]R^{\top} + R\mathbb{E}[Y_{n}(U_n a_n\Delta t_n)^{\top}] R^{\top} + R\mathbb{E}[(U_n a_n\Delta t_n)Y_{n}^{\top}]R^{\top} \\
				& + R\mathbb{E}[(U_n a_n\Delta t_n)(U_n a_n\Delta t_n)^{\top}]R^{\top} + R\mathbb{E}[(U_n b_n\Delta W_n)(U_n b_n\Delta W_n)^{\top}]R^{\top},
			\end{aligned}
		\end{equation}
	where $R$ is the square matrix obtained by the QR decomposition, denoted by  \texttt{QR} in Algorithm~\ref{alg: DLR EM SDE algorithm}.
	To derive the last equality, we have exploited the fact that all the terms linear in $\Delta W_n$ vanish since $\Delta W_n$ is independent of $\{U_{\ell},Y_{\ell}\}_{\ell=0}^n$ and has zero-mean.  
  
		We can now find an estimate on the smallest eigenvalue of $\mathbb{E}[Y_{n+1}Y_{n+1}^{\top}]$ through the Rayleigh-quotient. 
		Let us take a unit vector $v \in \mathbb{R}^k$
		and consider the following:
		\begin{equation}\label{rayleigh quo discr}
			\begin{aligned}
				v^{\top}\mathbb{E}[Y_{n+1}Y_{n+1}^{\top}]v  = & v^{\top}R\mathbb{E}[Y_{n}Y_{n}^{\top}]R^{\top}v + 2v^{\top}R\mathbb{E}[Y_{n}(U_n a_n\Delta t_n)^{\top}]R^{\top}v \\
				& + v^{\top}R\mathbb{E}[(U_n a_n\Delta t_n)(U_n a_n\Delta t_n)^{\top}]R^{\top}v + v^{\top}R\mathbb{E}[(U_n b_n\Delta W_n)(U_n b_n\Delta W_n)^{\top}]R^{\top}v.
			\end{aligned}
		\end{equation}
		Now the goal is to take the infimum over $v$ on both sides of \eqref{rayleigh quo discr} and to lower bound the right-hand side by a positive quantity. This will imply that the Rayleigh quotient $\inf\limits_{\|v\|=1, \ v \in \mathbb{R}^k}v^{\top}\mathbb{E}[Y_{n+1}Y_{n+1}^{\top}]v$ is always greater than a positive constant providing a lower bound on the smallest eigenvalue of $C_{Y_n}$. 
	    To achieve this goal,
		 it is necessary to also show that $R$ is of full rank, so that $R^{\top}v$ is never the null vector. For this we show that $\tilde{U}_{n+1}$ is full rank. 
		
		Indeed, given the QR decomposition $\tilde{U}_{n+1}^{\top} = QR$, where $Q$ has orthonormal columns, one has
		\begin{equation}\label{eq: norm and rank k}
			\| \tilde{U}_{n+1} \|_{\mathrm{F}} = 	\| \tilde{U}_{n+1}^{\top}\|_{\mathrm{F}} =	\| Q R \|_{\mathrm{F}} = 	\|R\|_{\mathrm{F}}
		\end{equation}
		and $\mathrm{rank}(\tilde{U}_{n+1}\tilde{U}^{\top}_{n+1}) \leq \mathrm{rank}(\tilde{U}_{n+1}) = \mathrm{rank}(Q R) = \mathrm{rank}(R)$. Let us compute $\tilde{U}_{n+1}\tilde{U}^{\top}_{n+1}$:
		\begin{equation}\label{eq: tilde U tilde U.T}
			\begin{aligned}
				\tilde{U}_{n+1}\tilde{U}^{\top}_{n+1}=&  U_nU_n^{\top} + U_n \left( C^{-1}_{Y_n}\mathbb{E}\left[Y_n a_n^{\top}\right]\left(I_{d \times d} - P_{U_n} \right) \Delta t_n\right)^{\top}  + \left( C^{-1}_{Y_n}\mathbb{E}\left[Y_n a_n^{\top}\right]\left(I_{d \times d} - P_{U_n} \right) \Delta t_n\right) U_n^{\top} \\
				& + \left(C^{-1}_{Y_n}\mathbb{E}\left[Y_n a_n^{\top}\right]\left(I_{d \times d} - P_{U_n} \right) \right) \left(C^{-1}_{Y_n}\mathbb{E}\left[Y_n a_n^{\top}\right]\left(I_{d \times d} - P_{U_n} \right) \right)^{\top}  (\Delta t_n)^2 \\
				= & I_{k \times k} +  \underbrace{\left(C^{-1}_{Y_n}\mathbb{E}\left[Y_n a_n^{\top}\right]\left(I_{d \times d} - P_{U_n} \right) \right)}_{C}  \cdot \left(C^{-1}_{Y_n}\mathbb{E}\left[Y_n a_n^{\top}\right]\left(I_{d \times d} - P_{U_n} \right) \right)^{\top} (\Delta t_n)^2\\
				= & I_{k \times k} +CC^{\top}(\Delta t_n)^2.
			\end{aligned}
		\end{equation}
		As $I_{k \times k}$ is symmetric positive definite with all eigenvalues equal to $1$, $(\Delta t_n)^2$ is a positive scalar, and $CC^{T}$ is symmetric positive semi-definite matrix, via Weyl's inequality \cite{horn2012matrix} all the eigenvalues of $\tilde{U}_{n+1}\tilde{U}^{\top}_{n+1}$ are greater than or equal to $1$.
		 This implies that $\tilde{U}_{n+1}\tilde{U}^{\top}_{n+1}$ is full rank and, therefore, the same holds for $\tilde{U}_{n+1}$ and $R$, which have rank equal to $k$.
		Moreover, one has that
		\begin{equation}\label{eq: R rel}
			\begin{aligned}
				R^{T}R & = R^{T}Q^{T}QR  = \tilde{U}_{n+1}\tilde{U}^{\top}_{n+1} 
				= I_{k \times k} + C^{T}C (\Delta t_n)^2,
			\end{aligned}
		\end{equation}
		hence via Weyl's inequality all eigenvalues of $R^{T}R$, and also the ones of $RR^{T}$, are greater than or equal to $1$.
		
		We now provide bounds for the terms on the right-hand side of \eqref{eq:intermed covar}.			
			For the last term in \eqref{eq:intermed covar}, since $\mathbb{E}[\Delta W_n \Delta W_n^{\top}] = \Delta t_n I_{m \times m}$ we have
			\begin{equation}\label{eq: bn}
				\begin{aligned}
					\mathbb{E}[(U_n b_n\Delta W_n)(U_n b_n\Delta W_n)^{\top}]  = & U_n \mathbb{E}[b_n\Delta W_n \Delta W_n^{\top} b_n^{\top}] U_n^{\top}  \\
					= & U_n \mathbb{E}[b_n  \mathbb{E}[ \Delta W_n \Delta W_n^{\top} \mid \mathcal{F}_{t_n}]· b_n^{\top}] U_n^{\top} \\
				 = & U_n \mathbb{E}[b_n b_n^{\top}]  U_n^{\top} \Delta t_n \succeq \sigma_{B} \cdot I_{k \times k},
				\end{aligned}	
			\end{equation}     
			by the orthogonality of $U_n$ and Assumption \ref{ass: diff}.
		
		Let us analyze the term $v^{\top}R\mathbb{E}[Y_{n}(U_n a_n\Delta t_n)^{\top}]R^{\top}v = \mathbb{E}[v^{\top}RY_{n} a_n ^{\top}U_n^{\top}R^{\top}v] \Delta t_n.$
		One has
		\begin{equation*}
			\begin{aligned}
				|\mathbb{E}[v^{\top}RY_{n} a_n ^{\top}U_n^{\top}R^{\top}v] | & \leq \mathbb{E}[|v^{\top}RY_{n} a_n ^{\top}U_n^{\top}R^{\top}v|]   \\
				& = \mathbb{E}[|v^{\top}RY_{n}| | a_n ^{\top}U_n^{\top}R^{\top}v|] \leq \frac{1}{2\varepsilon}\mathbb{E}[|v^{\top}RY_{n}|^{2}] + \frac{\varepsilon}{2} \mathbb{E}[|a_n ^{\top}U_n^{\top}R^{\top}v|^2] \\
				& = \frac{1}{2\varepsilon}  v^{\top}R\mathbb{E}[Y_{n}Y_{n}^{\top}]R^{\top}v + \frac{\varepsilon}{2}  v^{\top}R\mathbb{E}[U_n a_n a_n ^{\top}U_n^{\top}]R^{\top}v.  \\
			\end{aligned}
		\end{equation*}		
		
		Putting everything together, we get
		\begin{equation}\label{eq: intermediate rayleigh}
			\begin{aligned}
				v^{\top}\mathbb{E}[Y_{n+1}Y_{n+1}^{\top}]v \geq & v^{\top}R\mathbb{E}[Y_{n}Y_{n}^{\top}]R^{\top}v - \frac{1}{\varepsilon} v^{\top}R\mathbb{E}[Y_{n}Y_{n}^{\top}]R^{\top}v \Delta t_n- \varepsilon v^{\top}R\mathbb{E}[U_n a_n a_n ^{\top}U_n^{\top}]R^{\top}v \Delta t_n\\
				& +v^{\top}R\mathbb{E}[U_n a_n a_n ^{\top}U_n^{\top}]R^{\top}v (\Delta t_n)^2 + v^{\top}R\mathbb{E}[U_n b_nb_n^{\top}U_n^{\top}]R^{\top}v \Delta t_n.\\
			\end{aligned}
		\end{equation}
		Dividing everything by $|R^{\top}v|^2$ and using \eqref{eq: bn}, one finds that 
		\begin{equation}\label{eq: intermediate rayleigh 2}
			\begin{aligned}
				v^{\top}\frac{\mathbb{E}[Y_{n+1}Y_{n+1}^{\top}]}{|R^{\top}v|^2} v
				\geq & \left(1 - \frac{\Delta t_n}{\varepsilon}\right) v^{\top}R\frac{\mathbb{E}[Y_{n}Y_{n}^{\top}]}{|R^{\top}v|^2}R^{\top}v  + v^{\top}R\frac{\mathbb{E}[U_n a_n a_n ^{\top}U_n^{\top}]}{|R^{\top}v|^2}R^{\top}v  \left( (\Delta t_n)^2 - \varepsilon\Delta t_n\right) \\
				&+ \sigma_{B} \Delta t_n.
			\end{aligned}
		\end{equation}
		Taking $\varepsilon = \Delta t_n$, we have
		\begin{equation*}	
			\begin{aligned}
				v^{\top}\frac{\mathbb{E}[Y_{n+1}Y_{n+1}^{\top}]}{|R^{T}v|^2} v & \geq  \sigma_{B} \Delta t_n.
			\end{aligned}
		\end{equation*}
		Therefore,
		\begin{equation}	\label{eq: ray rhs}
			\begin{aligned}
				v^{\top}\mathbb{E}[Y_{n+1}Y_{n+1}^{\top}] v & \geq  \sigma_{B} |R^{T}v|^2 \Delta t_n, \\
				& \geq  \sigma_{B} ( v^{\top}RR^{T}v) \Delta t_n, \\
				& \geq \sigma_{B}  \Delta t_n,
			\end{aligned}
		\end{equation}
		where on the last inequality we have used that all the eigenvalues of $RR^{\top}$ are greater or equal to $1$ and $v$ is unitary. Now the right-hand side is independent on $v$ and, hence, passing to the infimum on the left-hand side we find
		\begin{equation}	
			\begin{aligned}
				\sigma_{n+1}^k \geq \sigma_{B} \Delta t_n,
			\end{aligned}
		\end{equation}
	where $\sigma_{n+1}^k$ is the smallest eigenvalue of $C_{Y_{n+1}} = \mathbb{E}[Y_{n+1}Y_{n+1}^{\top}]$.
	
	As these computations are independent of $n$, by induction $\sigma_{n}^k>0$ and $C_{Y_n}$ is full rank for all $n$.
	\end{proof}
\end{Proposition}

We now provide a uniform bound in $\Delta t$ on the $L^2$-sup-in-time norm of the DLR approximate solution under a restrictive condition on the time step, involving the smallest singular value.  This result will be crucial for the proofs of Proposition \ref{prop: 2 discrete covariance full-rank - EM} and Theorem \ref{thm: convergence of DLR Euler-Maruyama}.

\begin{Lemma}[Boundedness of the stochastic modes]\label{lem: sup n yn}
	Let us consider a DLR Euler-Maruyama discretizazion with a uniform time step $\Delta t:= T/N$, with $N \in \mathbb{N}$. Suppose that 
	\begin{equation}\label{eq: dt sup n cond}
		\Delta t \leq \frac{ \sqrt{ \sigma_{n}^k} }{ \sqrt{ C_{\mathrm{lgb}}} \sqrt{ 5(1+\mathbb{E}\left[|X_0|^2 \right]) \exp\left(  C_{\mathrm{lgb}} (2T^2+8T) \right)} },
	\end{equation}
	and Assumptions \ref{eq:initial value}-\ref{eq:lin-growth} hold,
	then the following $L^2$-sup-in-time bound holds for all $n =\{1,\dots, N\}$
	\begin{equation}\label{eq: sup n Y_n}
		\begin{aligned}
			\mathbb{E}[\max\limits_{0 \leq k \leq n}|\tilde{Y}_{k}|^{2}] \leq&	\mathbb{E}[\max\limits_{0 \leq k \leq n}|Y_{k}|^{2}]\\
			=& \mathbb{E}[\max\limits_{0 \leq k \leq n}|X_{k}|^{2}] \leq 5(1+\mathbb{E}\left[|X_0|^2 \right]) \exp\left(  C_{\mathrm{lgb}} (2T^2+8T) \right)=:K_1(T).
		\end{aligned}
	\end{equation}
	\begin{proof} The proof can be found in Appendix \ref{app:DLR EM}.
	\end{proof}
\end{Lemma}

\begin{Remark}
If the diffusion coefficient is non-degenerate, then the stability bound \eqref{eq: sup n Y_n} holds under a time-step condition that can be made independent of $n$. Indeed, under Assumption \ref{ass: diff}, 
	Proposition~\ref{prop: 1 discrete covariance full-rank - EM} implies that condition \eqref{eq: dt sup n cond} in Lemma~\ref{lem: sup n yn} can be replaced by
\begin{equation}\label{eq: delta t stab bis}	
	\Delta t_n \leq  \frac{ \sigma_B}{ C_{\mathrm{lgb}} 5(1+\mathbb{E}[|X_{0}|^2] ) \exp\left(  C_{\mathrm{lgb}} (2T^2+8T) \right)}, \quad \text{ for all } n \in  \{1,\dots,N\}.
\end{equation}
\end{Remark}

The norm bound of Lemma \ref{lem: sup n yn} on the discretized DLR solution $X_n$ helps us obtaining a better estimate on the lower bound of the Gramian $C_{Y_n}= \mathbb{E}[Y_nY_n^{\top}]$ than the one given in Proposition \ref{prop: 1 discrete covariance full-rank - EM}.

\begin{Proposition} \label{prop: 2 discrete covariance full-rank - EM}
	Suppose that Assumptions \ref{linear-growth-bound}, \ref{ass: C_Y_0}, and \ref{ass: diff} hold, and consider a DLR EM discretizazion with a uniform time step $\Delta t:= T/N$, with $N \in \mathbb{N}$. Moreover, assume that $\Delta t$ satisfies \eqref{eq: dt sup n cond} for any $n=0,\dots, N$.
 Then, one has
	 \begin{equation}\label{eq: cov 2 EM}
	 	\begin{aligned}
	 			C_{Y_{n+1}}
	 			\succeq
	 			 & \min \{\sigma_{0}, \frac{\sigma_{B}^2}{4C_{\mathrm{lgb}}(1 + K_1(T))}\},
	 	\end{aligned}
	 \end{equation}
where $K_1(T)$ is defined in \eqref{eq: sup n Y_n}.
	\begin{proof}
We proceed similarly to the proof of Proposition \ref{prop: 1 discrete covariance full-rank - EM}. The result follows by using relation \eqref{eq: intermediate rayleigh 2}, which we recall here
\begin{equation*}
		v^{\top}\frac{\mathbb{E}[Y_{n+1}Y_{n+1}^{\top}]}{|R^{\top}v|^2} v 
		 \geq \left(1 - \frac{\Delta t}{\varepsilon}\right) v^{\top}R\frac{\mathbb{E}[Y_{n}Y_{n}^{\top}]}{|R^{\top}v|^2}R^{\top}v  + v^{\top}R\frac{\mathbb{E}[U_n a_n a_n ^{\top}U_n^{\top}]}{|R^{\top}v|^2}R^{\top}v  \left( (\Delta t)^2 - \varepsilon\Delta t \right) + \sigma_{B} \Delta t \\
\end{equation*}
and the norm bound obtained in Lemma \ref{lem: sup n yn}. Since $R$ is obtained by the reduced \texttt{QR} applied to $\tilde{U}_{n+1}$, we have
\begin{equation}
	\begin{aligned}
		 v^{\top}R\mathbb{E}[U_n a_n a_n ^{\top}U_n^{\top}]R^{\top}v 
		\geq & - |v^{\top}R\mathbb{E}[U_n a_n a_n ^{\top}U_n^{\top}]R^{\top}v |\\
	\geq & - |R^{\top}v|^2 \mathbb{E}[|U_n a_n |^2]\\
		\geq & - |R^{\top}v|^2 \mathbb{E}[|a_n |^2]\\
			\geq & - |R^{\top}v|^2 C_{\mathrm{lgb}}(1+K_1(T)).\\
	\end{aligned}
\end{equation}
Therefore,
	\begin{equation}
			\begin{aligned}
				v^{\top}\frac{\mathbb{E}[Y_{n+1}Y_{n+1}^{\top}]}{|R^{\top}v|^2} v 
			& \geq \left(1 - \frac{\Delta t}{\varepsilon}\right) v^{\top}R\frac{\mathbb{E}[Y_{n}Y_{n}^{\top}]}{|R^{\top}v|^2}R^{\top}v  - C_{\mathrm{lgb}}(1 + K_1(T))  \left| (\Delta t)^2 - \varepsilon\Delta t \right| + \sigma_{B} \Delta t .
			\end{aligned}
		\end{equation}
Choosing
$$\varepsilon = \Delta t + \frac{\sigma_B}{2C_{\mathrm{lgb}}(1 + K_1(T))},$$
one obtains
\begin{equation*}
	\begin{aligned}
		v^{\top}\frac{\mathbb{E}[Y_{n+1}Y_{n+1}^{\top}]}{|R^{\top}v|^2} v 
		 \geq & \left(1 - \frac{\Delta t}{\Delta t + \frac{\sigma_B}{2C_{\mathrm{lgb}}(1 + K_1(T))}}\right) v^{\top}R\frac{\mathbb{E}[Y_{n}Y_{n}^{\top}]}{|R^{\top}v|^2}R^{\top}v \\
		& - C_{\mathrm{lgb}}(1 + K_1(T))  \left| (\Delta t)^2 - (\Delta t + \frac{\sigma_B}{2C_{\mathrm{lgb}}(1 + K_1(T))})\Delta t \right| + \sigma_{B} \Delta t \\
		= & \left(1 - \frac{\Delta t}{\Delta t + \frac{\sigma_B}{2C_{\mathrm{lgb}}(1 + K_1(T))}}\right) \sigma_{n}^k + \frac{\sigma_{B}}{2}\Delta t.
	\end{aligned}
\end{equation*}
By taking the infimum over $v$ on 
	the left-hand side and setting $A:=\frac{\sigma_B}{2C_{\mathrm{lgb}}(1 + K_1(T))}$, we obtain
\begin{equation*}
	\begin{aligned}
			\sigma_{n+1}^k \geq & \left(1 - \frac{\Delta t}{\Delta t + A}\right) \sigma_{n}^k + \frac{\sigma_{B}}{2}\Delta t \\
	                  \geq & \left(1 - \frac{1}{1+ \frac{A}{\Delta t}}\right)^{n+1} \sigma_{0} + \frac{\sigma_{B}}{2}\Delta t \sum_{k=0}^{n} \left(1 - \frac{1}{1+ \frac{A}{\Delta t}}\right)^{k}\\
	                 = & \left(1 - \frac{1}{1+ \frac{A}{\Delta t}}\right)^{n+1}\sigma_{0} + \frac{\sigma_{B}}{2}\Delta t \frac{1 - \left(1 - \frac{1}{1+ \frac{A}{\Delta t}}\right)^{n+1}}{1 -  \left(1 - \frac{1}{1+ \frac{A}{\Delta t}}\right)} \\
	                 = & \left(1 - \frac{1}{1+ \frac{A}{\Delta t}}\right)^{n+1}\sigma_{0} + \frac{\sigma_{B}}{2}(1+ \frac{A}{\Delta t})\Delta t \left[1 - \left(1 - \frac{1}{1+ \frac{A}{\Delta t}}\right)^{n+1}\right]\\
	                 = & \left(1 - \frac{1}{1+ \frac{A}{\Delta t}}\right)^{n+1}\sigma_{0} + \frac{\sigma_{B}}{2}A \left[1 - \left(1 - \frac{1}{1+ \frac{A}{\Delta t}}\right)^{n+1}\right]  + \frac{\sigma_{B}}{2}\Delta t \left[1 - \left(1 - \frac{1}{1+ \frac{A}{\Delta t}}\right)^{n+1}\right] \\
	                 \geq & \min \{\sigma_{0}, \frac{\sigma_{B}^2}{4C_{\mathrm{lgb}}(1 + K_1(T))}\} + \frac{\sigma_{B}}{2}\Delta t \left[1 - \left(1 - \frac{1}{1+ \frac{A}{\Delta t}}\right)^{n+1}\right].
    \end{aligned}
\end{equation*}
	\end{proof}
\end{Proposition}

\begin{Remark}
	If a uniform time step $\Delta t$ satisfies \eqref{eq: dt sup n cond}, then the bound \eqref{eq: cov 2 EM} turns out to be identical to the one obtained in \cite{kazashi2025dynamical} for the smallest eigenvalue of the Gramian $C_{Y(t)}$ for the continuous DLRA \eqref{DLR-conditions}. 
	Hence, for such $\Delta t$, the same upper bound
	\[|C^{-1}_{Y(s)}|,|C^{-1}_{Y_{n}}| \leq \frac{1}{\gamma}\text{ for all s }\in [0,T]\text{ and for all }n,\]
	for $\gamma>0$,
	holds for the inverse Gramians in both the continuous DLRA and DLR Euler--Maruyama scheme.
\end{Remark}

Unlike the standard Euler--Maruyama scheme, Algorithm \ref{alg: DLR EM SDE algorithm} requires a non-linear operation at every time step due to the rotation of the modes $U_n$ via \texttt{QR} decomposition. Despite this, we are able to establish in the next Theorem \ref{thm: convergence of DLR Euler-Maruyama} that the DLR Euler--Maruyama scheme converges at the usual strong rate. 

We are now ready to state and prove the first main convergence theorem of this section. To this purpose, it is useful defining the following notation: for $s\in [0,T]$ we define
\begin{equation}\label{eq: n_s}
	n_s := \max\{n = 0,1, \dots, N : t_n \leq s\}.
\end{equation}
\begin{Theorem}[Strong Convergence of DLR EM to the DO solution]\label{thm: convergence of DLR Euler-Maruyama}
	Let Assumptions~$1$-$3$ hold. Consider a uniform partition $\Delta := \left\{t_n = t_0 + n \Delta t, \ n=0,\dots, N\right\}$ of $[0,T]$ with time-step $\Delta t$ and let $\{X_n\}_{0 \leq n \leq N}$ be the DLR Euler-Maruyama solution built on this partition.
	
	Let us suppose that there exists a constant $\gamma>0$ such that $|C^{-1}_{Y(s)}|,|C^{-1}_{Y_{n_s}}| \leq \frac{1}{\gamma}$ for all $s \in [0,T]$ and $\Delta t$ satisfies \eqref{eq: dt sup n cond} for all $n=0, \dots, N-1$. 
	Moreover, 
	if the SDE  \eqref{eq:SDE-diff} is non-autonomous, we assume that
	the drift $a$ and the diffusion $b$ satisfy the following Hölder condition in $t$: there exist $\alpha>0$ and $C_{\mathrm{na}} >0$ such that
	\begin{equation}\label{polynomial growth in t}
	|a(t,x) - a(s,x) | +|b(t,x) - b(s,x)| \leq C_{\mathrm{na}} (1+|x|) |t-s|^{\alpha} \quad \text{ for all } \ t,s \in [0,T], \ x \in \mathbb{R}^d.
	\end{equation}	

	Then, the DLR Euler--Maruyama method with constant time step $\Delta t$ has strong order of convergence equal to ${\min\{ \frac{1}{2}, \alpha\} }$, i.e.\ there exists a constant $C:=C(\gamma, k,T)$ independent of $\Delta t$ such that:
	\begin{equation}\label{eq:conv DLR EM}
	  \sqrt{\mathbb{E}[\max\limits_{0 \leq n \leq N} |X_{n} - X(t_n)|^2]}  \leq C (\Delta t)^{\min\{ \frac{1}{2}, \alpha\} }.
	\end{equation}
	\begin{proof}
		
		To prove the statement, we proceed as follows. First, we bound the error between the DLRA solution and the DLR EM surrogate in terms of the errors concerning the stochastic and deterministic bases. We then analyze these two errors separately, expressing then as sums of local error contributions and jump terms between the DLR-EM bases before and after re-orthonormalization by QR. Finally, we combine the error estimates on the deterministic and stochastic bases and conclude via Gronwall's lemma.
		
		We start by finding a first bound on the error $\mathbb{E}[\max\limits_{0 \leq n \leq N} |X_{n} - X(t_n)|^2]$. Fix $n$, and suppose that we have just completed Step 6 of Algorithm \ref{alg: DLR EM SDE algorithm}; then we know
		\begin{equation*}
		X_{n} = U_{n}^{\top} Y_{n} = \tilde{U}_{n}^{\top} \tilde{Y}_{n},
		\end{equation*}
		where $\tilde{U}_{n},\tilde{Y}_{n}$ are the updates of the deterministic and stochastic modes before the \texttt{QR} decomposition in Step 6 of Algorithm \ref{alg: DLR EM SDE algorithm}. 
		They satisfy \eqref{Euler--Maruyama eq}, and $\tilde{U}_{n}$ is not necessarily orthogonal.
					
	We notice that the error $|X(t_{n})-X_{n}|^2$  can be bounded by the error on the deterministic and stochastic modes:
		\begin{equation*}
		\begin{aligned}
		|X(t_{n})-X_{n}|^2 & = |U(t_{n})^{\top}Y(t_{n})-U_{n}^{\top}Y_{n}|^2 = |U(t_{n})^{\top}Y(t_{n})-\tilde{U}_{n}^{\top}\tilde{Y}_{n}|^2\\
		& =    |U(t_{n})^{\top}(Y(t_{n})-\tilde{Y}_{n}) + (U(t_{n})^{\top}-\tilde{U}_{n}^{\top})\tilde{Y}_{n}|^2 \\
		   & \leq 2 |U(t_{n})|^2|Y(t_{n})-\tilde{Y}_{n}|^2 + 2|U(t_{n})-\tilde{U}_{n}|^2|\tilde{Y}_{n}|^2  \\
		   & \leq 2 |Y(t_{n})-\tilde{Y}_{n}|^2 + 2\|U(t_{n})-\tilde{U}_{n}\|_{\mathrm{F}}^2|\tilde{Y}_{n}|^2 .
		\end{aligned}
		\end{equation*}	
	
	  Taking the supremum over all the time-grid points leads to
	    \begin{equation*}
	    	\begin{aligned}
	    		\max\limits_{0 \leq n \leq N} |X(t_{n})-X_{n}|^2 
	    		\leq & 2 \max\limits_{0 \leq n \leq N}  |Y(t_{n})-\tilde{Y}_{n}|^2 +  2 \max\limits_{0 \leq n \leq N}\|U(t_{n})-\tilde{U}_{n}\|_{\mathrm{F}}^2 \max\limits_{0 \leq n \leq N} |\tilde{Y}_{n}|^2,\\
	    	\end{aligned}
	    \end{equation*}		
hence
        \begin{equation}\label{X-error estimate}
        	\begin{aligned}
        		\mathbb{E}[\max\limits_{0 \leq n \leq N} |X(t_{n})-X_{n}|^2 ]
        		& \leq 2 \mathbb{E}[ \max\limits_{0 \leq n \leq N}  |Y(t_{n})-\tilde{Y}_{n}|^2] +  2 \max\limits_{0 \leq n \leq N} \|U(t_{n})-\tilde{U}_{n}\|_{\mathrm{F}}^2 \mathbb{E}[\max\limits_{0 \leq n \leq N} |\tilde{Y}_{n}|^2], \\   
        			& \leq C_1^2 \left( \mathbb{E}[ \max\limits_{0 \leq n \leq N}  |Y(t_{n})-\tilde{Y}_{n}|^2] + \max\limits_{0 \leq n \leq N} \|U(t_{n})-\tilde{U}_{n}\|_{\mathrm{F}}^2\right),
        	\end{aligned}
        \end{equation}	
       where we have used Lemma \ref{lem: sup n yn}, and we define $C_1:= \sqrt{\max\{2,2K_1(T)\}}$.
		Therefore we just need to show convergence of the deterministic and stochastic components $\tilde{U}_n$ and $\tilde{Y}_n$ to prove convergence of $X_{n}$ to $X(t_{n})$. 
		
		We first look at the error on the stochastic basis. It is favorable to notice that we have the following recursive relation with respect to the increment of the stochastic basis before the QR decomposition:
		\begin{equation}\label{eq update Y_n+1}
		    \begin{aligned}
		    	\tilde{Y}_{n+1} =& Y_{n} + U_n a(t_n, U_n^{\top} Y_n)\Delta t
		    	+ U_n b(t_n, U_n^{\top} Y_n)\Delta W_n \\
		    	=& Y_{n} - \tilde{Y}_{n} + \tilde{Y}_{n} +  \int_{t_n}^{t_{n+1}} U_n a(t_n, U_n^{\top} Y_n) \mathrm{d} r
		    	+  \int_{t_n}^{t_{n+1}} U_n b(t_n, U_n^{\top} Y_n) \mathrm{d} W_r \\
		    	= & \sum_{\ell=1}^{n} \left( Y_{\ell} - \tilde{Y}_\ell \right)+    \sum_{\ell=0}^{n} \left( \int_{t_\ell}^{t_{\ell+1}} U_\ell a(t_\ell, U_\ell^{\top} Y_\ell) \mathrm{d} r
		    	+  \int_{t_\ell}^{t_{\ell+1}} U_\ell b(t_\ell, U_\ell^{\top} Y_\ell) \mathrm{d} W_r \right) + Y_0.
		    \end{aligned}
		\end{equation}
		From this relation, we obtain the following decomposition for the error of the stochastic basis
		\begin{equation}\label{eq error y in EM}
			\begin{aligned}
				\mathbb{E}\left[ \max_{0 \le n \le N} \left| \tilde{Y}_{n} - Y(t_n) \right|^2 \right]
				=&  \mathbb{E}\Bigg[ \max_{0 \le n \le N} \Bigg| \sum_{\ell=1}^{n-1} \left( Y_{\ell} - \tilde{Y}_\ell \right)+  \sum_{\ell=0}^{n-1} \Big(  \int_{t_\ell}^{t_{\ell+1}} U_\ell a(t_\ell, U_\ell^{\top} Y_\ell) \mathrm{d} r \\
				&+    \int_{t_\ell}^{t_{\ell+1}} U_\ell b(t_\ell, U_\ell^{\top} Y_\ell) \mathrm{d} W_r \Big) + Y_0 - Y(t_n) \Bigg|^2 \Bigg] \\
				=&  \mathbb{E}\Bigg[ \max_{0 \le n \le N} \Bigg|  \sum_{\ell=1}^{n-1} (Y_\ell - \tilde{Y}_\ell) \\
				&+ \Bigg( \sum_{\ell=0}^{n-1} \int_{t_\ell}^{t_{\ell+1}} \left( U_\ell a(t_\ell, U_\ell^{\top} Y_\ell)
				- U(r)a(r, U(r)^{\top} Y(r)) \right) \mathrm{d}r \\
				&+ \int_{t_\ell}^{t_{\ell+1}} \left( U_\ell b(t_\ell, U_\ell^{\top} Y_\ell)
				- U(r)b(r, U(r)^{\top} Y(r)) \right) \mathrm{d}W_r \Bigg)
				+ Y_0 - Y(0)\Bigg|^2 \Bigg] \\
				=& 3 \Bigg( \underbrace{ \mathbb{E}\left[ \max_{0 \le n \le N} \left| \left( \sum_{\ell=1}^{n-1} (Y_\ell - \tilde{Y}_\ell ) \right) \right|^2 \right]}_{\mathrm{T}_0}\\
				&+ \underbrace{\mathbb{E}\left[ \max_{0 \le n \le N} \left| \int_{0}^{t_{n}} \left( U_{n_r} a(t_{n_r}, U_{n_r}^{\top} Y_{n_r})
				- U(r)a(r, U(r)^{\top} Y(r)) \right) \mathrm{d}r  \right|^2 \right]}_{\mathrm{T}_1} \\
			&+ \underbrace{\mathbb{E}\left[ \max_{0 \le n \le N} \left| \int_{0}^{t_{n}} \left( U_{n_r} b(t_{n_r}, U_{n_r}^{\top} Y_{n_r})
				- U(r)b(r, U(r)^{\top} Y(r)) \right) \mathrm{d}W_r \right|^2 \right] }_{\mathrm{T}_2}\Bigg), \\
			\end{aligned}
		\end{equation}
	where we have assumed $Y_0 = Y(0)$.
		We proceed by bounding each term $\mathrm{T}_0$-$\mathrm{T}_2$ in a suitable way in order to obtain a Gronwall-type bound. 
		
		Let us start from finding a bound for $\mathrm{T}_0$. Notice that for all $n$ one has that $Y_n = R_n \tilde{Y}_n$ with $\tilde{U}_n = R_n^{\top} U_n$,
	  and hence it holds that
	  	\begin{equation*}
	  	\begin{aligned}
	  		| Y_n - \tilde{Y}_n | = | R_n \tilde{Y}_n - \tilde{Y}_n| \leq |R_n -  I_{k \times k}| |  \tilde{Y}_n|.
	  	\end{aligned}
	  \end{equation*}
  We now define 
  \begin{equation*}
  	U(s) = U_{n-1} + A P_{U_{n-1}}^{\perp}s, \quad A= C_{Y_{n-1}}^{-1}\mathbb{E}[Y_{n-1}a(t_{n-1}, U_{n-1}^{\top}Y_{n-1})^{\top}],
  \end{equation*}
so that $U(\Delta t)= \widetilde{U}_n$, and invoke Lemma \ref{lem: analyticity of U} to conclude that the QR decomposition $s \mapsto (Q(s),R(s))$ is analytic in $s$, hence from \eqref{eq: error in R} $R_n$ admits the following error bound 
		\begin{equation}\label{eq: R-I}
			\begin{aligned}
				|R_n -  I_{k \times k}|
				\leq & | R_n -  I_{k \times k}| \\
				\leq &\frac{(\Delta t)^2 }{2 } \sqrt{2(1+T^2|C_{Y_{n-1}}^{-1}\mathbb{E}\left[Y_{n-1} a(t_{n-1},U_{n-1}^{\top}Y_{n-1})^{\top}\right]|^2)}\\
				& \quad  \cdot  \bigl(\|C_{Y_{n-1}}^{-1}\mathbb{E}\left[Y_{n-1} a(t_{n-1},U_{n-1}^{\top}Y_{n-1})^{\top}\right]\|^2_{\mathrm{F}}+2T^2\|C_{Y_{n-1}}^{-1}\mathbb{E}\left[Y_{n-1} a(t_{n-1},U_{n-1}^{\top}Y_{n-1})^{\top}\right]\|_{\mathrm{F}}^4\bigr)  \\
				\leq &\frac{(\Delta t)^2 }{2 }  \sqrt{2(1+T^2 \frac{1}{\gamma} C_{\mathrm{lgb}}(1 + K_1(T)) )} \bigl( \frac{1}{\gamma} C_{\mathrm{lgb}}(1 + K_1(T)) )+2T^2 \frac{1}{\gamma}^2 C_{\mathrm{lgb}}^2(1 + K_1(T))^2\bigr) \\
				= & (\Delta t)^2 \tilde{C}(k, \gamma, T),
			\end{aligned}
		\end{equation}
	where we employ Cauchy-Schwarz inequality and linear-growth bound property, and $\tilde{C}(k, \gamma, T):=  \frac{1}{2}\sqrt{2(1+T^2 \frac{1}{\gamma} C_{\mathrm{lgb}}(1 + K_1(T)) )} \bigl( \frac{1}{\gamma} C_{\mathrm{lgb}}(1 + K_1(T)) )+2T^2 \frac{1}{\gamma}^2 C_{\mathrm{lgb}}^2(1 + K_1(T))^2\bigr)$.

	Now, via using \eqref{eq: R-I} we can bound $\mathrm{T}_0$ in the following fashion
		\begin{equation*}
			\begin{aligned}
			\mathrm{T}_0 = \mathbb{E}\left[  \max_{0 \le n \le N} \left| \sum_{\ell=1}^{n-1} (Y_\ell - \tilde{Y}_\ell) \right|^2\right] &\leq  \mathbb{E}\left[\max_{0 \le n \le N} \left( \sum_{\ell=1}^{n-1} |Y_\ell - \tilde{Y}_\ell| \right)^2 \right] \\	
				&\le \mathbb{E}\left[ \max_{0 \le n \le N} \left( \sum_{\ell=1}^{n-1}|R_\ell - I_{k\times k}| |  \tilde{Y}_\ell| \right)^2 \right]\\
				&\le \mathbb{E}\left[ \max_{0 \le n \le N} n \sum_{\ell=1}^{n-1} \left( \tilde{C}(k, \gamma, T) \Delta t^2 \right)^2 | \tilde{Y}_\ell|^2 \right] \\
					&\le \mathbb{E}\left[ N \sum_{\ell=1}^{N-1} \left(  \tilde{C}(k, \gamma, T) \Delta t^2 \right)^2 \left(\max_{0 \le n \le N} | \tilde{Y}_n|^2 \right) \right] \\
				&\le \ \tilde{C}(k, \gamma, T)
				\mathbb{E}\left[\left( \max\limits_{0 \leq n \leq N} |\tilde{Y}_n|^2 \right)\right] \Delta t^2 \\
				&\le  \tilde{C}(k, \gamma, T)
			 K_1(T) \Delta t^2.
			\end{aligned}
		\end{equation*}

		Concerning $\mathrm{T}_1$, through the orthogonality hypothesis on $U$, Jensen's inequality, Assumption \ref{lipschitz} and relation \eqref{polynomial growth in t}, Lemma \ref{lem: sup n yn}, and \eqref{eq: sup n Y_n}, we obtain:
		\begin{equation*}
		\begin{aligned}
		\mathrm{T}_1 = & \mathbb{E}\big[\max_{0 \le n \le N} |\int_{0}^{t_{n}} U_{n_r}a(t_{n_r}, U_{n_r}^{\top}Y_{n_r}) - U(r) a(r,U(r)^{\top}Y(r))\mathrm{d}r|^2\big] \\
		\leq & T \mathbb{E}\big[ \int_{0}^{T} |U_{n_r}a(t_{n_r}, U_{n_r}^{\top}Y_{n_r}) - U(r) a(r,U(r)^{\top}Y(r))|^2\mathrm{d}r\big]\\
		= & T \mathbb{E}\big[ \int_{0}^{T} |U_{n_r}a(t_{n_r}, U_{n_r}^{\top}Y_{n_r}) - U(r) a(r,U(r)^{\top}Y(r))|^2\mathrm{d}r\big] \\
		\leq & T \mathbb{E}\big[ \int_{0}^{T}| U_{n_r}a(t_{n_r}, U_{n_r}^{\top}Y_{n_r}) - U(r)a(t_{n_r}, U_{n_r}^{\top}Y_{n_r}) + U(r)a(t_{n_r}, U_{n_r}^{\top}Y_{n_r}) 
		- U(r)a(r, U_{n_r}^{\top}Y_{n_r}) \\
		& + U(r)a(r, U_{n_r}^{\top}Y_{n_r})  - U(r)a(r, U(r)^{\top}Y_{n_r}) +
		U(r)a(r, U(r)^{\top}Y_{n_r}) - U(r) a(r,U(r)^{\top}Y(r))|^2\mathrm{d}r\big] \\
		\leq & 4T \big[ \int_{0}^{T} C_{\mathrm{lgb}}(1+K_1(T)) \|U_{n_r} - U(r)\|^{2}_{\mathrm{F}} + {C_{\mathrm{na}}(1+K_1(T))}  |t_{n_r}-r|^{2\alpha} \\
		& +  C_{\mathrm{Lip}}K_1(T)\|U_{n_r} - U(r)\|^{2}_{\mathrm{F}} + C_{\mathrm{Lip}}\mathbb{E}[|Y_{n_r}-Y(r)|^2] \mathrm{d}r\big]  \\
		\leq & 4T \big[ \int_{0}^{T}(C_{\mathrm{lgb}}+ C_{\mathrm{Lip}})(1+K_1(T)) \|U_{n_r} - U(r)\|^{2}_{\mathrm{F}} + {C_{\mathrm{na}}(1+K_1(T))}  |t_{n_r}-r|^{2\alpha} \\
		&+ C_{\mathrm{Lip}}\mathbb{E}[|Y_{n_r}-Y(r)|^2] \mathrm{d}r\big] \\
		\end{aligned}
		\end{equation*}	
		Notice that via Jensen's inequality, Cauchy-Schwarz inequality, and linear-growth bound, we can bound the local error with respect to the deterministic increment as
		\begin{equation}\label{eq: local U error}
			\begin{aligned}
				 \|U_{n_r} - U(r)\|^{2}_{\mathrm{F}} &=  \|U_{n_r} - U(t_{n_r}) - \int_{t_{n_r}}^{r} C_{Y(s)}^{-1}\mathbb{E}\left[Y(s) a(s,U(s)^{\top}Y(s))^{\top}\right]\left(I_{d \times d} - P^{\mathrm{row} }_{U(s)} \right) ds \|^{2}_{\mathrm{F}} \\
				 & \leq 2 \|U_{n_r} - U(t_{n_r}) \|^{2}_{\mathrm{F}} +2  \|\int_{t_{n_r}}^{r} C_{Y(s)}^{-1}\mathbb{E}\left[Y(s) a(s,U(s)^{\top}Y(s))^{\top}\right]\left(I_{d \times d} - P^{\mathrm{row} }_{U(s)} \right) ds \|^{2}_{\mathrm{F}} \\
				 & \leq 2 \|U_{n_r} - U(t_{n_r}) \|^{2}_{\mathrm{F}} +2  \frac{1}{\gamma} C_{\mathrm{lgb}}(1 + K_2(T))(\Delta t)^2,
			\end{aligned}
		\end{equation}
			where $K_2(T)$ is a positive constant bounding the second moment of $Y(t)$ \cite[Lemma 2.7]{kazashi2025dynamical} and we employed  Cauchy-Schwarz inequality and linear-growth bound. With a similar procedure, using the Itô's isometry and independence property of the Itô's integral, one finds that
			\begin{equation}\label{eq: local Y error}
			\begin{aligned}
				\mathbb{E}[|Y_{n_r}-Y(r)|^2] &=  \mathbb{E}[|Y_{n_r}-Y(t_{n_r}) - \int_{t_{n_r}}^{r}U(s) a(s,U(s)^{\top}Y(s)) \mathrm{d}s - \int_{t_{n_r}}^{r}  U(s) b(s,U(s)^{\top}Y(s))\mathrm{d}W_s |^2] \\
				& \leq 3 \mathbb{E}[|Y_{n_r}-Y(t_{n_r})  |^2] + 3 C_{\mathrm{lgb}}(1 + K_2(T)) ( (\Delta t)^2 + \Delta t)
				\end{aligned}
		\end{equation}
		From these inequalities we obtain the following bound for $T_1$
			\begin{equation*}
			\begin{aligned}
				\mathrm{T}_1 \leq & 4T \big[ \int_{0}^{T}2(C_{\mathrm{lgb}}+ C_{\mathrm{Lip}})(1+K_1(T)) \left( \|U_{n_r} - U(t_{n_r}) \|^{2}_{\mathrm{F}} +  \frac{1}{\gamma} C_{\mathrm{lgb}}(1 + K_2(T)) (\Delta t)^2\right) \\
				&+ {C_{\mathrm{na}}(1+K_1(T))}  |t_{n_r}-r|^{2\alpha} 
				+ 3 C_{\mathrm{Lip}} \left(\mathbb{E}[|Y_{n_r}-Y(t_{n_r})  |^2] +  C_{\mathrm{lgb}}(1 + K_2(T)) ( (\Delta t)^2 + \Delta t)\right) \mathrm{d}r\big] \\
				\leq & 4T \big[ \int_{0}^{T} 2(C_{\mathrm{lgb}}+ C_{\mathrm{Lip}})(1+K_1(T)) \left(\max_{0 \le p \le r} \|U_{n_p} - U(t_{n_p})\|^{2}_{\mathrm{F}} +  \frac{1}{\gamma} C_{\mathrm{lgb}}(1 + K_2(T)) (\Delta t)^2 \right) \\
				& + {C_{\mathrm{na}}(1+K_1(T))}  (\Delta t)^{2\alpha} + 3C_{\mathrm{Lip}}\mathbb{E}[\max_{0 \le p \le r}|Y_{n_p} - Y(t_{n_p})|^2] + 3C_{\mathrm{Lip}} C_{\mathrm{lgb}}(1 + K_2(T)) ( (\Delta t)^2 + \Delta t) \mathrm{d}r\big] \\
				\leq & 8T(C_{\mathrm{lgb}}+ C_{\mathrm{Lip}})(1+K_1(T)) \int_{0}^{T} \max_{0 \le p \le r} \|U_{n_p} - U(t_{n_p})\|^{2}_{\mathrm{F}}\mathrm{d}r \\
				&+ \left( \frac{8T^2(C_{\mathrm{lgb}}+ C_{\mathrm{Lip}})}{\gamma} C_{\mathrm{lgb}}(1 + K_1(T)) (1 + K_2(T))  (\Delta t)^2 \right) \\
				&+ 4T^2{C_{\mathrm{na}}(1+K_1(T))} (\Delta t)^{2\alpha} + 12TC_{\mathrm{Lip}} \int_{0}^{T} {\mathbb{E}[\max_{0 \le p \le r}|Y_{n_p} - Y(t_{n_p})|^2]\mathrm{d}r}  \\
				&·+ 12T^2C_{\mathrm{Lip}} C_{\mathrm{lgb}}(1 + K_2(T)) ( (\Delta t)^2 + \Delta t).\\
			\end{aligned}
		\end{equation*}

		We have similar estimates on $\mathrm{T}_2$, using \cite[Theorem 4.4.4]{durrett2019probability}, Itô's isometry, relations \eqref{eq: local U error} and \eqref{eq: local Y error} and $\Delta t \leq 1$:
		\begin{equation*}
		\begin{aligned}
		\mathrm{T}_2 = & \mathbb{E}[\max\limits_{0 \leq n \leq N}|\sum\limits_{\ell=0}^{n-1}\int_{t_{\ell}}^{t_{\ell+1}} U_{\ell}b(t_{\ell}, U_{\ell}^{\top}Y_{\ell}) - U(r) b(r,U(r)^{\top}Y(r))\mathrm{d}W_r |^2] \\
		\leq & 4\mathbb{E}[| \int_{0}^{T} U_{n_r}b(t_{n_r}, U_{n_r}^{\top}Y_{n_r}) - U(r) b(r,U(r)^{\top}Y(r))\mathrm{d}W_r |^2] \\
		\leq &4\mathbb{E}[\int_{0}^{T} |U_{n_r}b(t_{n_r}, U_{n_r}^{\top}Y_{n_r}) - U(r) b(r,U(r)^{\top}Y(r))|^2\mathrm{d}r] \\
		\leq &4\int_{0}^{T} \mathbb{E}[|U_{n_r}b(t_{n_r}, U_{n_r}^{\top}Y_{n_r}) - U(r) b(r,U(r)^{\top}Y(r))|^2]\mathrm{d}r \\
		 \leq & 32(C_{\mathrm{lgb}}+ C_{\mathrm{Lip}})(1+K_1(T)) \int_{0}^{T} \max_{0 \le p \le r} \|U_{n_p} - U(t_{n_p})\|^{2}_{\mathrm{F}}\mathrm{d}r \\
		 &  + \left( \frac{32T(C_{\mathrm{lgb}}+ C_{\mathrm{Lip}})}{\gamma} C_{\mathrm{lgb}}(1 + K_1(T))(1 + K_2(T)) (\Delta t)^2 \right) \\
		 & + 16T{C_{\mathrm{na}}(1+K_1(T))} (\Delta t)^{2\alpha} + 48C_{\mathrm{Lip}} \int_{0}^{T} \mathbb{E}[\max_{0 \le p \le r}|Y_{n_p} - Y(t_{n_p})|^2]\mathrm{d}r \\
		 &+ 48TC_{\mathrm{Lip}} C_{\mathrm{lgb}}(1 + K_2(T)) ( (\Delta t)^2 + \Delta t).\\
		\end{aligned}
		\end{equation*}

		Finally, putting all the bounds of $\mathrm{T}_0$-$\mathrm{T}_2$ together, we deduce from~\eqref{eq error y in EM} that
		\begin{equation*}
		\begin{aligned}
		\mathbb{E}[\max_{0 \le n \le N} |\tilde{Y}_{n} - Y(t_n)|^2] 
		\leq & 3 \tilde{C}(k, \gamma, T)
		K_1(T) \Delta t^2\\
		&+ 24(C_{\mathrm{lgb}}+ C_{\mathrm{Lip}})(1+K_1(T))(T+4) \int_{0}^{T} \max_{0 \le p \le r} \|U_{n_p} - U(t_{n_p})\|^{2}_{\mathrm{F}}\mathrm{d}r \\
		& +  \frac{24(T^2+4T)(C_{\mathrm{lgb}}+ C_{\mathrm{Lip}})}{\gamma} C_{\mathrm{lgb}}(1 + K_1(T))(1 + K_2(T)) (\Delta t)^2 \\
		& + 12{C_{\mathrm{na}}(1+K_1(T))}(T^2+4T) (\Delta t)^{2\alpha} \\
		&+ 24C_{\mathrm{Lip}}(T+4) \int_{0}^{T} \mathbb{E}[\max_{0 \le p \le r}|Y_{n_p} - Y(t_{n_p})|^2]\mathrm{d}r \\
		& + 24(T^2+4T)C_{\mathrm{Lip}} C_{\mathrm{lgb}}(1 + K_2(T)) ( (\Delta t)^2 + \Delta t) \\
		 \leq & A_y \Delta t \sum_{\ell=0}^{N-1} \max_{0 \le j \le \ell} \|U_{j} - U(t_{j})\|^{2}_{\mathrm{F}}  + B_y \Delta t \sum_{\ell=0}^{N-1} \mathbb{E}[\max_{0 \le 
		 			j \le \ell}|Y_{j} - Y(t_{j})|^2] \\
		& + C_y (\Delta t)^{2\alpha} + D_y  \Delta t + O(\Delta t^2), \\
		\end{aligned}
		\end{equation*}
		where $A_y=24(C_{\mathrm{lgb}}+ C_{\mathrm{Lip}})(1+K_1(T))(T+4)$, $B_y=24C_{\mathrm{Lip}}(T+4)$, $D_y= 48(T^2+4T)C_{\mathrm{Lip}} C_{\mathrm{lgb}}(1 + K_2(T)) $ are positive constants, and $C_y=12{C_{\mathrm{na}}(1+K_1(T))}(T^2+4T)$ is a non-negative constant, equal to zero if $a$ and $b$ are autonomous.
		
		We now turn to the error on the deterministic modes.
		The extra difficulty with respect to the stochastic case is the presence of the inverse Gramians $C^{-1}_{Y(s)}$, $C^{-1}_{Y_{n_s}}$ in the equation for the deterministic modes. We will use \cite[Lemma 3.5]{kazashi2021existence} which says that, if  $|C^{-1}_{Y(s)}|,|C^{-1}_{Y_{n_s}}|  \leq \frac{1}{\gamma}$ holds,  then there exists a constant $K_{\gamma}>0$ such that $\|C^{-1}_{Y_{n_s}} - C^{-1}_{Y(s)}\|^2_{\mathrm{F}} \leq K_{\gamma} \mathbb{E}[ |Y_{n_s}-Y(s)|^2]$ for all $s \in [0,T]$.
	
		Therefore, similarly to the strategy adopted for the stochastic increments we have
		\begin{equation*}
		\begin{aligned}
		&\max_{0 \le n \le N}   \|\tilde{U}_{n}-U(t_n)\|^{2}_{\mathrm{F}} = \max_{0 \le n \le N}  \| \left( \sum_{\ell=0}^{n-1} \tilde{U}_{\ell+1}-\tilde{U}_\ell \right) - U(t_n) + \tilde{U}_0\|^{2}_{\mathrm{F}} \\
		=& \max_{0 \le n \le N}   \| \left(\sum_{\ell=1}^{n-1}  U_{\ell}-\tilde{U}_\ell \right)+ \Big( \sum_{\ell=0}^{n-1}  \int_{t_\ell}^{t_{\ell+1}} C^{-1}_{Y_{\ell}} \mathbb{E}[Y_{\ell} a(t_{\ell},U_{\ell}^{\top}Y_{\ell})^{\top}][I_{d \times d}- U_{\ell}^{\top}U_{\ell}] \\
		& \qquad \qquad - C^{-1}_{Y(r)} \mathbb{E}[Y(r) a(r,U(r)^{\top}Y(r))^{\top}] [I_{d \times d}- U(r)^{\top}U(r)]\mathrm{d}r  \Big)\|_{\mathrm{F}}^{2} \\
		=& \max_{0 \le n \le N}   \|  \left( \sum_{\ell=1}^{n-1}  U_{\ell}-\tilde{U}_\ell \right) + \Big(\int_{0}^{t_n} C^{-1}_{Y_{n_r}} \mathbb{E}[Y_{n_r} a(t_{n_r},U_{n_r}^{\top}Y_{n_r})^{\top}][I_{d \times d}- U_{n_r}^{\top}U_{n_r}]  \\
		& \qquad \qquad - C^{-1}_{Y(r)} \mathbb{E}[Y(r) a(r,U(r)^{\top}Y(r))^{\top}] [I_{d \times d}- U(r)^{\top}U(r)]\mathrm{d}r \Big) \|_{\mathrm{F}}^{2} \\
		 \leq & 3\underbrace{\max_{0 \le n \le N}  \|  \left( \sum_{\ell=1}^{n-1}  U_{\ell}-\tilde{U}_\ell \right) \|_{\mathrm{F}}^2}_{\mathrm{T}_4} \\
	&+	3  T \underbrace{\int_{0}^{T} \|C^{-1}_{Y_{n_r}} \mathbb{E}[Y_{n_r} a(t_{n_r},U_{n_r}^{\top}Y_{n_r})^{\top}]- C^{-1}_{Y(r)} \mathbb{E}[Y(r) a(r,U(r)^{\top}Y(r))^{\top}] \|_{\mathrm{F}}^{2} \mathrm{d}r}_{\mathrm{T}_5}\\
		& + 3    T \underbrace{\int_{0}^{T} \|C^{-1}_{Y_{n_r}} \mathbb{E}[Y_{n_r} a(t_{n_r},U_{n_r}^{\top}Y_{n_r})^{\top}]U_{n_r}^{\top}U_{n_r} - C^{-1}_{Y(r)} \mathbb{E}[Y(r) a(r,U(r)^{\top}Y(r))^{\top}]U(r)^{\top}U(r) \|_{\mathrm{F}}^{2} \mathrm{d}r}_{\mathrm{T}_6} \\
		\end{aligned}
		\end{equation*}
	where in the last term we employ Jensen's inequality and basic norm inequalities.
	
	Similarly to the treatment of $\mathrm{T}_0$, we can bound $\mathrm{T}_4$ in the following way
	\begin{equation*}
		\begin{aligned}
			\max_{0 \le n \le N}  \|  \left( \sum_{\ell=1}^{n-1}  U_{\ell}-\tilde{U}_\ell \right) \|_{\mathrm{F}}^2 & \leq \max_{0 \le n \le N}   \left( \sum_{\ell=1}^{n-1}  \| U_{\ell}-\tilde{U}_\ell  \|_{\mathrm{F}} \right)^2 \\
			& \leq 	\max_{0 \le n \le N}   \left( \sum_{\ell=1}^{n-1}  \| U_\ell - U_\ell R_\ell \|_{\mathrm{F}} \right)^2 \\
	     	& \leq 	\max_{0 \le n \le N}  \left( \sum_{\ell=1}^{n-1}  \|I_{k \times k} - R_{\ell}\|_{\mathrm{F}} \right)^2 \leq \tilde{C}(k, \gamma, T)
	     	K_1(T) \Delta t^2\\
		\end{aligned}
	\end{equation*}

 Concerning $\mathrm{T}_5$, it holds that
		\begin{equation*}
		\begin{aligned}
		\mathrm{T}_5 
		\leq & \int_{0}^{T} \|C^{-1}_{Y_{n_r}} \mathbb{E}[Y_{n_r} a(t_{n_r},U_{n_r}^{\top}Y_{n_r})^{\top}]- 
		C^{-1}_{Y(r)} \mathbb{E}[Y_{n_r} a(t_{n_r},U_{n_r}^{\top}Y_{n_r})^{\top}]+
		C^{-1}_{Y(r)} \mathbb{E}[Y_{n_r} a(t_{n_r},U_{n_r}^{\top}Y_{n_r})^{\top}]\\
		& - C^{-1}_{Y(r)} \mathbb{E}[Y(r) a(t_{n_r},U_{n_r}^{\top}Y_{n_r})^{\top}]+
		C^{-1}_{Y(r)} \mathbb{E}[Y(r) a(t_{n_r},U_{n_r}^{\top}Y_{n_r})^{\top}]-
		C^{-1}_{Y(r)} \mathbb{E}[Y(r) a(r,U_{n_r}^{\top}Y_{n_r})^{\top}] \\
		&+ C^{-1}_{Y(r)} \mathbb{E}[Y(r) a(r,U_{n_r}^{\top}Y_{n_r})^{\top}] 
		- C^{-1}_{Y(r)} \mathbb{E}[Y(r) a(r,U(r)^{\top}Y_{n_r})^{\top}] \\
		&+
		C^{-1}_{Y(r)} \mathbb{E}[Y(r) a(r,U(r)^{\top}Y_{n_r})^{\top}]-
		C^{-1}_{Y(r)} \mathbb{E}[Y(r) a(r,U(r)^{\top}Y(r))^{\top}] \|_{\mathrm{F}}^{2} \mathrm{d}r \\
		\leq & \int_{0}^{T} 5K_{\gamma} K_1(T)C_{\mathrm{lgb}}(1+K_1(T)) \mathbb{E}[ |Y_{n_r}-Y(r)|^2] + 5\frac{1}{\gamma^2} C_{\mathrm{lgb}}(1+K_1(T)) \mathbb{E}[ |Y_{n_r}-Y(r)|^2] \\
		&+ 5\frac{1}{\gamma}{C_{\mathrm{na}}(1+K_1(T))} (\Delta t)^{2\alpha} + 5\frac{K_1(T)}{\gamma} C_{\mathrm{Lip}} \| U_{n_r}-U(r)\|^2_{\mathrm{F}} \\
		&+ 5\frac{1}{\gamma}C_{\mathrm{Lip}} \mathbb{E}[ |Y_{n_r}-Y(r)|^2] \mathrm{d}r \\
		\leq & \int_{0}^{T} 10[(K_{\gamma}K_1(T) + \frac{1}{\gamma^2} )C_{\mathrm{lgb}}(1+K_1(T))+ \frac{1}{\gamma} C_{\mathrm{Lip}} ] \\
		& \cdot \left( \mathbb{E}[ \max_{0 \le p \le r} |Y_{n_p} - Y(t_{n_p})|^2] + C_{\mathrm{lgb}}(1 + K_2(T)) ( (\Delta t)^2 + \Delta t) \right) \mathrm{d}r \\
		&+ 5\frac{1}{\gamma} T {C_{\mathrm{na}}(1+K_1(T))} (\Delta t)^{2\alpha} \\
		& + \int_{0}^{T}  10\frac{K_1(T)}{\gamma} C_{\mathrm{Lip}} \left(\max_{0 \le p \le r}\|U_{n_p} - U(t_{n_p}) \|^{2}_{\mathrm{F}} + \frac{1}{\gamma} C_{\mathrm{lgb}}(1 + K_2(T)) (\Delta t)^2 \right) \mathrm{d}r,\\
		\end{aligned}
		\end{equation*}	
		where we exploited Cauchy-Schwarz inequality and in the last line we used \eqref{eq: local U error} and \eqref{eq: local Y error}.
		For $\mathrm{T}_6$, we have similarly
		\begin{equation*}
		\begin{aligned}
		\mathrm{T}_6  
		= & \int_{0}^{T} \|C^{-1}_{Y(r)} \mathbb{E}[Y(r) a(r,U(r)^{\top}Y(r))^{\top}]U(r)^{\top}U(r) - C^{-1}_{Y(r)} \mathbb{E}[Y(r) a(r,U(r)^{\top}Y(r))^{\top}]U(r)^{\top}U_{n_r}\\
		& + C^{-1}_{Y(r)} \mathbb{E}[Y(r) a(r,U(r)^{\top}Y(r))^{\top}]U(r)^{\top}U_{n_r} \\
		& - C^{-1}_{Y(r)} \mathbb{E}[Y(r) a(r,U(r)^{\top}Y(r))^{\top}]U_{n_r}^{\top}U_{n_r} + C^{-1}_{Y(r)} \mathbb{E}[Y(r) a(r,U(r)^{\top}Y(r))^{\top}]U_{n_r}^{\top}U_{n_r} \\
		& - C^{-1}_{Y_{n_r}} \mathbb{E}[Y_{n_r} a(t_{n_r},U_{n_r}^{\top}Y_{n_r})^{\top}]U_{n_r}^{\top}U_{n_r} \|_{\mathrm{F}}^{2} \mathrm{d}r \\
		\leq & \int_{0}^{T} 12 \frac{1}{\gamma} C_{\mathrm{lgb}}(1+K_2(T))  \left( \max_{0 \le p \le r} \|U_{n_p} - U(t_{n_p}) \|^{2}_{\mathrm{F}} +   \frac{1}{\gamma} C_{\mathrm{lgb}}(1 + K_2(T)) (\Delta t)^2 \right)\mathrm{d}r \\
		& + 30 \int_{0}^{T} [(K_{\gamma}K_1(T) +\frac{1}{\gamma^2} )C_{\mathrm{lgb}}(1+K_1(T))+ \frac{1}{\gamma} C_{\mathrm{Lip}} ] \\
		& \cdot\left(\mathbb{E}[\max_{0 \le p \le r}  |Y_{n_p} - Y(t_{n_p})  |^2] + C_{\mathrm{lgb}}(1 + K_2(T)) ( (\Delta t)^2 + \Delta t) \right)\mathrm{d}r \\
		&+ 15\frac{1}{\gamma} T {C_{\mathrm{na}}(1+K_1(T))} (\Delta t)^{2\alpha}  \\
		&+ 30\int_{0}^{T} \frac{K_1(T)}{\gamma} C_{\mathrm{Lip}} \left( \max_{0 \le p \le r} \|U_{n_p} - U(t_{n_p}) \|^{2}_{\mathrm{F}} +  \frac{1}{\gamma} C_{\mathrm{lgb}}(1 + K_2(T)) (\Delta t)^2  \right)\mathrm{d}r .\\
		\end{aligned}
		\end{equation*}	
		
		This leads us to:
		\begin{equation} 
		\begin{aligned}\label{computations in U - convergence}
		\max\limits_{0 \leq n \leq N}  \|\tilde{U}_{n}-U(t_n)\|^{2}_{\mathrm{F}} 
		\leq &  3 \tilde{C}(k, \gamma, T)
		K_1(T) \Delta t^2 \\
		& + 3T  [12 \frac{1}{\gamma} C_{\mathrm{lgb}}(1+K_2(T)) + 40\frac{K_1(T)}{\gamma} C_{\mathrm{Lip}} ] \\
		& \quad \cdot \left( \int_{0}^{T} \max_{0 \le p \le r} \|U_{n_p} - U(t_{n_p})\|^2_{\mathrm{F}} \mathrm{d}r + \frac{1}{\gamma} C_{\mathrm{lgb}}(1 + K_2(T)) (\Delta t)^2 \right) \\
		& + 120T  [(K_{\gamma}K_1(T) +  \frac{1}{\gamma^2} )C_{\mathrm{lgb}}(1+K_1(T))+ \frac{1}{\gamma}C_{\mathrm{Lip}} ]\\
		& \cdot  \left( \int_{0}^{T}\mathbb{E}[\max_{0 \le p \le r}  |Y_{n_p} - Y(t_{n_p})  |^2] \mathrm{d}r + C_{\mathrm{lgb}}(1 + K_2(T)) ( (\Delta t)^2 + \Delta t) \right)\\
		&+ 60\frac{1}{\gamma} T^2 {C_{\mathrm{na}}(1+K_1(T))} (\Delta t)^{2\alpha} \\
		\leq & A_u \sum_{\ell=0}^{N-1} \max_{0 \le j \le \ell}\|U_{j} - U(t_{j}) \|^{2}_{\mathrm{F}} \Delta t + B_u \sum_{\ell=0}^{N-1}  \mathbb{E}[\max_{0 \le j \le \ell}  |Y_{j} - Y(t_{j})  |^2]\Delta t \\
		& + C_u (\Delta t)^{2\alpha} + D_u \Delta t + O(\Delta t^2), \\
		\end{aligned}
		\end{equation}
		where $A_u= 3T [12 \frac{1}{\gamma} C_{\mathrm{lgb}}(1+K_2(T)) + 40\frac{K_1(T)}{\gamma} C_{\mathrm{Lip}} ]$, $B_u= 120T[(K_{\gamma}K_1(T)+  \frac{1}{\gamma^2})C_{\mathrm{lgb}}(1+K_1(T)) + \frac{1}{\gamma} C_{\mathrm{Lip}} ]$, $D_u=120T^2  [(K_{\gamma}K_1(T) +  \frac{1}{\gamma} )C_{\mathrm{lgb}}(1+K_1(T))+ \frac{1}{\gamma} C_{\mathrm{Lip}} ] C_{\mathrm{lgb}}(1 + K_2(T)) $ are positive constants and $C_u=60\frac{1}{\gamma} T^2 {C_{\mathrm{na}}(1+K_1(T))}$ is non-negative and equal to zero if the drift $a$ and the diffusion $b$ are autonomous.
		
		Finally, using Gronwall's lemma and $\Delta t\leq 1$ we find:
		\begin{equation}\label{final error on Y Gronwall - Convergence}
		\begin{aligned}
		\mathbb{E}[\max_{0 \le n \le N} |\tilde{Y}_{n} - Y(t_n)|^2]  + & \max_{0 \le n \le N}  \|\tilde{U}_{n}-U(t_n)\|^{2}_{\mathrm{F}}\\
		\leq& A \Delta t \sum_{\ell=0}^{N-1} \max_{0 \le j \le \ell} \|U_{j} - U(t_{j})\|^2_{\mathrm{F}} \\
		&+ B \Delta t \sum_{\ell=0}^{N-1} \max_{0 \le j \le \ell} | Y_{j} - Y(t_{j})|^2+ C (\Delta t)^{2\alpha} + D (\Delta t) + O(\Delta t^2) \\
		\leq & \Delta t \max\{A,B\} \left(\sum_{\ell=0}^{N-1} \max_{0 \le j \le \ell} \|U_{j} - U(t_{j})\|^2_{\mathrm{F}} +  \mathbb{E}[\max_{0 \le j \le  \ell} | Y_{j} - Y(t_{j})|^2]\right) \\
		&+ C (\Delta t)^{2\alpha} + D ( \Delta t) + O(\Delta t^2)\\
		\leq &\exp^{\max(A,B)T} \cdot 2(C+D)((\Delta t )^{\min\{ 1, 2\alpha\} })) + O(\Delta t^2),
		\end{aligned}
		\end{equation}
		where $A= A_u + A_y$, $B=B_u + B_y$, $D=D_u+D_y$ are positive constants and $C= C_u+C_y$ is non-negative and equal to zero if  $a$ and $b$ are autonomous. The thesis  therefore is obtained from \eqref{X-error estimate} and \eqref{final error on Y Gronwall - Convergence}:
		\begin{equation*}
	    \sqrt{\mathbb{E}[\max\limits_{0 \leq n \leq N} |X_{n} - X(t_n)|^2]} \leq  C_1\sqrt{(C+D)}  \exp^{\max(A,B)\frac{T}{2}} \cdot ((\Delta t)^{\min\{\frac{1}{2}, \alpha \} }) + O(\Delta t).
		\end{equation*}
	\end{proof}
\end{Theorem}
	
\begin{Remark}
	The convergence result of Theorem \ref{thm: convergence of DLR Euler-Maruyama} holds under the assumption that the Gramian (of both the continuous and discrete dynamics) is lower bounded by a positive constant in a matrix sense. Lower bounds on the Gramian are given in Propositions \ref{prop: 1 discrete covariance full-rank - EM} and \ref{prop: 2 discrete covariance full-rank - EM} under suitable assumptions.
\end{Remark}
\begin{Remark}
    A closer inspection on the constant $C:=C(\gamma,k,T)$ in \eqref{eq:conv DLR EM} shows that it depends exponentially on $\frac{1}{\gamma}$ (see inequality \eqref{computations in U - convergence} in the proof of Theorem \ref{thm: convergence of DLR Euler-Maruyama}), i.e.\ it blows up exponentially fast when the smallest singular values of the DLR solution and its DLR EM approximation goes to zero. This shows that the Euler-Maruyama discretization is not robust with respect to the smallest singular value. We verify this statement numerically in Section \ref{sec: numerical experiments}.
\end{Remark}

As a final result of
this subsection, we study the error between the exact solution $X^{\mathrm{true}}$ of the SDE and its DLR Euler-Maruyama approximation. 
\begin{Theorem}[Error estimate of DLR-EM]\label{thm: Approx real sol}
 Denote by $X(t)$ and $\{X_{n}\}_n$ the DLR solution of chosen rank $k$ and the DLR approximation via DLR EM method, respectively. Under the same assumptions of Theorem \ref{thm: convergence of DLR Euler-Maruyama}, let $\varepsilon:=\varepsilon(k) \geq0$  be defined as
\begin{equation}\label{eps bound projection}
	\varepsilon^2 := \sup_{\substack{t \in [0,T] \\ Z \in L^2(\Omega,\mathbb{R}^d) \text{ of rank } k \\
			\text{ with } ·\mathbb{E}[|Z|^2] \leq \tilde{K}}} \max\{\mathbb{E}[|a(t,Z) -  P_{Z}a(t, Z) |^2], \ \mathbb{E}[\| b(t, Z) - \mathcal{P}_{\mathcal{U}(Z)}b(t, Z)\|^2_{\mathrm{F}} ] \},
\end{equation}
	 where $\tilde{K}:=\min \{ K > 0 \text{ such that } \max\{\mathbb{E}[\sup\limits_{0 \leq t \leq T} |X(t)|^2], \mathbb{E}[\max\limits_{0 \leq n \leq N} |X_n|^2] \} \leq K \}$, $P_{Z}[ \ \cdot\ ]:= (I_{d\times d}-\mathcal{P}_{\mathcal{U}(Z)})[\ \cdot \ ]\mathcal{P}_{\mathcal{Y}(Z)}+\mathcal{P}_{\mathcal{U}(Z)}[\ \cdot \ ]$ is an orthogonal projector, whereas $P_\mathcal{U(Z)}$ and $\mathcal{P}_{\mathcal{Y}(Z)}$ denote the projection onto range and corange of $Z$, respectively. 
	Then, for all $n=1, \dots, N$
	\begin{equation}\label{approx real equation}
		\begin{aligned}
	\sqrt{\mathbb{E}[\max\limits_{0 \leq n \leq N} |X^{\mathrm{true}}(t_n) - X_{n} |^2]} \leq &c_1(T) \sqrt{\mathbb{E}[|X^{\mathrm{true}}_0 - X_0 |^2]} + c_2(k,T) \varepsilon \\
	&+ c_3(\gamma,k,T) (\Delta t)^{\min\{\frac{1}{2}, \alpha \} },
	\end{aligned}
	\end{equation}
	where $c_1:=c_1(T)$, $c_2:=c_2(k,T)$ are positive constants dependent on the final time $T$ and the rank $k$, and $c_3:=c_3(\gamma, k,T)$ is the positive constant $C$ appearing in Theorem \ref{thm: convergence of DLR Euler-Maruyama}, which depends on the final time $T$, the rank $k$, and $\frac{1}{\gamma}$.
	\begin{proof}
		In order to prove the statement, we will bound the studied error with an error between the continuous DLRA and the true solution plus an error between the continuous DLRA and the DLR EM algorithm. Then, we will look for reasonable bounds for the first term in order to conclude via Gronwall's lemma and Theorem \ref{thm: convergence of DLR Euler-Maruyama}.
		
		For the solution $X^{\mathrm{true}}(t)$ of \eqref{eq:SDE-diff} over the interval $[0,T]$, there exists a constant $K(T)$, depending on the final time $T$, such that
		 $\mathbb{E}[\|X^{\mathrm{true}}(t)\|^2] \leq K(T)$ for all $t \in  [0,T]$ \cite[Theorem 4.5.4]{kloeden1992stochastic}. 
		 
		Now, notice that the following relation holds
		 \begin{equation}\label{eq: sum of error}
			\begin{aligned}
				\sqrt{\mathbb{E}[\max\limits_{0 \leq n \leq N} |X^{\mathrm{true}}(t_n) - X_{n} |^2]} =& \sqrt{\mathbb{E}[\big( \max\limits_{0 \leq n \leq N} |X^{\mathrm{true}}(t_n) - X_{n} | \big)^2]}\\
	    		\leq &\sqrt{ \mathbb{E}[ \sup\limits_{0 \leq s \leq t} |X^{\mathrm{true}}(s) - X(s) |^2]} + \sqrt{  \mathbb{E}[\max\limits_{0 \leq n \leq N} |X(t_n) - X_{n} |^2 ]}.
	    	\end{aligned}
	    \end{equation}
       
       In order to bound the right-hand side in \eqref{eq: sum of error}, we need a convenient expression of the DLR solution. So recall that from the Itô's formula one has
		\begin{equation*}
		\begin{aligned}
		\mathrm{d}X(t)  = (\mathrm{d}U(t)^{\top})Y(t)+ U(t)^{\top} \mathrm{d}Y(t)+ \sum\limits_{i=1}^{k} \underbrace{ \mathrm{d} \langle U^{i}(t),Y^{i}(t)\rangle }_{=0},
	\end{aligned}
     \end{equation*}
 which is equivalent to
      \begin{equation*}
      	\begin{aligned}
  X(t)  =& X_0 + \int_{0}^{t} \left(I_{d \times d} - P_{U(s)} \right) \mathbb{E}\left[a(s,X(s))Y(s)^{\top}\right] C^{-1}_{Y(s)}Y(s) \mathrm{d}s \\
  &+ \int_{0}^{t} U(s)^{\top} U(s) a(s,X(s)) \mathrm{d}s + \int_{0}^{t} U(s)^{\top}  U(s) b(s,X(s))\mathrm{d}W_s.
        	\end{aligned}
		\end{equation*}
		
		Let us analyze the difference $X^{\mathrm{true}}(s) - X(s)$: we have
		\begin{equation*}
		\begin{aligned}
		&\mathbb{E}[\sup\limits_{0 \leq s \leq t} | X^{\mathrm{true}}(s) - X(s) |^2]\\
		  \leq &  3 \mathbb{E}[|X^{\mathrm{true}}_0 - X_0 |^2] + 3 \underbrace{\mathbb{E}[\sup\limits_{0 \leq s \leq t} |\int_{0}^{s}b(r,X^{\mathrm{true}}(r))-U(r)^{\top}  U(r) b(r,X(r))\mathrm{d}W_r|^2]}_{:=\mathrm{T}_2} \\
		& + 3 \underbrace{\mathbb{E}[\sup\limits_{0 \leq s \leq t} |\int_{0}^{s}a(r,X^{\mathrm{true}}(r)) -  P_{X(r)}[a(r,X(r))]  \mathrm{d}r|^2]}_{:=\mathrm{T}_1} .
		\end{aligned}
		\end{equation*}
		We want to find a specific bound for the terms $\mathrm{T}_1$ and $\mathrm{T}_2$.
		Regarding $\mathrm{T}_1$, we exploit condition~\eqref{eps bound projection}
			\begin{equation*}
			\begin{aligned}
			\mathrm{T}_1 \leq & \mathbb{E}[t \int_{0}^{t}| a(r,X^{\mathrm{true}}(r)) - P_{X(r)}[a(r,X(r))]|^2 \mathrm{d}r] \\
			\leq  & t \int_{0}^{t}\mathbb{E}[| a(r,X^{\mathrm{true}}(r))  - a(r,X(r)) + a(r,X(r)) - P_{X(r)} a(r, X(r)) |^2] \mathrm{d}r \\
			\leq & 2T \int_{0}^{t}\mathbb{E}[\sup\limits_{0 \leq s \leq t}  | a(r,X^{\mathrm{true}}(r))  - a(r,X(r))|^2]\mathrm{d}r + 2T \int_{0}^{t}\mathbb{E}[|  a(r,X(r)) - P_{X(r)} a(r, X(r)) |^2] \mathrm{d}r .\\
			\leq & 2T C_{\mathrm{Lip}}  \int_{0}^{t}\mathbb{E}[\sup\limits_{0 \leq s \leq t}  |X^{\mathrm{true}}(r) -X(r)|^2]\mathrm{d}r + 2T \int_{0}^{t}\mathbb{E}[|  a(r,X(r)) - P_{X(r)} a(r, X(r)) |^2] \mathrm{d}r .\\
			\end{aligned}
			\end{equation*}
		Whereas for $\mathrm{T}_2$, we have
			\begin{equation*}
			\begin{aligned}
			T_2 &= \mathbb{E}[\sup\limits_{0 \leq s \leq t} |\int_{0}^{s}b(r,X^{\mathrm{true}}(r))- b(r,X(r)) + b(r,X(r))-U(r)^{\top}  U(r) b(r,X(r))\mathrm{d}W_r|^2 ]\\
			 & \leq 4 \mathbb{E}[\int_{0}^{t}|b(r,X^{\mathrm{true}}(r))- b(r,X(r)) + b(r,X(r))-U(r)^{\top}  U(r) b(r,X(r))|^2\mathrm{d}r] \\
			& \leq 8 \mathbb{E}[\int_{0}^{t}  C_{\mathrm{Lip}} | X^{\mathrm{true}}(r) - X(r) |^2 \mathrm{d}r]+ 8 \mathbb{E}[\int_{0}^{t} | b(r,X(r))-U(r)^{\top}  U(r) b(r,X(r)) |^2 \mathrm{d}r], \\
			\end{aligned}
			\end{equation*}
		where in the second line we employ the Itô's isometry and the Doob's inequality.
		Putting all together we find by Gronwall's lemma and \eqref{eps bound projection}
		\begin{equation*}
		\begin{aligned}
		 \mathbb{E}[\sup\limits_{0 \leq s \leq t} | X^{\mathrm{true}}(s) - X(s) |^2]  \leq & 3 \mathbb{E}[|X^{\mathrm{true}}_0 - X_0 |^2] + 6(T+4)T \varepsilon^2 \\
		 &+ 6C_{\mathrm{Lip}}(T+4)T \left(\int_{0}^{t} \mathbb{E}[\sup\limits_{0 \leq s \leq r} |X^{\mathrm{true}}(s) - X(s)|^2] \mathrm{d}r \right) \\
		\leq & \left[3 \mathbb{E}[|X^{\mathrm{true}}_0 - X_0 |^2] + 6(T+4)T \varepsilon^2 \right]e^{6C_{\mathrm{Lip}}(T+4)T }.
		\end{aligned}
		\end{equation*}
		
		Finally, using Theorem \ref{thm: convergence of DLR Euler-Maruyama} and properties of square root yields the thesis:
		\begin{equation*}
		\begin{aligned}
		 \sqrt{\mathbb{E}[\max\limits_{0 \leq n \leq N}|X^{\mathrm{true}}(t_n) - X_{n} |^2]} 
		 \leq & \sqrt{ \mathbb{E}[\sup\limits_{0 \leq s \leq t} |X^{\mathrm{true}}(s) - X(s) |^2]} +  \sqrt {   \mathbb{E}[\max\limits_{0 \leq n \leq N}|X(t_n) - X_{n} |^2]} \\
		 \leq & \sqrt{ 3 }e^{6C_{\mathrm{Lip}}(T+4)\frac{T}{2}} \sqrt{\mathbb{E}[|X^{\mathrm{true}}_0 - X_0 |^2]}  \\
		& + \sqrt{6(T+4)T} e^{6C_{\mathrm{Lip}}(T+4)\frac{T}{2}}\varepsilon + C(\gamma, k,T) (\Delta t)^{\min\{\frac{1}{2}, \alpha \} }.
		\end{aligned}
		\end{equation*}
	\end{proof}
\end{Theorem}

\begin{Remark}[On the modelling error \eqref{eps bound projection}]
The quantity in \eqref{eps bound projection} defines the maximal projection error of the drift and diffusion terms, when evaluated on a rank-$k$ function onto the tangent space to the manifold of rank-$k$ processes. Roughly speaking, this condition estimates the low-rankness of our SDE full system \eqref{eq:SDE-diff} through time: indeed, the smaller the $\varepsilon,$ the better the DLRA of rank-$k$ surrogate approximate \eqref{eq:SDE-diff} over time.
This is similar to standard assumptions in the DLRA literature (see e.g. \cite{koch2007dynamical,kieri2016discretized}).
In light of this, it is reasonable to expect that the rank has an effect on $\varepsilon$: the larger the rank $k$, the smaller $\varepsilon$.
\end{Remark}
\begin{Remark}[Comments on the constants of Theorem \ref{thm: Approx real sol}]
	We observe that the constant $c_1$ on the right-hand side of \eqref{approx real equation} does not depend on the rank $k$. This is expected, as the only source of rank dependence in the error associated with the rank-$k$ approximation of the initial datum. 
	On the other hand, although the constant $c_2$ in the second term does not explicitly depend on the rank $k$, this term still implicitly depends on $k$ through  $\varepsilon$, which is given in \eqref{eps bound projection}. 
\end{Remark}
\begin{Remark}[Rank error in the initial condition $X^{\mathrm{true}}_0$]
		One can also go deeper in analyzing the error of $\mathrm{T}_1$ in the proof of Theorem \ref{thm: Approx real sol}. For instance consider the covariance $\mathbb{E}[ X^{\mathrm{true}}_0(X^{\mathrm{true}}_0)^{\top}] \in \mathbb{R}^{d \times d}$, which is a symmetric and semi-positive definite matrix, and, hence, admits the following decomposition
		$$
		\mathbb{E}[ X^{\mathrm{true}}_0(X^{\mathrm{true}}_0)^{\top}] = Q \mathrm{diag}(\lambda^1,\dots,\lambda^d)Q^{\top},$$ 
		where $Q = [q_i]_i \in \mathbb{R}^{d \times d}$ is an orthogonal matrix with $q_i$ the $i$-th orthonormal column and $\lambda^i$ is the non-negative eigenvalue associated to $q_i$, with $(\lambda^i)_i$ ordered decreasingly.
		Moreover, there exists an orthonormal system $(\zeta_i)_{i=1,\dots,d}$, with $\zeta_i \in L^2\left(\Omega, \mathbb{R}\right)$, such that one has
		\begin{equation*}
			X^{\mathrm{true}}_0 =  \sum_{i = 1 }^{d} \sqrt{\lambda^i} q_i\zeta_i ;
		\end{equation*}
		for details see e.g. \cite[Section 2]{kazashi2025dynamical}. Then, if one chooses $X_0 = \sum\limits_{i = 1}^{k} \sqrt{\lambda^i} q_i \zeta_i$, the error is
		\begin{equation*}
			\mathbb{E}[\|X^{\mathrm{true}}_0 - X_0 \|^2_{\mathrm{F}}] = \sum_{i = k+1}^{d} \lambda^i.
		\end{equation*}
\end{Remark}

\section{Analysis of the DLR Projector Splitting for SDEs}\label{sec: convergence PS SDE}
We now study the discretization scheme defined by Algorithm~\ref{alg: Stoch DLR Proj algorithm}.
As we saw in Section~\ref{sec: PS SDE}, this algorithm can be viewed as a projector-splitting scheme.
Importantly, we show in this section that it converges independently of the smallest singular value of the Gram matrix. This is in stark contrast to the DLR Euler–Maruyama method, for which  the constants in the error bound ---at least under our proof strategy--- may blow up exponentially fast with respect to the inverse of the smallest singular value and the method may require a time-step restriction proportional to the same singular value. 
To establish such singular-value-independent results, we first present the following lemma on an $L^2$-norm bound. Notably, we do not impose any condition on $\Delta t_n$, in contrast to the analogous result in Lemma~\ref{lem: sup n yn} for the DLR EM scheme.
\begin{Lemma}[$L^2$-norm bound of 
	DLR Projector Splitting for SDEs]\label{lem: bound stoch dlr proj}
	Let $X_n= U_n^{\top}Y_n$, $n=0,\dots,N$, be the solution produced by
	Algorithm \ref{alg: Stoch DLR Proj algorithm} with an arbitrary sequence $\{\Delta t_n\}_n$ of step-sizes such that $\sum_{n=0}^{N-1} \Delta t_n=  T$. Under Assumptions \ref{linear-growth-bound}-\ref{eq:initial value}, the following bound on the mean square norm of $X_n$ holds:
	\begin{equation}\label{Stoch DLR Proj-bound 1}
		\mathbb{E}[|X_{n}|^2] \leq \left( \mathbb{E}[|X_0|^2] + 1\right) e^{( 1+ C_{\mathrm{lgb}}(2+ T))T} - 1 =K_3(T).
	\end{equation}
	\begin{proof}
		The proof is deferred to Appendix \ref{app:DLR PS SDEs}.
	\end{proof}
\end{Lemma}

In Lemma \ref{lem: sup n yn}, we derived a mean-squared bound of the DLR EM solution, under a time-step restriction dependent on the smallest singular value of the Gramian of $Y_n$. 
This result is subsequently applied in Proposition \ref{prop: 2 discrete covariance full-rank - EM} to derive a lower bound on the smallest singular value of the Gramian under assumption of uniformly elliptic noise. 

In contrast, for the DLR PS SDE method, Lemma \ref{lem: bound stoch dlr proj} imposes no restrictions on 
$\Delta t_n$. 
Consequently, the following Proposition, analogous to Propositions~\ref{prop: 1 discrete covariance full-rank - EM} and~\ref{prop: 2 discrete covariance full-rank - EM}, derives a lower-bound on the smallest singular value of the Gramian under no time step-restrictions if the noise is uniformly elliptic.

\begin{Proposition}[Lower bound of the Gramian of the stochastic modes] \label{prop: discrete covariance full-rank - Stoch DLR Proj}
	Suppose Assumption \ref{ass: diff} is satisfied. Then for any sequence of step-sizes $\{\Delta t_n\}_n$ with $\Delta t_n = t_{n+1} - t_n$ we have:
	\begin{equation}\label{eq: cov 1 Stoch DLR Proj}
		C_{Y_{n+1}} \succ \sigma_{B} \Delta t_n \cdot I_{k \times k}.
	\end{equation}
	If Assumption \ref{ass: C_Y_0} holds as well, then for a uniform step-size $\Delta t$ one has
	\begin{equation}\label{eq: cov 2 Stoch DLR Proj}
		\begin{aligned}
			C_{Y_{n+1}} \succeq  & \min \{\sigma_{0}, \frac{\sigma_{B}^2}{4C_{\mathrm{lgb}}(1 + K_3(T))}\}.
		\end{aligned}
	\end{equation}
		\begin{proof}
		The proof closely follows those of Propositions \ref{prop: 1 discrete covariance full-rank - EM} and \ref{prop: 2 discrete covariance full-rank - EM}.
	\end{proof}
\end{Proposition}
Unlike the DLR EM, Algorithm \ref{alg: Stoch DLR Proj algorithm} converges to the true solution without any dependence on the smallest singular value of the involved Gramians. To prove this results, in contrast to Theorem \ref{thm: convergence of DLR Euler-Maruyama}, we do not compare our discretized modes with the continuous ones, rather we compare the DLRA numerical solution $X_n$ with the Euler-Maruyama solution $X^{\mathrm{EM}}_n$ of \eqref{eq:SDE-diff}. This strategy better exploits the projector-configuration of the algorithm to result in an error made by 3 components: the approximation of the initial condition, the Euler-Maruyama time discretization error, and the modelling error \eqref{eps bound projection}, without involving any inverse of the Gramians.
	\begin{Theorem}[Numerical Convergence of Algorithm \ref{alg: Stoch DLR Proj algorithm} to the Euler-Maruyama approximation of the True solution] \label{thm: Numerical Convergence Stoch DLR Proj - DLRA}
		Consider a uniform partition $\Delta := \left\{t_n = t_0 + n \Delta t, \ n=0,\dots, N\right\}$ of $[0,T]$ with time-step $\Delta t=\frac{T}{N}$. Let us denote by $X^{\mathrm{EM}}_n$ the numerical approximation of \eqref{eq:SDE-diff} obtained by a standard Euler-Maruyama method and by $X_n$ the numerical DLR solution obtained with Algorithm \ref{alg: Stoch DLR Proj algorithm}. Assume \eqref{polynomial growth in t} to be valid. Furthermore, let us suppose that there exists a constant $\gamma>0$ such that $|C^{-1}_{Y(s)}|,|C^{-1}_{Y_{n}}|  \leq \frac{1}{\gamma}$ for all $s \in [0,T]$ and for all $n=0,\dots,N-1$. For $\varepsilon$ defined as in \eqref{eps bound projection}, there exists a positive constant $C$ independent of $\gamma$ such that it holds 
		\begin{equation*}
			\sqrt{\mathbb{E}[\max\limits_{0 \leq n \leq N}|X^{\mathrm{EM}}_n-X_{n}|^2]} 
			\leq C \left( \sqrt{\mathbb{E}[|X^{\mathrm{true}}_0 - X_0 |^2]} + (\Delta t)^{\frac{1}{2}} + \varepsilon\right).
		\end{equation*}
		\begin{proof}
			Recall that $a_\ell = a(t_{\ell},U_{\ell}^{\top}Y_{\ell})$, $b_\ell = b(t_{\ell},U_{\ell}^{\top}Y_{\ell})$. We want to compare the numerical DLRA solution and the Euler-Maruyama approximation of \eqref{eq:SDE-diff}. Then we can write:
			\begin{equation*}
				\begin{aligned}
					\mathbb{E}\left[\max\limits_{0 \leq n \leq N} |X_{n}-X^{\mathrm{EM}}_n|^2\right]=& \mathbb{E}\left[ \max\limits_{0 \leq n \leq N} |X_0 - X_0^{\mathrm{true}} +  \sum_{\ell=0}^{n - 1} (X_{n+1}-X_n-(X^{\mathrm{EM}}_{n+1}-X^{\mathrm{EM}}_n))|^2\right] \\
					\leq & \Big( \mathbb{E}\bigg[ \max\limits_{0 \leq n \leq N} \bigg|X_0 -X^{\mathrm{true}}_0 + \sum_{\ell=0}^{n - 1} \Big(P_{U_\ell^{\top}\tilde{Y}_{\ell+1} } \left[ a_\ell \right]\Delta t + P_{U_\ell}\left[b_\ell \Delta W_\ell \right] \\
					& - \left( a(t_{\ell},X^{\mathrm{EM}}_{\ell}) \Delta t + b(t_{\ell},X^{\mathrm{EM}}_{\ell})\Delta W_\ell \right)\Big)\bigg|^2 \bigg] \Big)\\
					\leq &3 \Big(\mathbb{E}[|X_0 -X^{\mathrm{true}}_0|^2]  + \mathbb{E}\left[\max\limits_{0 \leq n \leq N}\left| \sum_{\ell=0}^{n - 1}   \left(P_{U_{\ell}^{\top}\tilde{Y}_{\ell+1} }  a_{\ell} - a(t_{\ell},X^{\mathrm{EM}}_{\ell})\right) \Delta t \right|^2\right] \\
					&+ \mathbb{E}\left[\max\limits_{0 \leq n \leq N}\left| \sum_{\ell=0}^{n - 1} \left( P_{U_\ell} \left[  b_\ell \right]-  b(t_\ell,X^{\mathrm{EM}}_\ell)\right)\Delta W_\ell \right|^2\right] \Big) \\
					\leq &3 \mathbb{E}[|X_0 -X^{\mathrm{true}}_0|^2]  + 3\mathbb{E}\left[\max\limits_{0 \leq n \leq N} n \sum_{\ell=0}^{n - 1}  \left| \left(P_{U_{\ell}^{\top}\tilde{Y}_{\ell+1} }  a_{\ell} - a(t_{\ell},X^{\mathrm{EM}}_{\ell})\right) \Delta t \right|^2\right] \\
					&+ 3\mathbb{E}\left[\max\limits_{0 \leq n \leq N}\left| \sum_{\ell=0}^{n - 1} \left( P_{U_\ell} \left[  b_\ell \Delta W_\ell \right]-  b(t_\ell,X^{\mathrm{EM}}_\ell)\Delta W_\ell \right)\right|^2\right]  \\
					\leq &3 \mathbb{E}[|X_0 -X^{\mathrm{true}}_0|^2]  + 3  T  \underbrace{\sum_{\ell=0}^{N- 1}  \mathbb{E}\left[ \left| \left(P_{U_{\ell}^{\top}\tilde{Y}_{\ell+1} }  a_{\ell} - a(t_{\ell},X^{\mathrm{EM}}_{\ell})\right)  \right|^2 \right]\Delta t  }_{:=\mathrm{T}_1}\\
					&+ 3  \underbrace{\mathbb{E}\left[\max\limits_{0 \leq n \leq N}\left| \left(\sum_{\ell=0}^{n - 1} \left(P_{U_\ell}  \left[ b_\ell \right]- b(t_\ell,X^{\mathrm{EM}}_\ell) \right)\Delta W_\ell \right)\right|^2\right] }_{:=\mathrm{T}_2}, \\
				\end{aligned}
			\end{equation*}
			where in the last line we employ the Jensen's inequality.
			Concerning $\mathrm{T}_1$ we have by \eqref{polynomial growth in t}, \eqref{eps bound projection}, \eqref{Stoch DLR Proj-bound 1} and Assumptions \ref{lipschitz} and \ref{linear-growth-bound} that
			\begin{equation*}
				\begin{aligned}
					\mathbb{E}[|P_{U_{\ell}^{\top}\tilde{Y}_{\ell+1} } a_{\ell}  -  a(t_{\ell},X^{\mathrm{EM}}_{\ell})]|^2 
					= &\mathbb{E}[|- P^{\perp}_{U_{\ell}^{\top}\tilde{Y}_{{\ell}+1} } a_{\ell} + P^{\perp}_{U_{\ell}^{\top}\tilde{Y}_{{\ell}+1} } a_{\ell} + P_{U_{\ell}^{\top}\tilde{Y}_{{\ell}+1} } a_{\ell}   - a(t_{\ell},X^{\mathrm{EM}}_{\ell})|^2]\\
					= &\mathbb{E}[| -P^{\perp}_{U_{\ell}^{\top}\tilde{Y}_{{\ell}+1} } a_{\ell}  + a_{\ell} - a(t_{\ell},X^{\mathrm{EM}}_{\ell})|^2]\\
					= & \mathbb{E}[| P^{\perp}_{U_{\ell}^{\top}\tilde{Y}_{{\ell}+1} } [-a_{\ell} - a(t_{\ell},U_{\ell}^{\top}\tilde{Y}_{{\ell}+1}) + a(t_{\ell},U_{\ell}^{\top}\tilde{Y}_{{\ell}+1}) ] + a_{\ell} - a(t_{\ell},X^{\mathrm{EM}}_{\ell})|^2]\\
					\leq & 3\mathbb{E}[| P^{\perp}_{U_{\ell}^{\top}\tilde{Y}_{{\ell}+1} } [a_{\ell} - a(t_{\ell},U_{\ell}^{\top}\tilde{Y}_{{\ell}+1})]|^2] \\
					&+ 3\mathbb{E}[| P^{\perp}_{U_{\ell}^{\top}\tilde{Y}_{{\ell}+1} }a(t_{\ell},U_{\ell}^{\top}\tilde{Y}_{{\ell}+1}) |^2]  + 3\mathbb{E}[| a_{\ell} - a(t_{\ell},X^{\mathrm{EM}}_{\ell})|^2]\\
					\leq & 3\mathbb{E}[| P^{\perp}_{U_{\ell}^{\top}\tilde{Y}_{{\ell}+1} } [a(t_{\ell},U_{\ell}^{\top}Y_{\ell}) - a(t_{\ell},U_{\ell}^{\top}\tilde{Y}_{{\ell}+1})]|^2] + 3\varepsilon^2 + 3C_{\mathrm{Lip}}\mathbb{E}[| X_{\ell}-X^{\mathrm{EM}}_{\ell}|^2]\\
					\leq & 3 C_{\mathrm{Lip}} \mathbb{E}[| Y_{\ell} - \tilde{Y}_{{\ell}+1}|^2] + 3\varepsilon^2 + 3C_{\mathrm{Lip}}\mathbb{E}[| X_{\ell}-X^{\mathrm{EM}}_{\ell}|^2]\\
					\leq & 3 C_{\mathrm{Lip}} \mathbb{E}[| U_{\ell} a_{\ell} \Delta t + U_{\ell} b_{\ell} \Delta W_{\ell} |^2] + 3\varepsilon^2 + 3C_{\mathrm{Lip}}\mathbb{E}[| X_{\ell}-X^{\mathrm{EM}}_{\ell}|^2]\\
					\leq & 6 C_{\mathrm{Lip}}  C_{\mathrm{lgb}} (1+K_3(T))  \left((\Delta t)^2+ \Delta t\right) + 3\varepsilon^2 + 3C_{\mathrm{Lip}}\mathbb{E}\left[\max\limits_{0 \leq j \leq {\ell}} |X_{j}-X^{\mathrm{EM}}_j|^2\right].
				\end{aligned}
			\end{equation*}
			
			On the other hand, for $\mathrm{T}_2$ we have
			\begin{equation*}
				\begin{aligned}
					P_{U_\ell}  \left[ b_\ell \right]- b(t_\ell,X^{\mathrm{EM}}_\ell) = &\left(P_{U_\ell}\right)^{\perp}  \left[ b_\ell  \right]-  \left(P_{U_\ell}\right)^{\perp}  \left[ b_\ell \right]  +  P_{U_\ell} b_\ell  - b(t_\ell,X^{\mathrm{EM}}_\ell) \\
					= &-  \left(P_{U_\ell}\right)^{\perp}  \left[ b_\ell  \right] 
					+ \left(b_\ell -  b(t_\ell,X^{\mathrm{EM}}_\ell) \right) . \\
				\end{aligned}
			\end{equation*}
			and, hence,
			\begin{equation*}
				\begin{aligned}
					\mathrm{T}_2 \leq & 2 \mathbb{E}[\max\limits_{0 \leq n \leq N}| \sum_{\ell=0}^{n - 1}  \left(P_{U_\ell}\right)^{\perp}  \left[ b_\ell \right] \Delta W_\ell |^2] + 2 \mathbb{E}[ \max\limits_{0 \leq n \leq N} |\sum_{\ell=0}^{n - 1} \left( b_{\ell} -  b(t_{\ell},X^{\mathrm{EM}}_{\ell}) \right) \Delta W_\ell |^2]. \\
				\end{aligned}
			\end{equation*}
			Using Doob's maximal inequality \cite[Theorem 4.4.4]{durrett2019probability}, the independence of the Brownian increments, \eqref{eps bound projection}, and Lipschitz continuity one obtains that
			\begin{equation*}
				\begin{aligned}
					\mathrm{T}_2  \leq  & 8 \mathbb{E}[ | \sum_{\ell=0}^{N - 1}  \left(P_{U_\ell}\right)^{\perp}  \left[ b_\ell \right] \Delta W_\ell |^2] + 8 \mathbb{E}[  |  \sum_{\ell=0}^{N - 1} \left( b_{\ell} -  b(t_{\ell},X^{\mathrm{EM}}_{\ell}) \right) \Delta W_\ell |^2] \\
					 =  & 8 \sum_{\ell=0}^{N - 1} \mathbb{E}[  |  \left(P_{U_\ell}\right)^{\perp}  \left[ b_\ell \right] \Delta W_\ell |^2] + 8 \sum_{\ell=0}^{N - 1} \mathbb{E}[  | \left( b_{\ell} -  b(t_{\ell},X^{\mathrm{EM}}_{\ell}) \right) \Delta W_\ell |^2] \\
					\leq  & 8 \sum_{\ell=0}^{N - 1} \mathbb{E}[\|\left(P_{U_\ell}\right)^{\perp}b_{\ell}\|_{\mathrm{F}}]^2 \Delta t + 8 \sum_{\ell=0}^{N - 1} \mathbb{E}[  | \left( b_{\ell} -  b(t_{\ell},X^{\mathrm{EM}}_{\ell}) \right) |^2] \Delta t \\
					\leq  & 8  T\varepsilon^2 + 8 \sum_{\ell=0}^{N - 1} C_{\mathrm{Lip}} \mathbb{E}[| X_{\ell} - X^{\mathrm{EM}}_{\ell}|^2] \Delta t \\
					\leq  & 8  T\varepsilon^2 + 8 \sum_{\ell=0}^{N - 1} C_{\mathrm{Lip}} \mathbb{E}[ \max\limits_{0 \leq j \leq {\ell}}  | X_{j} - X^{\mathrm{EM}}_{j}|^2] \Delta t. \\
				\end{aligned}
			\end{equation*}
			
			Finally, by Gronwall's lemma we have
			\begin{equation*}
				\begin{aligned}
					\mathbb{E}\left[ \max\limits_{0 \leq n \leq N} |X_{n}-X^{\mathrm{EM}}_n|^2\right] \leq & 3 \mathbb{E}[| X_0-X^{\mathrm{EM}}_0 |^2] + 24T\varepsilon^2  \\
					&+ T \big( 18 C_{\mathrm{Lip}}C_{\mathrm{lgb}} (1+K_3(T))  \left((\Delta t)^2+ \Delta t\right) + 9T^2\varepsilon^2 \\
					&+ 3C_{\mathrm{Lip}}  (3T+8 ) \sum_{\ell=0}^{N - 1} \mathbb{E}[ \max\limits_{0 \leq j \leq {\ell}}  | X_{j} - X^{\mathrm{EM}}_{j}|^2] \Delta t  \\
					\leq & \Big[3 \mathbb{E}[|X^{\mathrm{true}}_0 - X_0 |^2] + \left((\Delta t)^2+ \Delta t\right) \\
					& \cdot \left[ 18 C_{\mathrm{Lip}} C_{\mathrm{lgb}} (1+K_3(T))T \right] + 3T(8+3T)\varepsilon^2\Big] e^{3C_{\mathrm{Lip}}  (3T+8 ) T}.
				\end{aligned}
			\end{equation*}
			These computations yield the result.
		\end{proof}
	\end{Theorem}
	
	\begin{Theorem}[Numerical Convergence of the DLR Projector Splitting for SDEs] \label{thm: Numerical Convergence Stoch DLR Proj - True}
		Denote by $X^{\mathrm{true}}(t)$ the strong solution of \eqref{eq:SDE-diff} and by
		$X_n$ the numerical solution obtained by Algorithm \ref{alg: Stoch DLR Proj algorithm}. Suppose Assumptions \eqref{lipschitz}-\eqref{linear-growth-bound} hold as well as \eqref{polynomial growth in t} and \eqref{eps bound projection}. Then, there exists a constant $C>0$ independent of the smallest singular value of the covariance matrices $\mathbb{E}[Y(t)Y^{\top}(t)]$ for all $t \in [0,T]$ and $\mathbb{E}[Y_nY^{\top}_n]$ for all $n \in \{0,1, \dots, N\}$ such that
		\begin{equation*}
			\sqrt{\mathbb{E}[\max\limits_{0 \leq n \leq N}|X^{\mathrm{true}}(t_n)-X_{n}|^2]} 
			\leq C \left( \sqrt{\mathbb{E}[|X^{\mathrm{true}}_0 - X_0 |^2]} + (\Delta t)^{\min\{\frac{1}{2},\alpha\}} + \varepsilon\right).
		\end{equation*}
			\begin{proof}
			Let us denote by $X^{\mathrm{EM}}_n$ the Euler-Maruyama numerical solution of \eqref{eq:SDE-diff}. It holds
			\begin{equation*}
				\sqrt{\mathbb{E}[\max\limits_{0 \leq n \leq N}|X_{n} - X^{\mathrm{true}}(t_n)|^2]} \leq \sqrt{\mathbb{E}[\max\limits_{0 \leq n \leq N}|X_{n} - X^{\mathrm{EM}}_n|^2]} + \sqrt{ \mathbb{E}[\max\limits_{0 \leq n \leq N}| X^{\mathrm{EM}}_n - X^{\mathrm{true}}(t_n)|^2]}.
			\end{equation*}
			Considering the convergence of Euler--Maruyama method \cite[Theorem 10.22]{kloeden1992stochastic}, one has
			\begin{equation*}
				\sqrt{\mathbb{E}[\max\limits_{0 \leq n \leq N}|X^{\mathrm{EM}}_n - X^{\mathrm{true}}(t_n)|^2]} \leq C (\Delta t)^{\min\{ \frac{1}{2}, \alpha\} }.
			\end{equation*}
			Using Theorem \ref{thm: Numerical Convergence Stoch DLR Proj - DLRA}, the thesis follows. 		
		\end{proof}
	\end{Theorem}
	We finally state a last result comparing the DLR PS SDE solution to the DLRA one. This result will be needed in Part II when analyzing the stochastic discretization error. Unfortunately, the error bound given in the next lemma does have constants that blow up when the smallest singular value of either the continuous DLRA or its numerical approximation goes to zero.
	\begin{Theorem}[Strong Convergence of DLR Projector Splitting for SDEs with respect to the DO solution]\label{thm: convergence of Stoch DLR Proj}
		Let Assumptions \eqref{lipschitz} and \eqref{linear-growth-bound}, as well as \eqref{polynomial growth in t} hold. 
		Consider a uniform partition $\Delta := \left\{t_n \ : \ 0= t_0 < t_1 < \ldots < t_{N-1} < t_{N} = T\right\}$ of $[0,T]$ and the time step $\Delta t = t_{n+1}-t_n$ for all $n$.
		Furthermore, let us suppose that there exists a constant $\gamma>0$ such that $|C^{-1}_{Y(s)}|,|C^{-1}_{Y_{n_s}}|  \leq \frac{1}{\gamma}$ for all $s \in [0,T]$. 
		
		Then, the DLR PS SDE method with constant time step $\Delta t$ has strong order of convergence equal to ${\min\{\frac{1}{2}, \alpha \} }$, i.e.\ there exists a constant $C:=C(\gamma, k,T)$ independent of $\Delta t$ such that for all $0 \leq t \leq T$:
		\begin{equation*}
			\sqrt{\mathbb{E}[\max\limits_{0 \leq n \leq N} |X_{n} - X(t_n)|^2]} \leq C (\Delta t)^{\min\{ \frac{1}{2}, \alpha\} } + o((\Delta t)^{\min\{ \frac{1}{2}, \alpha\} }).
		\end{equation*}
		\begin{proof} The proof is deferred to Appendix \ref{app:DLR PS SDEs}.
		\end{proof}
	\end{Theorem}

	\subsection{Mean Square Stability of the DLR Projector Splitting for SDEs}\label{sec: Stoc Proj Stab}
	In this section, we study the mean-square stability of the proposed DLR Projector Splitting for SDEs. Its projected nature turns out to be extremely beneficial to establish an easy-to-verify sufficient condition for mean-square stability. 
	Assuming that
	\begin{equation*}
		a(t,0) \equiv 0, \qquad b(t,0) \equiv 0, \quad t \in [0,T],
	\end{equation*}
	we will consider the following definition of $p$-stability \cite{khasminskii2011stochastic, mao2007stochastic} throughout this work.
	\begin{Definition}[$p$-stability and asymptotical $p$-stability]
		Let us denote with $X^{s,x}$ the solution of \eqref{eq:SDE-diff} with initial condition $X^{\mathrm{true}}(s) = x \in \mathbb{R}^d$, with $s>0$. A solution $X^{\mathrm{true}}(t,\omega) \equiv 0$ of \eqref{eq:SDE-diff} is said to be $p$-stable for $t \geq 0$ if
		$$\lim\limits_{\delta \to 0} \sup\limits_{\{|x|\leq\delta, t \geq s\}} \mathbb{E}[|X^{s,x}(t)|^p]  = 0.$$
		Moreover, if there exists $\delta := \delta(s)>0$ such that
		$$\lim\limits_{t \to +\infty} \mathbb{E}[|X^{s,x}(t)|^p]  = 0$$
		also holds for all $s >0$ and $|x|<\delta$, then $X(t,\omega) \equiv 0$ is said to be asymptotically $p$-stable. 
	\end{Definition}
	We will focus, in particular, to the case $p=2$, which is the so-called \textit{mean-square stability}. A mean-square stable system is also mean stable ($p=1$). For the sake of completeness, we recall that for autonomous linear equations (asymptotical) $p$-stability implies (asymptotical) stability in probability \cite{arnold1974stochastic,khasminskii2011stochastic}, namely
	for all $s\geq0$, and for all $\varepsilon>0$
	$\lim\limits_{x \to 0} \mathbb{P}\{\sup\limits_{ t \geq s} | X^{s,x} (t)| > \varepsilon\} = 0$,
	and the asymptotical $p$-stability implies the asymptotical stability in probability, i.e.\ $\lim\limits_{x \to 0} \mathbb{P}\{\lim\limits_{ t \to \infty} | X^{s,x} (t)| > \varepsilon\} = 0$.
	
	To study the stability of the DLR Projector splitting scheme for SDEs, we restrict to the case of a linear SDE \cite{saito1996stability,higham2000stability,saito2002mean}, as it is common in the literature.
	More precisely, given time-dependent deterministic matrices $A(t), B_{k}(t) \in \mathbb{R}^{d \times d}$, for $k = 1, \dots, m$, we consider the SDE:
	\begin{equation}\label{eq: moltiplic sdes}
		\begin{aligned}
			\mathrm{d}X^{\mathrm{true}}(t) & = A(t)X^{\mathrm{true}}(t)\mathrm{d}t+ \sum_{k=1}^{m} B_k(t) X^{\mathrm{true}}(t) \mathrm{d}W^{k}_t, \quad \text{ for all } t \in [0, +\infty),
			X_0 = c \in \mathbb{R}^d,
		\end{aligned} 
	\end{equation}
	where the initial datum $X_0$ can be deterministic or random, as it does not have any influence on the stability analysis \cite{mao2007stochastic}. Via an It\^o's formula argument we find the following inequality
	\begin{equation}\label{eq: ams bound}
		\begin{aligned}		
			\mathrm{d}\mathbb{E}[(X^{\mathrm{true}}(t))^{\top}X^{\mathrm{true}}(t)]
			\leq & \left(\lambda_{\max}(A(t)+ A^{\top}(t)) + \sum\limits_{k=1}^{m}  |B_k(t)|^2 \right) \mathbb{E}[|X^{\mathrm{true}}(t)|^2] \mathrm{d}t,
		\end{aligned}
	\end{equation}
	which implies that \eqref{eq: moltiplic sdes} is asymptotically mean-square (AMS) stable if 
	\begin{equation}\label{eq: cond stab sde}
		\lambda_{\max}(A(t)+ A^{\top}(t)) < - \sum\limits_{k=1}^{m} |B_k(t)|^2.
	\end{equation}
	As the next proposition shows, a sufficient condition for AMS stability of the DLR approximation of \eqref{eq: moltiplic sdes} is still \eqref{eq: cond stab sde}, that is, under \eqref{eq: cond stab sde}, both the SDE and its continuous DLRA are AMS stable.
	\begin{Proposition}[Mean-square stability of the continuous DLRA]\label{prop: stab DLRA}
		Under condition  \eqref{eq: cond stab sde}, the DLR approximation of \eqref{eq: moltiplic sdes} is AMS stable.
		\begin{proof}
			The proof is deferred to Appendix \ref{app:DLR PS SDEs}.
		\end{proof}
	\end{Proposition}
	One can notice that no dependence on the smallest singular value is present in \eqref{eq: ams bound}, as the deterministic modes evolve independently of the stochastic ones for this simple linear model. 
	
	Let us now turn to the numerical discretization. The usual Euler-Maruyama method for equation \eqref{eq: moltiplic sdes} reads as follows:
	\begin{equation*}
		X_{n+1} = X_n + A(t_n) X_n \Delta t_n + \sum_{k=i}^{m} B_k(t_n) X_n \Delta W^{k}_n,
	\end{equation*}
	and, its mean-square norm can be bounded as
	\begin{equation*}
		\begin{aligned}
			\mathbb{E}[|X_{n+1}|^2] = &\mathbb{E}[|X_{n}|^2] + \mathbb{E}[X_n^{\top}A(t_n)X_n ]\Delta t_n + \mathbb{E}[X_n^{\top}A^{\top}(t_n) X_n] \Delta t_n +  \mathbb{E}[X_n^{\top}A^{\top}(t_n) A(t_n)X_n] (\Delta t_n)^2\\
			& + \mathbb{E}[X_n^{\top}\sum_{k=1}^{m} B_k^{\top}(t_n) B_k(t_n)  X_n] (\Delta t_n) \\
			\leq & \mathbb{E}[|X_{n}|^2] \left(1 + \lambda_{\max}(A(t_n) + A^{\top}(t_n) )\Delta t_n + \sigma_{\max}(A(t_n))^2 (\Delta t_n)^2 + \sum_{k=1}^{m}|B_k(t_n)|^2 (\Delta t_n)\right).
		\end{aligned}
	\end{equation*}
	Therefore, a sufficient condition for AMS stability for the Euler-Maruyama scheme is
	\begin{equation}\label{eq: stab em cond}
		\left|1 + \lambda_{\max}(A(t_n) + A^{\top}(t_n) )\Delta t_n + \sigma_{\max}(A(t_n))^2 (\Delta t_n)^2 + \sum_{k=1}^{m} |B_k(t_n)|^2 \Delta t_n\right| < 1.
	\end{equation}
	
	The next proposition shows that the DLR Projector Splitting for SDEs (Algorithm \ref{alg: Stoch DLR Proj algorithm}) is AMS stable under the same condition \eqref{eq: stab em cond} as the standard Euler-Maruyama scheme. In particular the stability condition \eqref{eq: stab em cond} is independent of the smallest singular value of the Gramians $C_{Y_n}$.
	\begin{Proposition}[Mean-square stability of the DLR Projector Splitting for SDEs]\label{prop: PS SDE stab}
		Under condition  \eqref{eq: stab em cond}, the DLR Projector Splitting for SDEs applied to \eqref{eq: moltiplic sdes} is AMS stable.
	\begin{proof}
		The DLR PS SDE scheme applied to equation \eqref{eq: moltiplic sdes} reads as follows:
		\begin{equation*}
			X_{n+1} = X_n + \underbrace{\left(\left(I_{d \times d} - P_{{U}_n} \right) P_{\tilde{Y}_{n+1}} + P_{{U}_n}\right)}_{P_{{U}^{\top}_n\tilde{Y}_{n+1}}} [ A(t_n) X_n ] \Delta t_n + P_{{U}_n} [\left(\sum_{k=1}^{m} B_k(t_n) X_n \Delta W^{k}_n\right)].
		\end{equation*}
		It follows that
		\begin{equation*}
			\mathbb{E}[X_n^{\top}\left( X_{n+1}-X_n\right)] = \mathbb{E}[X_n^{\top} A(t_n) X_n]  \Delta t_n,
		\end{equation*}
		and, via using Lemma \ref{lem: technical PS}, that
		\begin{equation*}
			\begin{aligned}
				\mathbb{E}[\left( X_{n+1}-X_n\right)^{\top}\left( X_{n+1}-X_n\right)]
				& \leq \mathbb{E}[|A(t_n) X_n|^2] (\Delta t_n)^2 + \sum_{k=1}^{m} \mathbb{E}[|B_k(t_n)  X_n|^2] (\Delta t_n) \\
				& \leq \left(\sigma_{\max}(A(t_n))^2 (\Delta t_n)^2 + \sum_{k=1}^{m}|B_k(t_n)|^2 (\Delta t_n)\right)\mathbb{E}[|X_n|^2].
			\end{aligned}
		\end{equation*}
		Exploiting the relation
		\begin{equation}\label{eq: cross product relation}
			\mathbb{E}[X_n^{\top}\left( X_{n+1}-X_n\right)] + \mathbb{E}[\left( X_{n+1}-X_n\right)^{\top} X_n] = \mathbb{E}[|X_{n+1}|^2]  - \mathbb{E}[|X_{n}|^2] -\mathbb{E}[\left( X_{n+1}-X_n\right)^{\top}\left( X_{n+1}-X_n\right)],
		\end{equation}
		we obtain
		\begin{equation}\label{eq: KNV mean square stab}
			\begin{aligned}
				\mathbb{E}[|X_{n+1}|^2]  \leq & \mathbb{E}[|X_{n}|^2] +  \mathbb{E}[X_n^{\top} (A(t_n) + A^{\top}(t_n) ) X_n] \Delta t_n +  \Big(\sigma_{\max}(A(t_n))^2 (\Delta t_n)^2 \\
				&+ \sum_{k=1}^{m}|B_k(t_n)|^2 (\Delta t_n)\Big)\mathbb{E}[|X_n|^2] \\
				\leq & \mathbb{E}[|X_{n}|^2] \left(1 + \lambda_{\max}(A(t_n) + A^{\top}(t_n) )\Delta t_n + \sigma_{\max}(A(t_n))^2 (\Delta t_n)^2 + \sum_{k=1}^{m}|B_k(t_n)|^2 (\Delta t_n)\right),
			\end{aligned}
		\end{equation}
		hence the scheme is AMS stable under \eqref{eq: stab em cond}.
	\end{proof}
	\end{Proposition}

\section{Analysis of the DLR Projector Splitting for EM}\label{sec: convergence PS EM}
In this section, {convergence and mean-square stability analyses} of Algorithm \ref{alg: Eva Proj Splitt SDE algorithm} are presented. By being a projected scheme, one can retrieve computational results akin to the ones of Section \ref{sec: convergence PS SDE} via similar proofs. The main difference with the previous two schemes, is that Algorithm \ref{alg: Eva Proj Splitt SDE algorithm} converges to the solution of the modified DLR equations \eqref{eq: DO cao-lu}, rather than \eqref{DLR-conditions}. Moreover, we have been unable to obtain convergence results independent of the smallest singular value, although, numerically, this Algorithm seem to perform even better than Algorithm \ref{alg: Stoch DLR Proj algorithm}.

 \begin{Lemma}[$L^2$-norm bound of DLR Projector Splitting for EM]\label{lem: discrete gronwall PS EM}
 	Let $X_n= U_n^{\top}Y_n$, $n=0,\dots,N$ be the solution produced by
 	Algorithm \ref{alg: Eva Proj Splitt SDE algorithm} with an arbitrary sequence $\{\Delta t_n\}_n$ of (non-zero) step-sizes such that $\sum_{n=0}^{N-1} \Delta t_n=  T$. Under Assumptions \ref{linear-growth-bound}-\ref{eq:initial value}, the following bound on the mean square norm of $X_n$ holds
 	\begin{equation}\label{eva gronwall-bound 1}
 		\mathbb{E}[|X_{n}|^2] \leq \left( \mathbb{E}[|X_0|^2] + 1\right) e^{( 1+ C_{\mathrm{lgb}}(2+ T))T} - 1 :=K_3(T).
 	\end{equation}
 	\begin{proof}
 		The proof follows closely to the one of Lemma \ref{lem: bound stoch dlr proj}.
 	\end{proof}
 \end{Lemma}

\begin{Proposition} \label{prop: discrete covariance full-rank - KNV}
		Suppose Assumption \ref{ass: diff} is satisfied. Then for any sequence of step-sizes $\{\Delta t_n\}_n$ we have:
		\begin{equation}\label{eq: cov 1 KNV}
			C_{Y_{n+1}} \succ \sigma_{B} \Delta t_n \cdot I_{k \times k}.
		\end{equation}
		If, moreover, Assumption \ref{ass: C_Y_0} holds as well, then for a uniform step-size $\Delta t$ one has
		\begin{equation}\label{eq: cov 2 KNV}
			\begin{aligned}
					C_{Y_{n+1}} \succeq  & \min \{\sigma_{0}, \frac{\sigma_{B}^2}{4C_{\mathrm{lgb}}(1 + K_3(T))}\}.
			\end{aligned}
		\end{equation}
			\begin{proof}
			The proof follows closely to the one of Propositions \ref{prop: 1 discrete covariance full-rank - EM} and \ref{prop: 2 discrete covariance full-rank - EM}.
		\end{proof}
\end{Proposition}
 
The following result establishes the accuracy of the DLR solution defined by \eqref{eq: DO cao-lu} and its numerical approximation given by Algorithm \ref{alg: Eva Proj Splitt SDE algorithm}.

\begin{Theorem}[Numerical Convergence of Algorithm \ref{alg: Eva Proj Splitt SDE algorithm} to the true solution] \label{thm: Numerical Convergence Eva}
	Consider a uniform partition $\Delta := \left\{t_n = t_0 + n \Delta t, \ n=0,\dots, N\right\}$ of $[0,T]$ with time-step $\Delta t=\frac{T}{N}$. Let us denote by $X^{\mathrm{true}}(t)$ the solution of \eqref{eq:SDE-diff} and by $X_n$ the numerical DLR solution obtained with Algorithm \ref{alg: Eva Proj Splitt SDE algorithm}. Assume \eqref{polynomial growth in t} to be valid. Furthermore, let us suppose that there exists a constant $\gamma>0$ such that $|C^{-1}_{Y(s)}|,|C^{-1}_{Y_{n}}|  \leq \frac{1}{\gamma}$ for all $s \in [0,T]$ and for all $n=0,\dots,N-1$, where $Y(s)$ is the stochastic basis in \eqref{eq: DO cao-lu} and $Y_n$ is the numerical approximation of the stochastic basis. Then, given $\varepsilon$ defined as in \eqref{eps bound projection} it holds 
	\begin{equation*}
		\sqrt{\mathbb{E}\left[ \max\limits_{0 \leq n \leq N}|X_{n} - X^{\mathrm{true}}(t_n)|^2\right]} \leq C( \sqrt{\mathbb{E}[| X_0 - X^{\mathrm{true}}_0 |^2]} + (\Delta t)^{\min\{\frac{1}{2},\alpha\}}  + \varepsilon),
	\end{equation*}
	for a positive constant $C:=C(\gamma,k,T)$.
	\begin{proof}
	To obtain the thesis, we bound the sought error with the one between the true solution and the continuous DLRA $X(t)$ given by \eqref{eq: DO cao-lu} and the one between the DLRA and the numerical solution $X_n$. The latter can be obtained following closely the proof of Theorem \ref{thm: convergence of DLR Euler-Maruyama}, i.e.\ via bounding the error between the bases.  Due to its similarity to the proof of Theorem \ref{thm: convergence of DLR Euler-Maruyama}, we streamline the argument and we only highlight the main differences.
		
		First, let us define the DLRA $X(t)=U(t)^{\top}Y(t)$ as the product of the bases defined by \eqref{eq: DO cao-lu}, i.e.\ via Itô's formula
		  \begin{equation}\label{eq: DLRA cao}
			\begin{aligned}
				X(t)  =& X_0 + \int_{0}^{t} \left(P_{X(s)} \left[a(s,X(s)) \right] + P_{U(s)}^{\perp}\mathbb{E} \left[b(s,X(s)) b(s,X(s))^{\top}\right] U(s)^{\top}C^{-1}_{Y(s)}Y(s) \right) \mathrm{d}s \\
				&+  \int_{0}^{t} P_{U(s)} b(s,X(s))\mathrm{d}W_s.
			\end{aligned}
		\end{equation}
		Furthermore, we can assume that there exists a positive constant $K_4(T)$ such that $\mathbb{E}[ \sup\limits_{0 \leq s \leq t} |X(s) |^2] \leq K_4(T)$ (see \cite{kazashi2025dynamical}).
	Then, similarly to \eqref{eq: sum of error}, one has
		 \begin{equation*}
			\begin{aligned}
				\sqrt{\mathbb{E}[\max\limits_{0 \leq n \leq N} |X^{\mathrm{true}}(t_n) - X_{n} |^2]}
				\leq &\sqrt{ \mathbb{E}[ \sup\limits_{0 \leq s \leq t} |X^{\mathrm{true}}(s) - X(s) |^2]} + \sqrt{  \mathbb{E}[\max\limits_{0 \leq n \leq N} |X(t_n) - X_{n} |^2 ]}.
			\end{aligned}
		\end{equation*}
		For the error between $X^{\mathrm{true}}(t)$ and $X(t)$ one has
			\begin{equation*}
			\begin{aligned}
				&\mathbb{E}[\sup\limits_{0 \leq s \leq t} | X^{\mathrm{true}}(s) - X(s) |^2]\\
				\leq &  4 \mathbb{E}[|X^{\mathrm{true}}_0 - X_0 |^2]  + 4 \mathbb{E}[\sup\limits_{0 \leq s \leq t} |\int_{0}^{s}a(r,X^{\mathrm{true}}(r)) -  P_{X(r)}[a(r,X(r))]  \mathrm{d}r|^2] \\
				& + 4 \mathbb{E}[\sup\limits_{0 \leq s \leq t} |\int_{0}^{s}P_{U(r)}^{\perp}\mathbb{E} \left[b(r,X(r)) b(r,X(r))^{\top}\right] U(r)^{\top}C^{-1}_{Y(r)}Y(r) \mathrm{d}r|^2] \\
				&+ 4 \mathbb{E}[\sup\limits_{0 \leq s \leq t} |\int_{0}^{s}b(r,X^{\mathrm{true}}(r))-P_{U(r)} b(r,X(r))\mathrm{d}W_r|^2].
			\end{aligned}
		\end{equation*}
		Via adding and subtracting the drift $a(r,X(r))$ and the diffusion $b(r,X(r))$, Lipschitzianity of $a$ and $b$, Jensen's inequality, Doob's inequality and using \eqref{eps bound projection}, one obtains
			\begin{equation*}
			\begin{aligned}
				&\mathbb{E}[\sup\limits_{0 \leq s \leq t} | X^{\mathrm{true}}(s) - X(s) |^2]\\
				\leq &  4 \mathbb{E}[|X^{\mathrm{true}}_0 - X_0 |^2]  + 8C_{\mathrm{Lip}} (T+4)T \int_{0}^{T}\mathbb{E}[\sup\limits_{0 \leq p \leq r}  | X^{\mathrm{true}}(p) - X(p) |^2]  \mathrm{d}r \\
				& + \left(8(T+4)T + 8T^2 C_{\mathrm{lgb}}(1+K_4(T))\gamma^{-1} \right) \varepsilon^2 \\
				\leq & C_1(\gamma, T)\left( \mathbb{E}[|X^{\mathrm{true}}_0 - X_0 |^2] + \varepsilon^2\right),
			\end{aligned}
		\end{equation*}
		where in the last line we use the Gronwall's lemma and $C_1(\gamma, T)$ is a positive constant.
		
		Concerning the error $\sqrt{  \mathbb{E}[\max\limits_{0 \leq n \leq N} |X(t_n) - X_{n} |^2 ]}$, the procedure is very close to the proof of Theorem \ref{thm: convergence of DLR Euler-Maruyama}. We detail the extra difference term that appears here, where we use the notation $b_{n_r} = b(t_{n_r}, X_{n_r})$, where $n_r$ is defined as in \eqref{eq: n_s}:
			\begin{equation*}
			\begin{aligned}
				T_7:=&\mathbb{E}\left[\max\limits_{0 \leq n \leq N} \left|  \left(\int_{0}^{t_n} P_{U(r)}^{\perp}\mathbb{E} \left[b(r,X(r)) b(r,X(r))^{\top}\right] U(r)^{\top}C^{-1}_{Y(r)}Y(r) - P_{U_{n_r}}^{\perp} \mathbb{E}\left[ b_{n_r}  b_{n_r}^{\top}  \right]U_{n_r}^{\top}C^{-1}_{\tilde{Y}_{n_r+1}}\tilde{Y}_{n_r+1}  \mathrm{d}r \right) \right|^2 \right] \\
			 \leq & \mathbb{E}\left[\sup\limits_{0 \leq n \leq N} t_n \int_{0}^{t_n} \left| P_{U(r)}^{\perp}\mathbb{E} \left[b(r,X(r)) b(r,X(r))^{\top}\right] U(r)^{\top}C^{-1}_{Y(r)}Y(r)  - P_{U_{n_r}}^{\perp} \mathbb{E}\left[ b_{n_r}  b_{n_r}^{\top}  \right]U_{n_r}^{\top}C^{-1}_{\tilde{Y}_{n_r+1}}\tilde{Y}_{n_r+1} \right|^2 \mathrm{d}r \right] \\
				\leq & T\int_{0}^{T} \mathbb{E}\left[ \left| P_{U(r)}^{\perp}\mathbb{E} \left[b(r,X(r)) b(r,X(r))^{\top}\right] U(r)^{\top}C^{-1}_{Y(r)}Y(r)  - P_{U_{n_r}}^{\perp} \mathbb{E}\left[ b_{n_r}  b_{n_r}^{\top}  \right]U_{n_r}^{\top}C^{-1}_{\tilde{Y}_{n_r+1}}\tilde{Y}_{n_r+1} \right|^2 \right] \mathrm{d}r 
			\end{aligned}
		\end{equation*}
	  Adding and subtracting cross terms and employing usual inequalities and assumptions, we finally get
	  	\begin{equation*}
	  	\begin{aligned}
	  		T_7
	  		\leq & C_2(\gamma,k,T)\left( \left(\int_{0}^{T} \mathbb{E}\left[\max\limits_{0 \leq p \leq r} \left| X(t_{n_p}) - X_{n_p}\right|^2\right] \mathrm{d}r\right) + \varepsilon^2 + \Delta t + O(\Delta t^2)\right),\\
	  	\end{aligned}
	  \end{equation*}
	  for a positive constant $C_2(\gamma,k,T)$. Then, the conclusion follows applying Gronwall's lemma.
	\end{proof}
\end{Theorem}

\begin{Remark}[Comments on convergence in Theorem \ref{thm: Numerical Convergence Eva}]\label{rmk: conv PS EM}
	The presence of the full diffusion term in the expectation given by the $L^2$-projection onto the stochastic basis makes the numerical solution be non-adapted to the filtration, which adds more difficulties in proving the convergence of the DLR PS EM. In the special case of zero model error $$\sup_{\substack{t \in [0,T] \\ Z \in L^2(\Omega,\mathbb{R}^d) \text{ of rank } k \\
			\text{ with } ·\mathbb{E}[|Z|^2] \leq \tilde{K}}} \mathbb{E}[\| b(t, Z) - \mathcal{P}_{\mathcal{U}(Z)}b(t, Z)\|^2_{\mathrm{F}} ] = 0,$$
		 the DLRA equations described by \eqref{DLR-conditions} and \eqref{eq: DO cao-lu} coincide. Therefore, Algorithm \ref{alg: Eva Proj Splitt SDE algorithm} converges to the DLRA for SDEs defined by \eqref{DLR-conditions}, too. 
\end{Remark}
 
\subsection{Mean Square Stability of the DLR Projector Splitting for EM}\label{sec: KNV stab}
Finally, let us analyze the AMS stability of the DLR PS EM applied to the linear SDE \eqref{eq: moltiplic sdes}. Having the same projected nature of the DLR PS SDE turns out to be favorable in term of stability. Indeed, the following proposition shows that the DLR projector splitting scheme for EM is AMS stable under a condition that does not depend on the smallest singular value of the numerical DLR solution. 

\begin{Proposition}[Mean-square stability of the DLR Projector Splitting for EM]\label{prop: stab proj spli}
		Under condition \eqref{eq: stab em cond}, the DLR PS EM applied to \eqref{eq: moltiplic sdes} is AMS stable.
	\begin{proof}
		The proof follows similarly to the one of Proposition \ref{prop: PS SDE stab}.
	\end{proof}
\end{Proposition}

\section{Numerical Experiments}
\label{sec: numerical experiments}
In this section, we validate the results of Sections \ref{sec: DLR EM},  \ref{sec: convergence PS SDE}, and \ref{sec: convergence PS EM} via computational experiments. 
All the averages in the computations in Algorithms \ref{alg: DLR EM SDE algorithm}, \ref{alg: Stoch DLR Proj algorithm}, and \ref{alg: Eva Proj Splitt SDE algorithm} as well as those needed in computing the $L^2(\Omega)$-errors, are calculated using a Monte-Carlo method (sample averages) using independent Brownian motions. This effectively leads to a coupled particle system which is fully analyzed in our companion paper Part II. Here we assume that the particle system converges to its continuous version of Algorithms \ref{alg: DLR EM SDE algorithm}, \ref{alg: Stoch DLR Proj algorithm}, and \ref{alg: Eva Proj Splitt SDE algorithm} when the number of Monte Carlo samples go to $+\infty$.
The orthonormalization procedure done in the three aforementioned algorithms is always performed using a \texttt{QR} decomposition.

In all the simulations, we plot two types of errors for the three algorithms presented in Section~\ref{sec: discretization procedure}. 
The first is the strong error which is measured against the continuous DLRA defined by \eqref{DLR-conditions}. Since a closed-form of the continuous DLRA solution is not available, we approximate it using a DLR PS SDE on a very finer time grid, with the same number of particles as the numerical surrogate being tested. 
The reference DLRA is computed using the DLR PS SDE, as this is the only algorithm for which we have been able to prove convergence to the continuous DLRA with respect to the time discretization without any dependence on the smallest singular value of the Gramian (see Theorem~\ref{thm: convergence of Stoch DLR Proj}), as well as with respect to the stochastic discretization via Monte Carlo sampling (see \cite{kazashi2025dynamicalpartII} for more details). For the sake of notation, in the caption of the figures of this section sometimes we will denote the continuous DLRA as $DLR$, whereas $DLR_{n}$ will denote any of the proposed time discretization schemes. Furthermore, we highlight the fact that if the model error vanishes (i.e.\ the exact solution is low rank), then Algorithms \ref{alg: Stoch DLR Proj algorithm} and \ref{alg: Eva Proj Splitt SDE algorithm} coincide. Otherwise, we might observe that the error between the numerical solution of Algorithm \ref{alg: Eva Proj Splitt SDE algorithm} and the $DLR$ saturates, as $\Delta t \to 0$ since Algorithm \ref{alg: Eva Proj Splitt SDE algorithm} effectively approximates the alternative DLRA formulation \eqref{eq: DO cao-lu}.

The second type of error is also a strong error, but measured against the true solution $X^{\mathrm{true}}$ of the SDE under study. 
In all the simulations, we approximate the exact solution using a standard Euler-Maruyama method \cite{kloeden1992stochastic} with the same number of time steps and paths used to compute the reference ``continuous" DLR. We will refer to $L^2-\sup_{t \in [0,T]}$ as the following error between the two processes $X$, the continuous DLRA, and $X^{\mathrm{true}}$, the true solution,
\begin{equation}\label{eq: err-sup}
	\sqrt{\mathbb{E}[\sup_{t \in [0,T]}|X(t)-X^{\mathrm{true}}(t)|^2]} \approx 	\sqrt{\mathbb{E}[\max_{n \in [0,N]}|X_n-X^{\mathrm{true}}(t_n)|^2]}
\end{equation}
whereas the relative (rel.) $L^2-\sup_{t \in [0,T]}$ error will indicate
\begin{equation}\label{eq: rel. err-sup}
\sqrt{\frac{\mathbb{E}[\sup_{t \in [0,T]} |X(t)-X^{\mathrm{true}}(t)|^2]}{\mathbb{E}[\sup_{t \in [0,T]}|X^{\mathrm{true}}(t)|^2]}} \approx	\sqrt{\frac{\mathbb{E}[\max_{n \in [0,N]} |X_n-X^{\mathrm{true}}(t_n)|^2]}{\mathbb{E}[\max_{n \in [0,N]}|X^{\mathrm{true}}(t_n)|^2]}}.
\end{equation}
In the presented figures of this section, we will always consider \eqref{eq: rel. err-sup}.

Throughout this section, $\sigma^i(C)$ denotes
 the $i$-th largest singular value of a matrix $C \in \mathbb{R}^{n \times p}$. Moreover, for the sake of notation, in the legends of the presented figures, Algorithms \ref{alg: DLR EM SDE algorithm}, \ref{alg: Stoch DLR Proj algorithm}, and \ref{alg: Eva Proj Splitt SDE algorithm} will be denoted by DLR EM, DLR PS  SDE, and DLR PS EM, respectively.
\subsection{Basic SDE examples}\label{sec: basic sde}
First, we aim to corroborate the results of Propositions \ref{prop: 1 discrete covariance full-rank - EM}, \ref{prop: 2 discrete covariance full-rank - EM}, \ref{prop: discrete covariance full-rank - Stoch DLR Proj}, and \ref{prop: discrete covariance full-rank - KNV}, as well as numerically assess the convergence of the three algorithms. 
We simulate the following low-dimensional stochastic process 
\begin{equation}\label{ex: toy example}
	\begin{aligned}
	\hspace{-0.3cm}
	\begin{pmatrix}
		\mathrm{d}X^{\mathrm{true}}_1 (t) \\
		\mathrm{d}X^{\mathrm{true}}_2 (t) \\
		\mathrm{d}X^{\mathrm{true}}_3 (t)\\
	\end{pmatrix}
	=&
	\begin{pmatrix}
		-0.1 \cdot X^{\mathrm{true}}_1(t) &0.1 \cdot X^{\mathrm{true}}_2(t)&0.001 \cdot X^{\mathrm{true}}_3(t)\\
		-0.1 \cdot X^{\mathrm{true}}_1(t) &0.1 \cdot X^{\mathrm{true}}_2(t)&0.001 \cdot X^{\mathrm{true}}_3(t)\\
	-4 \cdot X^{\mathrm{true}}_1(t) &-4 \cdot X^{\mathrm{true}}_2(t)&-4 \cdot X^{\mathrm{true}}_3(t)\\
	\end{pmatrix}
	\mathrm{d}t\\
	&+ \sqrt{\sigma_B}
	\begin{bmatrix}
		1 + 6 \left(|X^{\mathrm{true}}_1(t)|+ |X^{\mathrm{true}}_2(t)|\right)& 0 & 0  \\
		0 & 	1 + 6 \left(|X^{\mathrm{true}}_1(t)|+ |X^{\mathrm{true}}_2(t)|\right) & 0  \\
		0 & 0 & 1 \\
	\end{bmatrix} \cdot
	\begin{pmatrix}
		\mathrm{d}W_1(t) \\
		\mathrm{d}W_2(t) \\
		\mathrm{d}W_3(t) \\
	\end{pmatrix}
	\end{aligned}
\end{equation}
with $t \in [0,10]$. 
The initial datum is given by the random vector of component $X^{\mathrm{true}}_i(0) = 0.1 - \text{Un}_{i}$, for $i=1,2$, where each  $\text{Un}_{i}$ is an independent uniform random variable taking values in the interval $[-0.0001,0.0001]$, whereas $X^{\mathrm{true}}_3(0)=0$. Therefore the initial condition $X^{\mathrm{true}}$ has rank equal to $2$. Moreover, we choose $\sigma_B=10^{-8}$.
For this numerical experiment we compute $M=10000$ paths and choose a rank $k=2$. 
	
To initialize all the DLRA algorithms, we employ a low-rank approximation of the ``discretized" initial datum. 
Specifically, the initial condition of Algorithms \ref{alg: DLR EM SDE algorithm}, \ref{alg: Stoch DLR Proj algorithm}, and \ref{alg: Eva Proj Splitt SDE algorithm}, is the rank-$k$ best approximation with respect to the Frobenius norm of $[X^{\mathrm{true}}_{1}(0), \dots, X^{\mathrm{true}}_{M}(0)] \in \mathbb{R}^{d \times M}$, namely the rank-$k$ truncated SVD of the matrix whose columns are the $M$ realizations of $X^{\mathrm{true}}_0$. 
In \eqref{ex: toy example}, we remark that the drift and the diffusion coefficients for the first two components are the same, although the Brownian motions differ. Moreover, the diffusion coefficient is constructed so that $\sigma_B$ is actually the lower-bound quantity appearing in Assumption \ref{ass: diff}. Then, we expect that for all the DLRA algorithms of rank $k=2$ we obtain a good approximation of the full-order solution of \eqref{ex: toy example} over time discretization and that the covariance of the stochastic basis is always strictly positive definite, according to results in Sections \ref{sec: DLR EM}, \ref{sec: convergence PS SDE}, and \ref{sec: convergence PS EM}.

Figure \ref{fig:toy example sig values} presents the smallest singular value $\sigma^k$ over time for all three algorithms, across varying time step-sizes. These plots also include the corresponding two quantities \eqref{eq: cov 1 EM} and \eqref{eq: cov 2 EM}. 
We observe that in all cases, the smallest singular value $\sigma^k$ of the dicretized DLR solutions is for above the lower bounds given by \eqref{eq: cov 1 EM} and \eqref{eq: cov 2 EM} (where $K_1(T)$ is estimate using the Monte Carlo estimator); this behavior supports Propositions~\ref{prop: 1 discrete covariance full-rank - EM},
\ref{prop: 2 discrete covariance full-rank - EM}, \ref{prop: discrete covariance full-rank - Stoch DLR Proj}, and \ref{prop: discrete covariance full-rank - KNV}.

\begin{figure}[!h]
	\centering
		\includegraphics[scale=0.16]{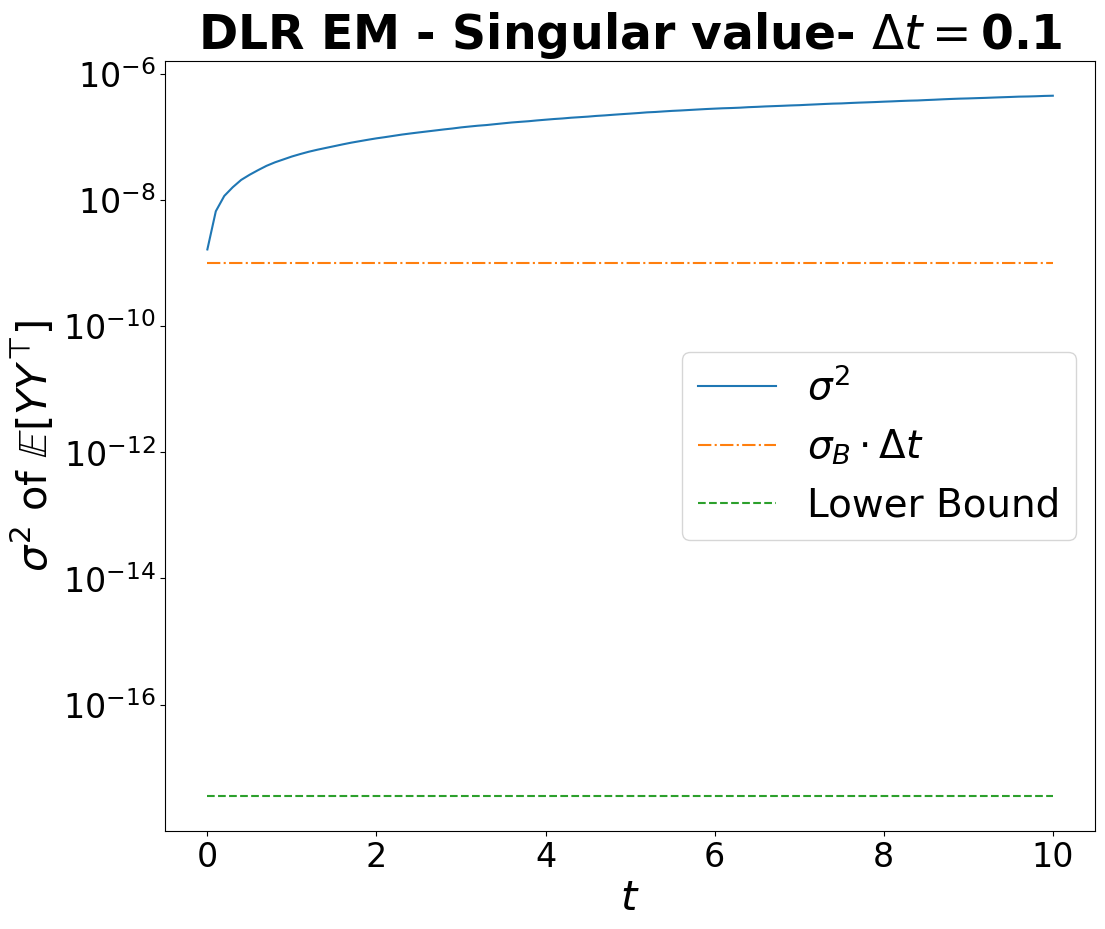}
	\includegraphics[scale=0.16]{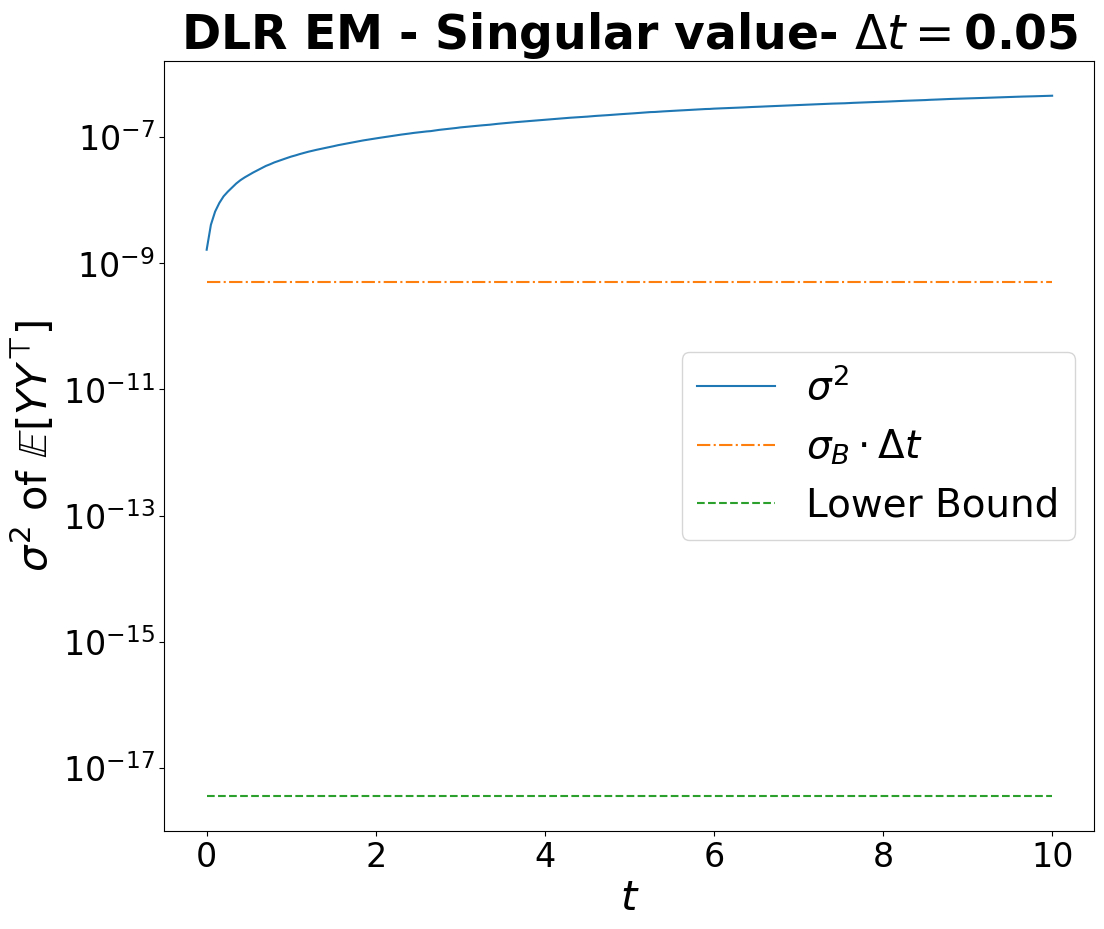}
	\includegraphics[scale=0.16]{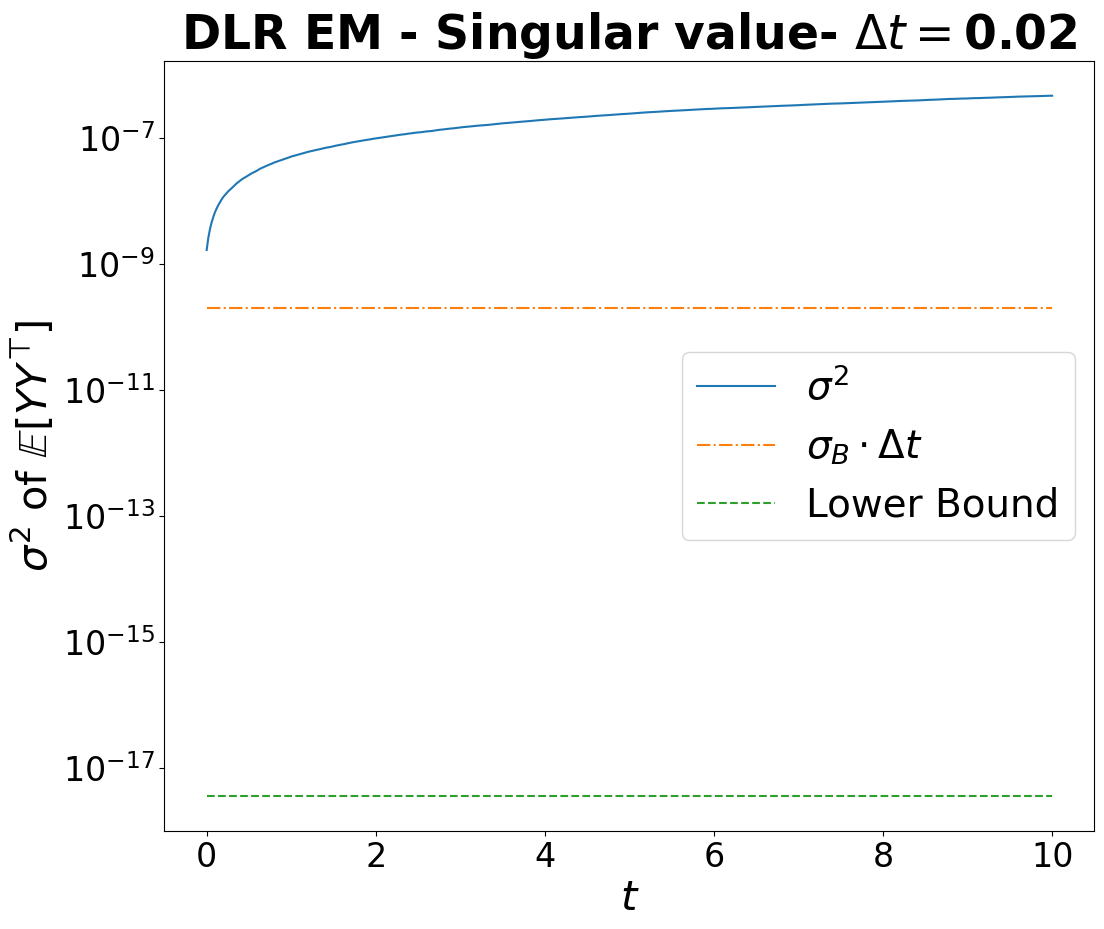}\\
	\includegraphics[scale=0.16]{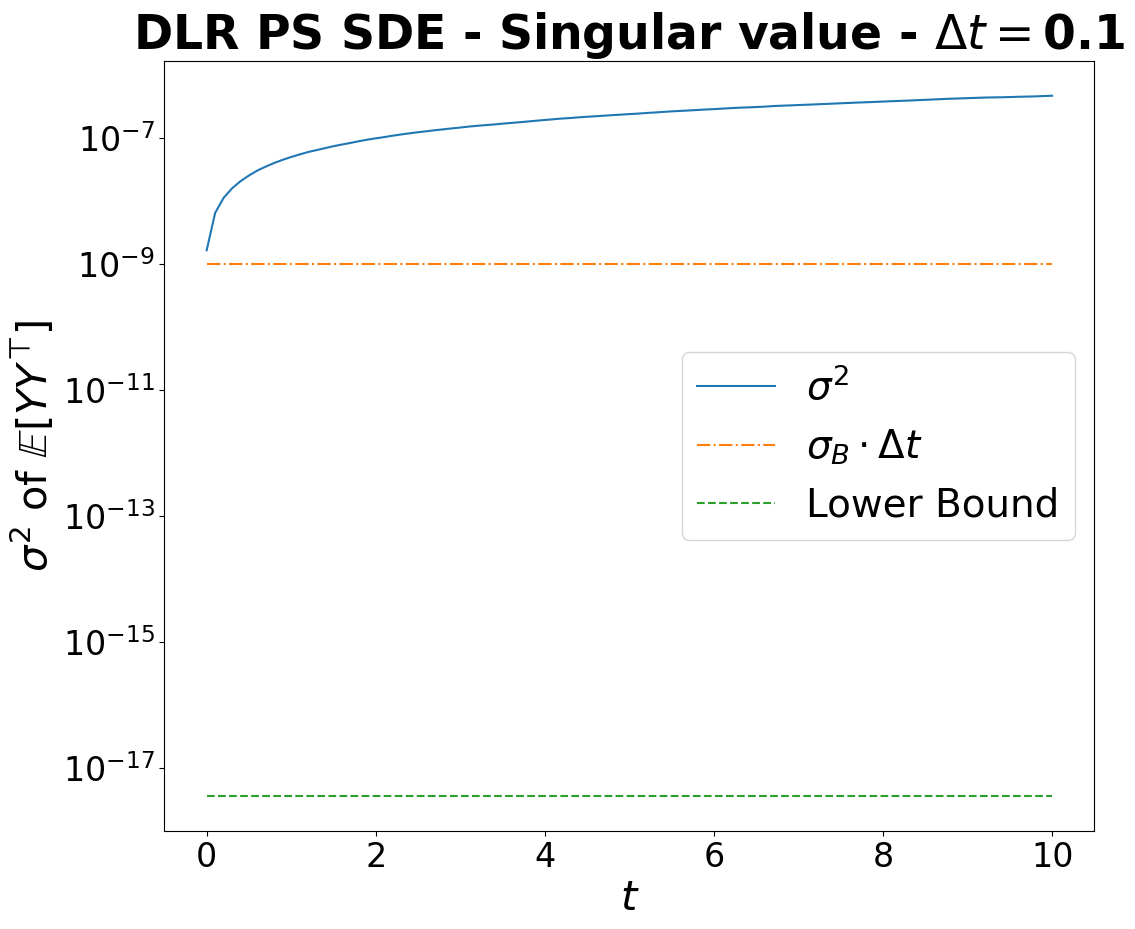}
	\includegraphics[scale=0.16]{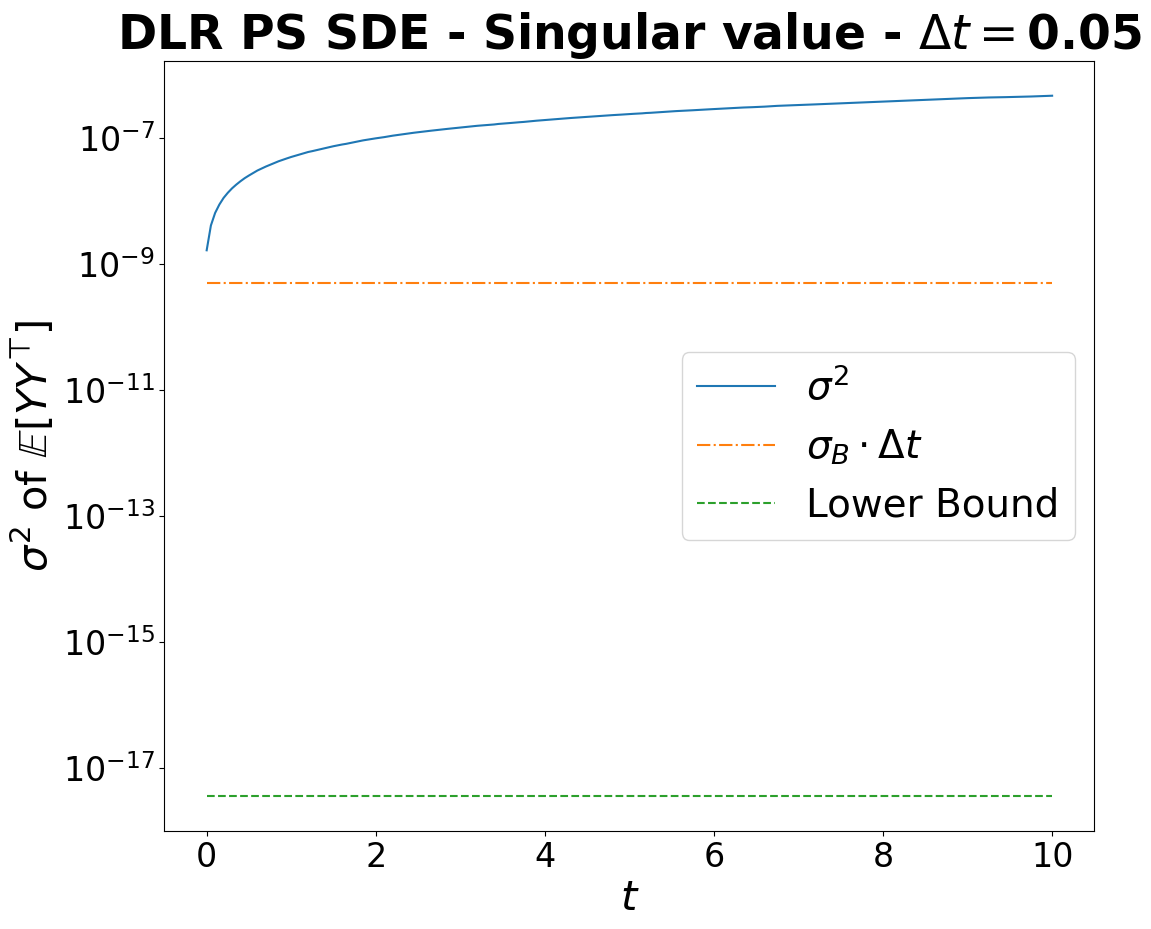}
	\includegraphics[scale=0.16]{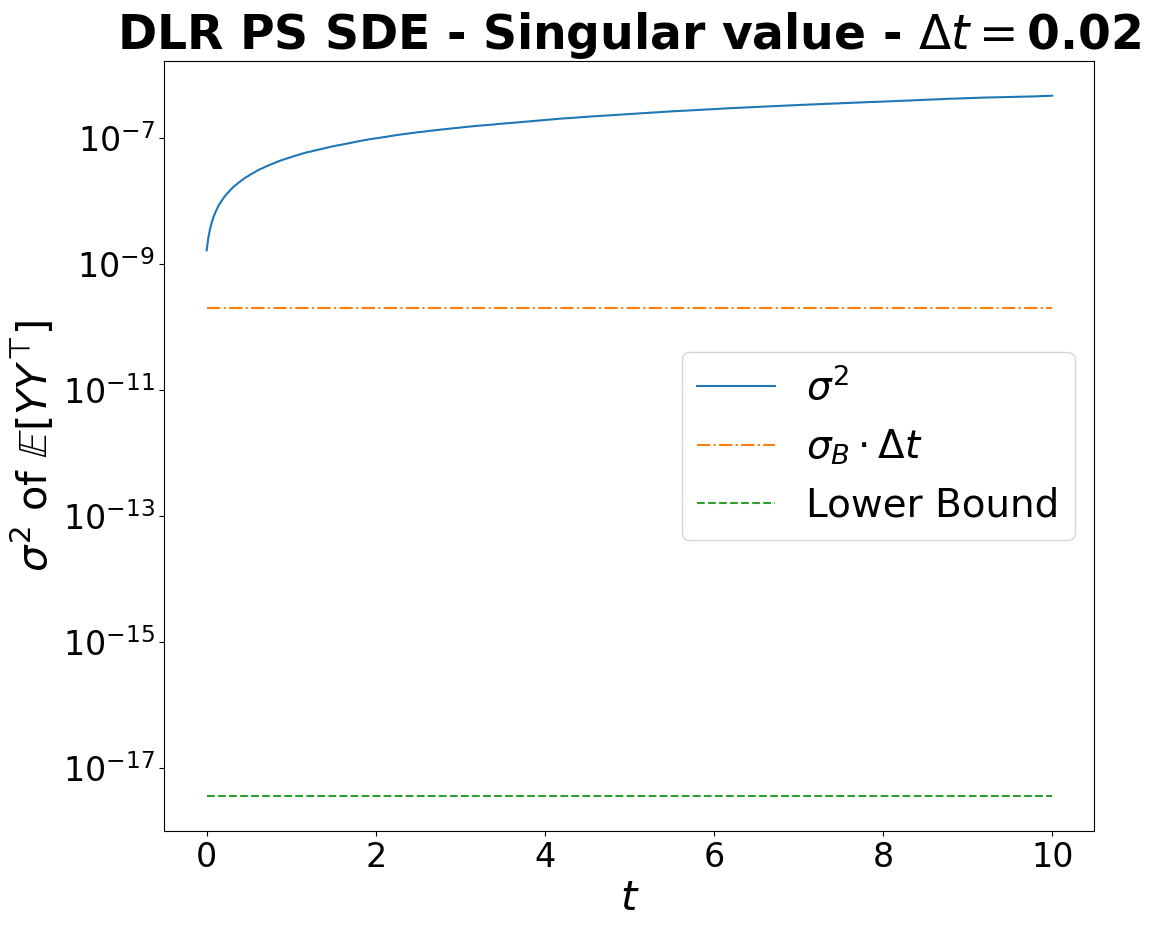}\\
	\includegraphics[scale=0.16]{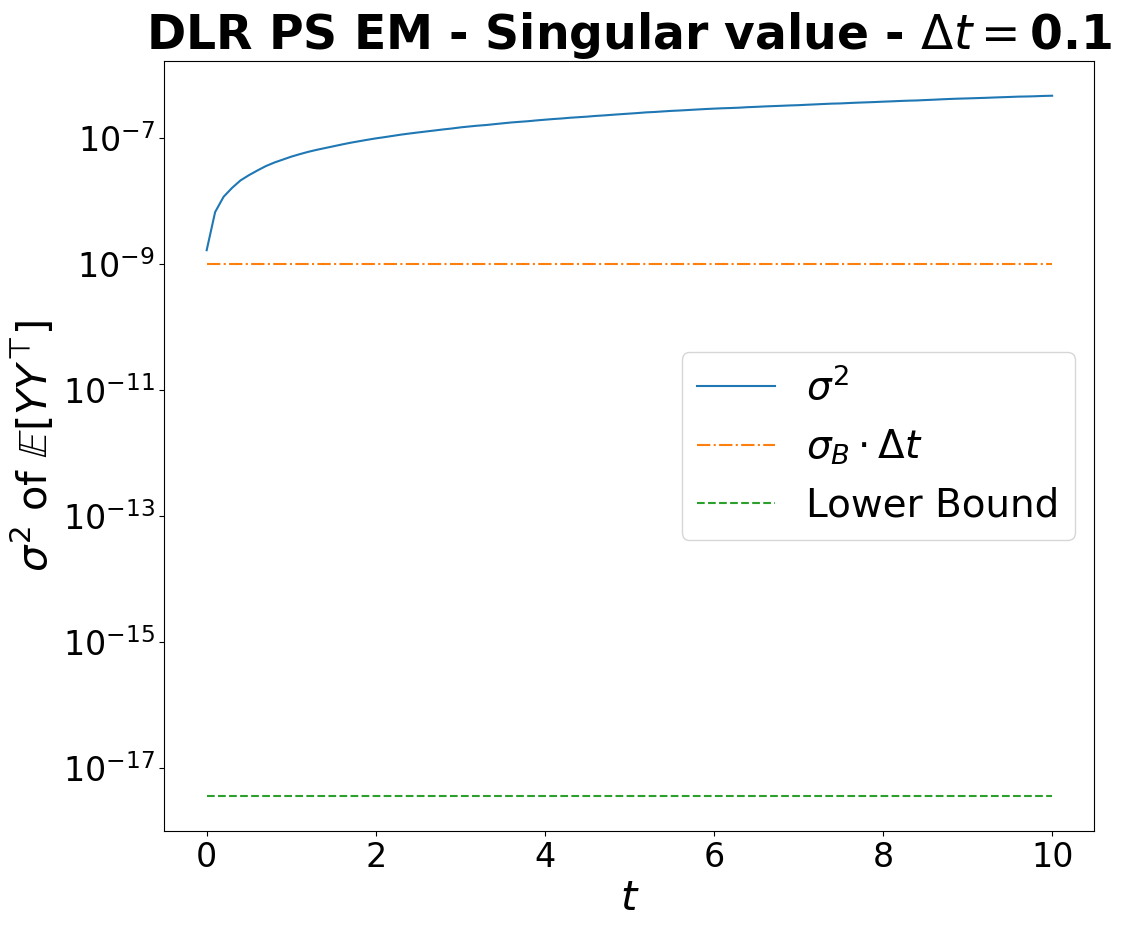}
	\includegraphics[scale=0.16]{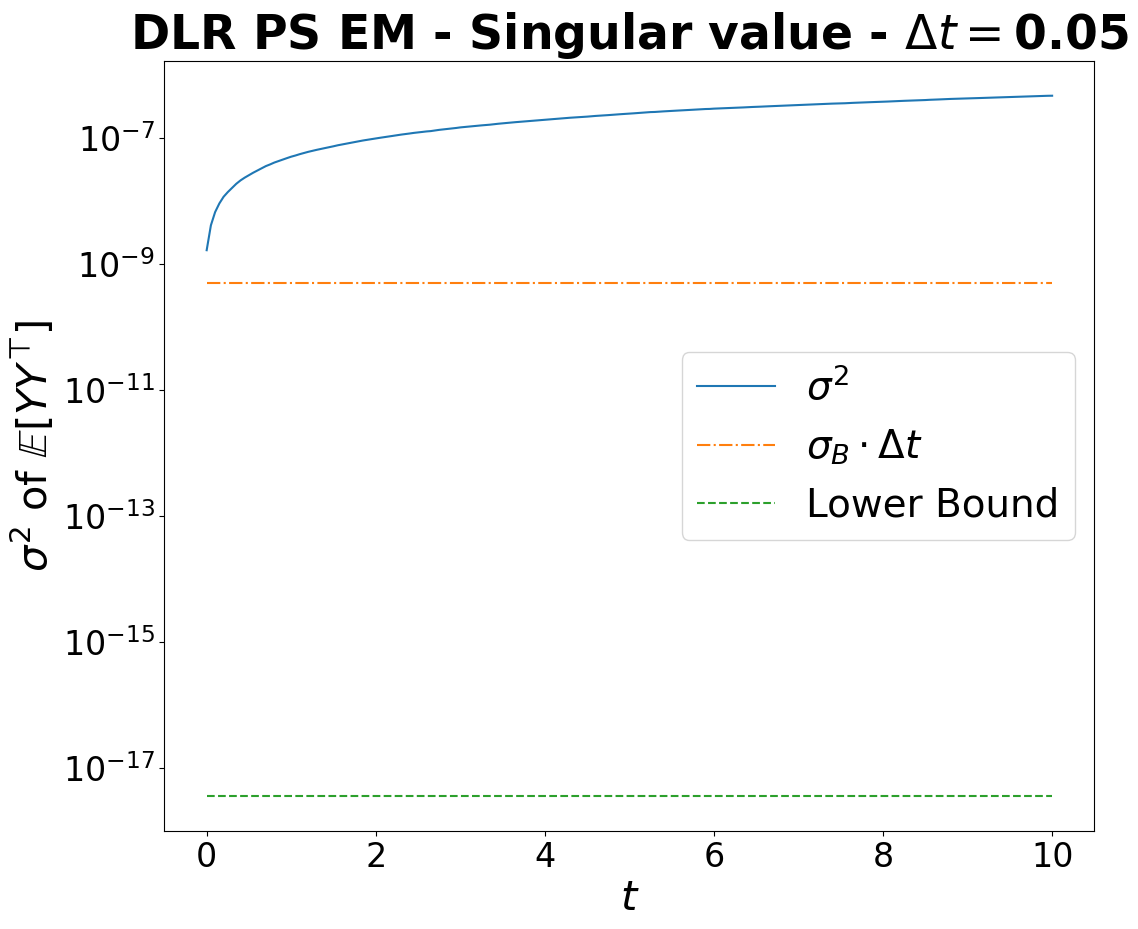}
	\includegraphics[scale=0.16]{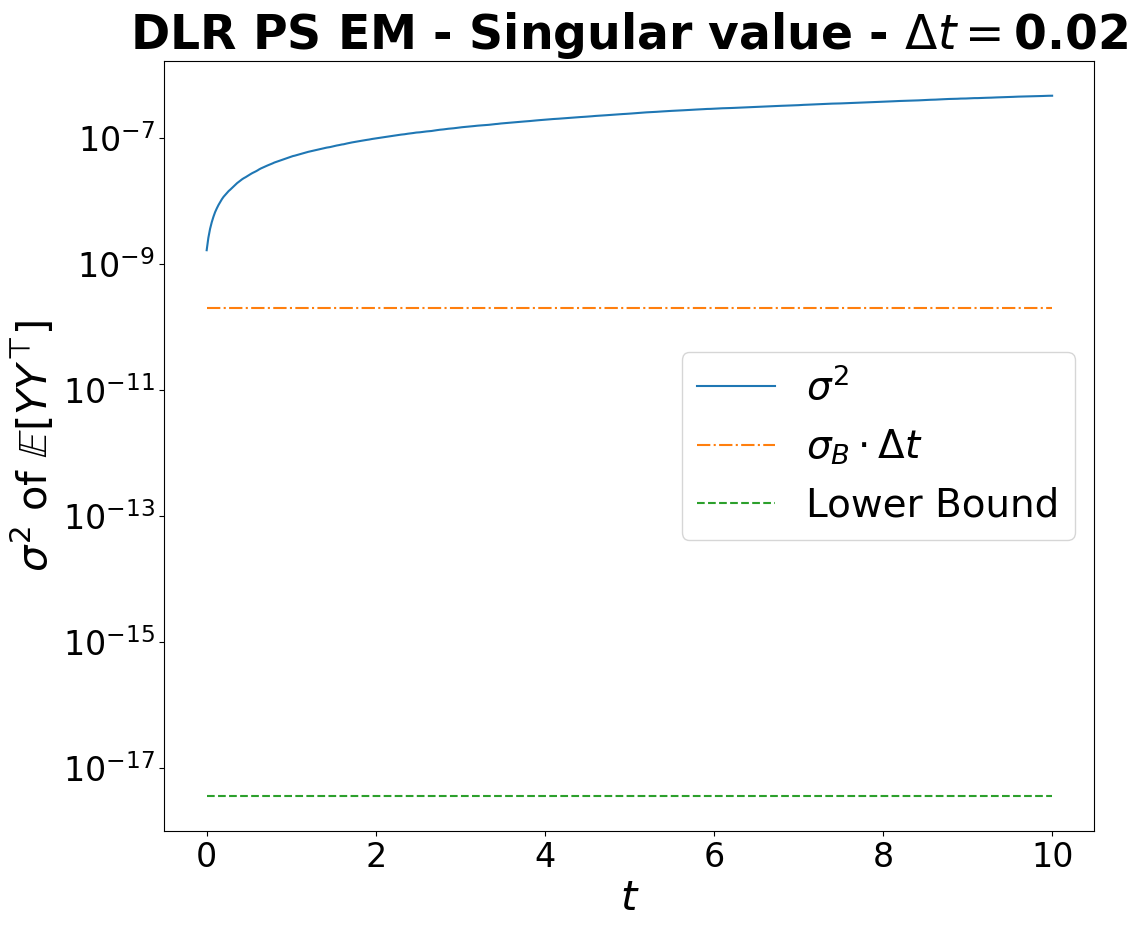}\\
	\caption{For $\Delta t=0.1,0.05,0.02$, $k=2$, $M=10000$, $\sigma_{B}=10^{-8}$ for \eqref{ex: toy example}. (Top) Smallest singular values of $\mathbb{E}[Y_nY_n^{\top}]$ for the DLR Euler-Maruyama, (Middle) for the DLR Projector Splitting for SDEs, (Bottom) for the DLR Projector Splitting for EM. Lower bounds of $\sigma^k$ refer to \eqref{eq: cov 2 EM} (and, hence, to \eqref{eq: cov 2 Stoch DLR Proj}, and \eqref{eq: cov 2 KNV}, too), for all the three algorithms, respectively.}
	\label{fig:toy example sig values}
\end{figure}

On the left panel of Figure \ref{fig:toy example error}, the singular values of the true solution $X^{\mathrm{true}}$ are shown, while
its right panel illustrates the $L^2$ strong errors between each discretization method of Sections \ref{sec: discretization procedure} and the exact solution $X^{\mathrm{true}}$ or its continuous DLRA. The reference $X^{\mathrm{true}}$ and the continuous DLR solutions have been computed with $\Delta t = 1 \cdot 10^{-3}$. 
We notice that $\sigma^{3}(\mathbb{E}[X^{\mathrm{true}}(t)(X^{\mathrm{true}})^{\top}(t)])$ approaches $10^{-9}$ almost immediately and stays nearly constant afterward, whereas $\sigma^{2}(\mathbb{E}[X^{\mathrm{true}}(t)(X^{\mathrm{true}})^{\top}(t)])$ slowly increases over time.
Consequently, the decision to employ DLRA algorithms with rank $k=2$ is confirmed by numerical evidence, too. In Figure \ref{fig:toy example error} we notice that the rank-$2$ DLRA discretization still manages to retrieve the strong convergence $O(\Delta t)$ for all the algorithms, with the Projector Splitting methods exhibiting a smaller constant than the DLR Euler-Maruyama. This empirical observation is consistent with Theorems~\ref{thm: convergence of DLR Euler-Maruyama}, \ref{thm: convergence of Stoch DLR Proj}, \ref{thm: Numerical Convergence Eva}. Moreover, notice that the DLR PS EM shows a slightly better accuracy with respect to the true solution $X^{\mathrm{true}}$ than DLR PS SDE for small $\Delta t$.
\begin{figure}[!h]
	\centering
	\includegraphics[scale=0.24]{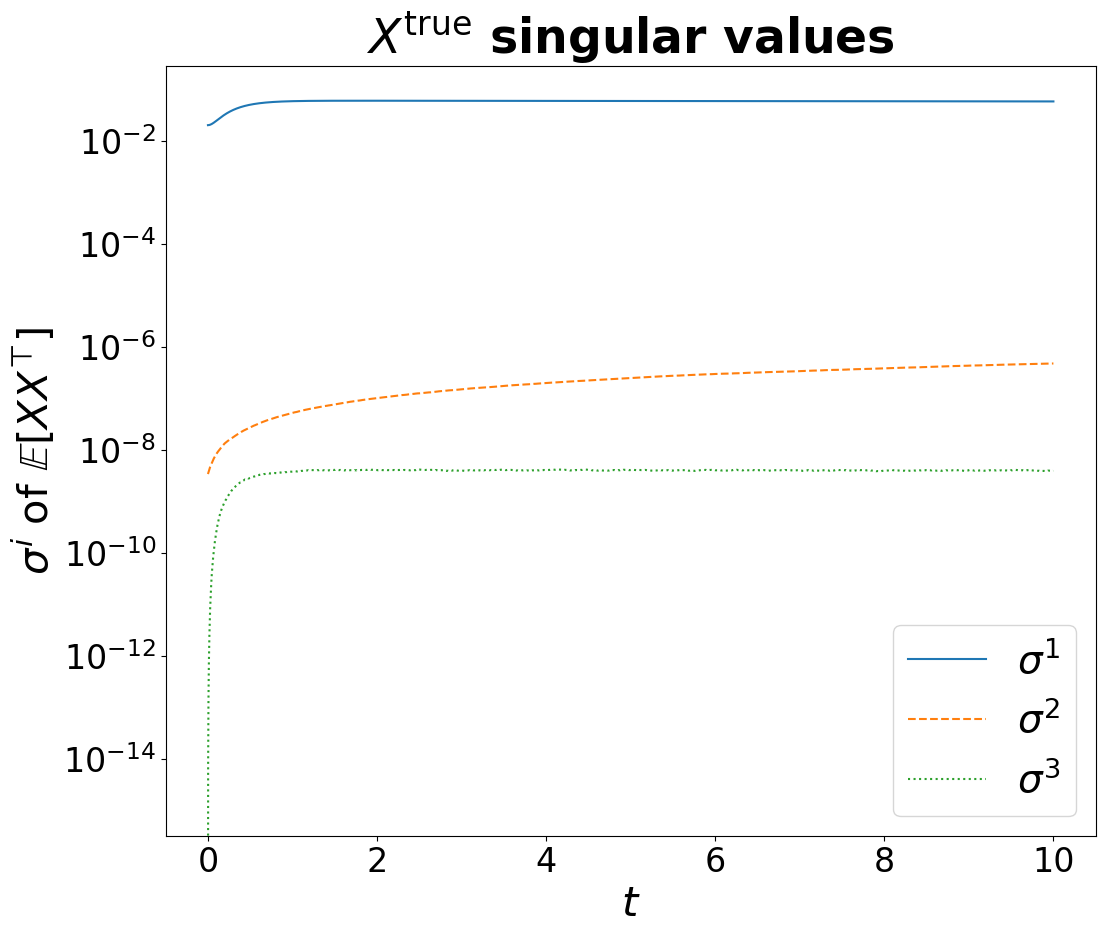}
	\includegraphics[scale=0.29]{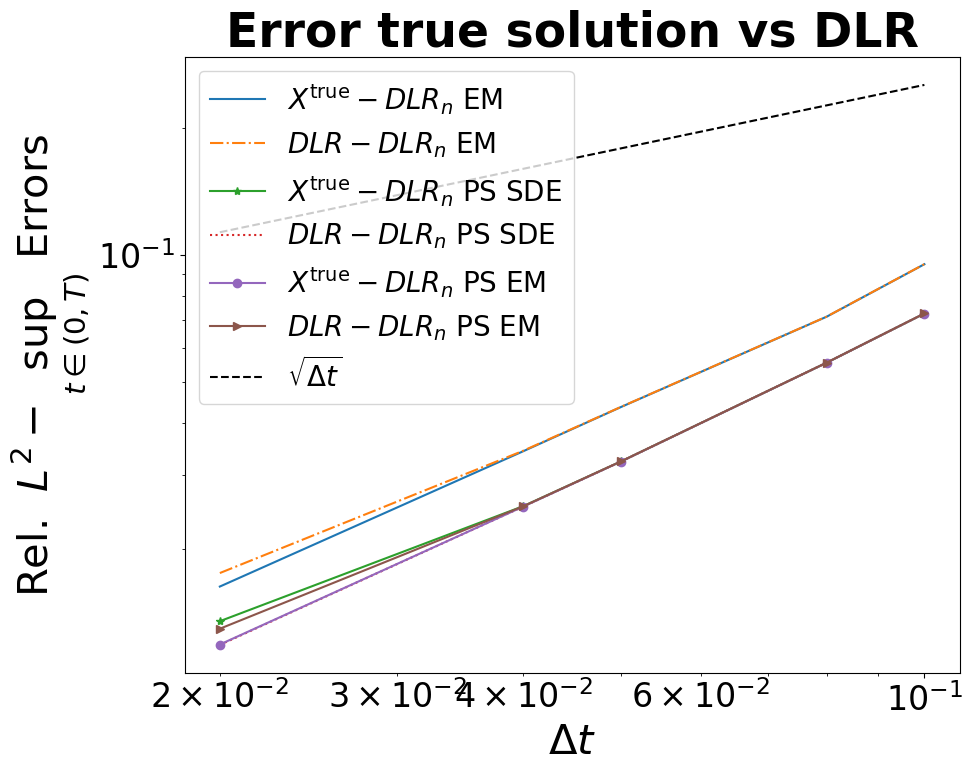}
	\caption{(Left) Singular values for the process $X^{\mathrm{true}}(t)$, solution of equation \eqref{ex: toy example}. (Right) Relative errors for the DLR Euler-Maruyama, DLR Projector Splitting for SDEs, and the DLR Projector Splitting for EM, with respect to the true solution $X^{\mathrm{true}}$ and the continuous DLRA of rank $k=2$. $M=10000$ paths are used in all methods to approximate expectations.}
	\label{fig:toy example error}
\end{figure}

It is worth noticing that Equation \eqref{ex: toy example} has a linear deterministic drift. In this case, the deterministic modes of the DLR Euler-Maruyama evolve independently of the Gramian $C_{Y_n}$ (see \eqref{eq: linear drift}). Therefore, in this setting the DLR Euler-Maruyama is not affected by the smallest singular value. To provide numerical evidence of this property let us consider a slight modification of the problem, given by
\begin{equation}\label{ex: toy example 2}
	\begin{aligned}
		\hspace{-0.3cm}
		\begin{pmatrix}
			\mathrm{d}X^{\mathrm{true}}_1 (t) \\
			\mathrm{d}X^{\mathrm{true}}_2 (t) \\
			\mathrm{d}X^{\mathrm{true}}_3 (t)\\
		\end{pmatrix}
		=&
		\begin{pmatrix}
			-0.1 \cdot X^{\mathrm{true}}_1(t) &0.1 \cdot X^{\mathrm{true}}_2(t)&0.001 \cdot X^{\mathrm{true}}_3(t)\\
			-0.1 \cdot X^{\mathrm{true}}_1(t) &0.1 \cdot X^{\mathrm{true}}_2(t)&0.001 \cdot X^{\mathrm{true}}_3(t)\\
			-4 \cdot X^{\mathrm{true}}_1(t) &-4 \cdot X^{\mathrm{true}}_2(t)&-4 \cdot X^{\mathrm{true}}_3(t)\\
		\end{pmatrix}
		\mathrm{d}t\\
		&+ \sqrt{\sigma_B}
		\begin{bmatrix}
			1 & 0 & 0  \\
			0 & 	1& 0  \\
			0 & 0 & 0\\
		\end{bmatrix} \cdot
		\begin{pmatrix}
			\mathrm{d}W_1(t) \\
			\mathrm{d}W_2(t) \\
			\mathrm{d}W_3(t) \\
		\end{pmatrix}
	\end{aligned}
\end{equation}
with $t \in [0,10]$.
The initial datum is given by the random vector of components $X^{\mathrm{true}}_1(0) = 0.1 - \text{Un}_{1}$, where $\text{Un}_{1}$ is a uniform random variable taking values in the interval $[-0.0001,0.0001]$, whereas $X^{\mathrm{true}}_2(0) = 0.1 - \text{Un}_{2}$ with $\text{Un}_{2}$ is a uniform random variable, independent of $\text{Un}_{1}$, taking values in the interval $[-10^{-9},10^{-9}]$, and $X^{\mathrm{true}}_3(0)=0$. Therefore, the initial condition $X^{\mathrm{true}}$ has rank equal to $2$, but is close to be rank deficient. Furthermore, we choose $\sigma_B=10^{-16}$, which implies the addition of noise close to the zero machine in absolute value. With these properties, \eqref{ex: toy example 2} is close to be a rank-$1$ dynamics. Indeed, in Figure \ref{fig:toy example sig values 2} we illustrate the smallest singular value $\sigma^2$ over time for all three algorithms, across varying time step-sizes, and we notice that their value is close to the zero machine. 
\begin{figure}[!h]
	\centering
	\includegraphics[scale=0.16]{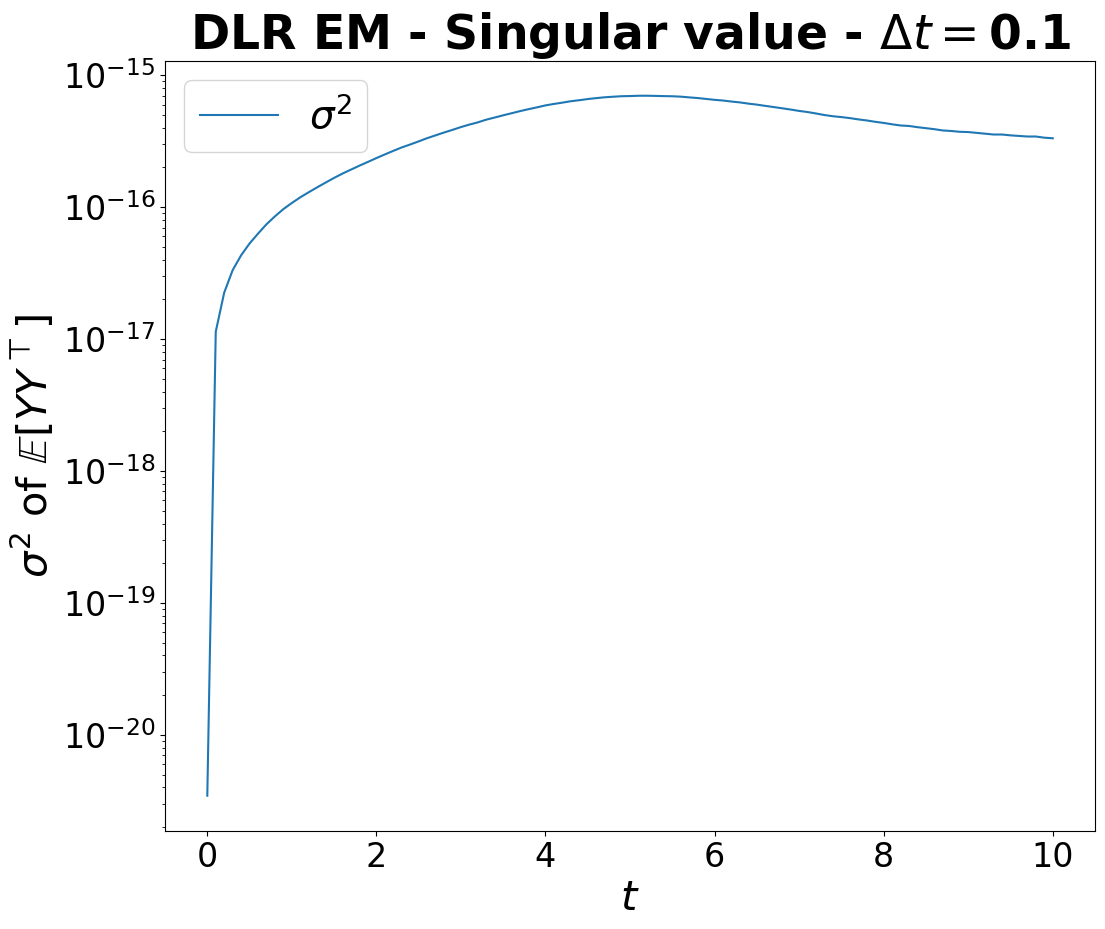}
	\includegraphics[scale=0.16]{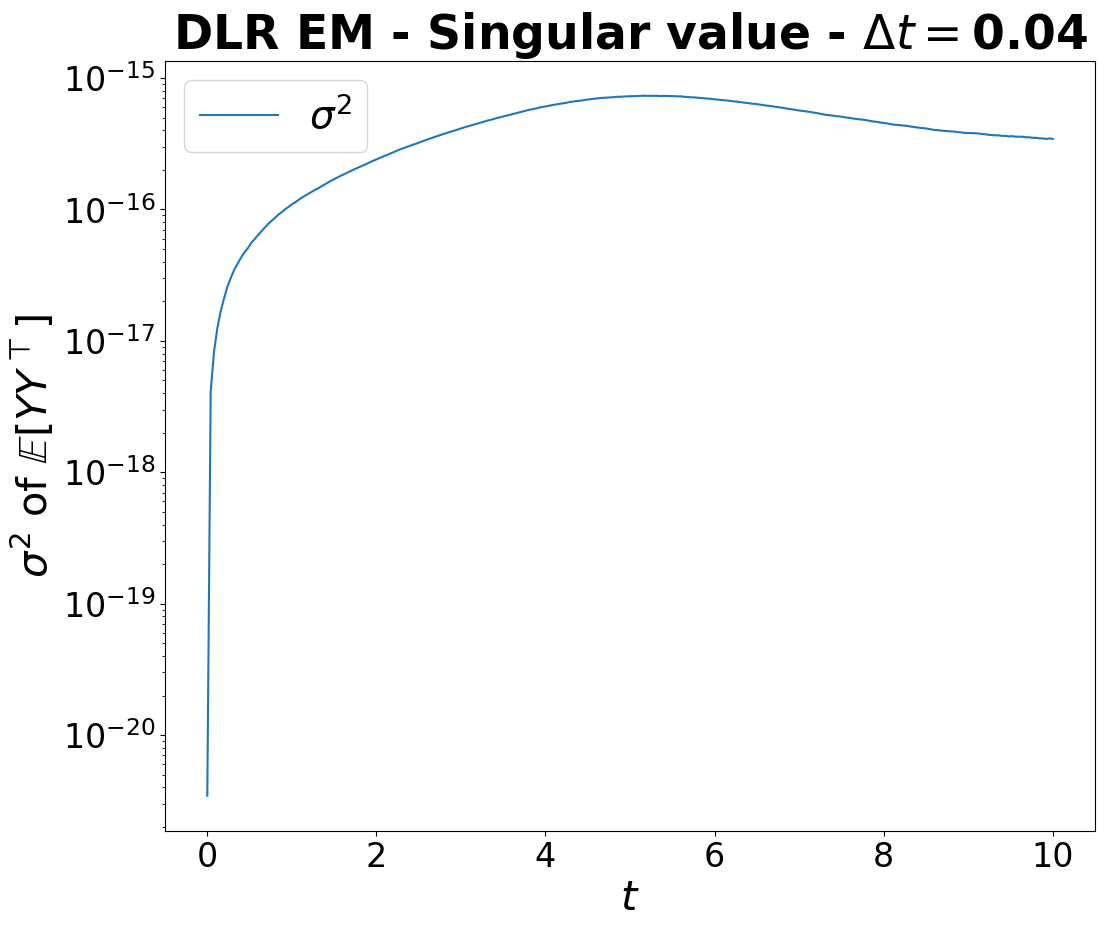}
	\includegraphics[scale=0.16]{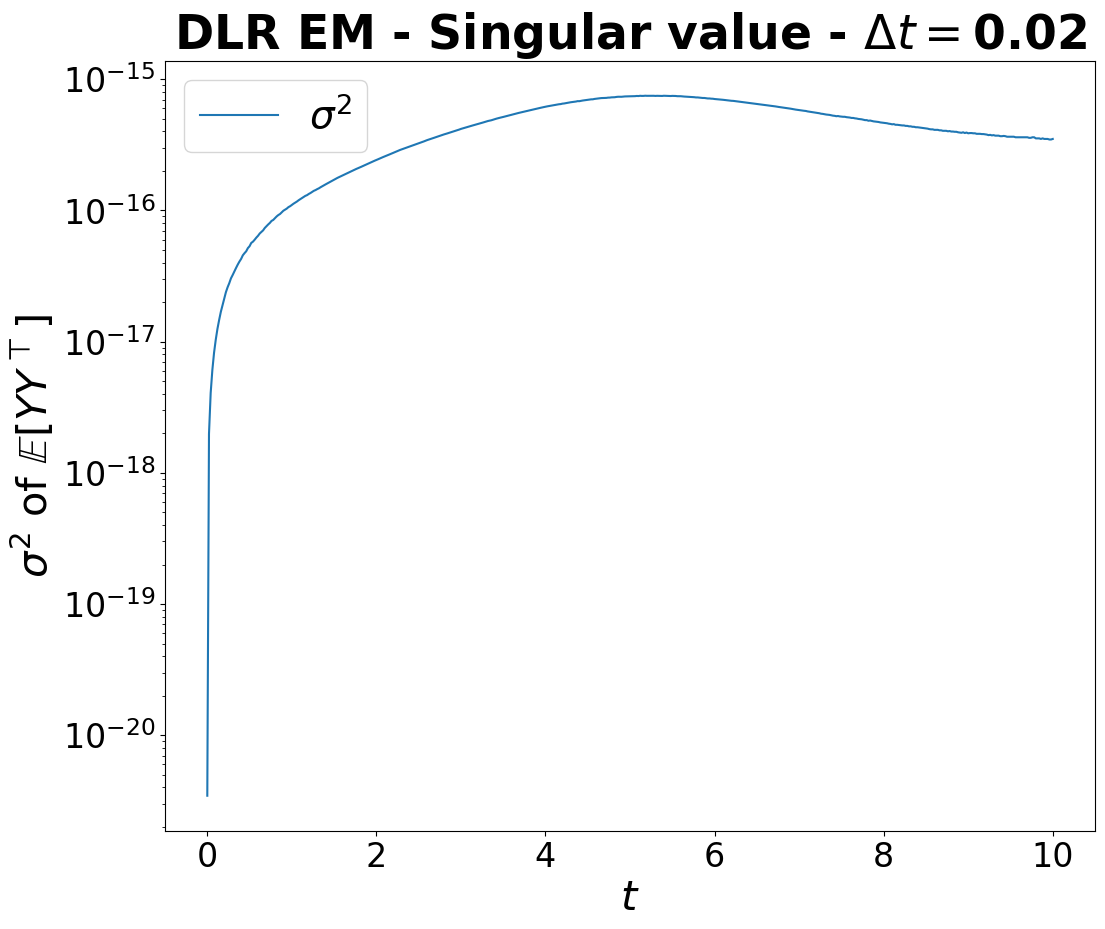}\\
	\includegraphics[scale=0.16]{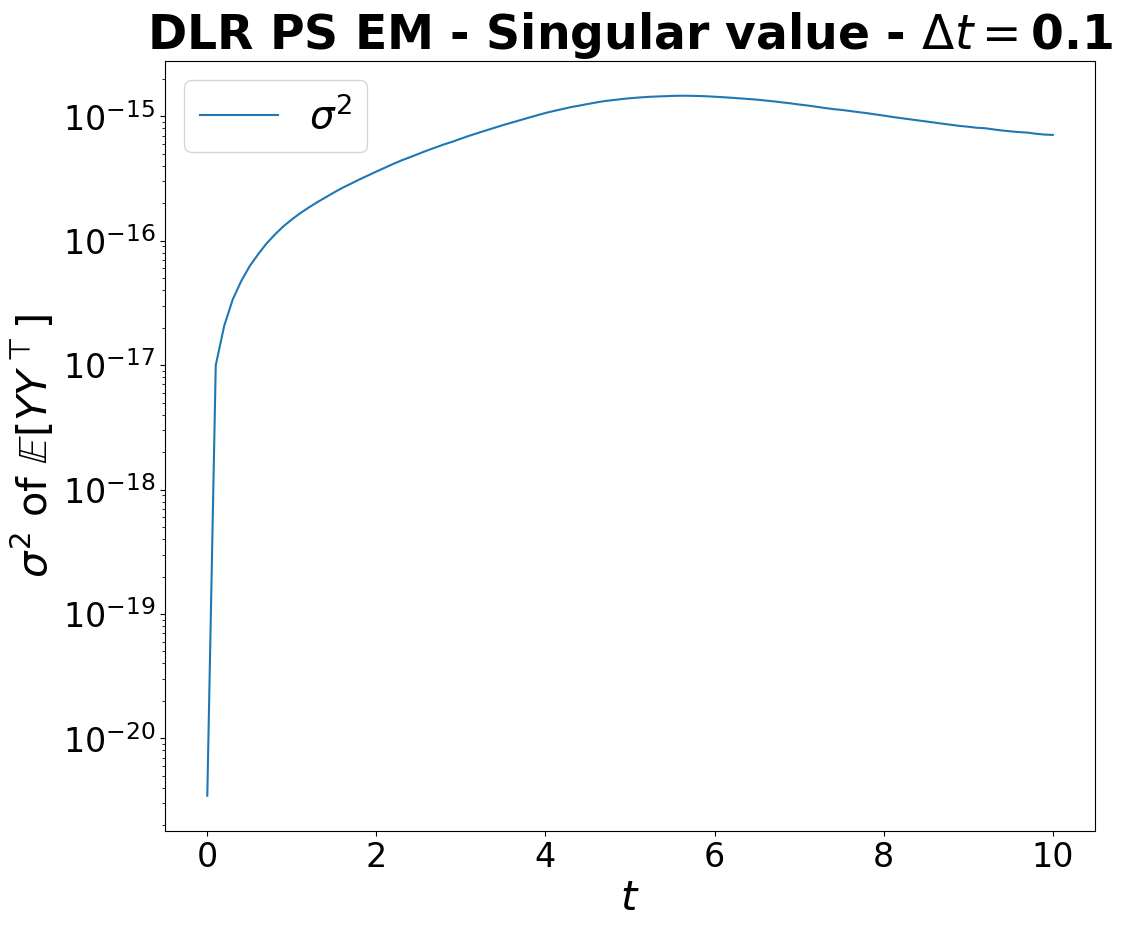}
	\includegraphics[scale=0.16]{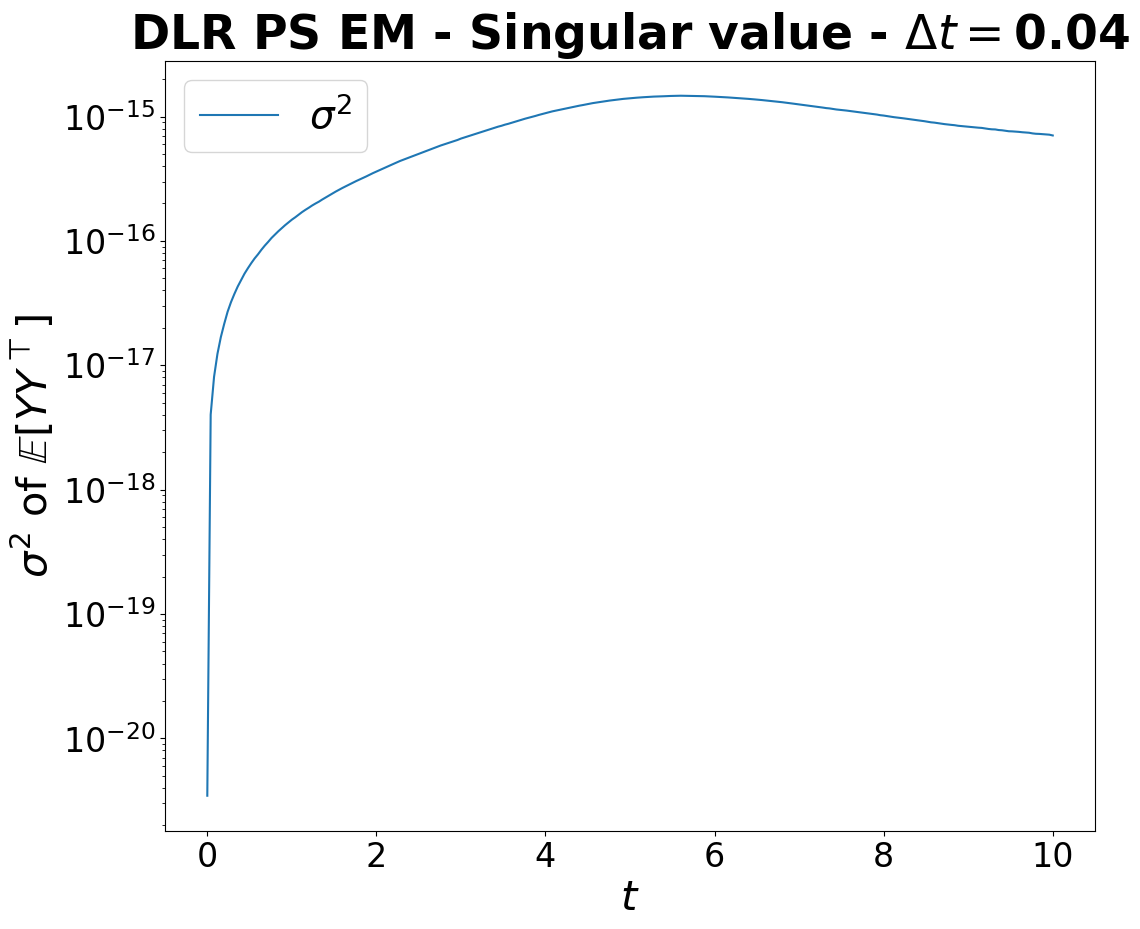}
	\includegraphics[scale=0.16]{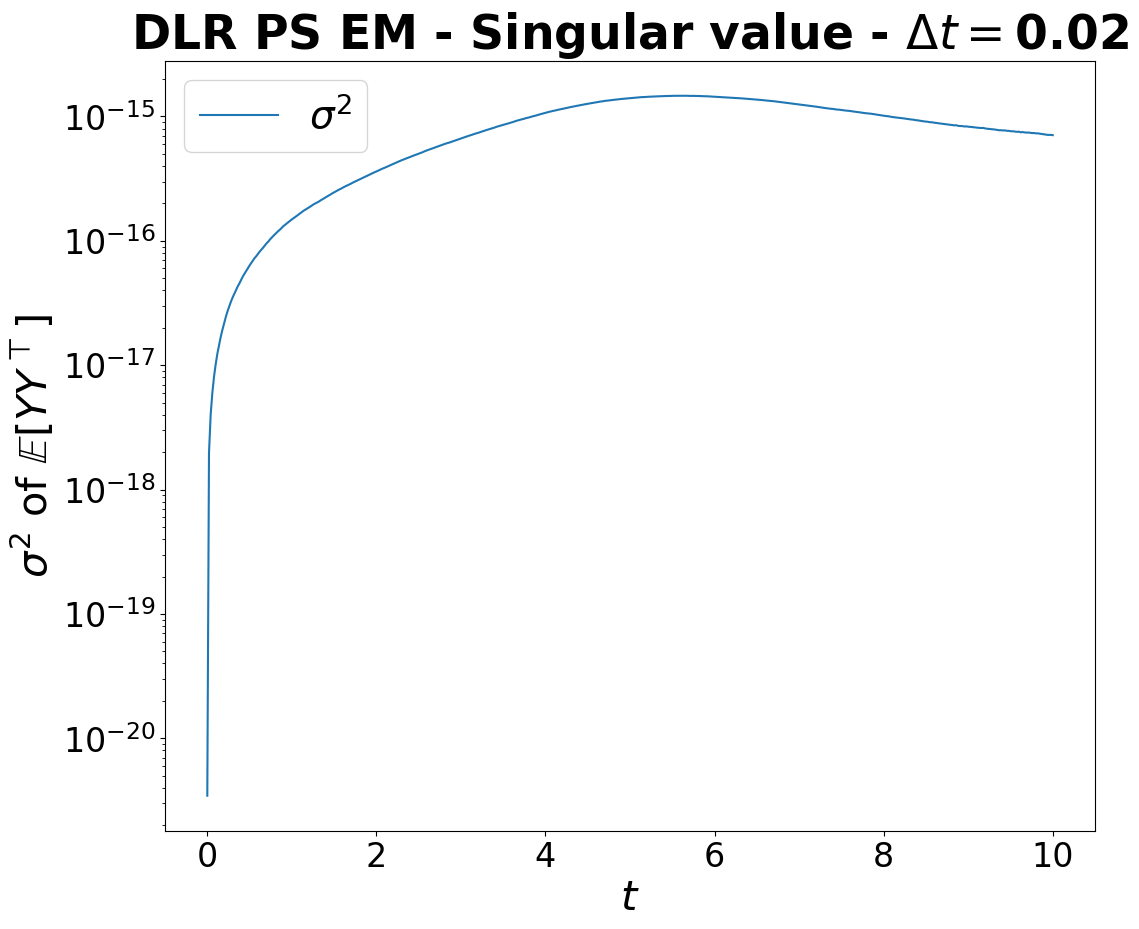}\\
	\includegraphics[scale=0.16]{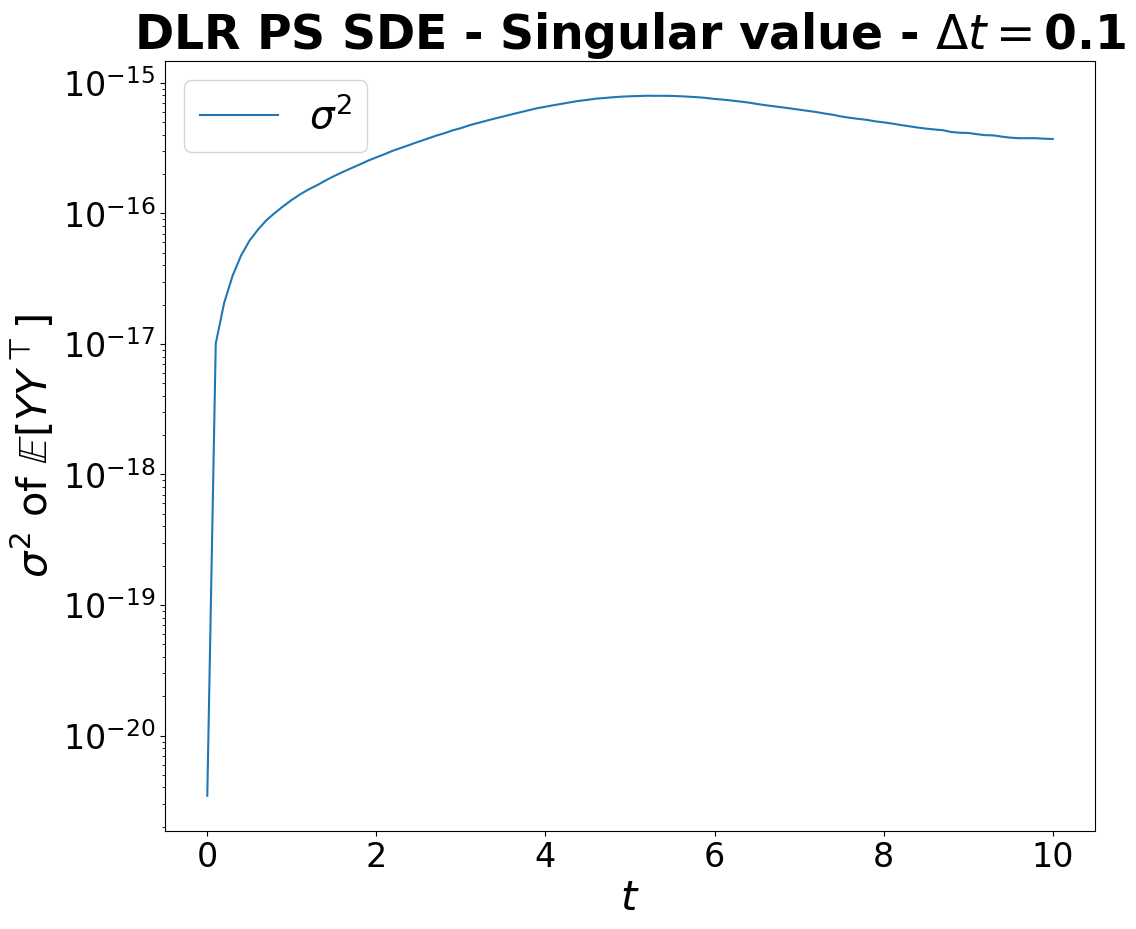}
	\includegraphics[scale=0.16]{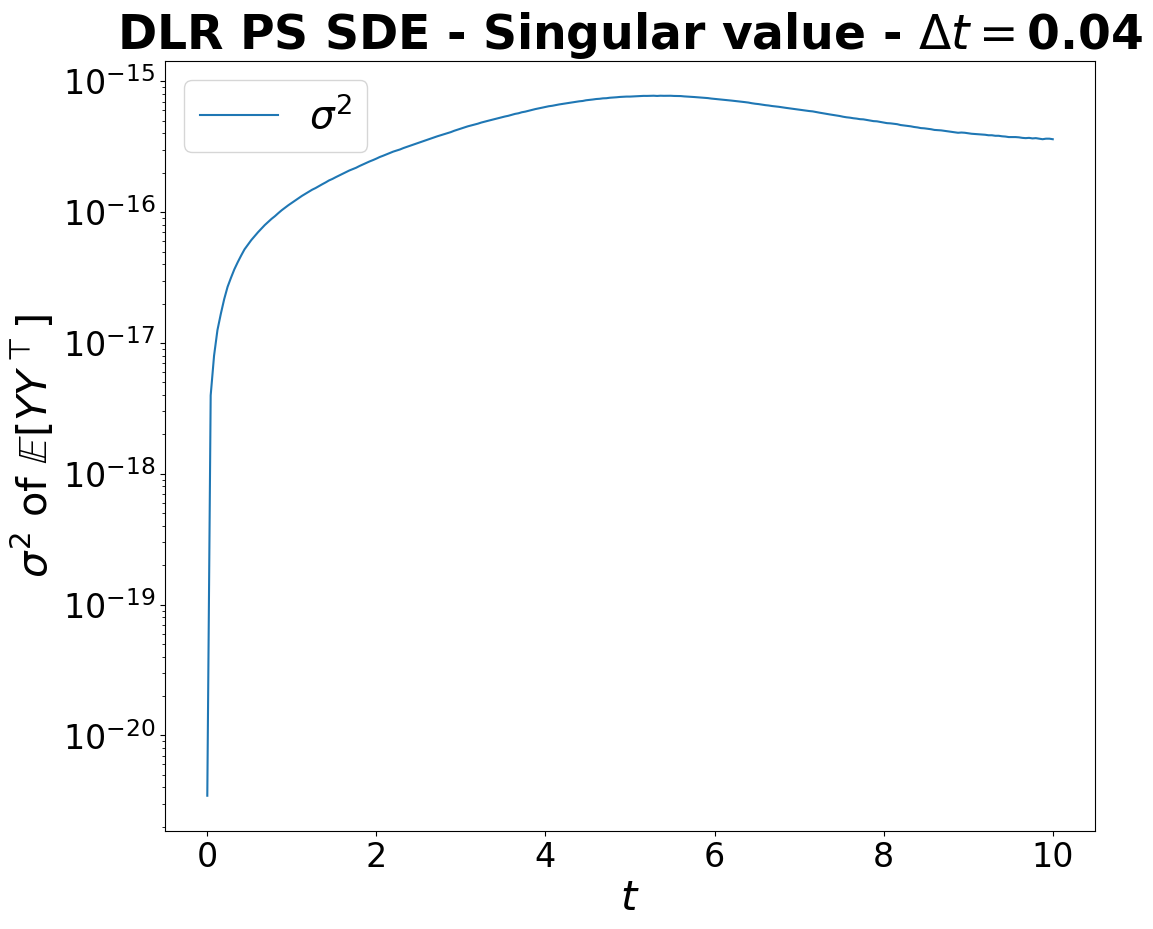}
	\includegraphics[scale=0.16]{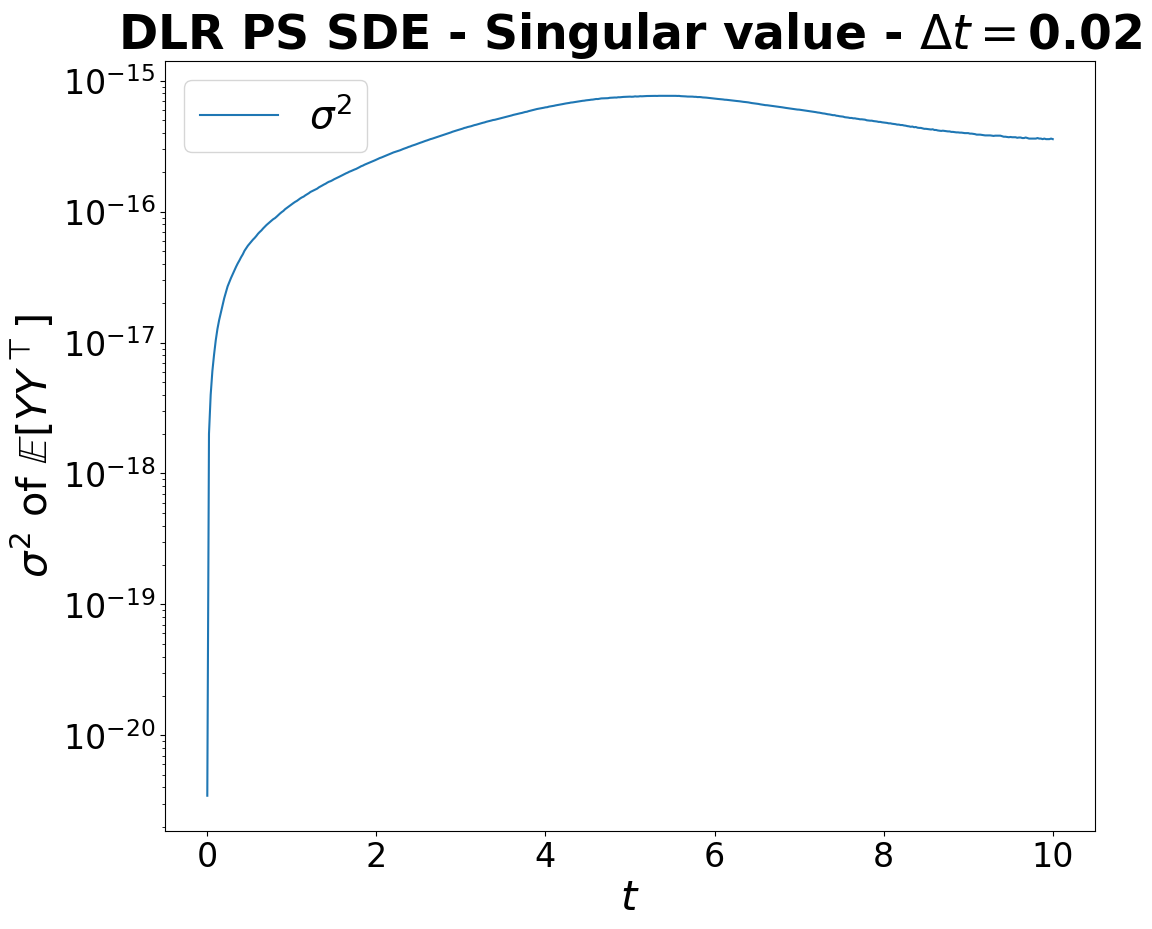}\\
	\caption{For $\Delta t=0.1,0.05,0.02$, $k=2$, $M=10000$, $\sigma_{B}=10^{-16}$. (Top) Smallest singular values of $\mathbb{E}[Y_nY_n^{\top}]$ for the DLR Euler-Maruyama, (Middle) for the DLR Projector Splitting for EM, (Bottom) for the DLR Projector Splitting for SDEs for Problem \eqref{ex: toy example 2}.}
	\label{fig:toy example sig values 2}
\end{figure}

However, one can see from the error plots (Figure \ref{fig:toy example error 2}) that the behavior of the smallest singular values of the stochastic basis does not affect the convergence of any of the algorithms. Moreover, one can notice a convergence in time of order $O(\Delta t)$; this numerical trend is in compliance with usual accuracy of the standard Euler-Maruyama method for additive noise \cite{milstein2004stochastic}.
\begin{figure}[!h]
	\centering
	\includegraphics[scale=0.29]{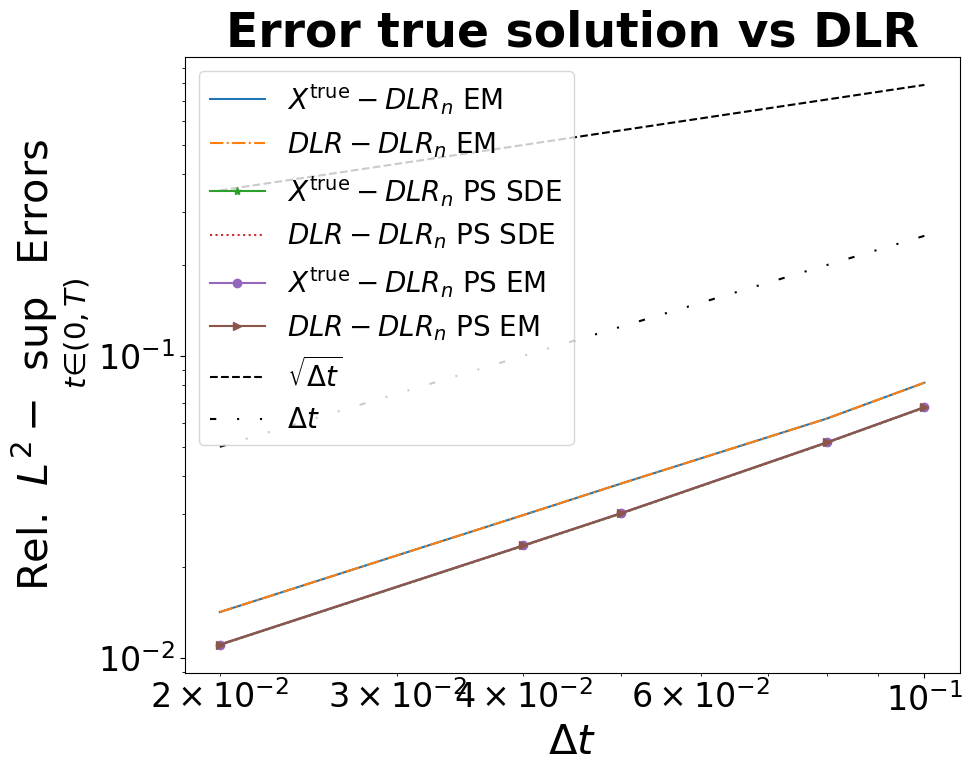}
	\caption{Relative errors for the DLR Euler-Maruyama, for the DLR Projector Splitting for SDEs, and the DLR Projector Splitting for EM, with respect to the true solution $X^{\mathrm{true}}$ and the continuous DLRA of rank $k=2$ for Problem \eqref{ex: toy example 2}. Notice that the DLR PS EM and DLR PS SDE present quite identically accuracy and their error line is overlapped. $M=10000$ paths are used in all methods to approximate expectations.}
	\label{fig:toy example error 2}
\end{figure}

In the last example of this section, we consider a slight modification of the drift in  \eqref{ex: toy example 2} by introducing some nonlinearity. In this case, the evolution of the deterministic basis $U$ of the DLR Euler-Maruyama is coupled with that of the stochastic basis $Y$, leading to possible instabilities in case of nearly singular Gramian $C_{Y_n}$. We consider the following system
\begin{equation}\label{ex: toy example 3}
	\begin{aligned}
		\hspace{-0.3cm}
		\begin{pmatrix}
			\mathrm{d}X^{\mathrm{true}}_1 (t) \\
			\mathrm{d}X^{\mathrm{true}}_2 (t) \\
			\mathrm{d}X^{\mathrm{true}}_3 (t)\\
		\end{pmatrix}
		=&
		\begin{pmatrix}
			-4 \cdot \sin{(\pi  X^{\mathrm{true}}_1(t))} &0.1 \cdot X^{\mathrm{true}}_2(t)&0.001 \cdot X^{\mathrm{true}}_3(t)\\
			-4 \cdot \sin{(\pi X^{\mathrm{true}}_1(t))} &0.1 \cdot X^{\mathrm{true}}_2(t)&0.001 \cdot X^{\mathrm{true}}_3(t)\\
			-4 \cdot \sin{(\pi X^{\mathrm{true}}_1(t))} &-4 \cdot X^{\mathrm{true}}_2(t) &-4 \cdot X^{\mathrm{true}}_3(t)\\
		\end{pmatrix}
		\mathrm{d}t\\
		&+ \sqrt{\sigma_B}
		\begin{bmatrix}
			1 & 0 & 0  \\
			0 & 	1& 0  \\
			0 & 0 & 0\\
		\end{bmatrix} \cdot
		\begin{pmatrix}
			\mathrm{d}W_1(t) \\
			\mathrm{d}W_2(t) \\
			\mathrm{d}W_3(t) \\
		\end{pmatrix}
	\end{aligned}
\end{equation}
for the exact same initial condition and parameters as in \eqref{ex: toy example 2}. In Figure \ref{fig:toy example sig values 3} we observe that the smallest singular value $\sigma^2$ remains close to the zero machine for all the three algorithms, for various values of the time step-sizes. For this setting, we expect the convergence of the Euler-Maruyama to be badly affected by the inverse of $\sigma^2$ (see Theorem \ref{thm: convergence of DLR Euler-Maruyama}), unlike the DLR PS SDE algorithm (see Theorems~\ref{thm: convergence of Stoch DLR Proj}). 
\begin{figure}[!h]
	\centering
	\includegraphics[scale=0.16]{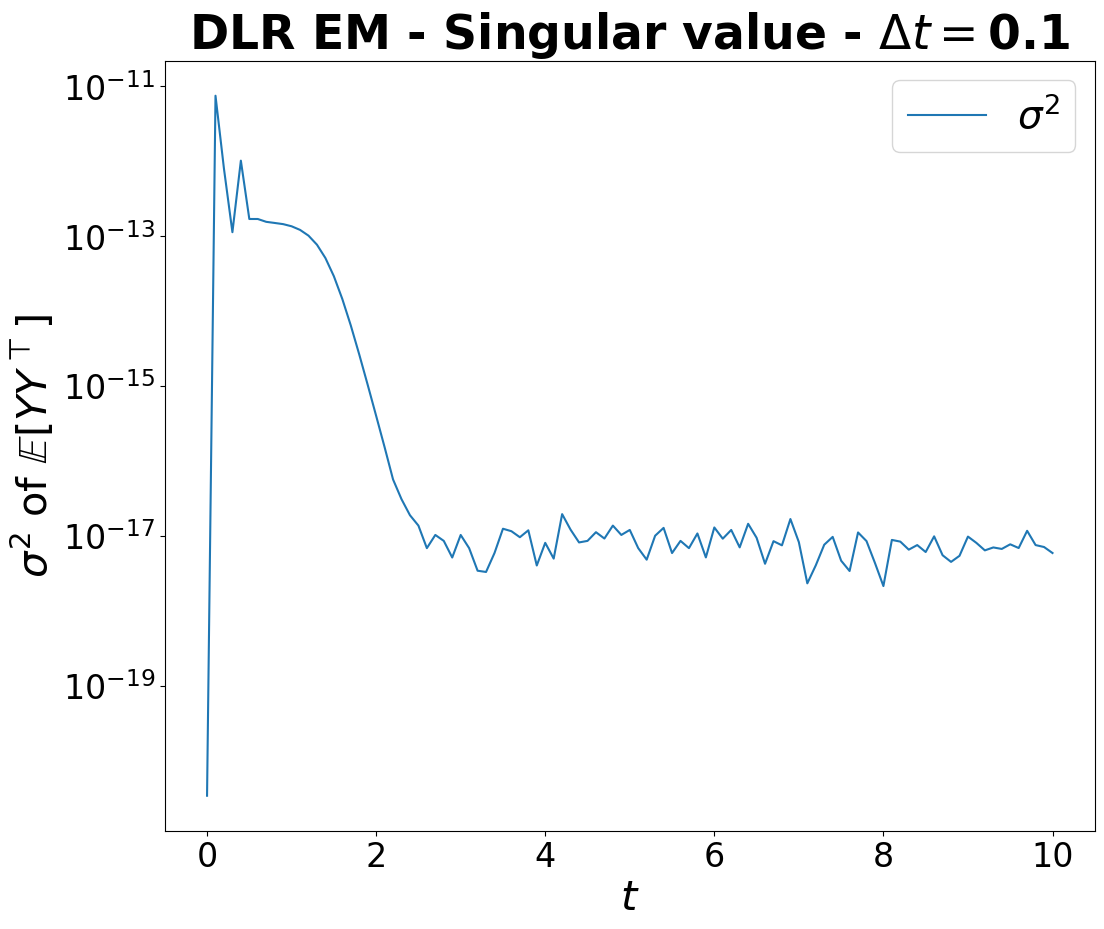}
	\includegraphics[scale=0.16]{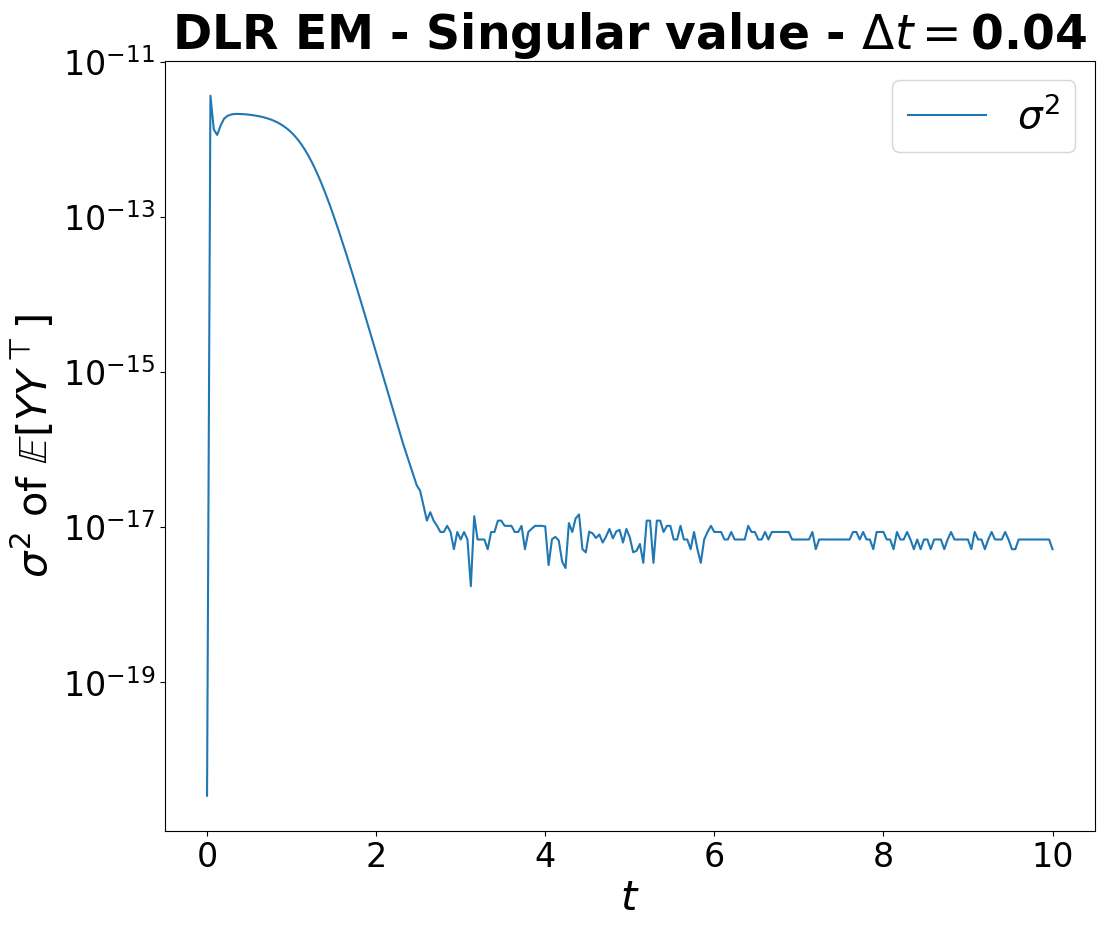}
	\includegraphics[scale=0.16]{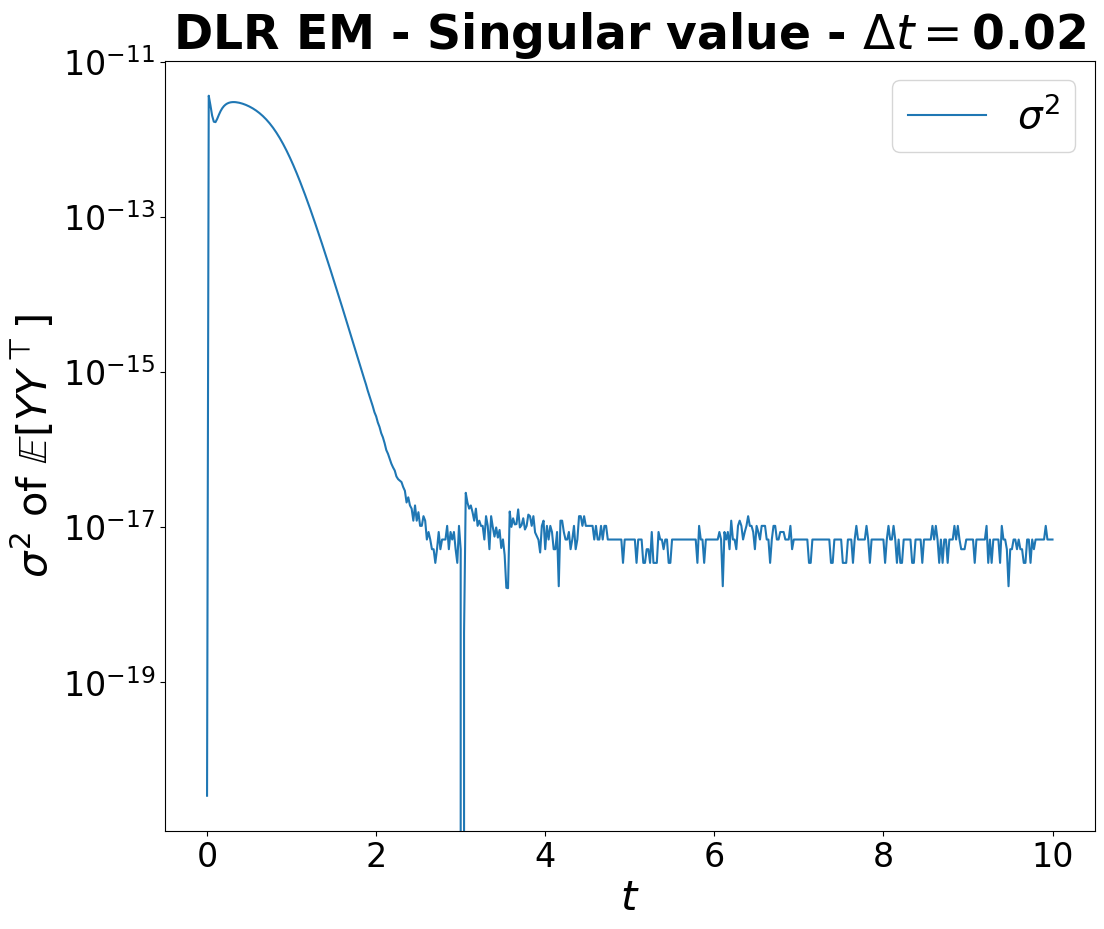}\\
	\includegraphics[scale=0.16]{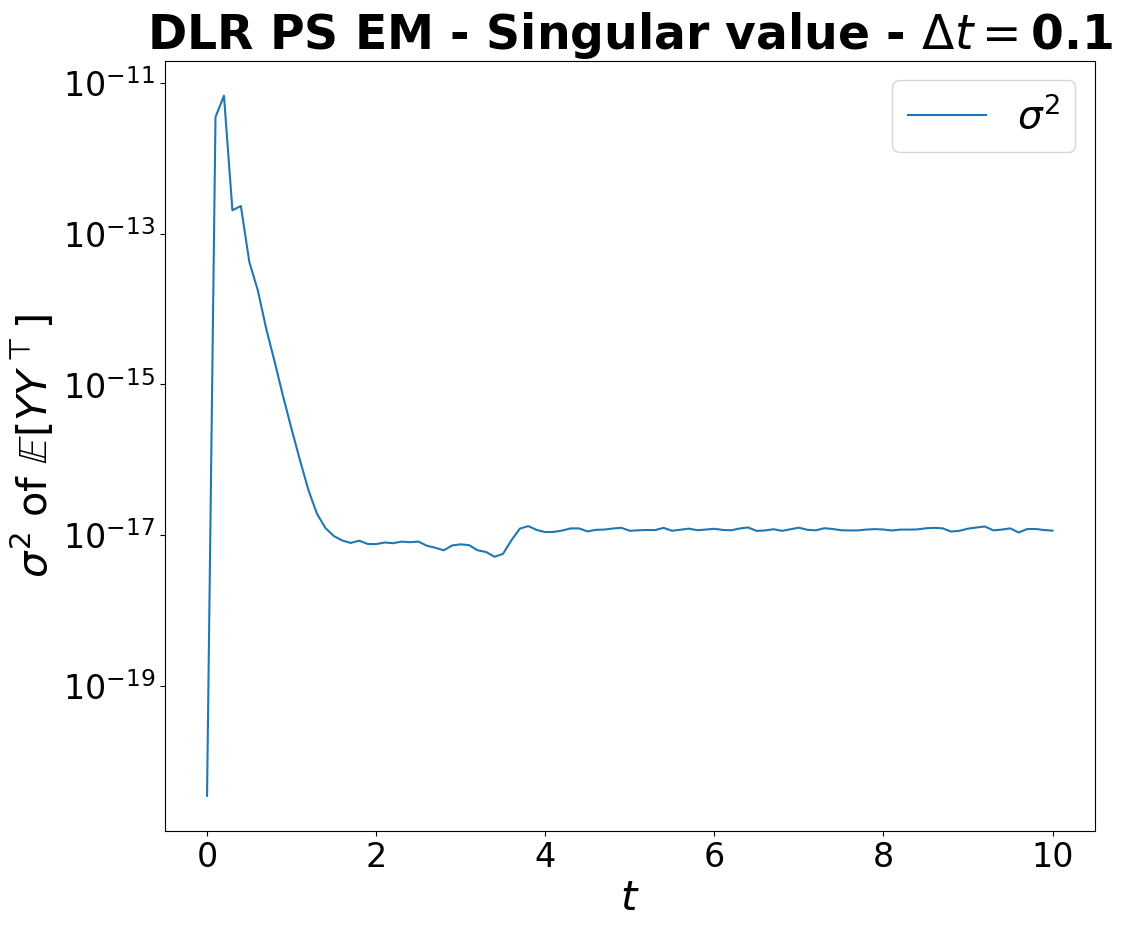}
	\includegraphics[scale=0.16]{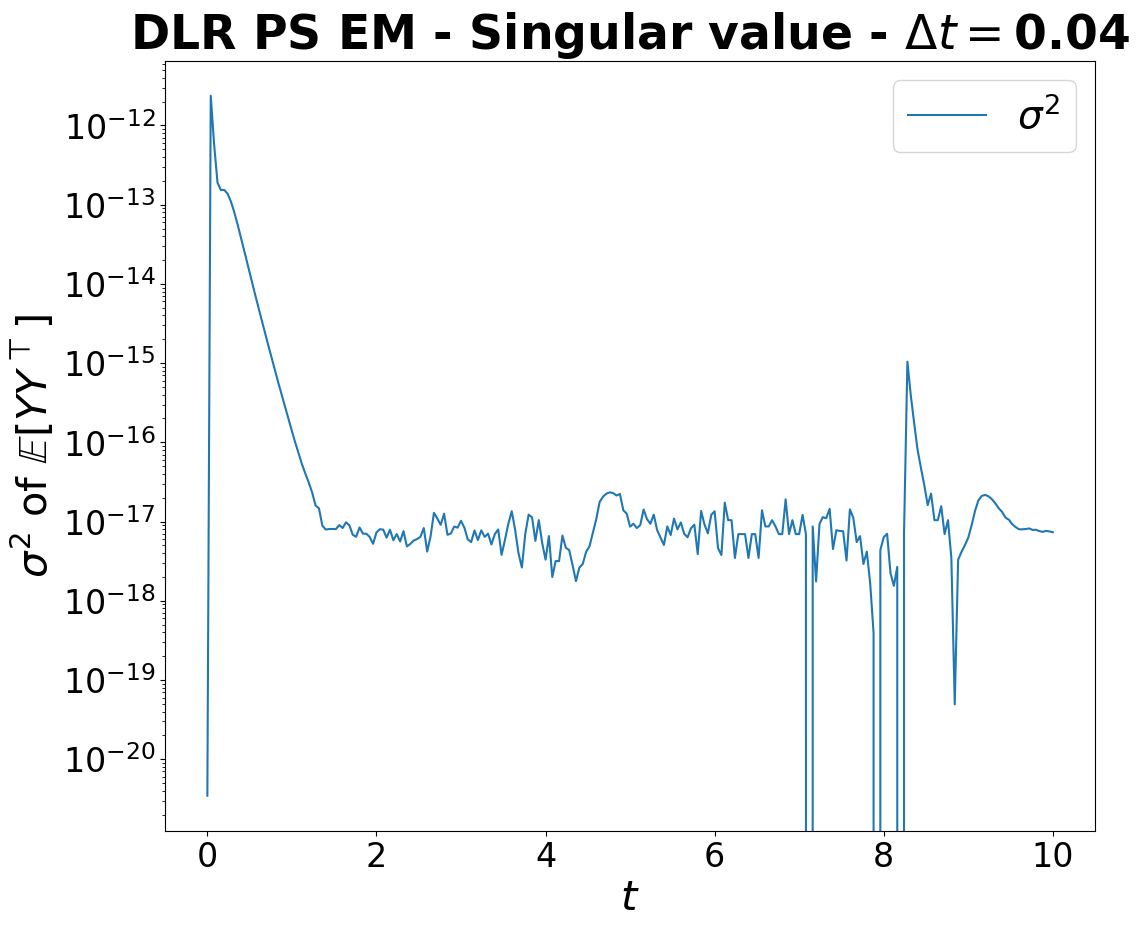}
	\includegraphics[scale=0.16]{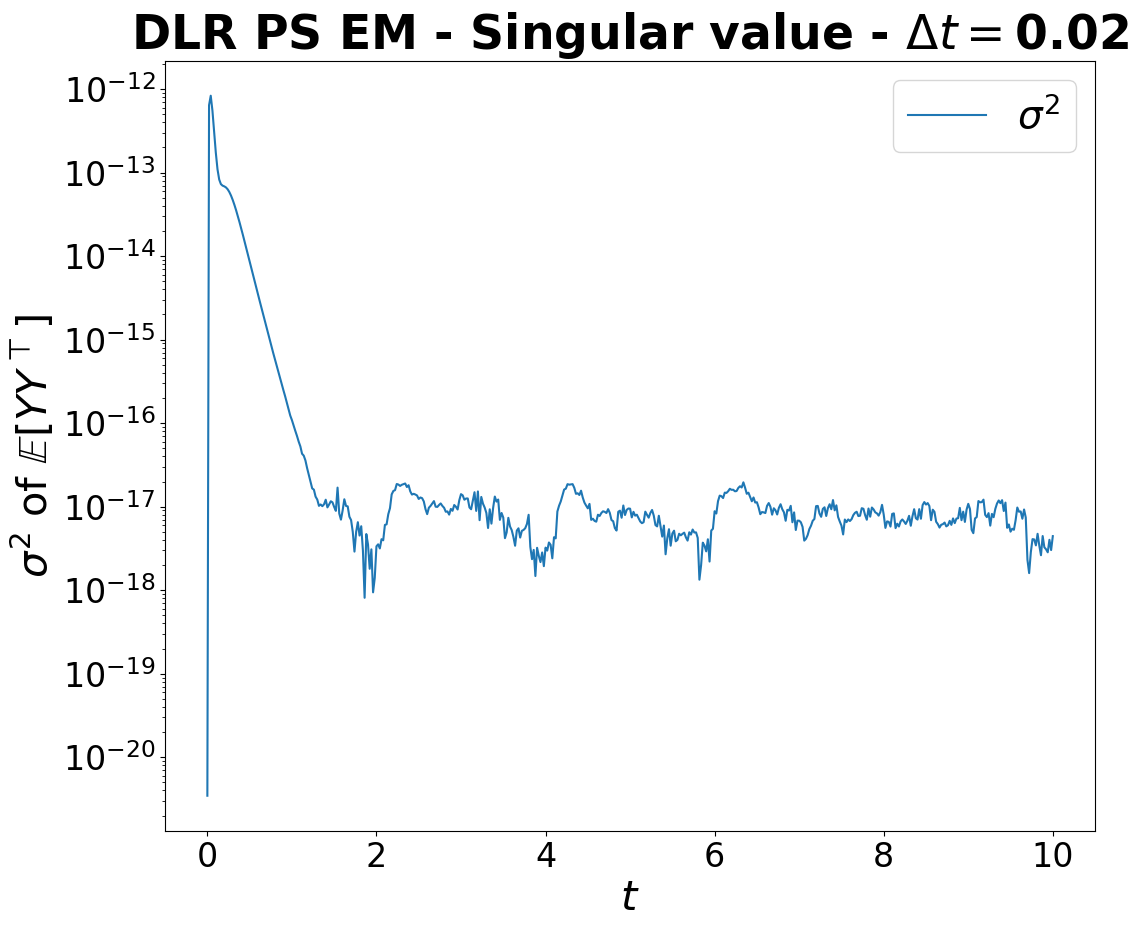}\\
	\includegraphics[scale=0.16]{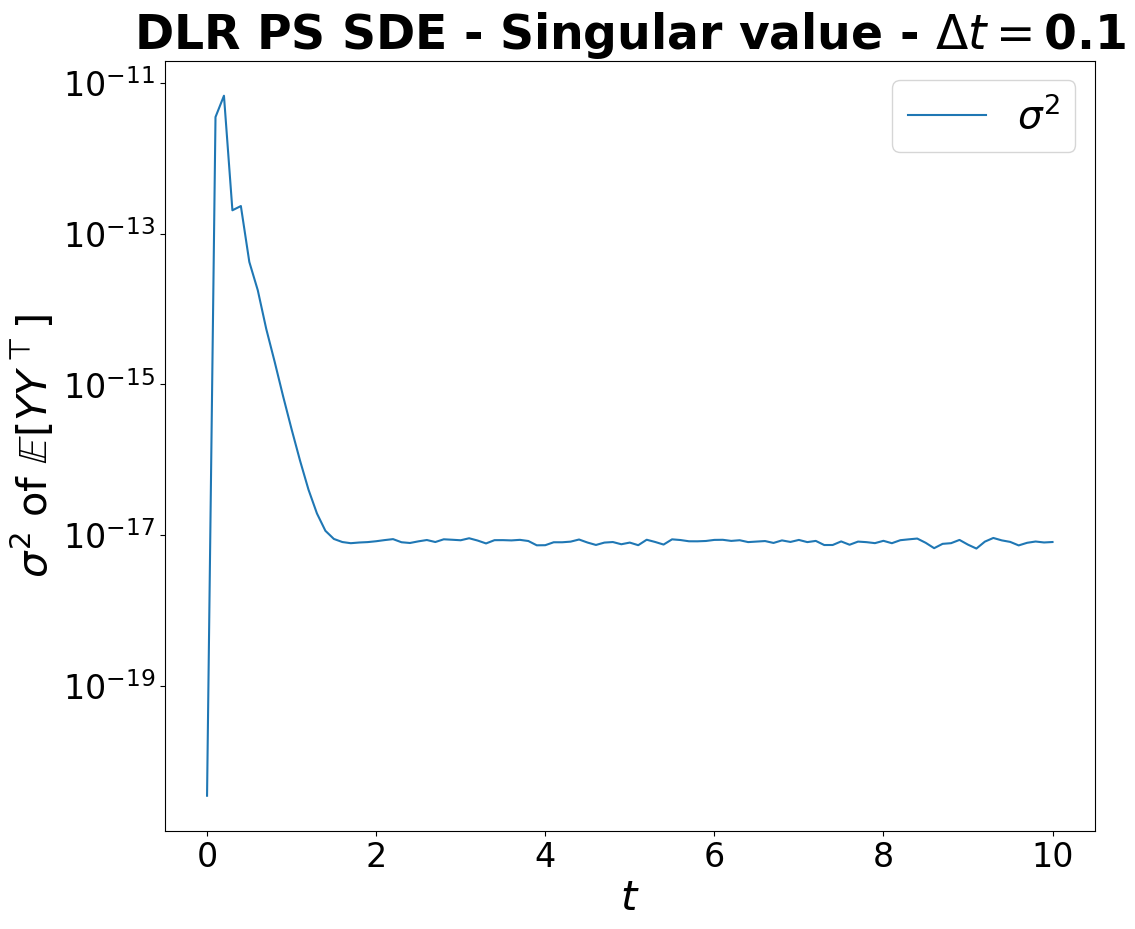}
	\includegraphics[scale=0.16]{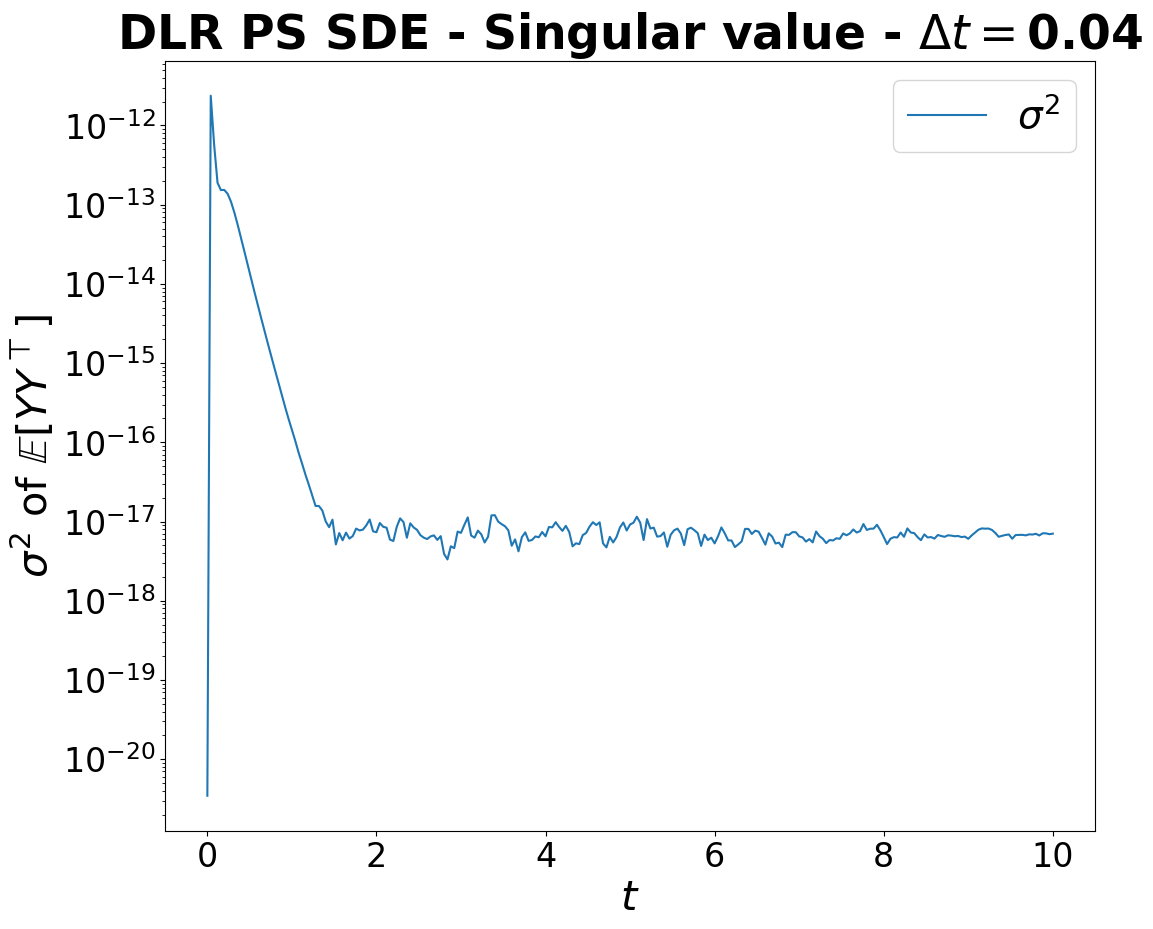}
	\includegraphics[scale=0.16]{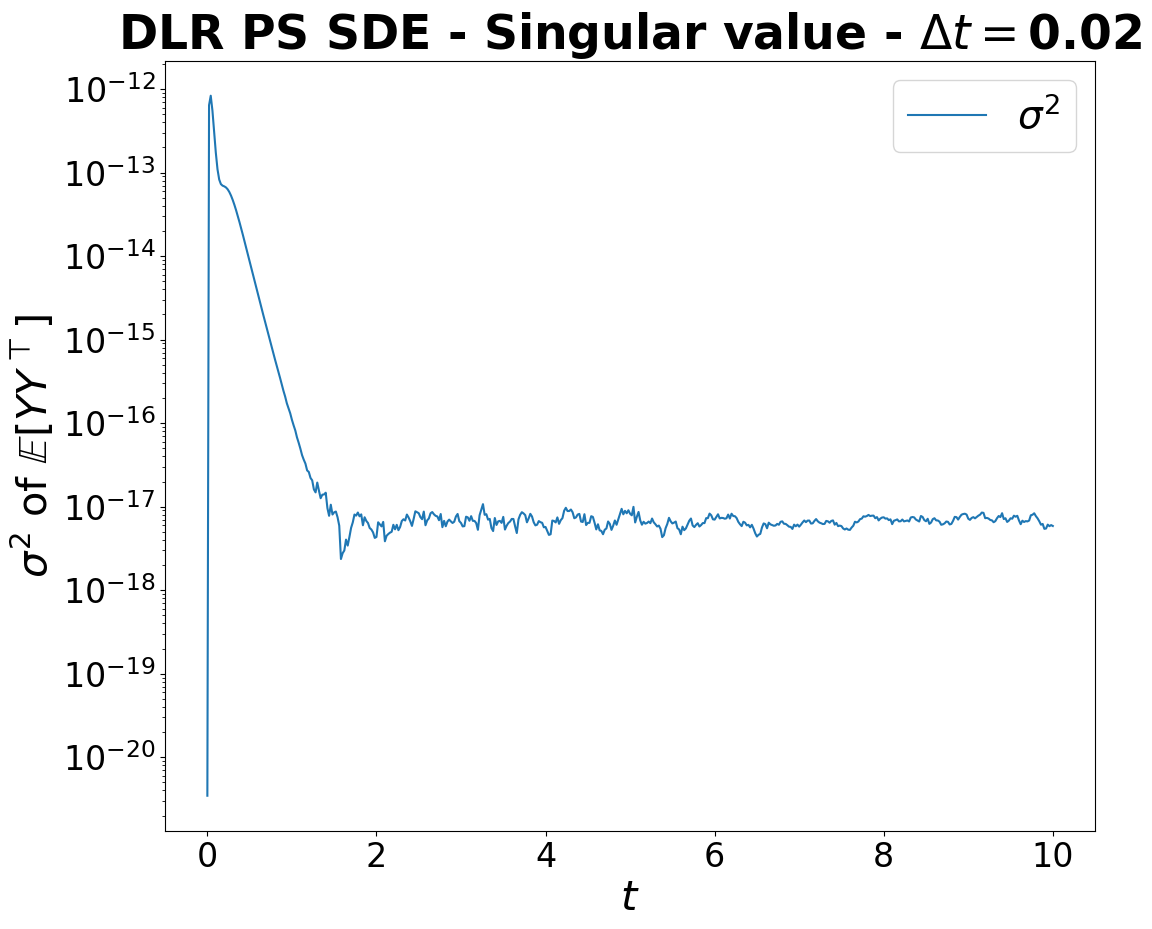}\\
	\caption{For $\Delta t=0.1,0.04,0.02$, $k=2$, $M=10000$, $\sigma_{B}=10^{-16}$. For Problem \eqref{ex: toy example 3} (Top) smallest singular values of $\mathbb{E}[Y_nY_n^{\top}]$ for the DLR Euler-Maruyama, (Middle) for the DLR Projector Splitting for SDEs, (Bottom) for the DLR Projector Splitting for EM.}
	\label{fig:toy example sig values 3}
\end{figure}
Figure \ref{fig:toy example error 3} provides numerical evidence for this behavior. Indeed, the accuracy of the DLR EM is not acceptable, as it is showing convergence over time but with a huge constant. On the other hand, the accuracy of DLR Projector Splitting for SDEs and for EM is unaffected by the trend of the smallest singular value $\sigma^{2}$.
\begin{figure}[!h]
	\centering
	\includegraphics[scale=0.29]{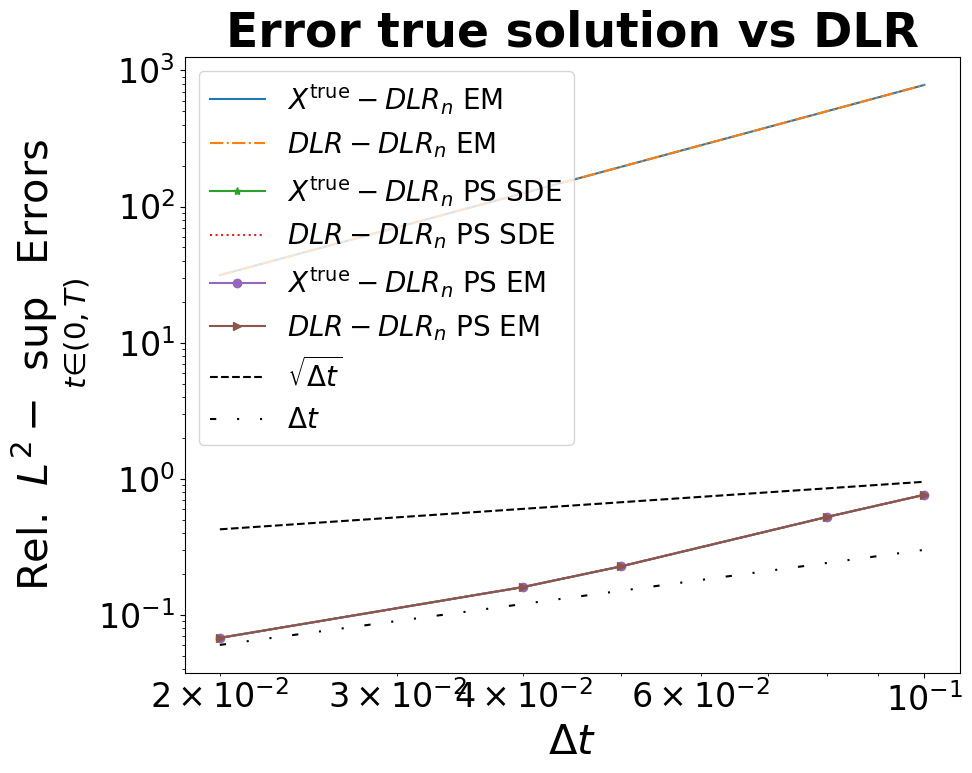}
	\caption{Relative errors for the DLR Euler-Maruyama, DLR Projector Splitting for SDEs, and the DLR Projector Splitting for EM, with respect to the true solution $X^{\mathrm{true}}$ and the continuous DLRA of rank $k=2$ for equation \eqref{ex: toy example 3}. $M=10000$ paths are used in all methods to approximate expectations. Again, the DLR PS EM and DLR PS SDE present quite identically accuracy and their error line is overlapped}.
	\label{fig:toy example error 3}
\end{figure}

\subsection{Mean-square stability}
In this section, we verify that the conditions found in Sections \ref{sec: Stoc Proj Stab} and \ref{sec: KNV stab} are sufficient for the mean-square stability of the studied system. Furthermore, we see that for linear systems, mean-square stability is independent of the smallest singular value of the initial condition and of the ill-conditioning of the Gramian of the stochastic basis.
We consider a process $X^{\mathrm{true}}(t,\omega)\in \mathbb{R}^{10}$ for all $\omega \in \Omega$ and $t \in [0,+\infty)$, whose $i$-th component $X^{\mathrm{true}}_i(t)$ satisfies the SDE
\begin{equation}\label{eq: sde stab}
	\mathrm{d}X^{\mathrm{true}}_i(t) = a_{ii}(t) X^{\mathrm{true}}_i(t)\mathrm{d}t + B_{i}X^{\mathrm{true}}_i(t) \mathrm{d}W^{i}_t,
\end{equation}
where  $A(t) = (a_{ij}(t))_{ij} \in \mathbb{R}^{d \times d}$, with $a_{ii}(t) =-22$ if $i=1,2,3$, else  $a_{ii}(t) =-22+\sin(3\pi t)$ for $i=4,\dots,10$, and $a_{ij}(t)=0$ for $i \neq j$, whereas $B_{i} = (b_{ij})_{ij}$ with $b_{11} =b_{22} =0.1 \sqrt{5}$, $b_{ii} = 10^{-16}$ for $i>2$, otherwise $b_{ij}=0$. $\{W^{i}_t\}_i$ are independent Brownian motions, also independent of the initial condition $X^{\mathrm{true}}(0)$. The initial datum is $X^{\mathrm{true}}_i(0) = 1 + 0.005 \sum_{i=1}^{5} \sin(\pi \frac{i}{10})\mathcal{N}_i(0,1)$, where $(\mathcal{N}_i(0,1))_{i=1,\dots,5}$ are independent standard normal random variables. Notice that the drift is linear, hence in this case the equation \eqref{eq: DLR EM 1v} does simplify to \eqref{eq: DLR EM 3v} and, hence, does not depend on the Gramian $C_{Y_n}$. We consider an approximation of rank $k=5$. 
Then, the stability condition on $\Delta t$ for the (standard) Euler-Maruyama for SDEs, the DLR Projector Splitting for EM and the DLR Projector Splitting for SDEs applied to \eqref{eq: sde stab} is 
\begin{equation}\label{eq: stab cond}
	\begin{aligned}
	\Delta t &\leq \frac{-\lambda_{\max}(A(t_n) + A^{\top}(t_n) ) - \sum_{k=1}^{m}|B_k(t_n)|^2 }{\sigma_{\max}(A(t_n))^2}  \approx 0.0907
	\end{aligned}
\end{equation}
We initialize all the DLR algorithms as described in Section \ref{sec: basic sde}. 
Figure \ref{fig:mean square norm 1} shows the time evolution of the second moment of the solution for $\Delta t = 0.0907$, $\Delta t= 0.0909$, $\Delta t=0.0911$. One notice that the second moment remains numerically stable as expected, when the chosen $\Delta t$ satisfies \eqref{eq: stab cond}. On the other hand, it is illustrated that the three algorithms are not stable when \eqref{eq: stab cond} is violated. Furthermore, Figure \ref{fig:mean square norm 1} confirm us that \eqref{eq: stab cond} is pretty sharp, as for $\Delta t=0.0911$, whose value is higher than \eqref{eq: stab cond}, all algorithms do not converge to the null solution.
\begin{figure}[!h]
\includegraphics[scale=0.135]{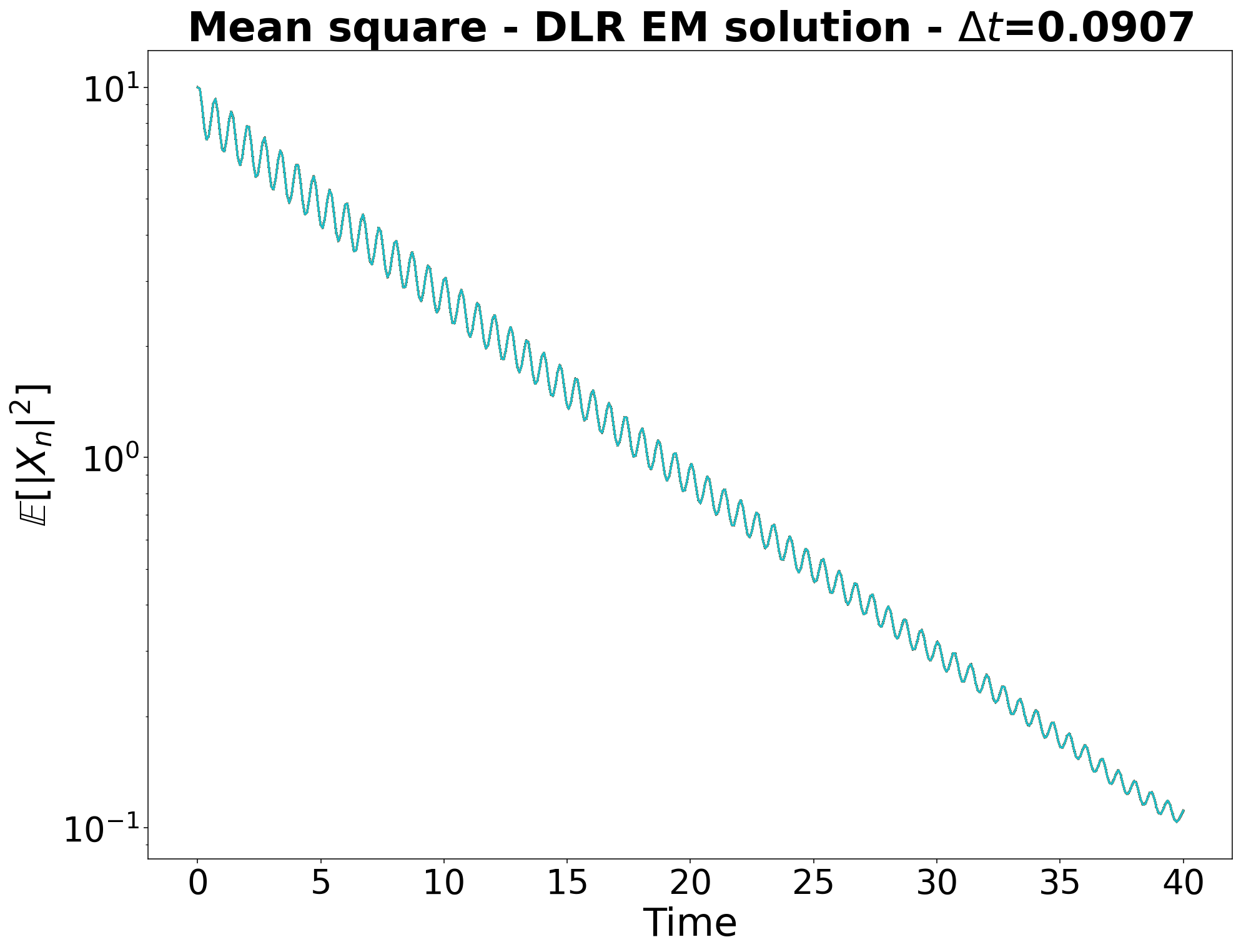}
	\includegraphics[scale=0.135]{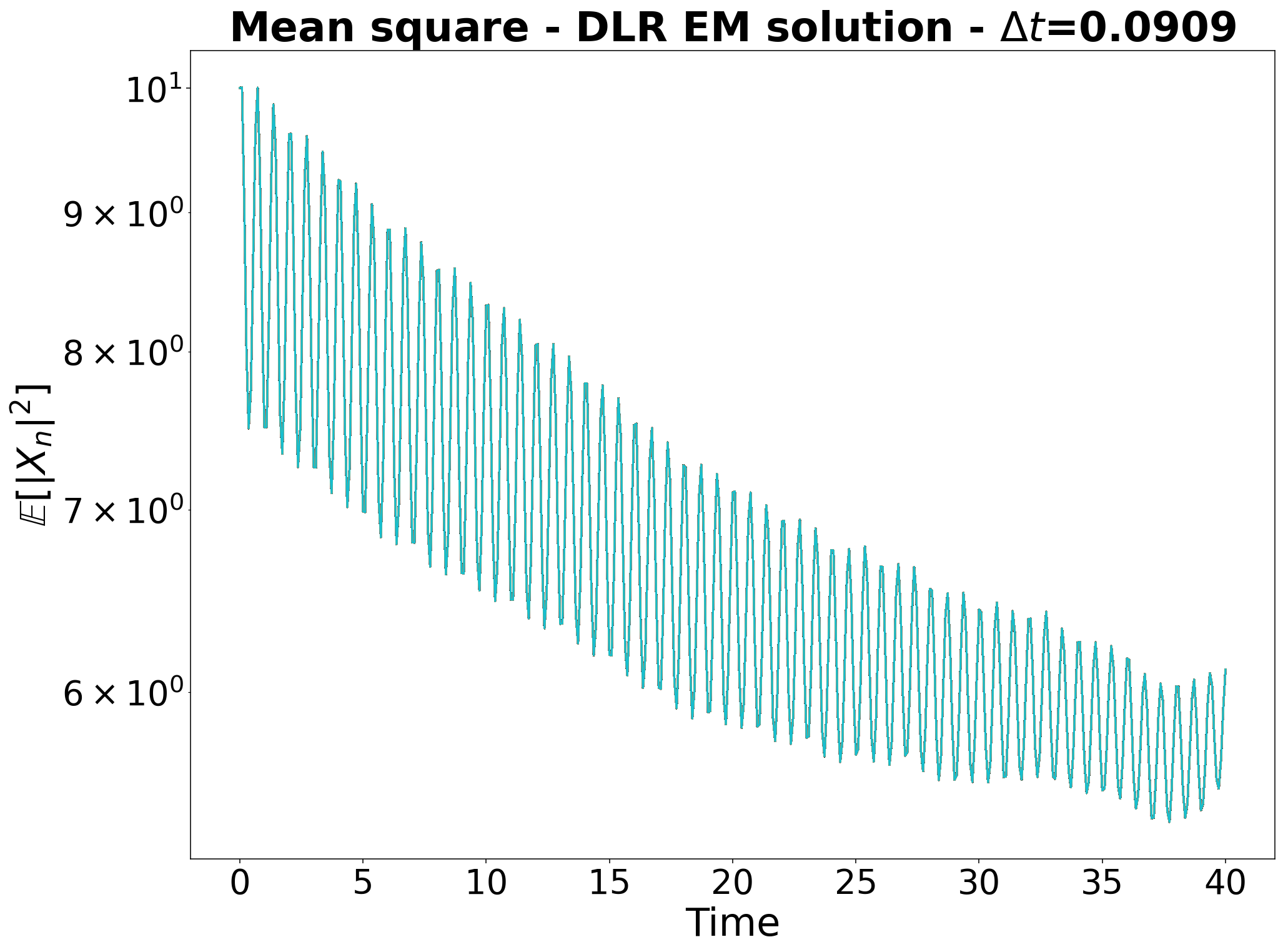}
		\includegraphics[scale=0.135]{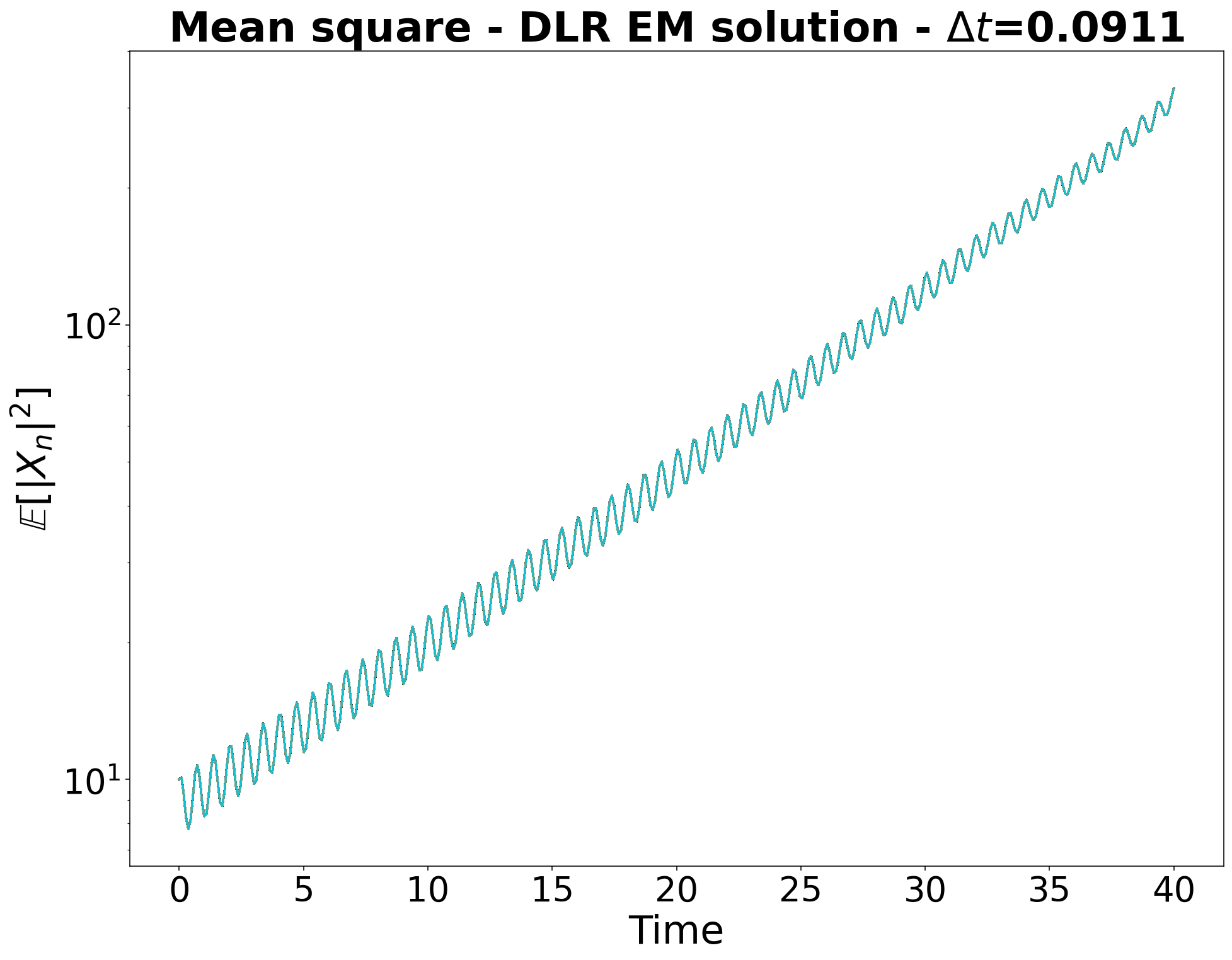}\\
\includegraphics[scale=0.135]{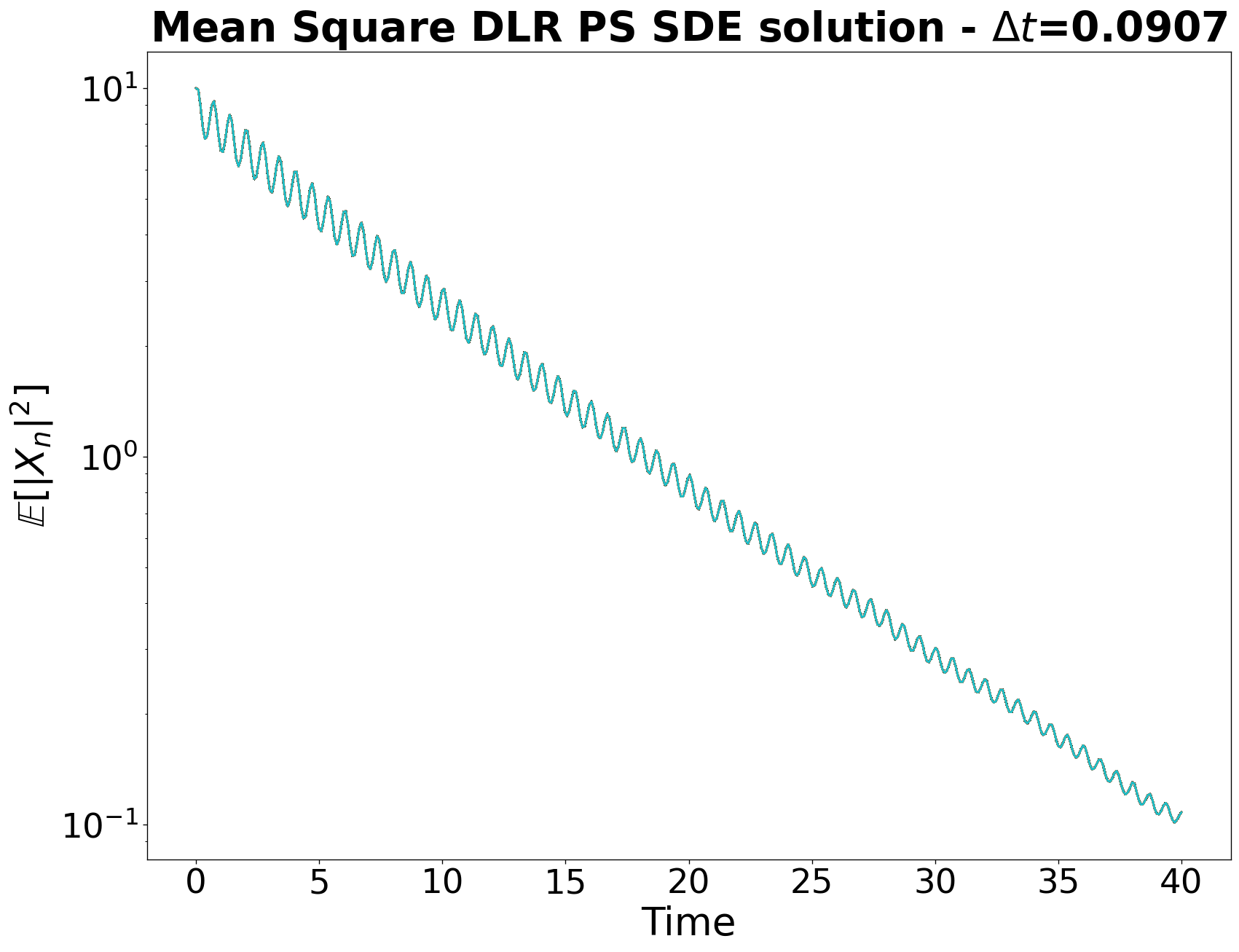}
\includegraphics[scale=0.135]{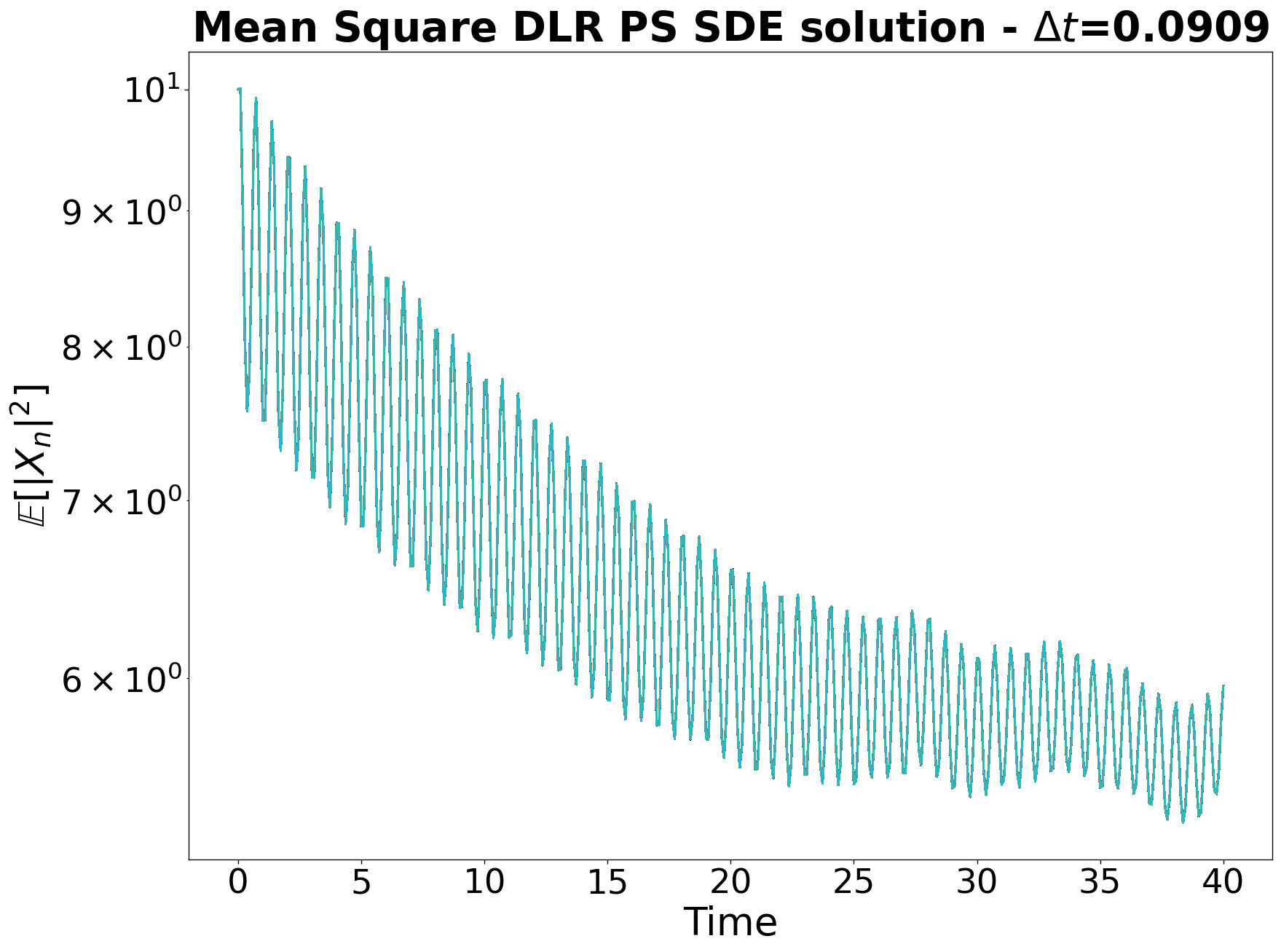}
\includegraphics[scale=0.135]{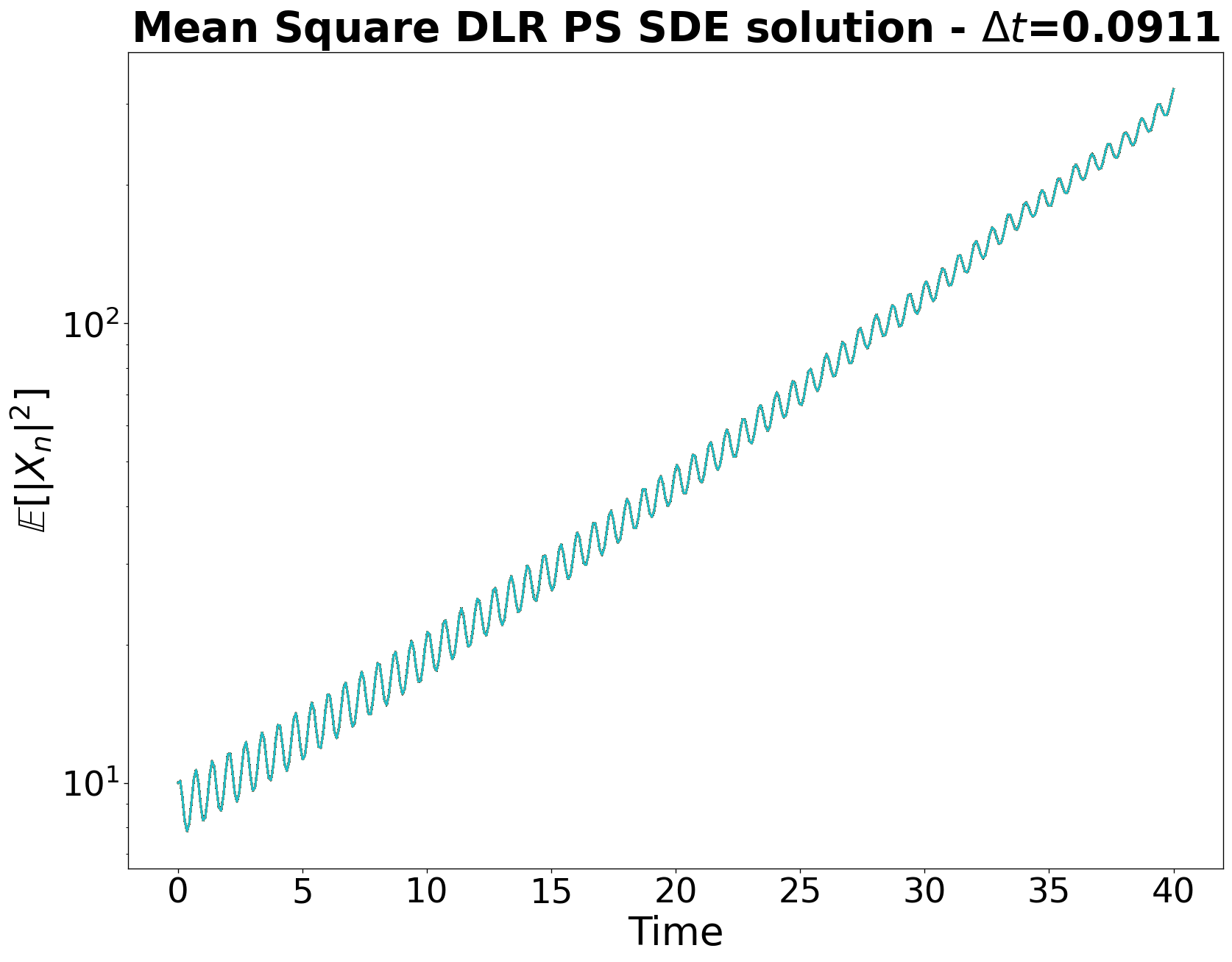}\\
\includegraphics[scale=0.137]{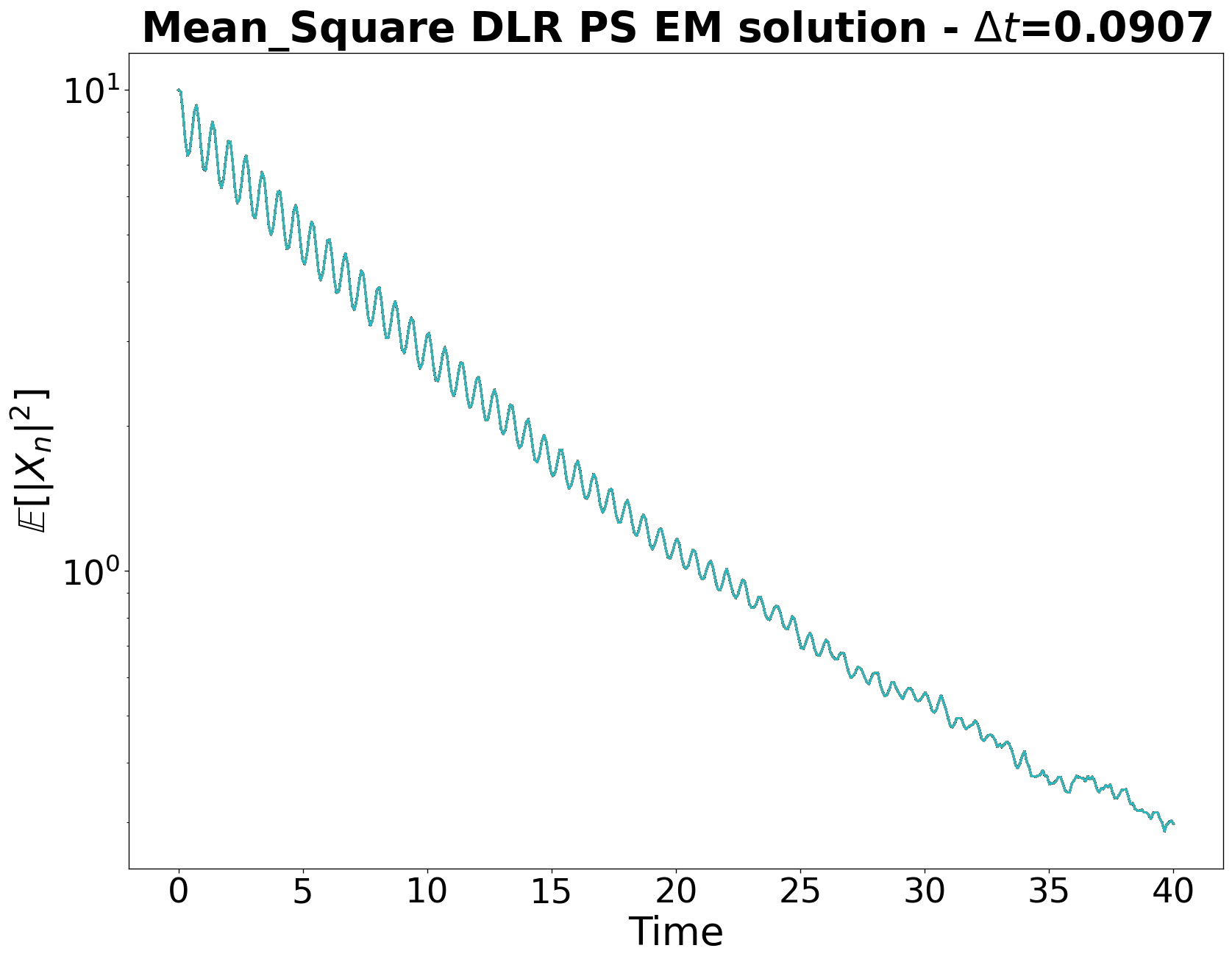}
	\includegraphics[scale=0.137]{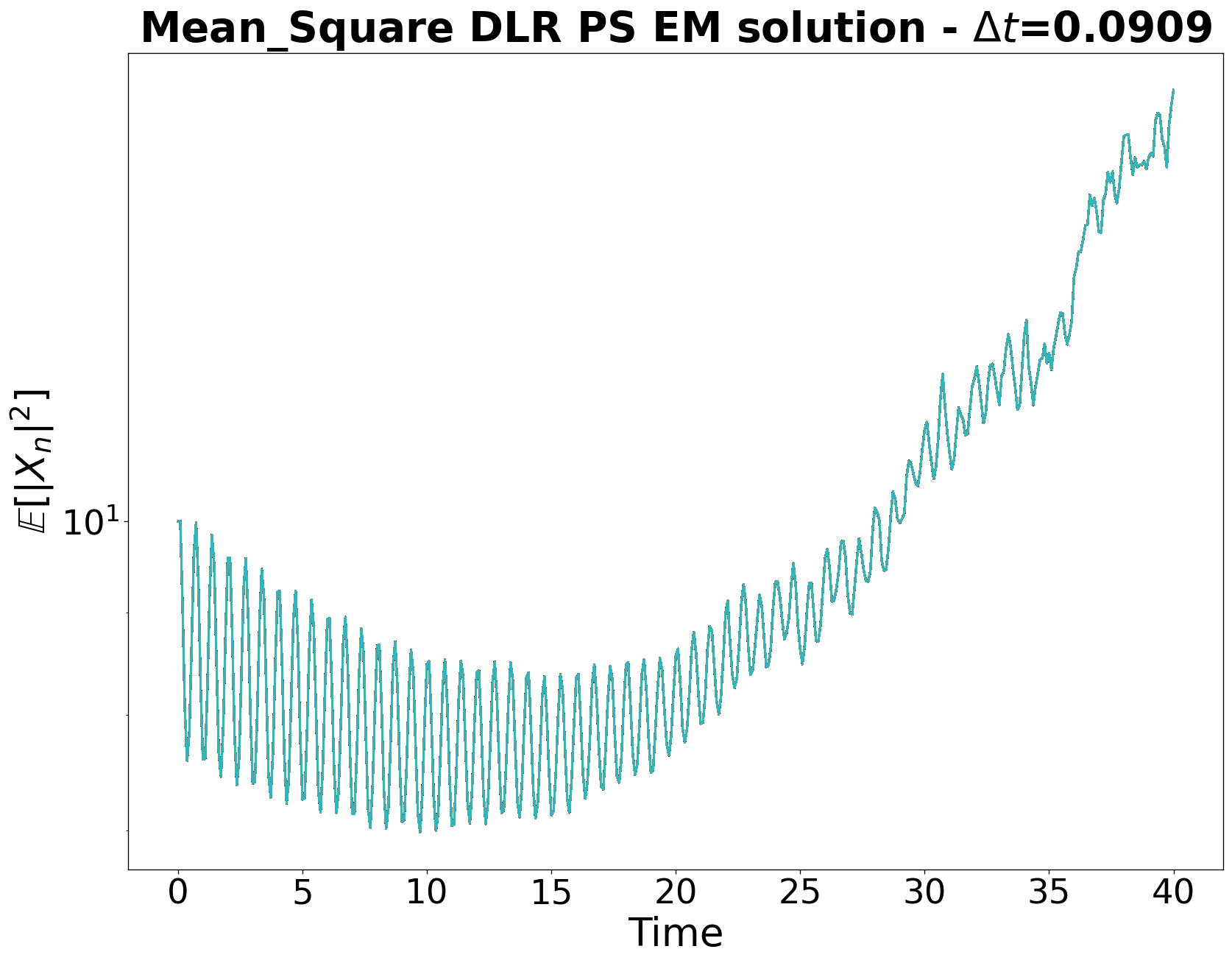}
\includegraphics[scale=0.137]{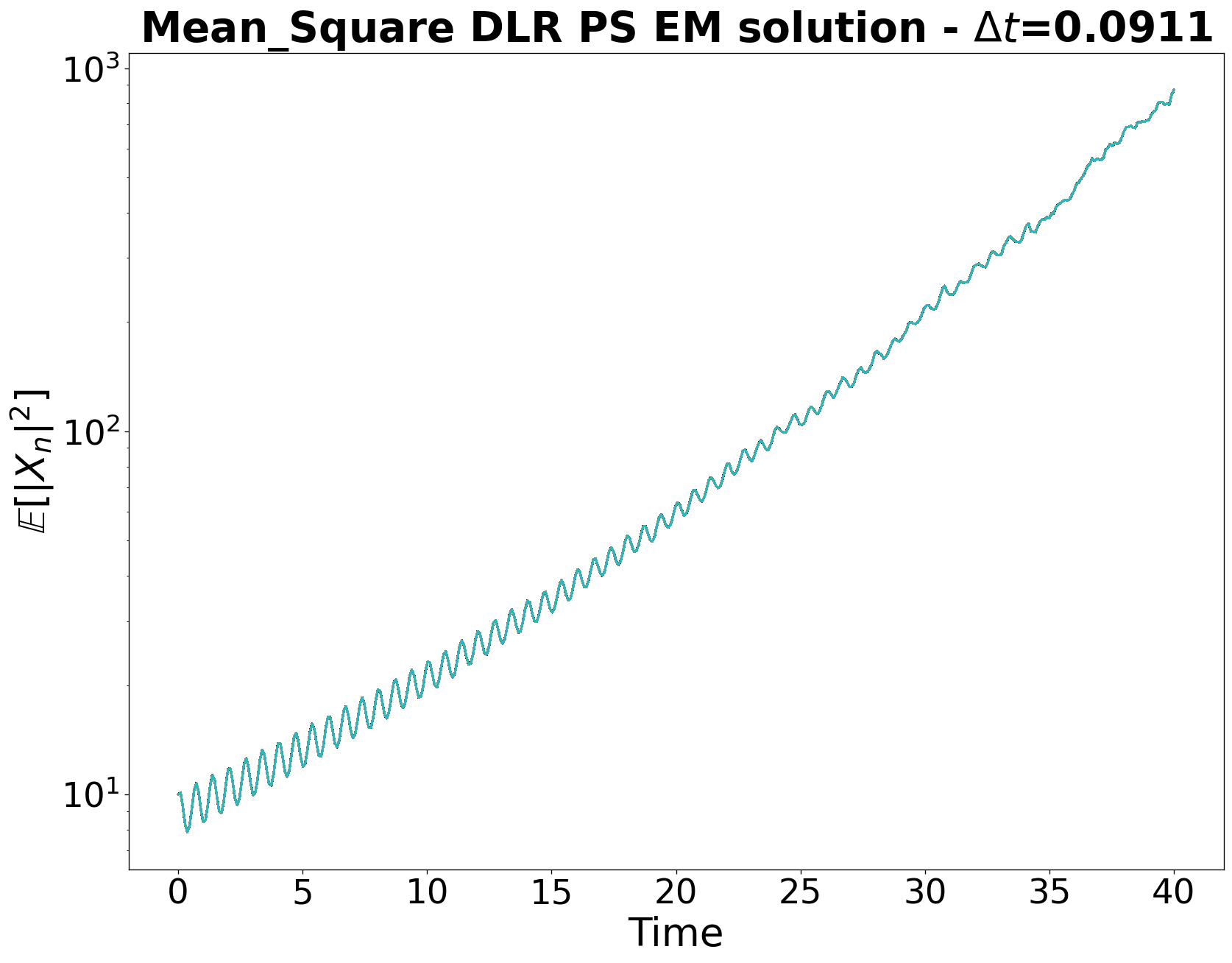}\\
\caption{For $\Delta t=0.0907, 0.909, 0.0911$, and equation \eqref{eq: sde stab}:  Mean-square norm of the numerical solution $X_n$, i.e. $\mathbb{E}[|X_n|^2]$, (Top) for the DLR Euler-Maruyama, (Middle) for the DLR Projector Splitting for SDEs, (Bottom) for the DLR Projector Splitting for EM, $M=2000$ paths.}
\label{fig:mean square norm 1}
\end{figure}

We aim to verify that the time step condition for mean square stability does not depend on the smallest singular value of the solution (at least fro the DLR PS SDE and DLR PS EM schemes).
As \eqref{eq: sde stab} is a system characterized by a rank-deficient diffusion matrix, to obtain a nearly rank-deficient dynamics one can just simulate \eqref{eq: sde stab} with a nearly rank-deficient initial condition of the following form
\begin{equation}\label{eq: 2 init cond stab}
	X^{\mathrm{true}}_i(0) = 1+5 \sum_{i=1}^{5} 10^{-15} \sin\left(\frac{ \pi  i x }{10}\right) \text{Un}_{i},
\end{equation}
where $(\text{Un}_{i})_{i=1,\dots,5}$ are independent uniform random variables with values in $[0,1]$. In Figure \ref{fig:mean square norm 1 - overapp} we show that the asymptotic stability behavior follows a similar trend to the one of the previous example, corroborating our thesis. Indeed, we see that the smallest singular value of the DLR EM is always approaching numerical error, also for the time steps which guarantees the numerical stability (see Figure \ref{fig:stability sig values DLR EM}). Despite this behaviour, when $\Delta t$ satisfies \eqref{eq: stab cond} the DLR EM is still asymptotically stable.

\begin{figure}[!h]
	\includegraphics[scale=0.135]{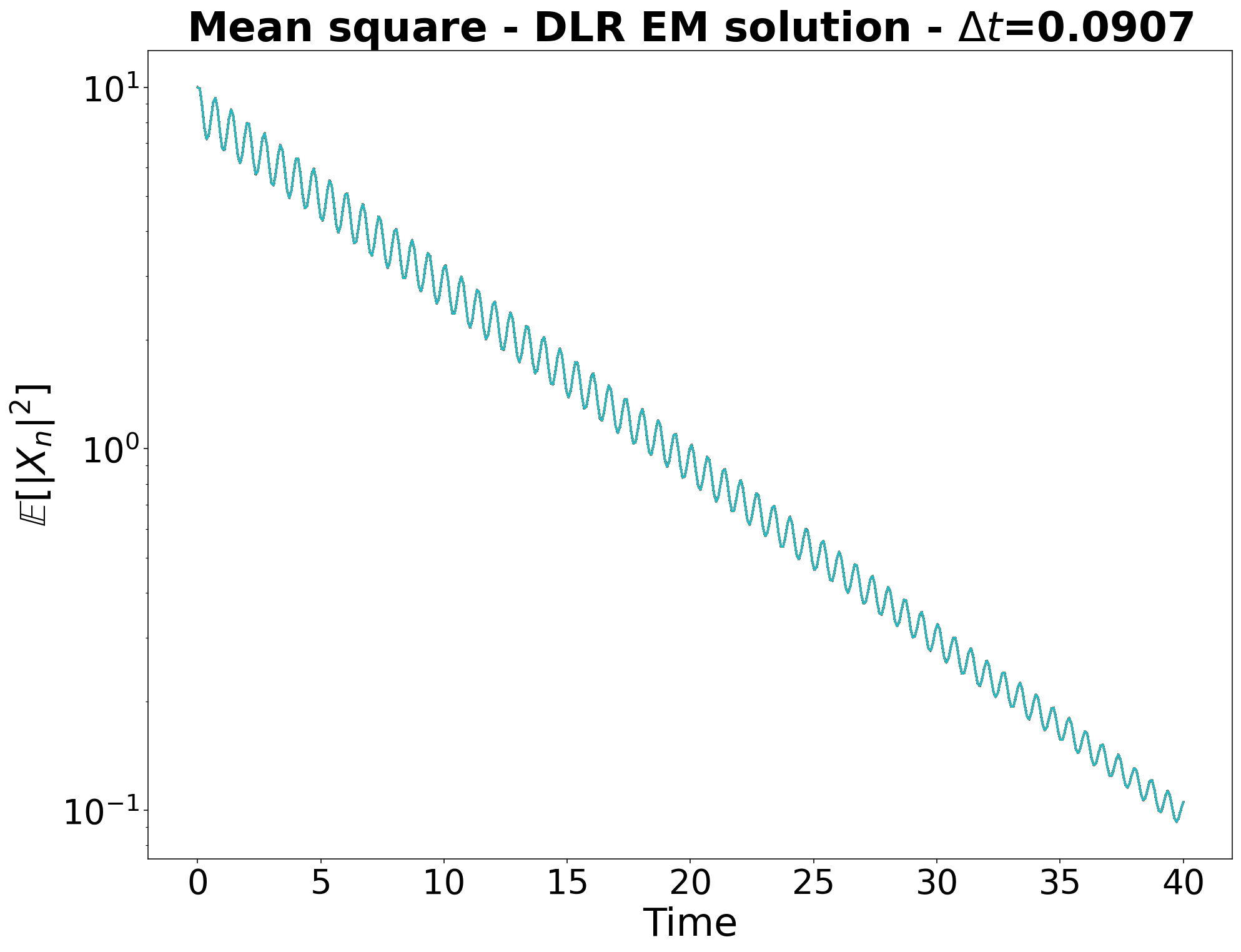}
		\includegraphics[scale=0.137]{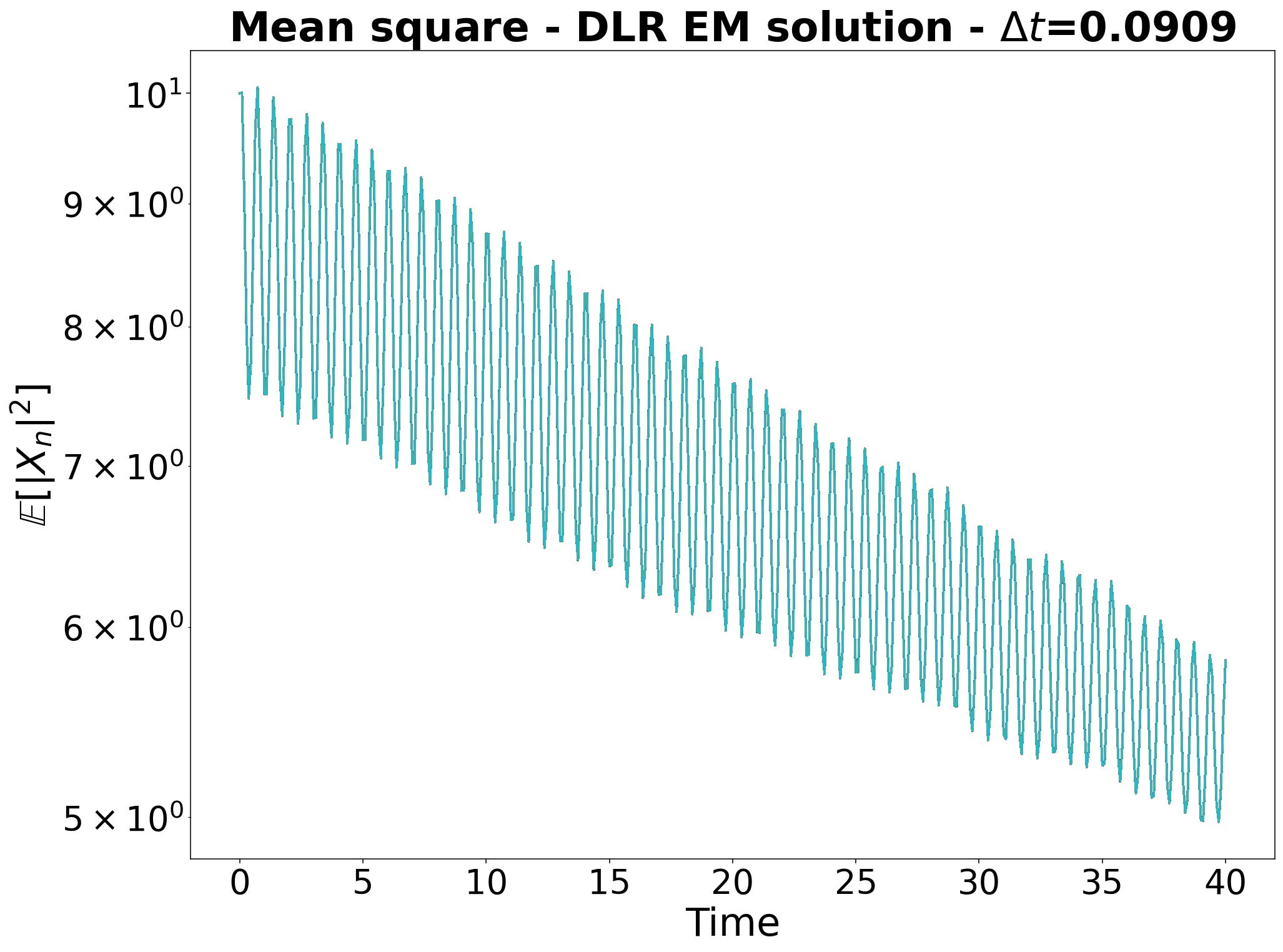}
			\includegraphics[scale=0.137]{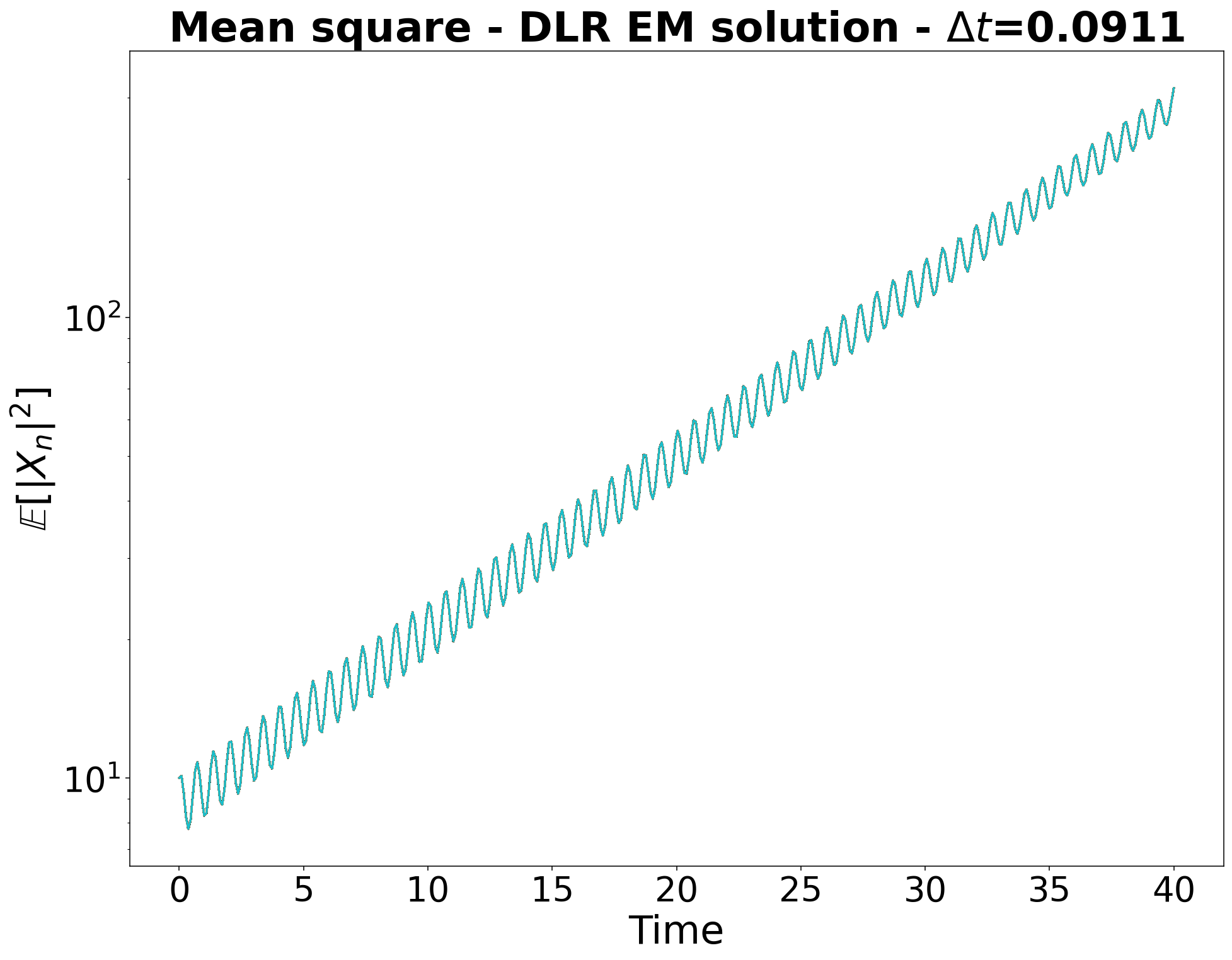}\\
	\includegraphics[scale=0.14]{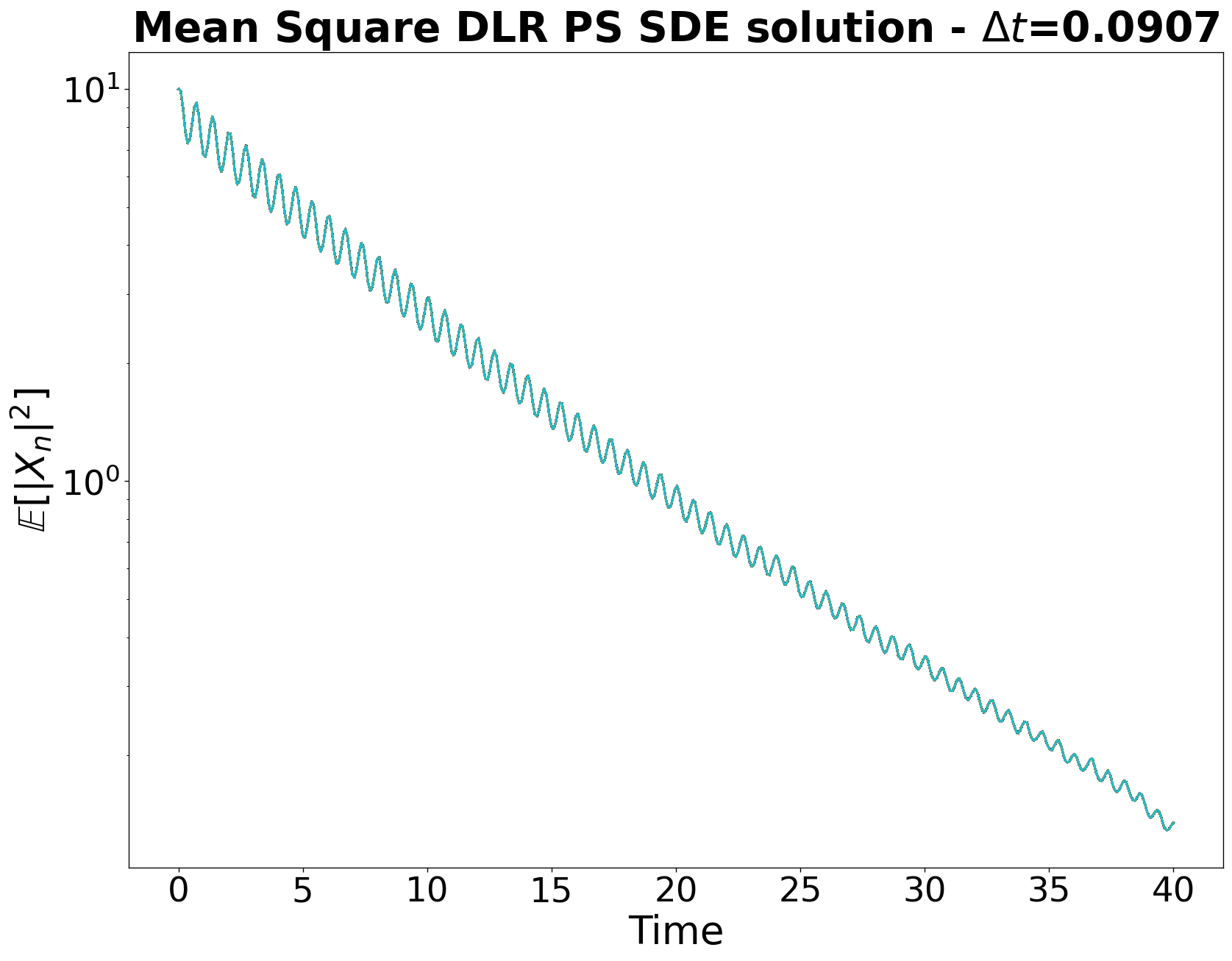}
	\includegraphics[scale=0.136]{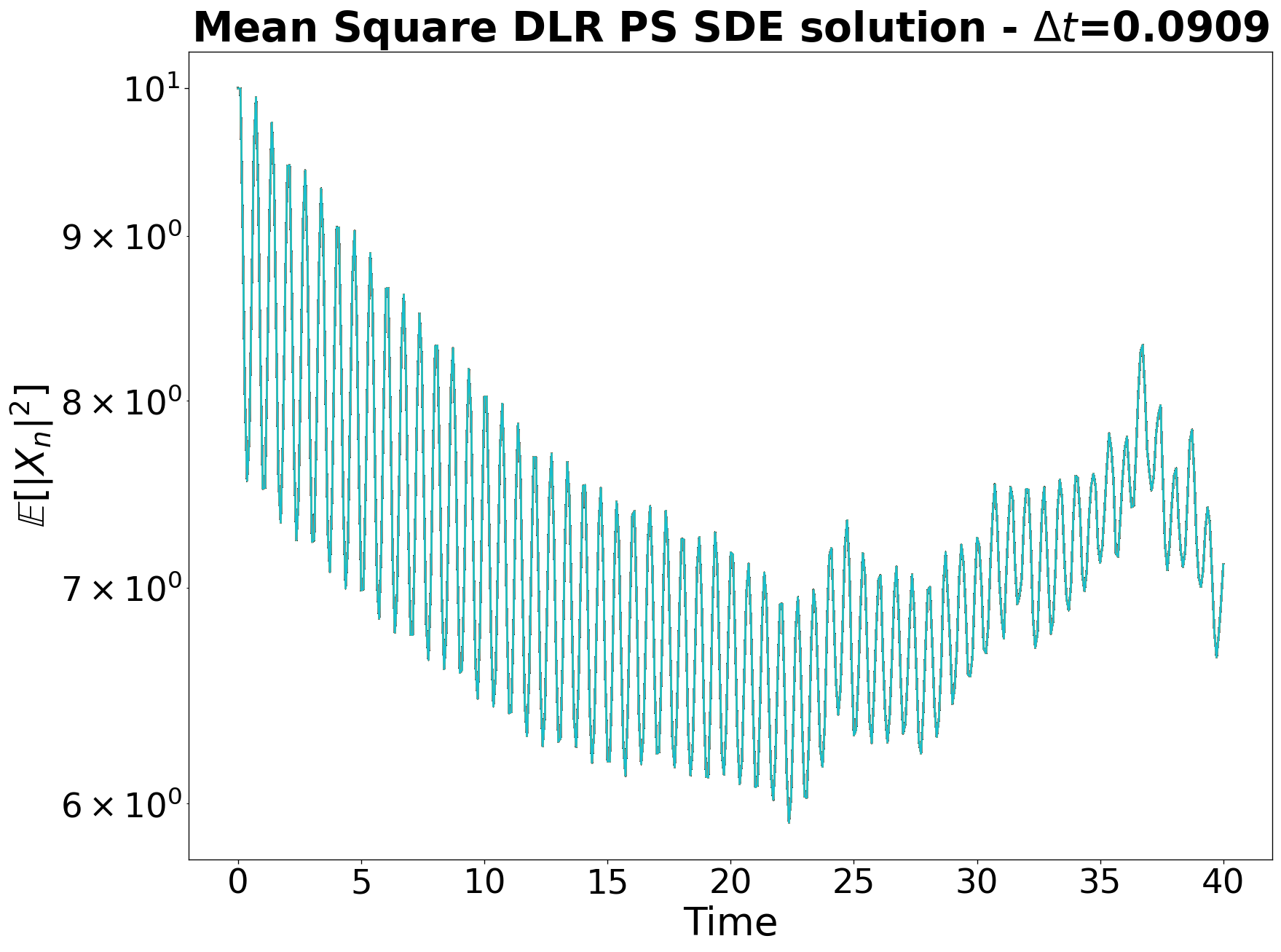}
	\includegraphics[scale=0.136]{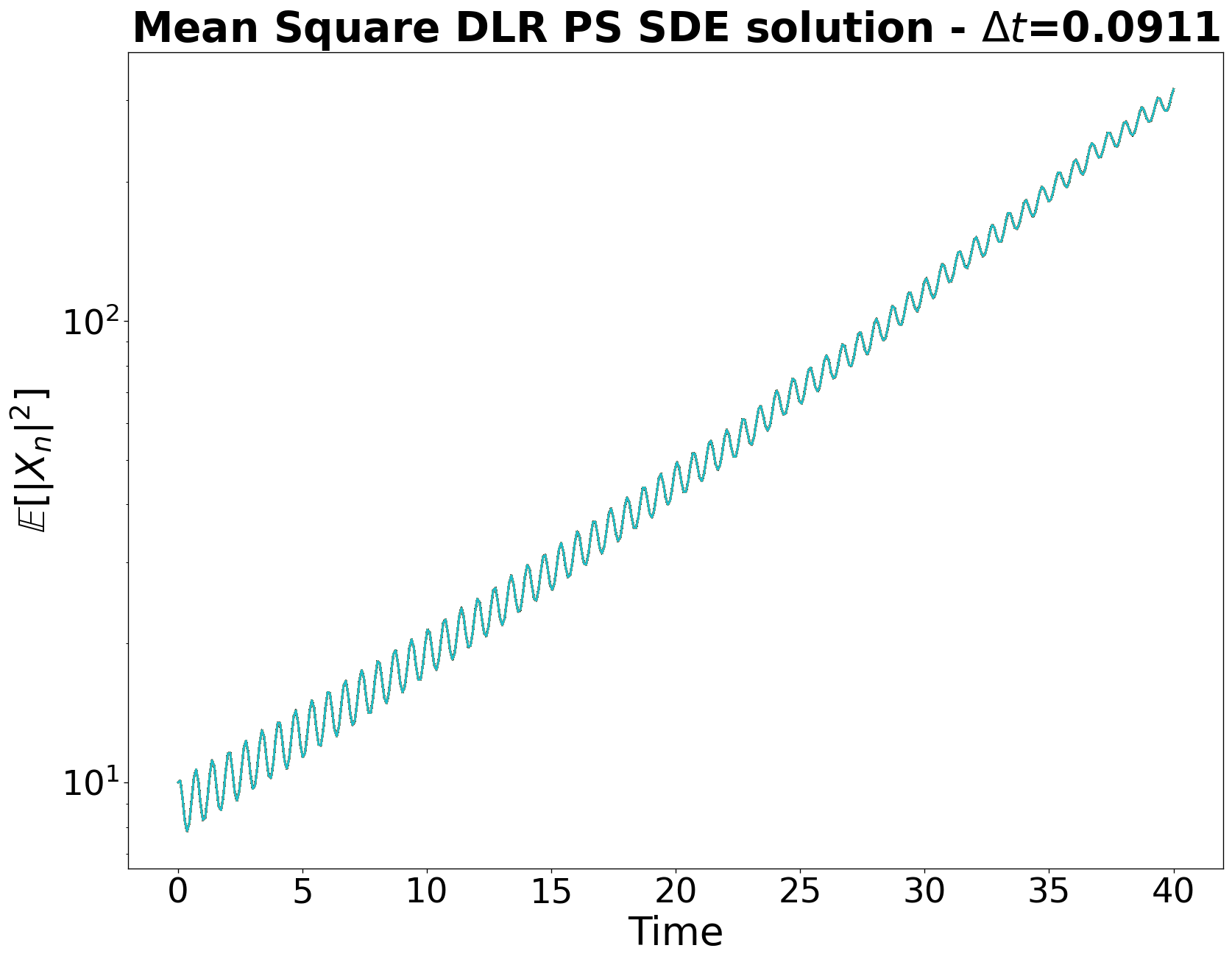}\\
	\includegraphics[scale=0.14]{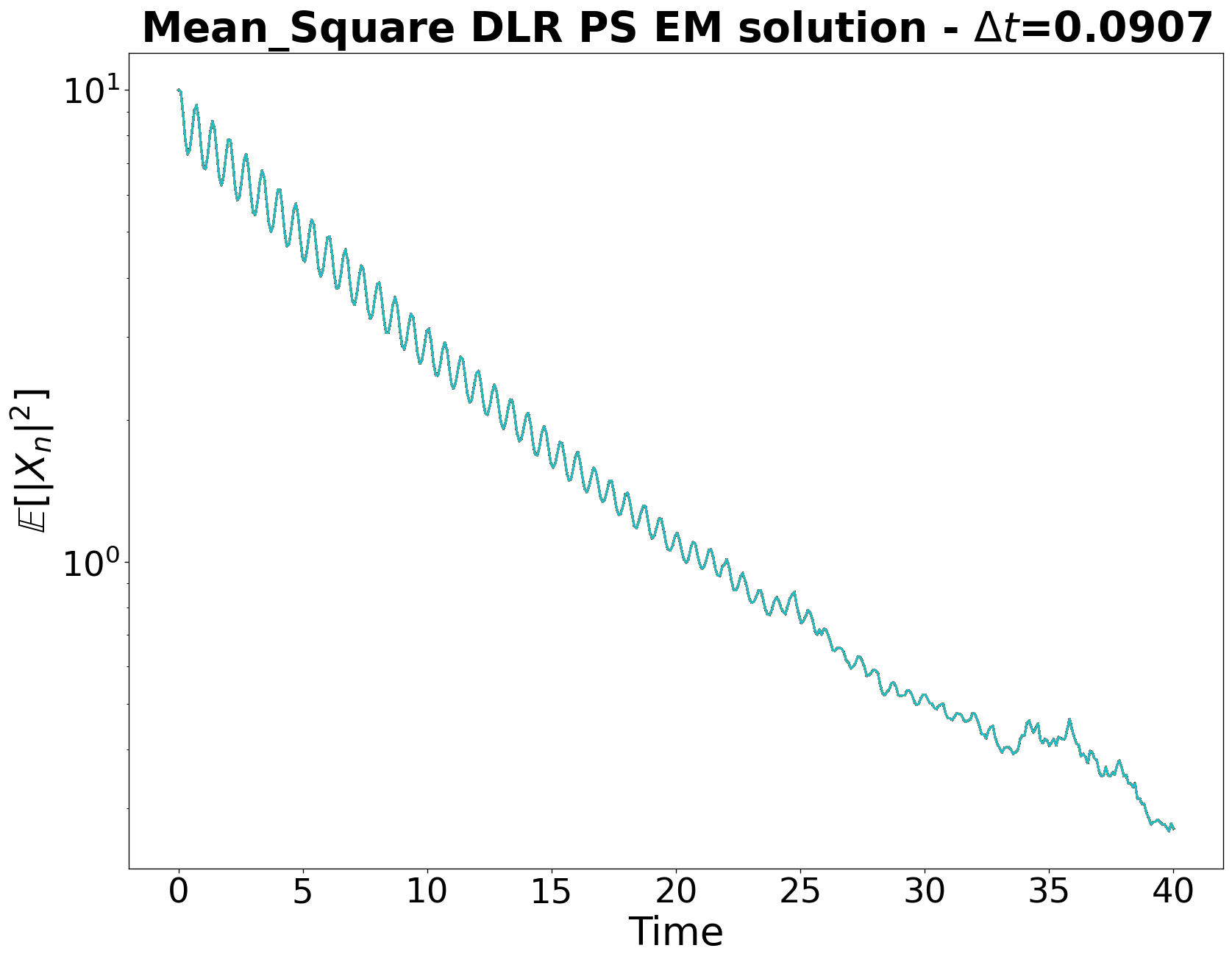}
		\includegraphics[scale=0.138]{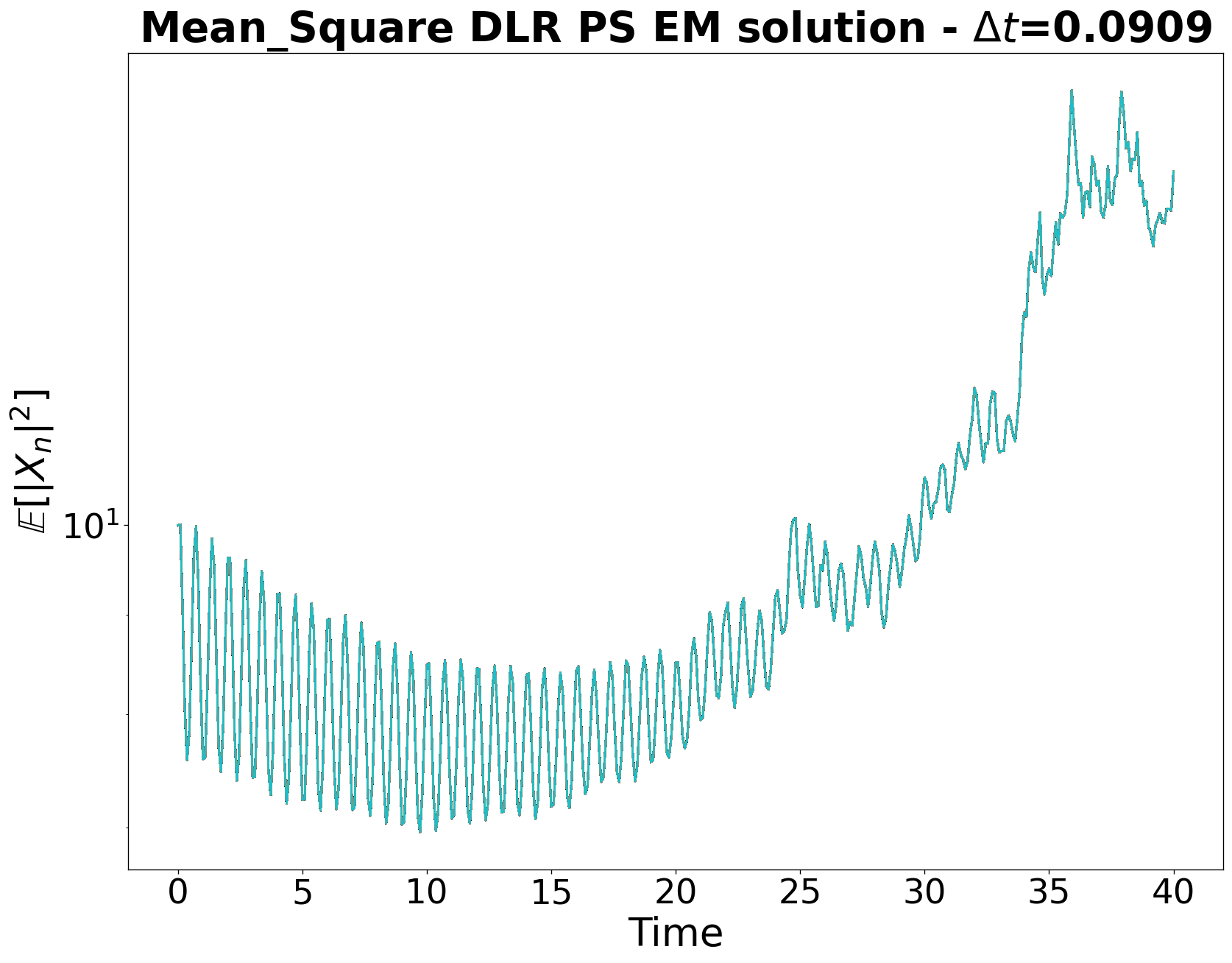}
		\includegraphics[scale=0.138]{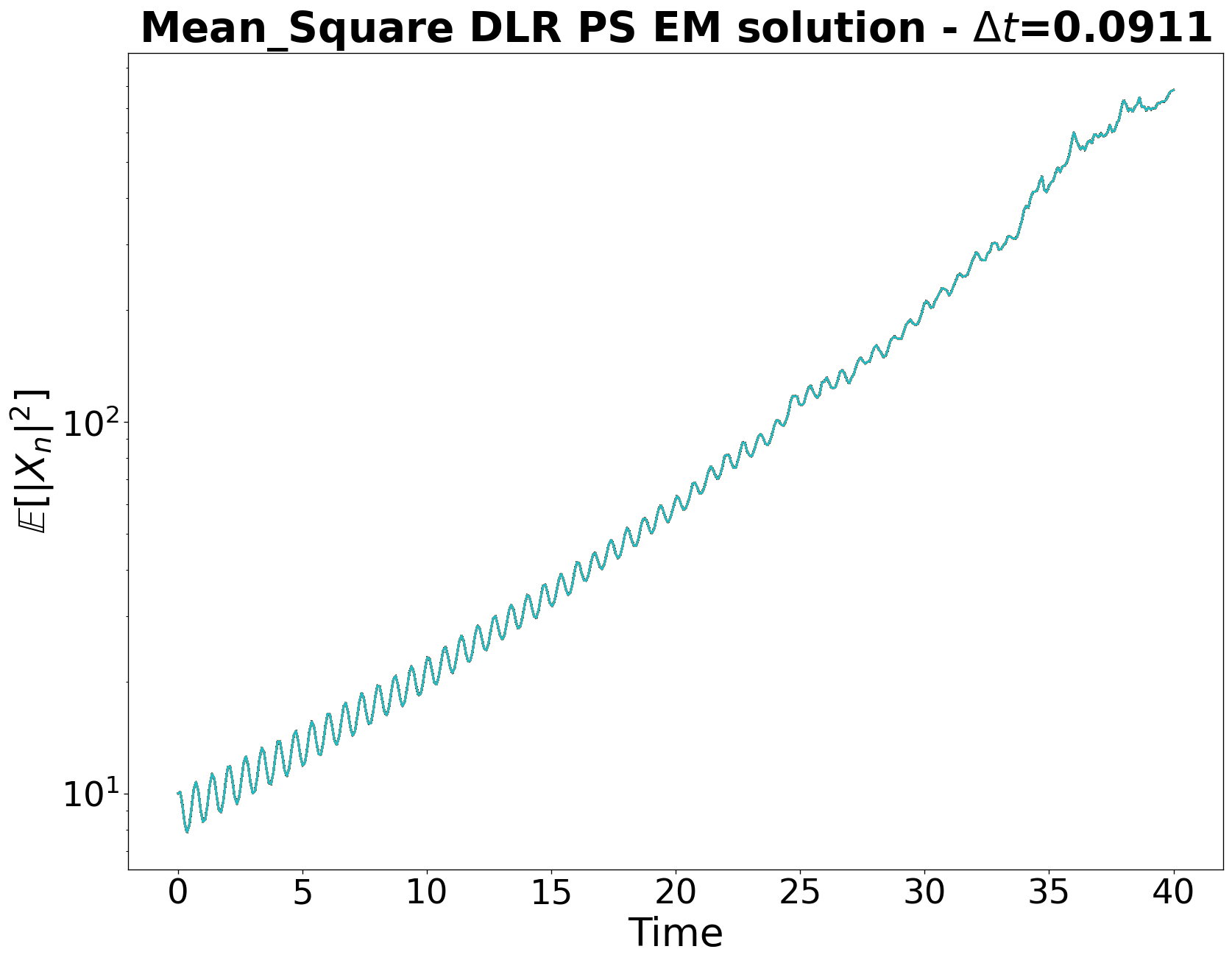}\\
	\caption{For $\Delta t=0.0907, 0.909, 0.0911$, and equation \eqref{eq: sde stab}: Mean-square norm of the numerical solution $X_n$, i.e. $\mathbb{E}[|X_n|^2]$, (Top) for the DLR Euler-Maruyama, (Middle) for the DLR Projector Splitting for SDEs, (Bottom) for the DLR Projector Splitting for EM, $M=2000$ paths, initial condition $X^{\mathrm{true}}(0)$ \eqref{eq: 2 init cond stab}.}
	\label{fig:mean square norm 1 - overapp}
\end{figure}

\begin{figure}[!h]
	\centering
	\includegraphics[scale=0.16]{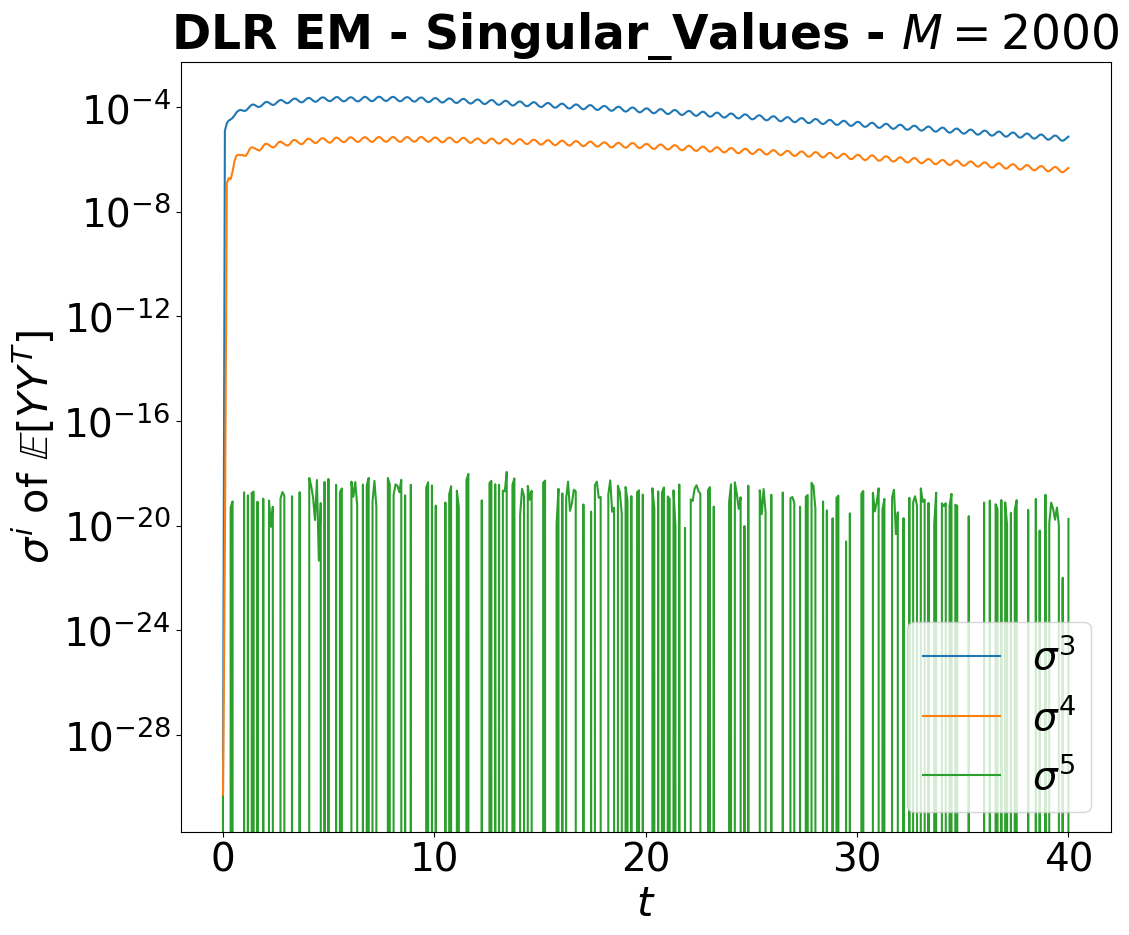}
	\includegraphics[scale=0.16]{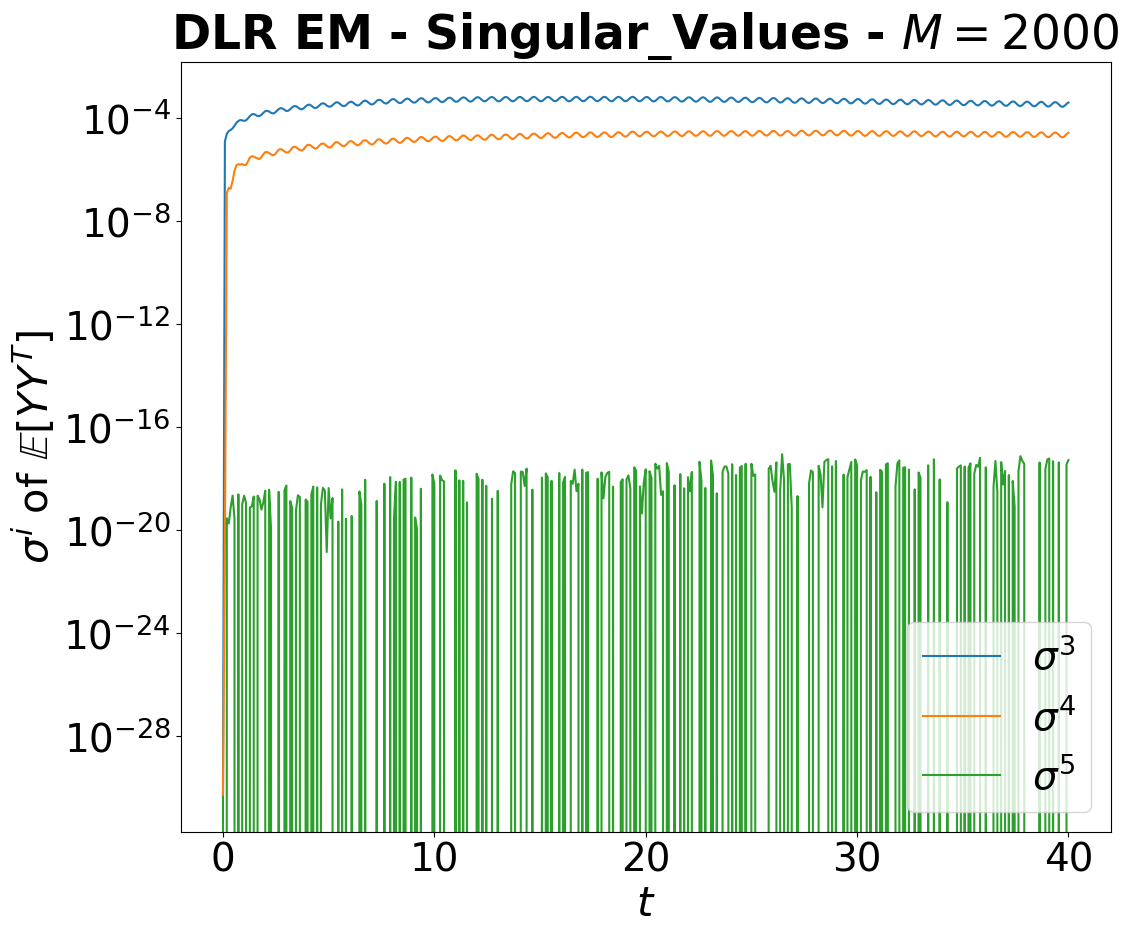}
	\includegraphics[scale=0.16]{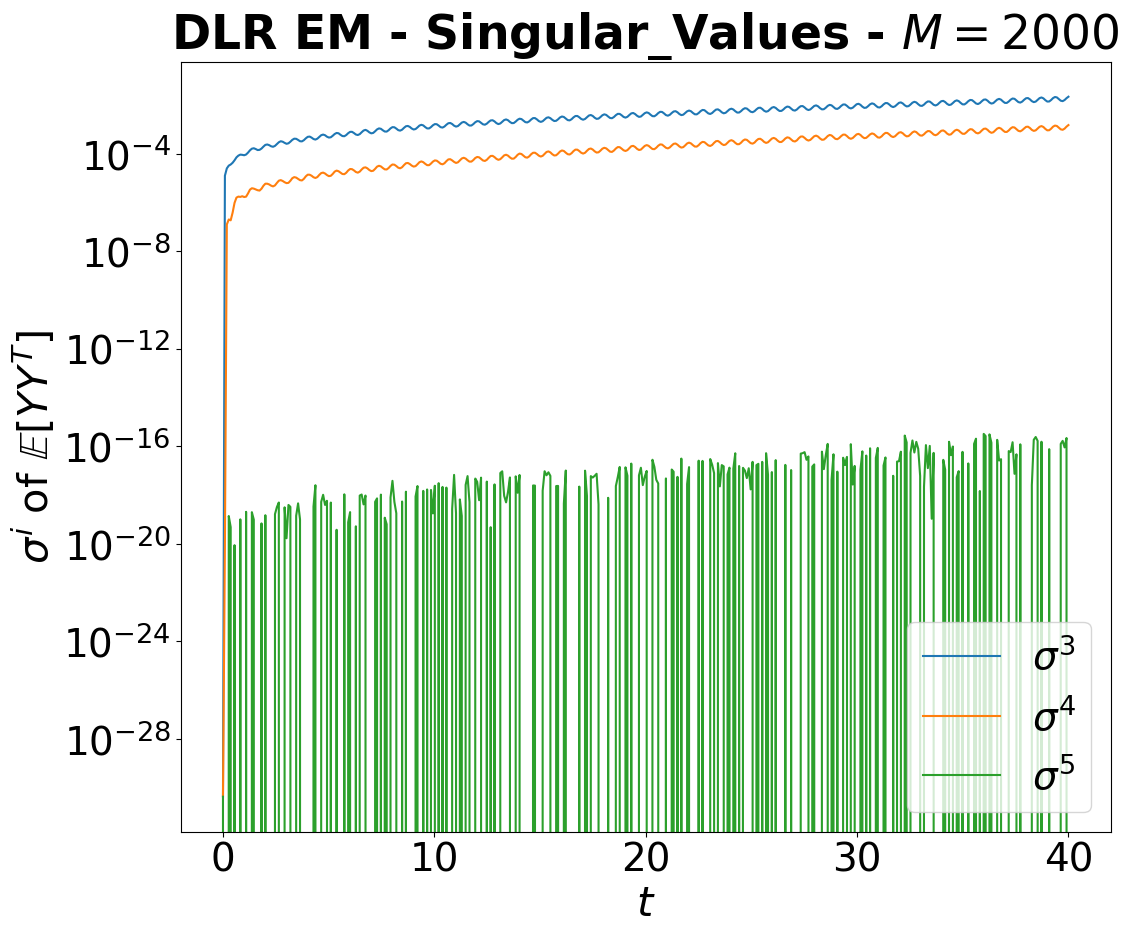}\\
	\caption{For $\Delta t=0.0907,0.0909,0.0911$, $k=3$, $M=2000$ smallest singular values of $\mathbb{E}[Y_nY_n^{\top}]$ for the DLR Euler-Maruyama of Problem \eqref{eq: stab cond} with initial condition $X^{\mathrm{true}}(0)$ \eqref{eq: 2 init cond stab}.}
	\label{fig:stability sig values DLR EM}
\end{figure}

\subsection{Additive Noise: Stochastic Advection-Diffusion-Reaction PDE in 1D}
In this section, we provide a numerical example consisting of a one-dimensional advection-diffusion-reaction system under low-rank additive diffusion suitably discretized in space. We aim to test the numerical convergence of the three algorithms under suitable conditions on the time step based on the smallest singular values of the Gramian of the stochastic basis $Y_n$. The spatial domain is $[0,L]$ with $L=1$, where we consider Neumann boundary conditions, while the temporal domain is $[0,T]$, with $T=10$. The studied equation is the following
\begin{equation}\label{ex: SADR}
	\begin{cases}
		\begin{aligned}
		\partial_t u(x,t,\omega)=& L u(x,t,\omega) + \sum_{i=1}^{m}\phi_i(x) \dot{W}_i(\omega), \quad (x,t,\omega) \in [0,L] \times [0,T] \times \Omega, \\
		u_{0}(x,\omega)=&  \sum_{i = 1}^m \frac{1}{(2 \pi  i)^2}\cos\left( \frac{\pi i x}{ L} \right)  \left(0.5 -\text{Un}_{i}(-10^{-4},10^{-4})\right) \\
		&+  \sum_{i = m+1}^R 10^{-15}\cos\left( \frac{\pi i x}{ L} \right)  \left(\text{Un}_{i}(-0.1,0.1)\right) , \quad (x,\omega) \in [0,L] \times \Omega\\
		\partial_x u(0,t, \omega)=& \partial_x u(1,t,\omega) = 0, \quad (t,\omega) \in [0,T] \times \Omega,
		\end{aligned}
	\end{cases}
\end{equation}
where in \eqref{ex: SADR} the operator $L$ is given by $L u(x,t,\omega) = a \partial_x^2 u(x,t,\omega) - v \partial_x u(x,t,\omega) + r \sin(u(x,t,\omega))$, which is not linear in $u$,  $\{\text{Un}_{i}(-10^{-4},10^{-4})\}_{i=1,\dots,m}$ and $\{\text{Un}_{i}(-0.1,0.1)\}_{i=m+1,\dots,R}$ are i.i.d.\ uniform random variables independent of the one-dimensional Brownian motions $(W_i)_{i=1,\dots,m}$, $a=0.005$, $v= 0.3$, and $r = 0.1$. Concerning the diffusion term, we set $\phi_{i}(x) = 10^{-5} \cos(\frac{i \pi x}{L})$ for all $i=1,\cdots, m$. We discretize in space \eqref{ex: SADR} by second order centered finite differences, with a first order upwinding treatment of the advection term, using a uniform grid with mesh size $\mathrm{d}x = 0.04$. Therefore, the physical dimension is $d=25$. The time discretization for the reference true solution $u$ of \eqref{ex: SADR} is given by the forward Euler-Maruyama method with mesh size $\Delta t=2\cdot 10^{-3}$. The number of paths simulated is $M=4000$. We choose $m=5$ and we consider a DLR approximation of rank $k=18$. Moreover, we initiate all the DLRA algorithms as described in Section \ref{sec: basic sde}. Notice that with this choice of the rank we are overapproximating the initial condition and by the noise term, that affect the system: indeed we have 13 modes very close to the zero machine. We point out that the operator $L$ has a non-linear reaction term. Hence, the evolution of the deterministic modes and the DLR-EM scheme is governed by \eqref{eq: DLR EM 1v}, rather than \eqref{eq: linear drift} and we do expect the performance of the scheme to be affected by the smallest singular value of $\mathbb{E}[Y_nY_n^{\top}]$. According to Theorem \ref{thm: convergence of DLR Euler-Maruyama}, we may not observe any convergence for $\Delta t$ not satisfying \eqref{eq: DLR EM sde stab}, and for $\Delta t \to 0$ the $O(\sqrt{\Delta t})$ convergence may feature a constant that blows up as the smallest singular value goes to zero. In contrast, Theorems \ref{thm: Numerical Convergence Stoch DLR Proj - DLRA} and \ref{thm: Numerical Convergence Eva} do not require any restriction on $\Delta t$ for the DLR Projector Splitting schemes and the former predicts a constant independent of the smallest singular value. We aim to illustrate numerically these effects.

To this end, in Figures \ref{fig: SADR_M_4000 singular values DLR EM}, \ref{fig: SADR_M_4000 singular values Stoch Proj}, and  \ref{fig: SADR_M_4000 singular values DLR KNV}, we plot the following quantity related to \eqref{eq: dt sup n cond}
\begin{equation}\label{eq: delta t stab 2}	
	\widehat{\Delta t_n}^{j} := \frac{ \sqrt{\sigma_n^j}}{ \sqrt{C_{\mathrm{lgb}}} \sqrt{(1+\max\limits_{n \in \{1,\dots,N\}}\mathbb{E}[|X_{n}|^2] )}},
\end{equation}
where we recall that $\sigma_{n}^j$ is the $j$-th largest singular value of the Gramian $C_{Y_n}$. Notice that the constraint $\Delta t_n \leq \widehat{\Delta t_n}^k$ is less restrictive than \eqref{eq: dt sup n cond} which uses, instead an a priori upper bound on $\max\limits_{n \in \{1,\dots,N\}} \mathbb{E}[|X_n|^2]$.

\begin{figure}[!h]
	\centering
	\includegraphics[scale=0.16]{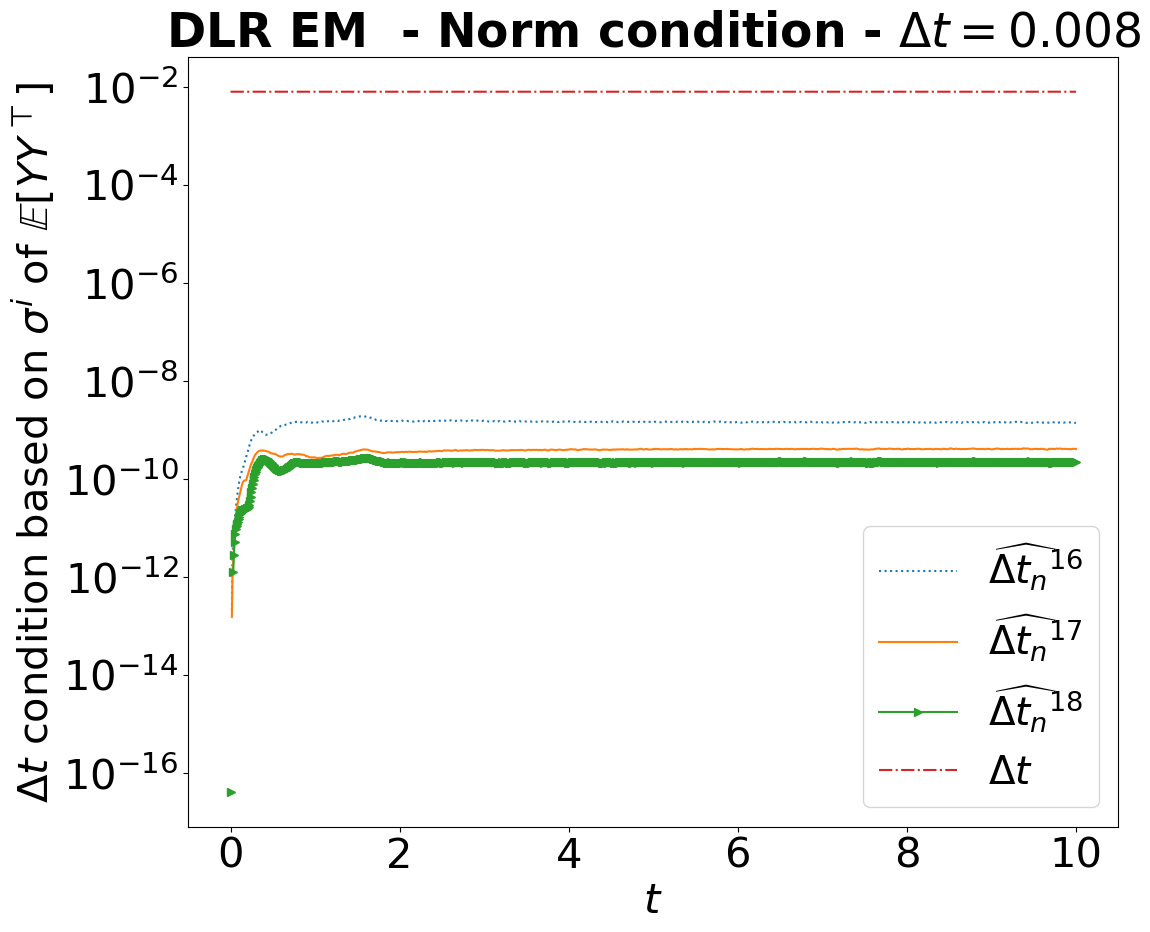}
	\includegraphics[scale=0.16]{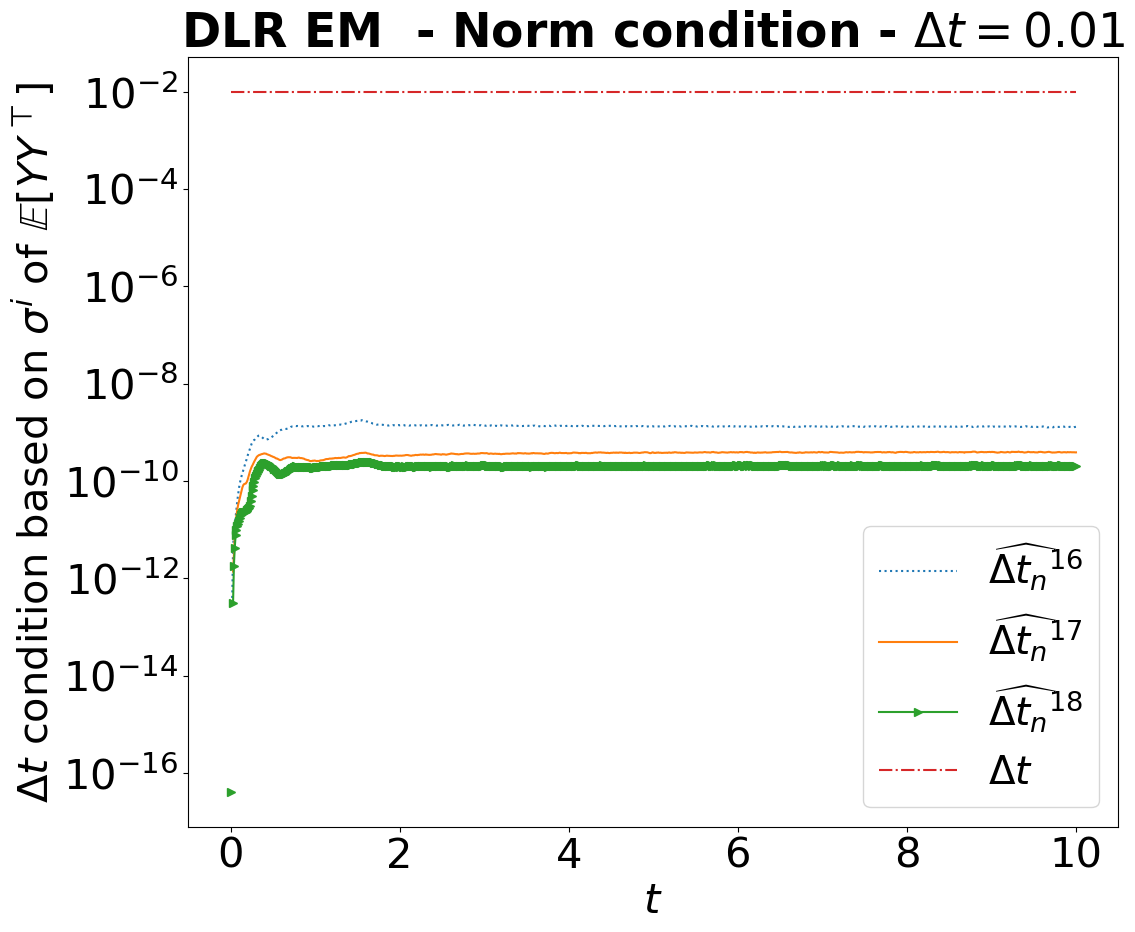}
		\includegraphics[scale=0.16]{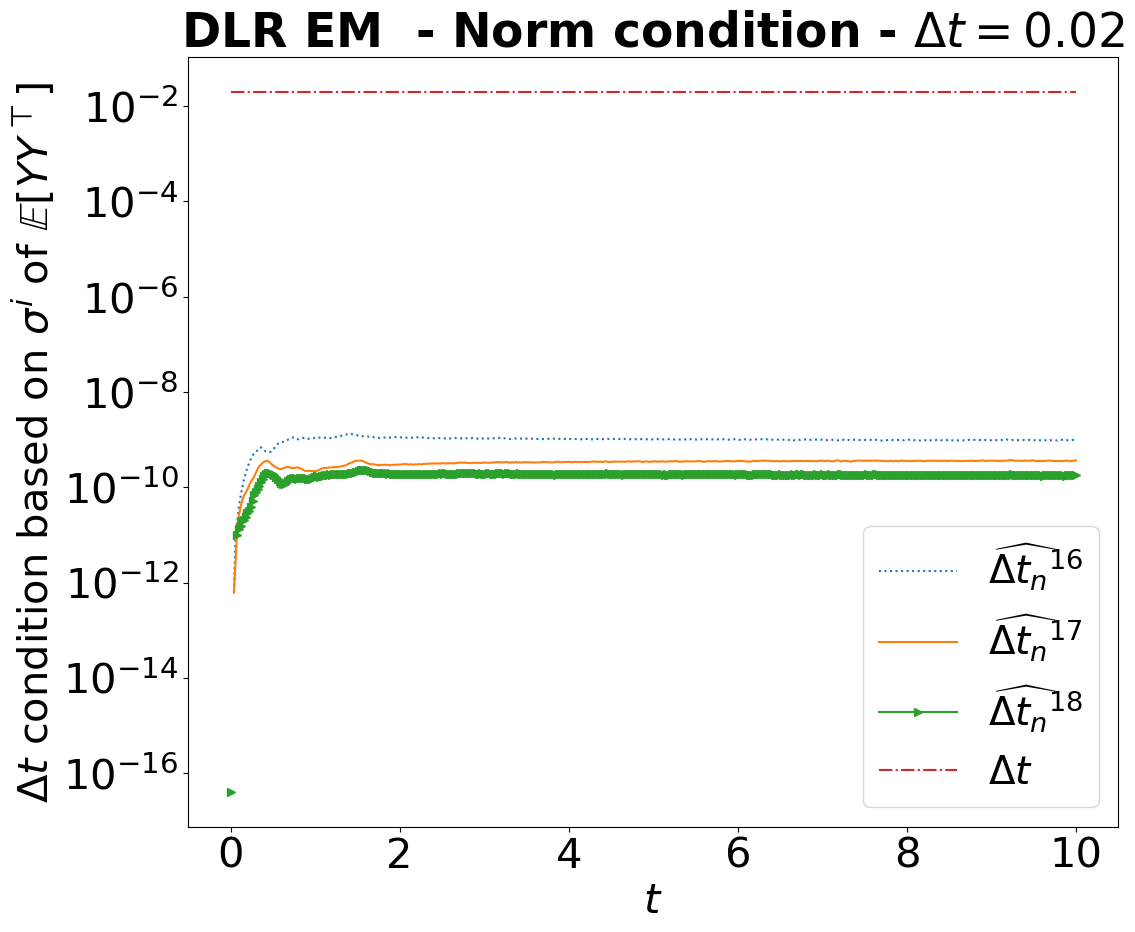}
	\caption{Condition \eqref{eq: delta t stab 2} related to the square root of the singular values $\sigma^i$ of the Gramian $\mathbb{E}[X_nX^{\top}_n]$ for DLR Euler-Maruyama with $\Delta t = 0.008,0.01,0.02$ for equation \eqref{ex: SADR}.}
	\label{fig: SADR_M_4000 singular values DLR EM}
\end{figure}

\begin{figure}[!h]
	\centering
	\includegraphics[scale=0.16]{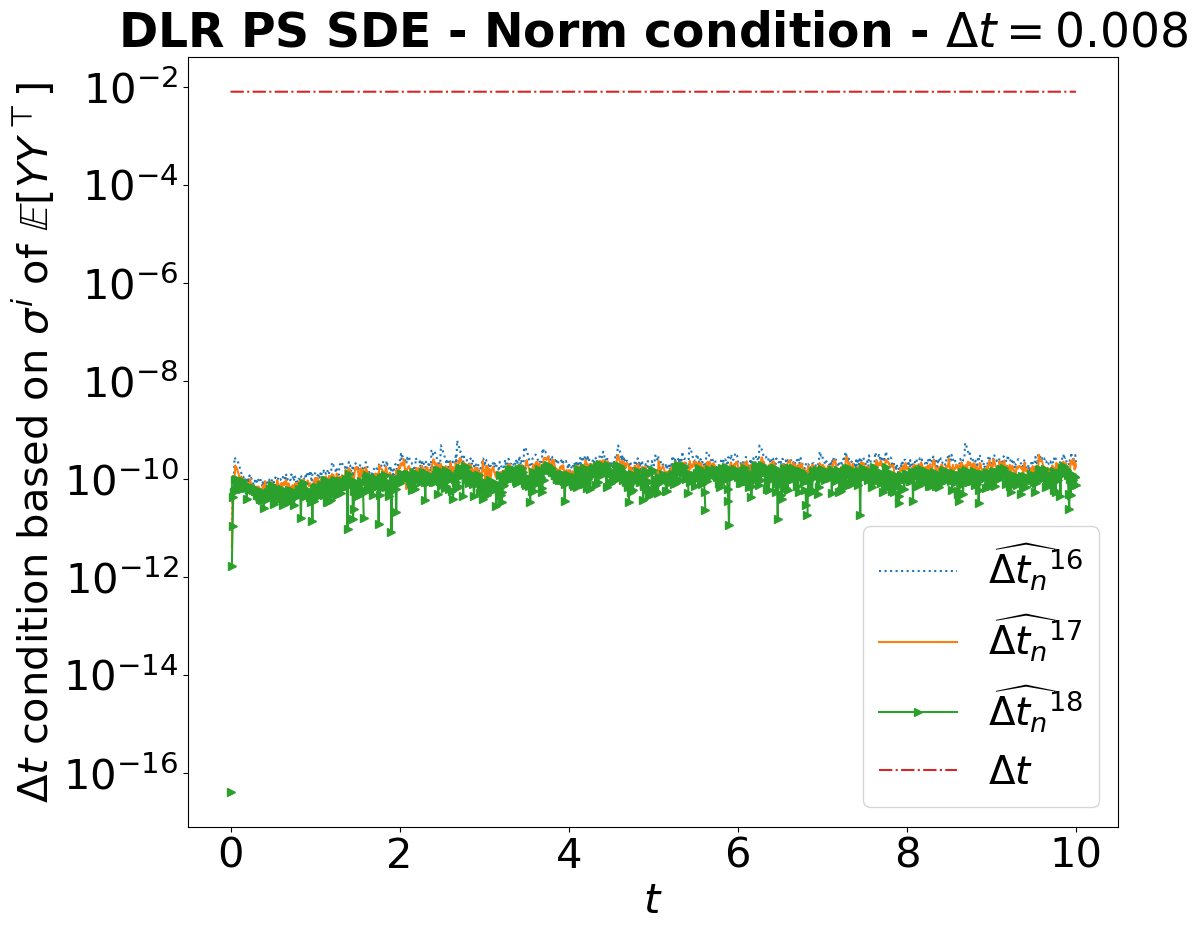}
	\includegraphics[scale=0.16]{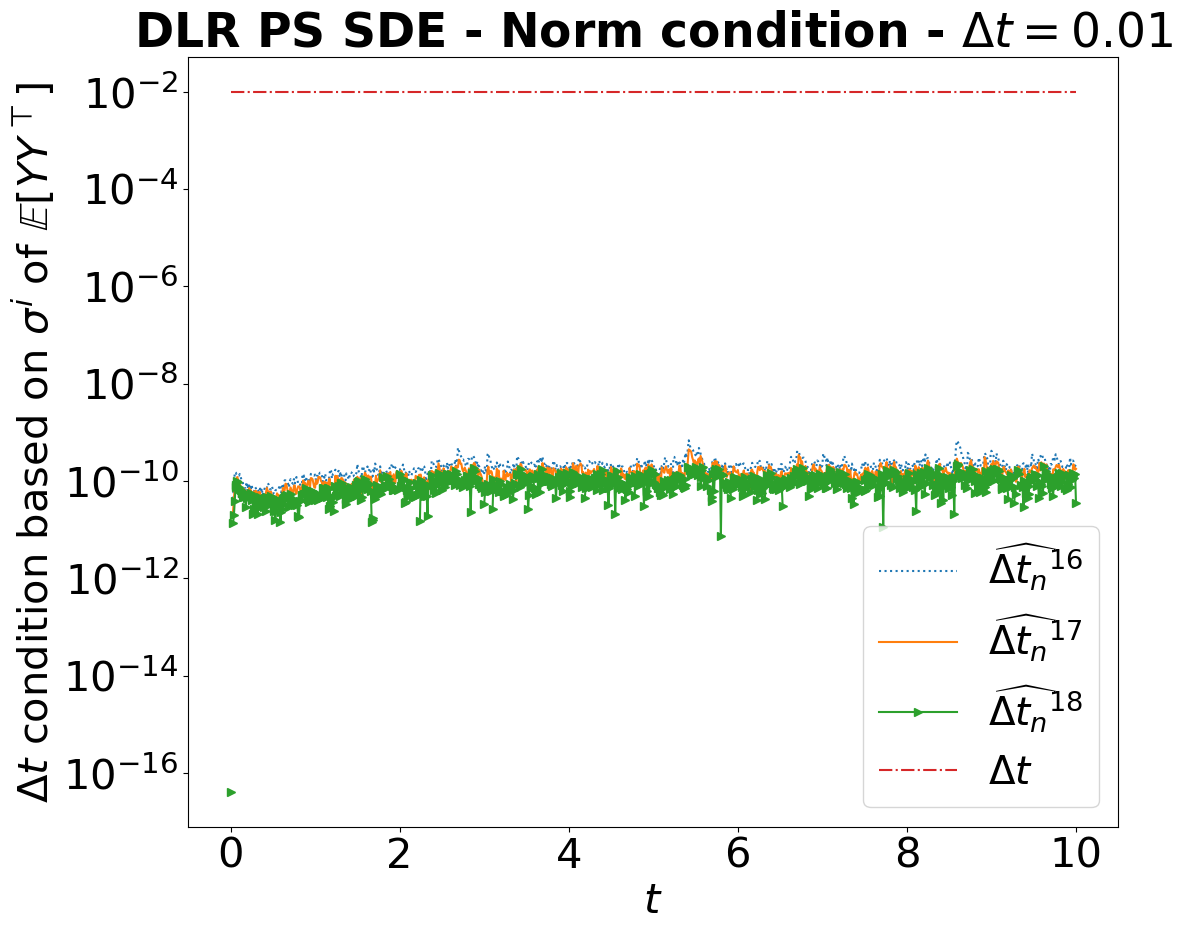}
	\includegraphics[scale=0.16]{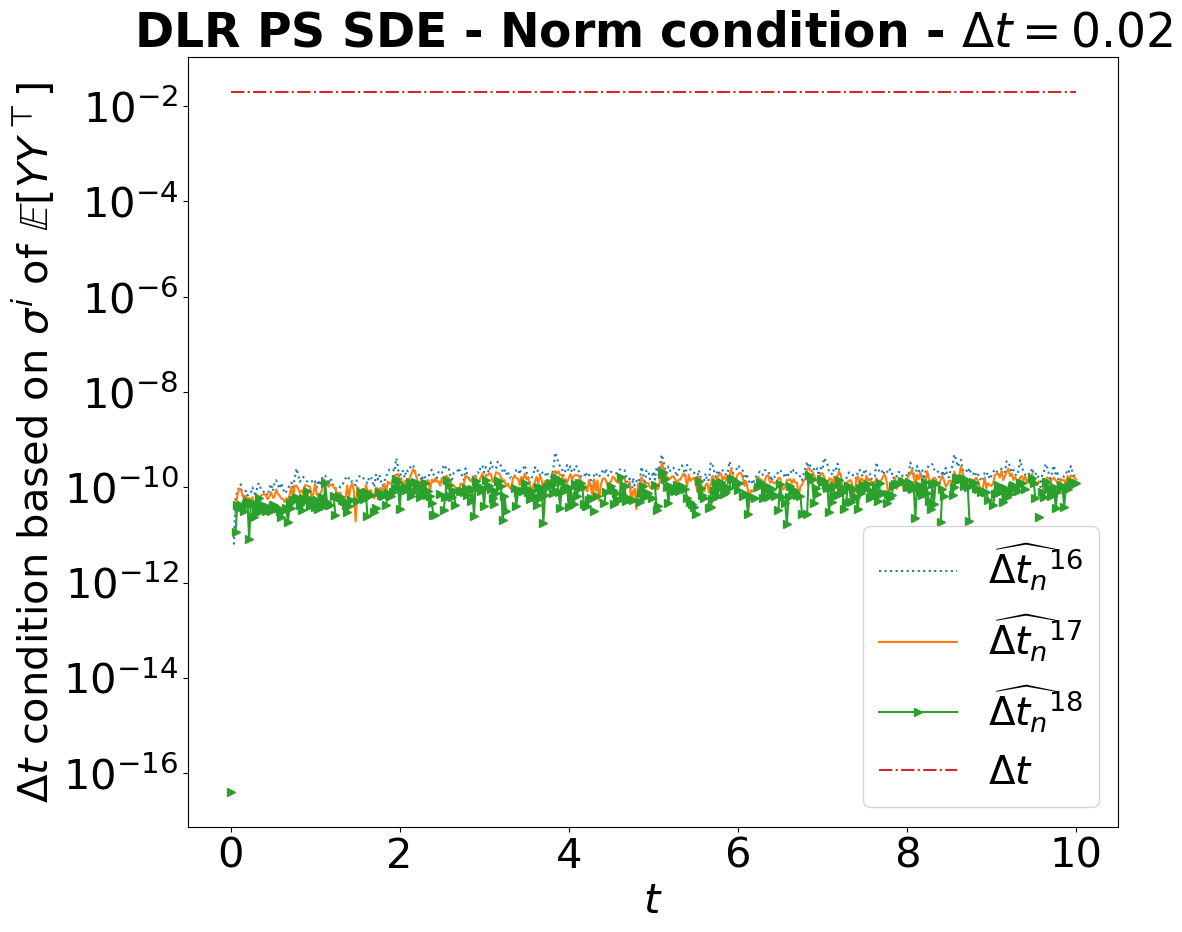}
	\caption{Condition \eqref{eq: delta t stab 2} related to the square root of the singular values $\sigma^i$ of the Gramian $\mathbb{E}[X_nX^{\top}_n]$ for DLR Projector Splitting for SDEs with $\Delta t = 0.008,0.01,0.02$ for equation \eqref{ex: SADR}.}
	\label{fig: SADR_M_4000 singular values Stoch Proj}
\end{figure}

\begin{figure}[!h]
	\centering
	\includegraphics[scale=0.16]{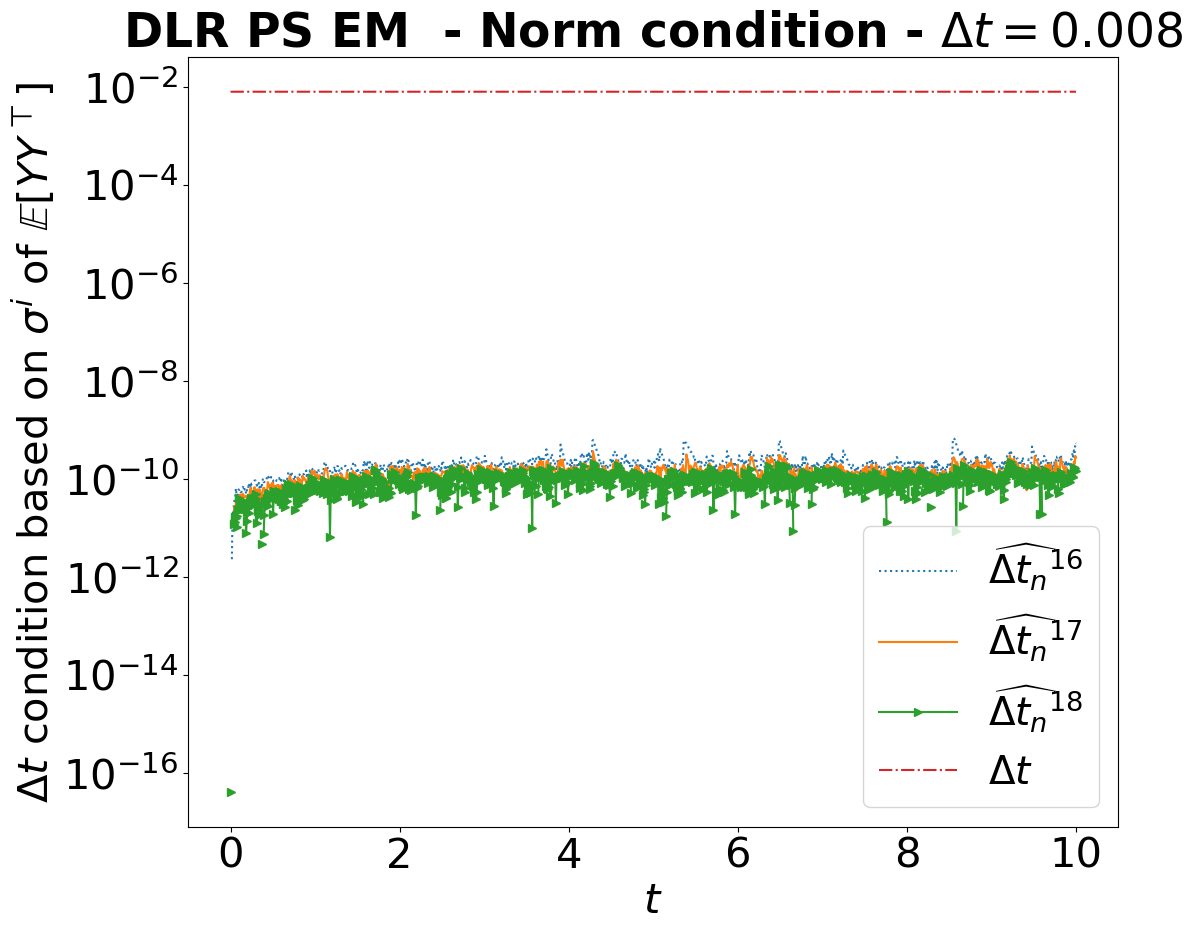}
	\includegraphics[scale=0.16]{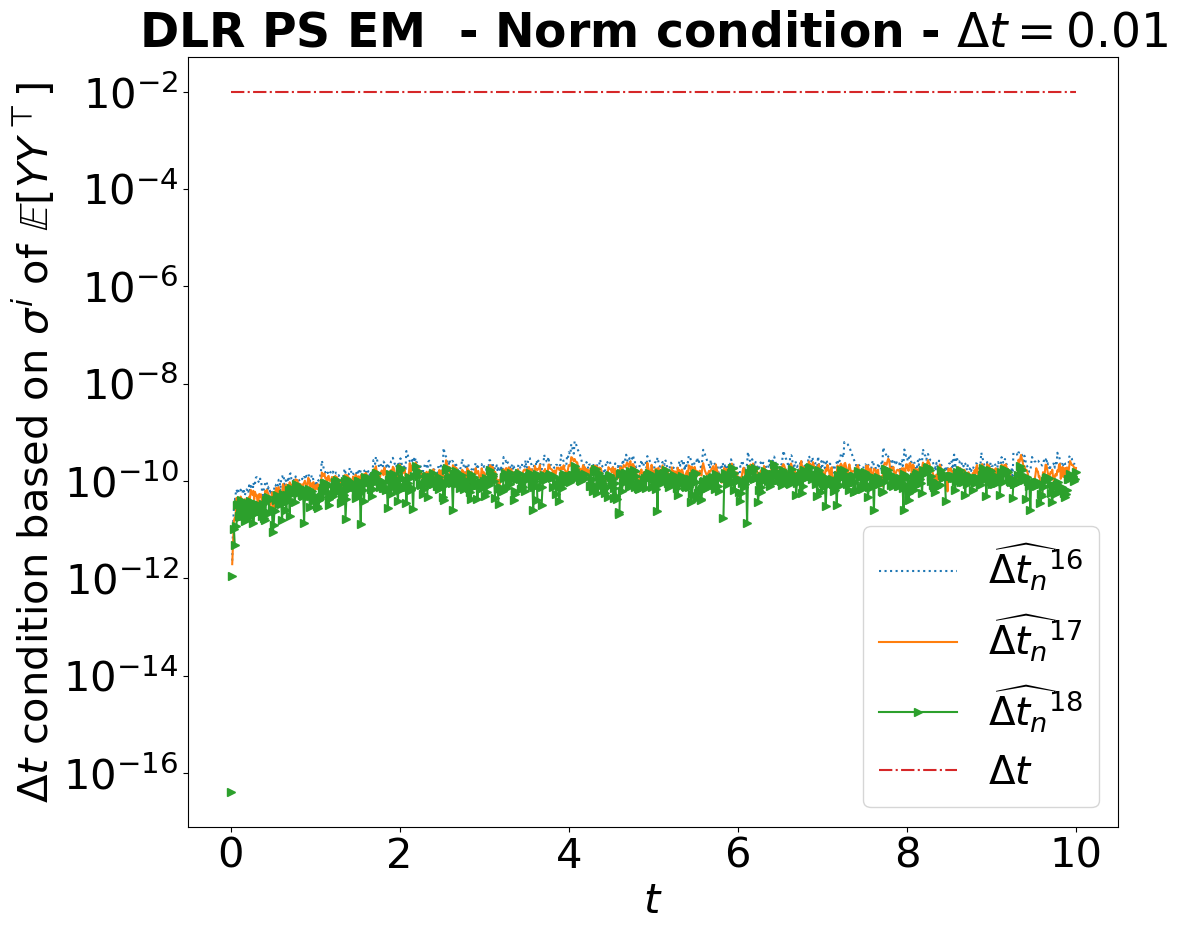}
	\includegraphics[scale=0.16]{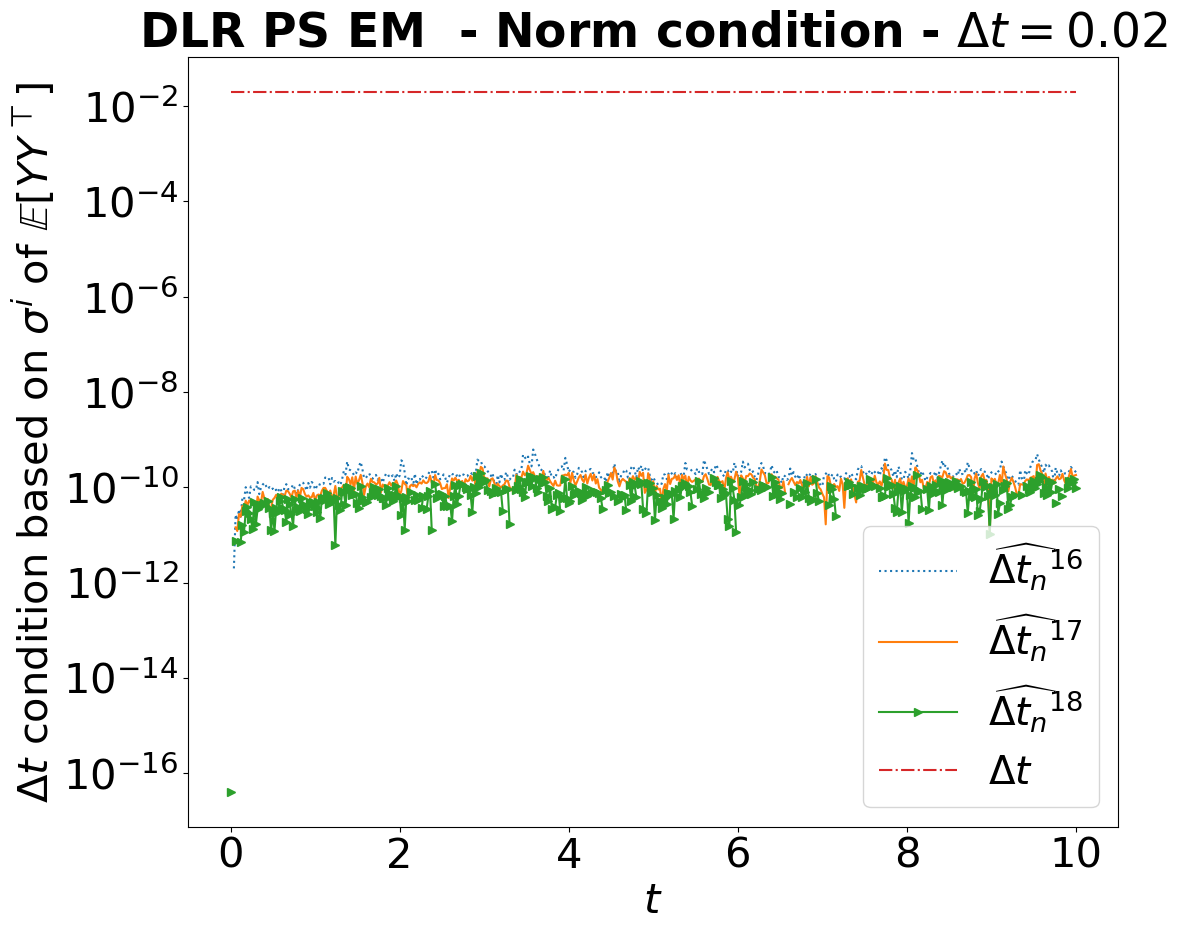}
	\caption{Condition \eqref{eq: delta t stab 2} related to the square root of the singular values $\sigma^i$ of the Gramian $\mathbb{E}[X_nX^{\top}_n]$ for DLR Projector Splitting for EM with $\Delta t = 0.008,0.01,0.02$ for equation \eqref{ex: SADR}.}
	\label{fig: SADR_M_4000 singular values DLR KNV}
\end{figure}

Figures \ref{fig: SADR_M_4000 singular values DLR EM}, \ref{fig: SADR_M_4000 singular values Stoch Proj} and \ref{fig: SADR_M_4000 singular values DLR KNV} show how the quantity in \eqref{eq: delta t stab 2}, related to the square root of the $i$-th largest singular value $\sigma^i$ of the Gramian $\mathbb{E}[Y_nY_n^{\top}]$, evolves over time for the three algorithms, with $\Delta t$ held fixed in each case, but varying across the plots. Figure \ref{fig: SADR_M_4000 singular values DLR EM} corresponds to the DLR EM method, Figure \ref{fig: SADR_M_4000 singular values Stoch Proj} to the DLR Projector Splitting for SDEs, and Figure \ref{fig: SADR_M_4000 singular values DLR KNV} to the DLR Projector Splitting for EM.
In all cases, we see that the time-step never satisfies \eqref{eq: delta t stab 2} in $[0,T]$, which we interpret as a violation of the condition \eqref{eq: dt sup n cond}. Therefore, we expect that the DLR EM method loses accuracy, unlike the two other projector-splitting methods, which are theoretically unaffected by condition \eqref{eq: dt sup n cond}. Furthermore, we see in Figure  \ref{fig:SADR L2 norm} that the norm of the solution of the DLR EM does not resemble the one of the $X^{\mathrm{true}}$, showing again that the condition  \eqref{eq: dt sup n cond} is necessary in Lemma \ref{lem: sup n yn}.

\begin{figure}[!h]
	\centering
		\includegraphics[scale=0.16]{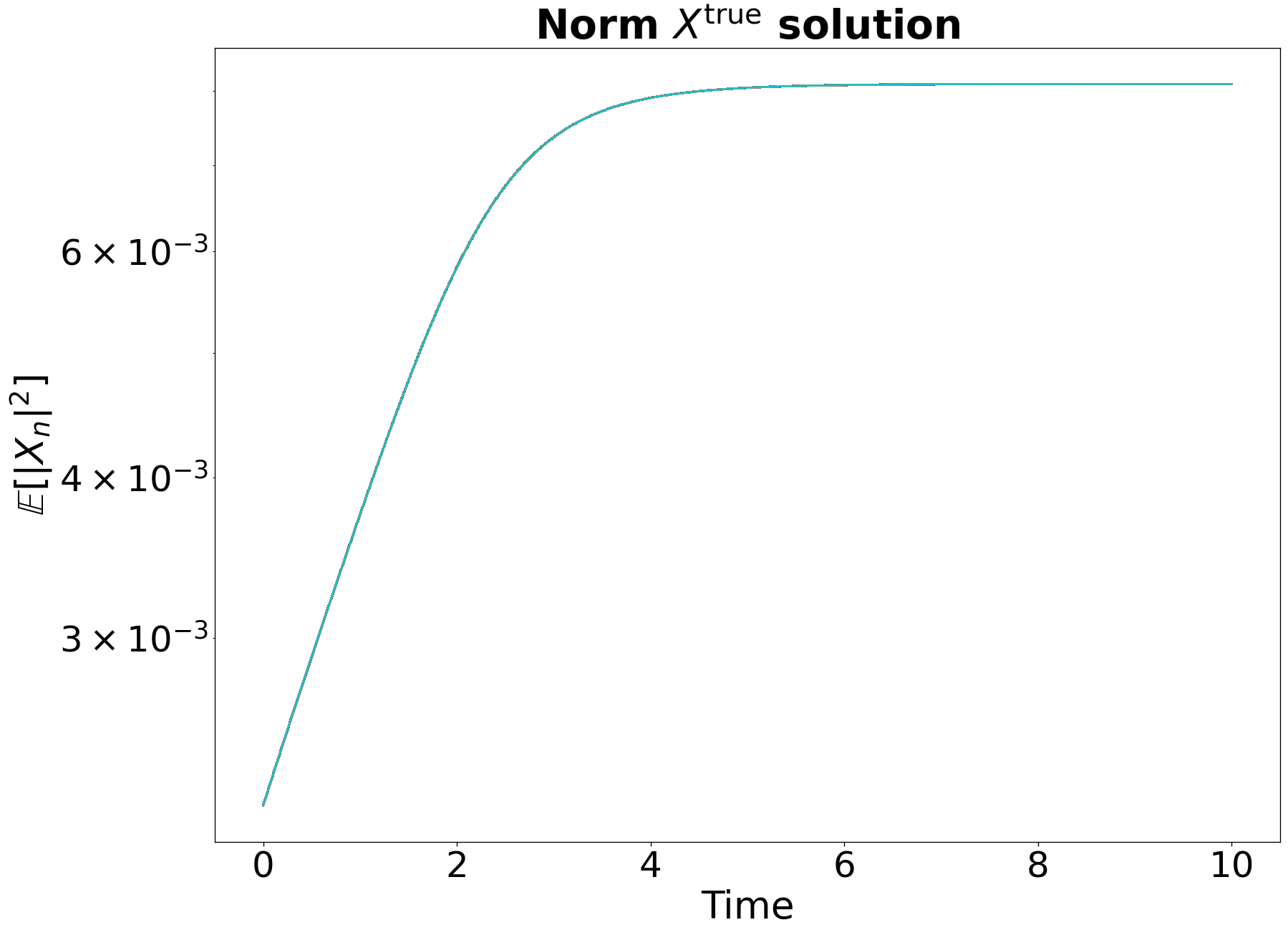}
	\includegraphics[scale=0.16]{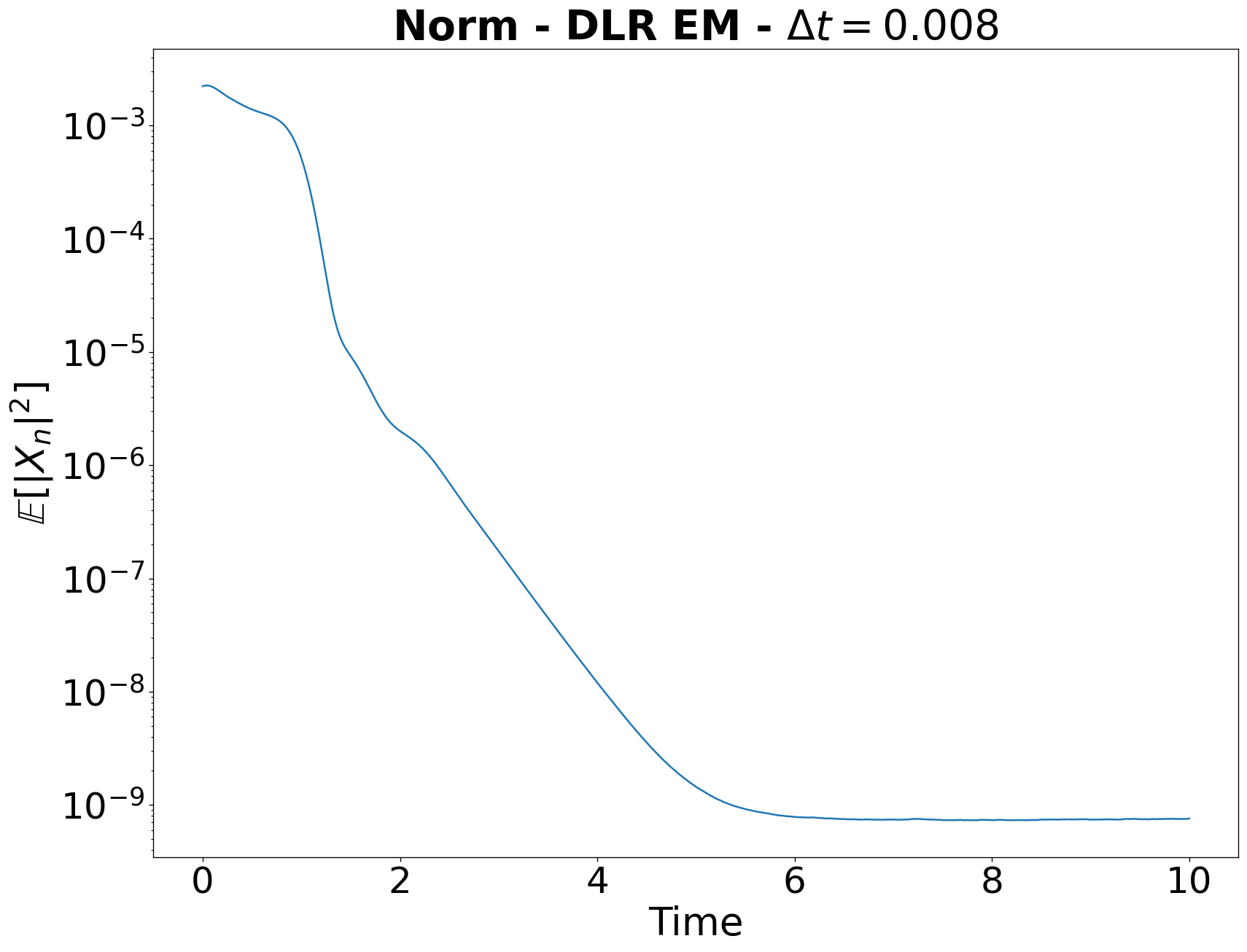}
		\includegraphics[scale=0.16]{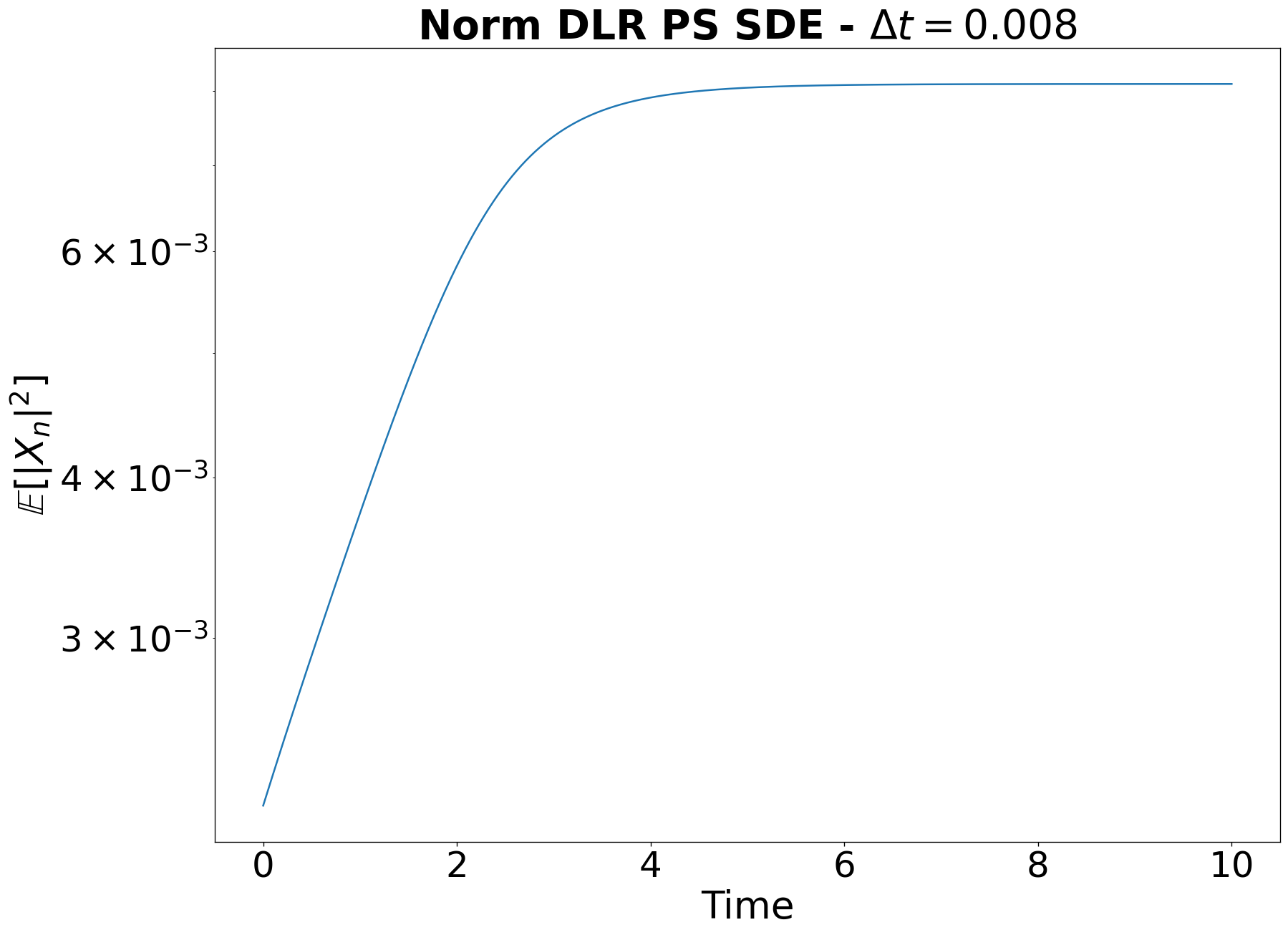}
		\includegraphics[scale=0.16]{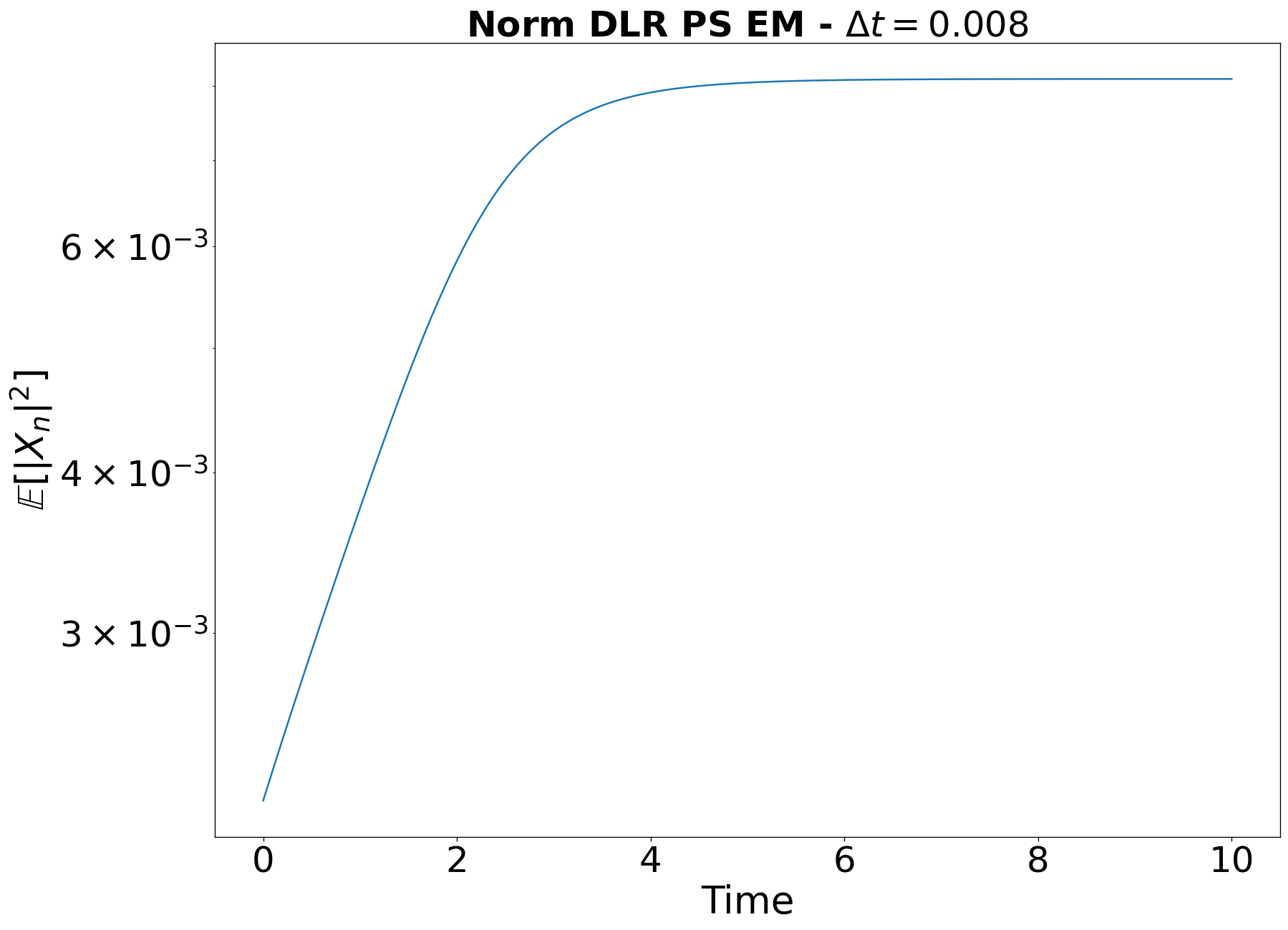}
	\caption{$L^{2}$ norm for equation \eqref{ex: SADR} for the  $X^{\mathrm{true}}$ and DLRA algorithms, $k=18$, $M=4000$ paths, $\Delta t=0.008$.}
	\label{fig:SADR L2 norm}
\end{figure}

In Figure \ref{fig:SADR example errors}, we show the relative errors between the DLR algorithms and the continuous DLRA or the true solution for different values of $\Delta t$. 
Projection-based methods, Algorithms \ref{alg: Stoch DLR Proj algorithm} and \ref{alg: Eva Proj Splitt SDE algorithm}, appear to be more accurate than the DLR EM scheme. This trend supports Theorems~\ref{thm: convergence of DLR Euler-Maruyama}, \ref{thm: Numerical Convergence Stoch DLR Proj - DLRA}, and \ref{thm: Numerical Convergence Eva}.
Moreover, we notice that Algorithm \ref{alg: Eva Proj Splitt SDE algorithm} has smaller error with respect to the true solution than the DLR Projector Splitting for SDEs. This trend may motivate further investigation of the theoretical convergence of Algorithm \ref{alg: Eva Proj Splitt SDE algorithm}, removing the dependence of the constant on the smallest singular value. Furthermore, Figure \ref{fig:SADR example errors} shows that the convergence rate of the projector-splitting methods stagnates for very small time steps. In this regime, the time discretization error is no longer the dominant source of error. Instead, the error is governed by the approximability assumption $\varepsilon$ defined in \eqref{eps bound projection}, which is independent of the time step.

\begin{figure}[!h]
	\centering
	\includegraphics[scale=0.26]{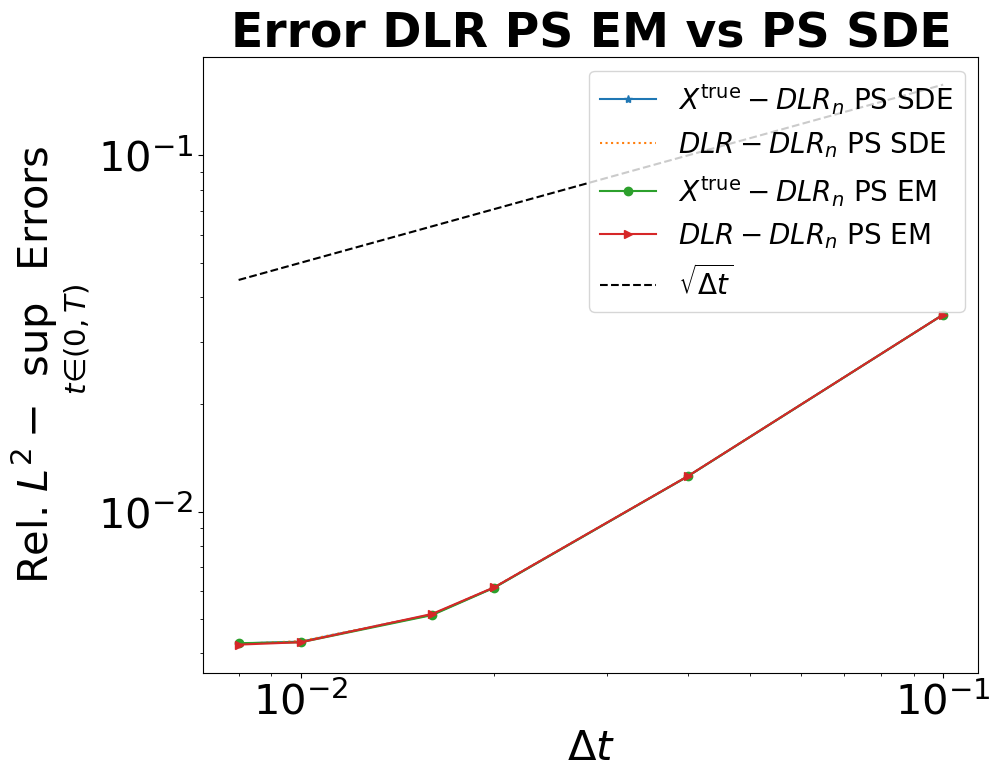}
	\includegraphics[scale=0.26]{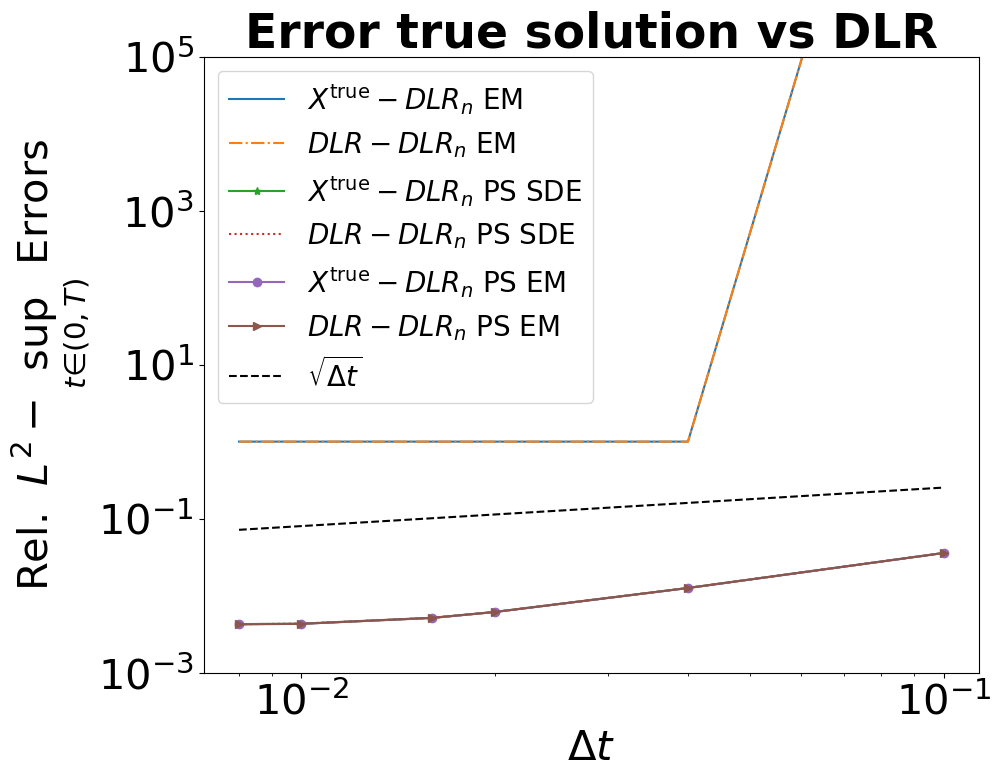}
	\caption{Rel. errors in $L^{2}-\sup\limits_{t \in [0,T]}$ norm for equation \eqref{ex: SADR}, $k=18$, $M=4000$ paths, (\underline{Blue} $(-)$) error between the true solution and the DLRA algorithms and between the continuous DLR and the DLRA algorithms}
	\label{fig:SADR example errors}
\end{figure}
\subsection{Multiplicative noise: Stochastic Laplacian in 1D}
In this numerical example we consider a stochastic Laplacian problem in 1D with Dirichlet boundary conditions under multiplicative colored noise. Like in the previous section, we want to test the numerical convergence of the three algorithms and their dependence on the smallest singular values of the Gramian of the stochastic basis. We will see that the numerical results validate once again the theoretical statements (Theorems~\ref{thm: convergence of DLR Euler-Maruyama}, \ref{thm: Numerical Convergence Stoch DLR Proj - True}, \ref{thm: Numerical Convergence Eva}) in the case of multiplicative noise. The spatial domain is $[0,1]$, while the temporal domain is $[0,T]$, with $T=5$. More precisely, we aim to find the numerical solution of the following system:
	\begin{equation}\label{ex: laplacian}
	\begin{cases}
		\partial_t u(x,t,\omega)&= 0.001 \cdot \Delta u(x,t,\omega) + F(x,t) + G(u,x)\dot{W}(\omega), \quad (x,t,\omega) \in [0,1] \times [0,5] \times \Omega, \\
		u_{0}(x,\omega)&=   \sum_{\ell = 1}^{13} \sin(\pi \ell x)\cdot \mathcal{N}_\ell(0,1)(\omega) +  8 \cdot 10^{-7} \sin(8\pi  x)\cdot \mathcal{N}_{14}(0,1)(\omega),  \quad  (x,\omega) \in [0,1] \times \Omega\\
		u(0,t, \omega)&= u(1,t,\omega) = 0, \quad (t,\omega) \in [0,5] \times \Omega.
	\end{cases}
\end{equation}
In \eqref{ex: laplacian},  $\{\mathcal{N}_\ell(0,1)\}_{\ell=1,\dots,14},$ are i.i.d.\ normal random variables, 
while $G\dot{W}$ represents a colored noise in space, which is rough in time, defined as
\begin{equation*}
	G(u,x)\dot{W}(\omega) = \sum\limits_{\ell=1}^{N} \gamma_{\ell} \langle u(x,t,\omega), \cos(2\pi \ell x)+ \sin(2\pi \ell x)\rangle_{L^2(0,1)} \dot{W}^{\ell}_t(\omega),
\end{equation*}
where $W^{\ell}_t$, for $\ell\in \{1,\dots, N\}$, are independent Brownian motions, also independent of the initial condition $u_0$, 
with $\gamma_\ell = \frac{1}{2 \pi \ell } \exp^{-2\pi \ell}$ and $N=26$. The deterministic force term $F(t)$ is a ``moving window" defined as
\begin{equation}
	F(x,t) = \begin{cases}
		3, \quad& x \in \ ]vt,l+vt[, \quad \text{if} \ t \leq \frac{1-l}{v},\\
		0, \quad& x \notin \  ]vt,l+vt[, \quad \text{if} \  t \leq \frac{1-l}{v},	\end{cases}
\end{equation}
with $l= 0.12$, $v=0.4$, which is reflected symmetrically at the boundary and it is repeated periodically in time with period $\bar{T} = 2(\frac{1-l}{v})=4.4$.
To approximate $u$, we consider a DLR solution $X$ of rank $k=14$, simulating $M=2200$ paths. This choice of the rank has been made considering the rank of the initial condition. We highlight that in $u_0$ the term $8 \cdot 10^{-7} \sin(8 \pi  x)\cdot \mathcal{N}_{14}(0,1)(\omega)$ is nearly negligible compared to the other 13 terms and it may cause problems of overapproximation. Moreover, notice that \eqref{ex: laplacian} is well posed in the sense of mild solutions \cite[Theorem 6.7]{da2014stochastic}.

We discretize the laplacian operator $\Delta$ using centered finite differences for each fixed path $\omega$ on a uniform mesh $\{ (ih); \ i \in \{0,\dots, N\} \ : \ Nh = 1 \}$ with mesh size $h=0.04$
\begin{equation*}
	\displaystyle
	\begin{aligned}
		  \Delta u(ih, \omega) \approx& \frac{ u((i-1)h,\omega)-2\cdot u(ih,\omega)+ u((i+1)h,\omega)}{h^{2}}. \\
	\end{aligned}
\end{equation*}
This leads to a SDE of effective dimension $d=24$, to which we apply the different DLRA algorithms. We initiate all the DLRA procedures as described in Section \ref{sec: basic sde}. The time discretization for the reference true solution $u$ of \eqref{ex: laplacian} and the continuous DLR are made with forward Euler-Maruyama and DLR Projector Splitting for SDEs, respectively, with mesh size $\Delta t=2 \cdot 10^{-3}$.  Notice that, even though the laplacian operator is linear and deterministic, unfortunately the presence of the forcing term $F(x,t)$ does not allow us to simplify the expression for the deterministic basis to \eqref{eq: linear drift}, as the stochastic basis is not assumed as a centered random variable. Therefore, we expect that the accuracy of the DLR EM will be affected by the smallest singular value of its stochastic basis.

Also in this example, we show in Figures \ref{fig: laplacian singular values DLR EM}, \ref{fig: laplacian singular values Stoch Proj}, and \ref{fig: laplacian singular values DLR KNV}, the quantity $\widehat{\Delta t_n}^{i}$ \eqref{eq: delta t stab 2} related to the $i$-th largest singular values $\sigma^i$ of the Gramian $\mathbb{E}[Y_nY_n^{\top}]$ for all the three algorithms over time. We consider these plots to understand if the condition on $\Delta t$ guaranteeing boundedness of the numerical solution is satisfied and if the accuracy remains unaffected by the smallest singular values. 

\begin{figure}[!h]
	\centering
	\includegraphics[scale=0.163]{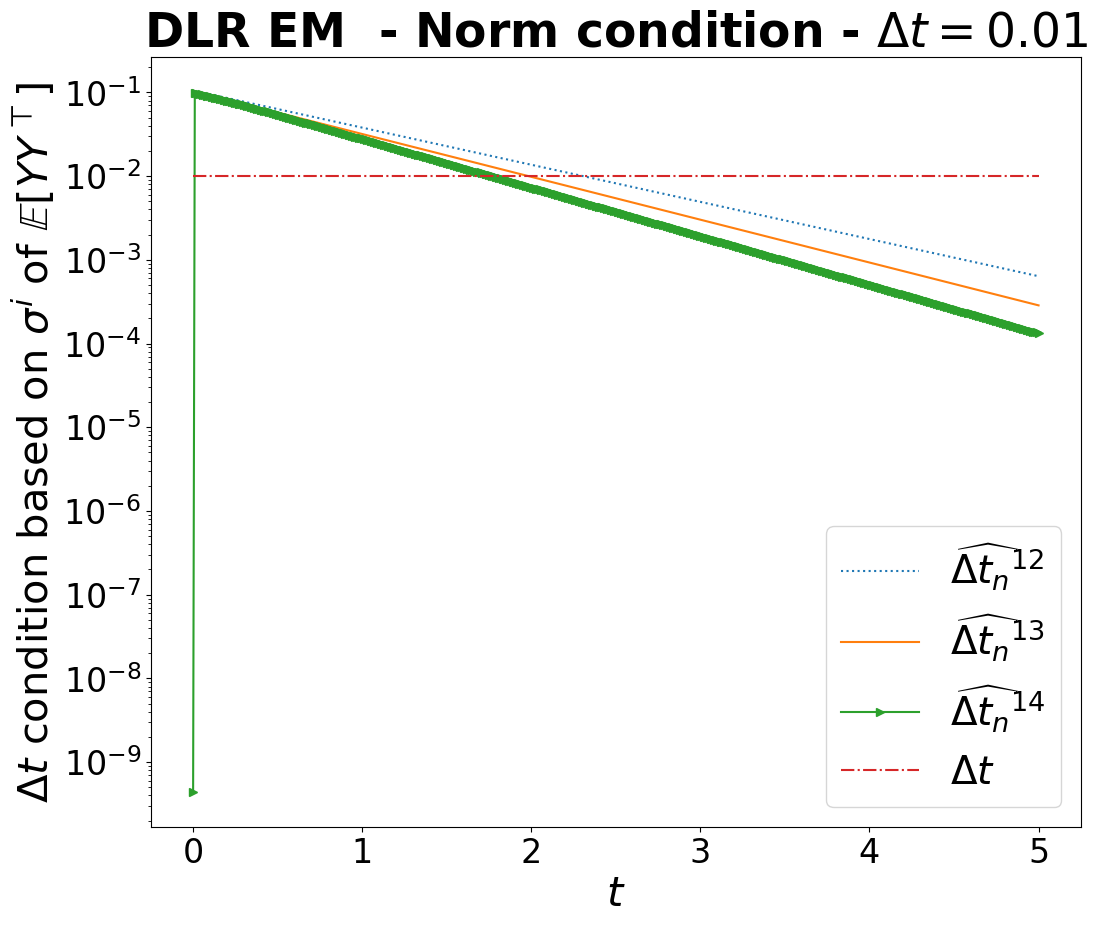}
	\includegraphics[scale=0.163]{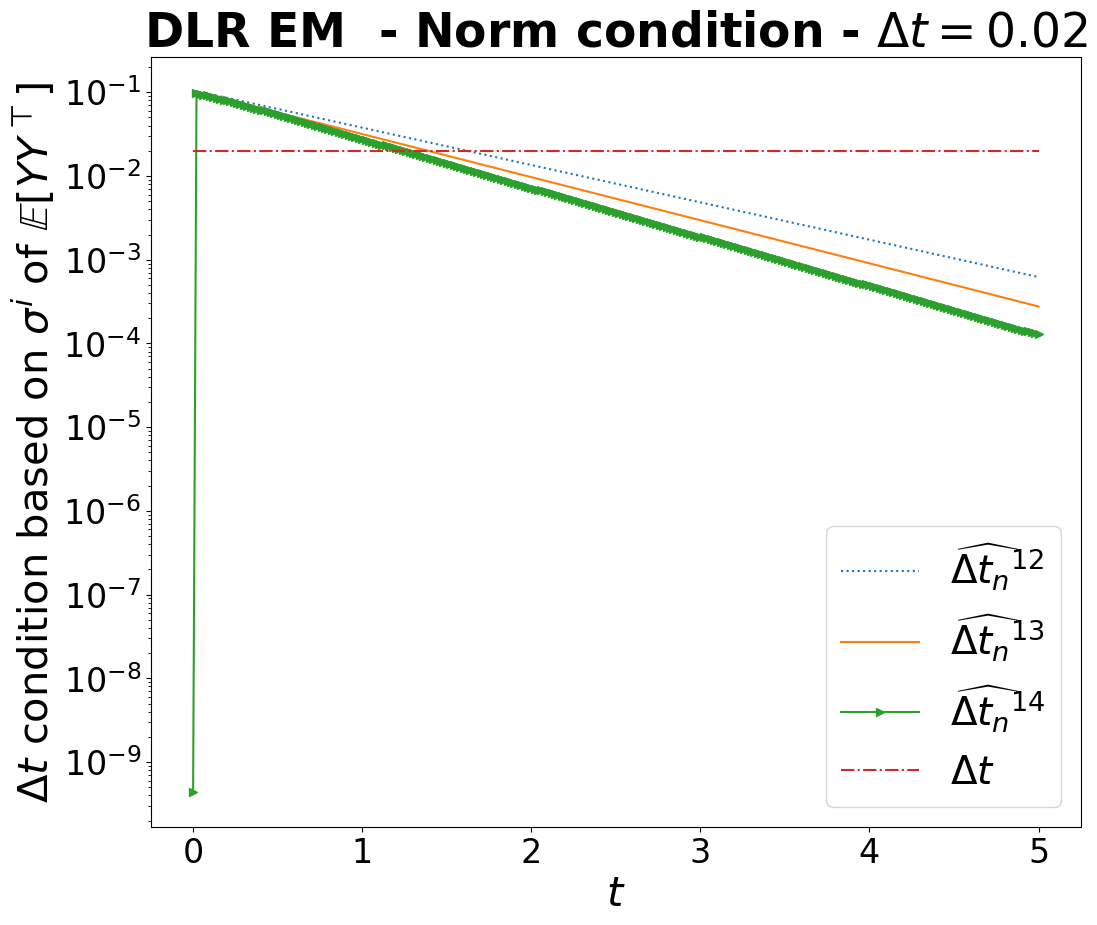}\\
	\includegraphics[scale=0.165]{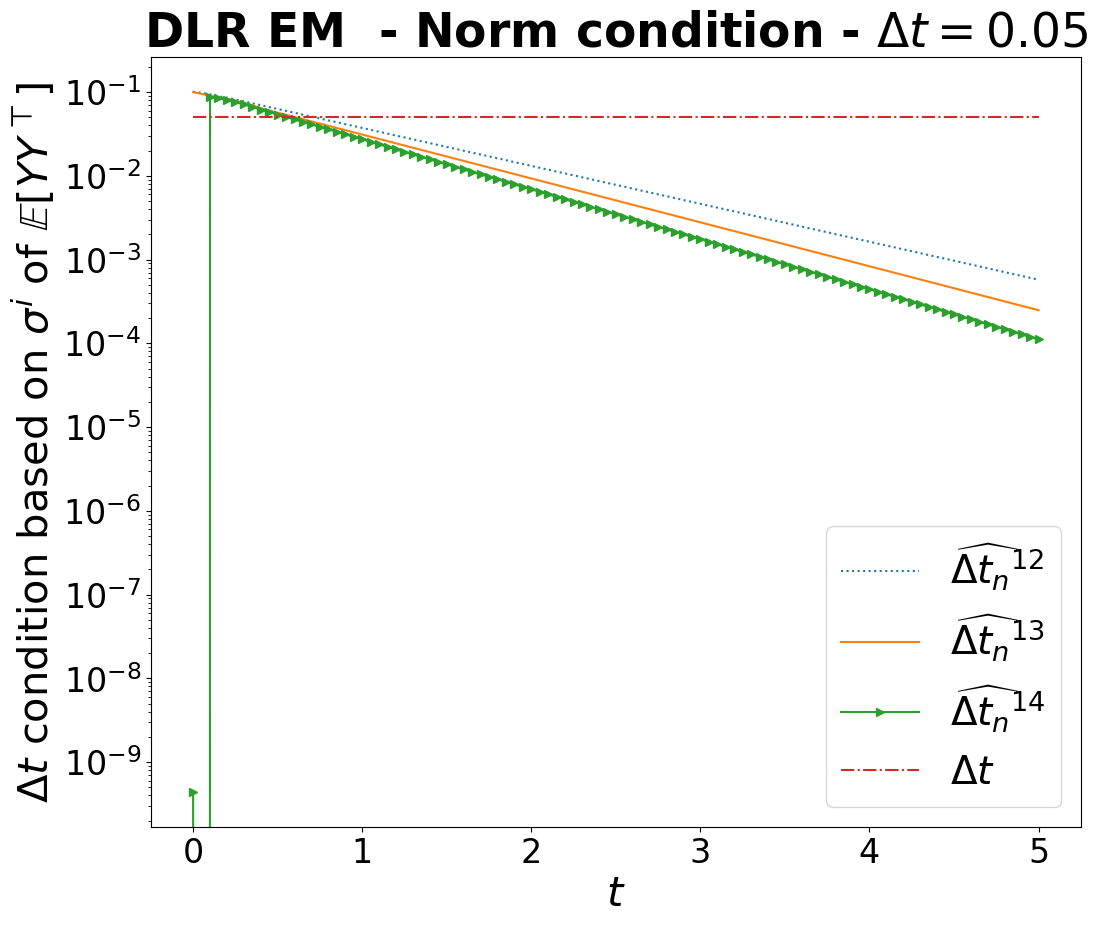}
	\includegraphics[scale=0.165]{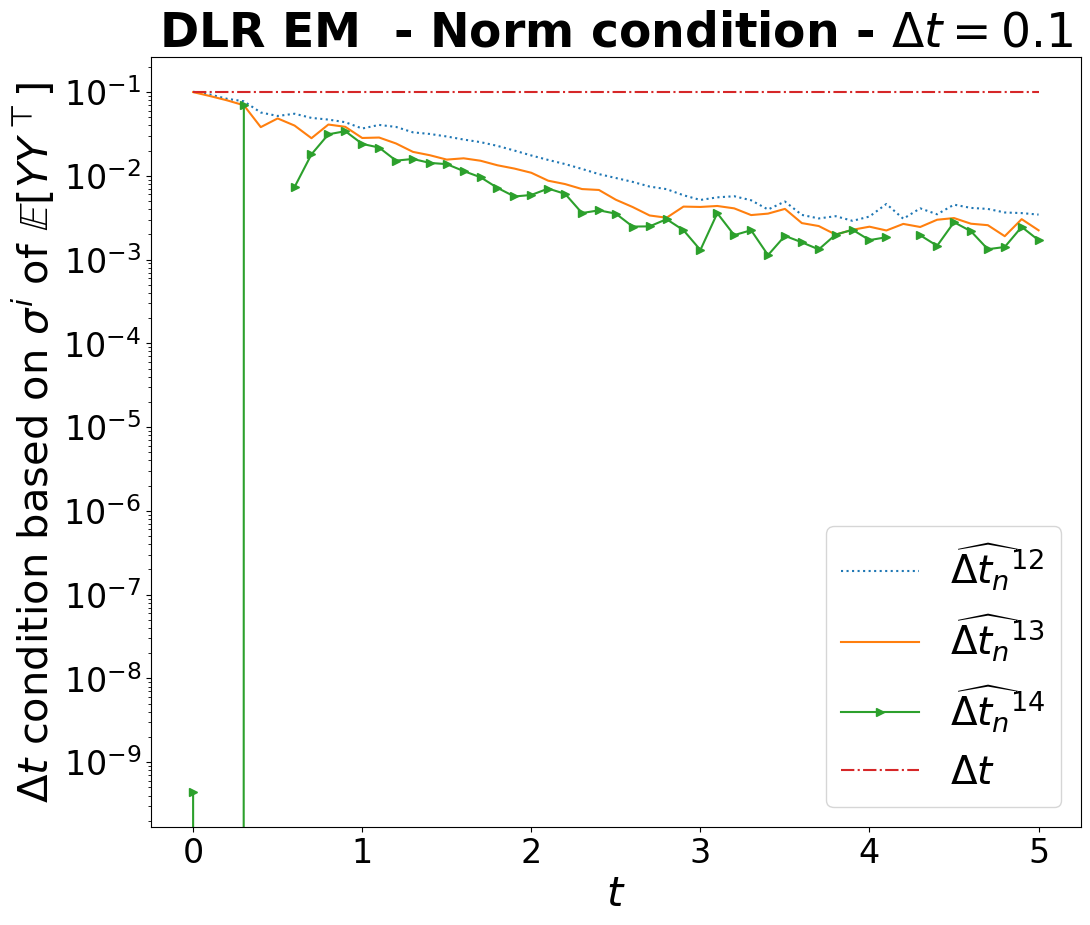}
	\caption{Condition \eqref{eq: delta t stab 2} related to the square root of the singular values $\sigma^i$ of the Gramian $\mathbb{E}[X_nX^{\top}_n]$ for DLR Euler-Maruyama with $\Delta t = 0.01,0.02,0.05,0.1$ for Problem \eqref{ex: laplacian}.}
	\label{fig: laplacian singular values DLR EM}
\end{figure}

\begin{figure}[!h]
	\centering
	\includegraphics[scale=0.165]{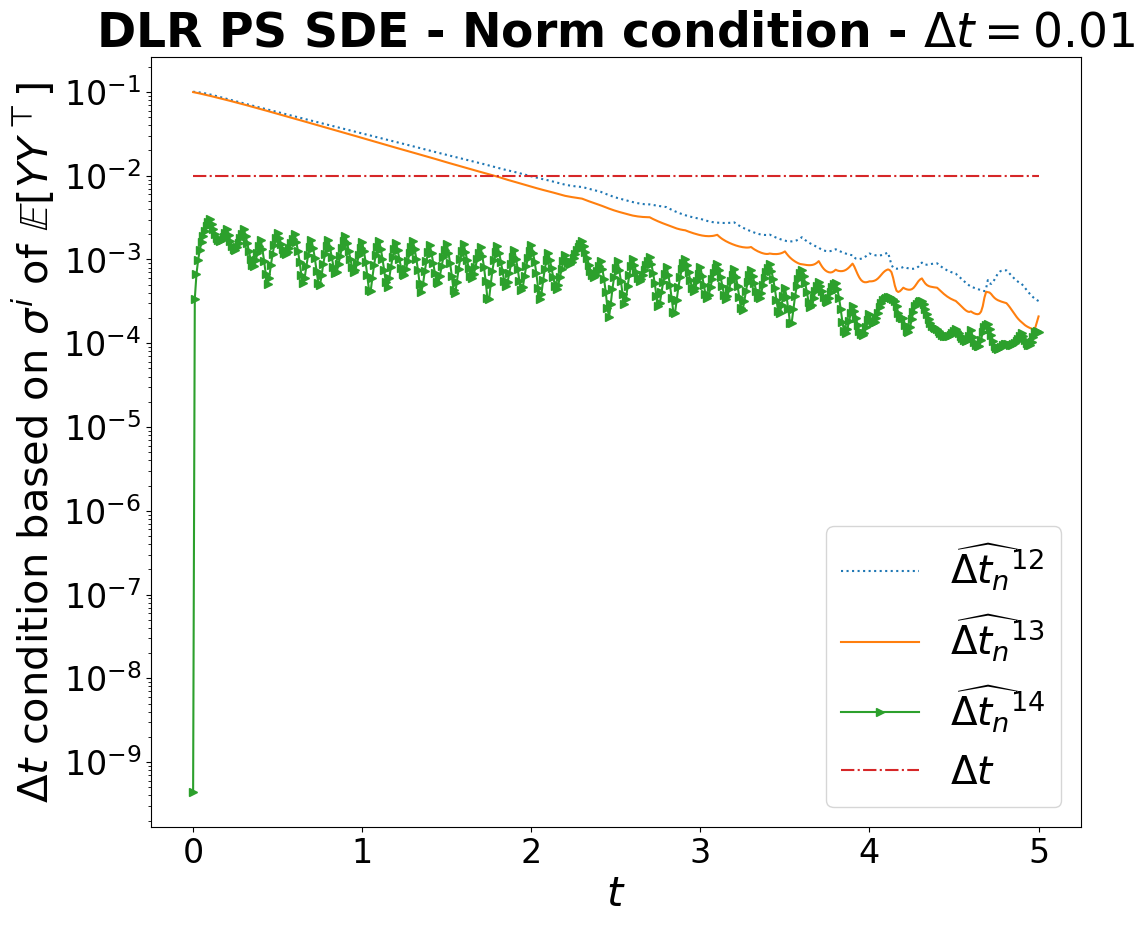}
	\includegraphics[scale=0.165]{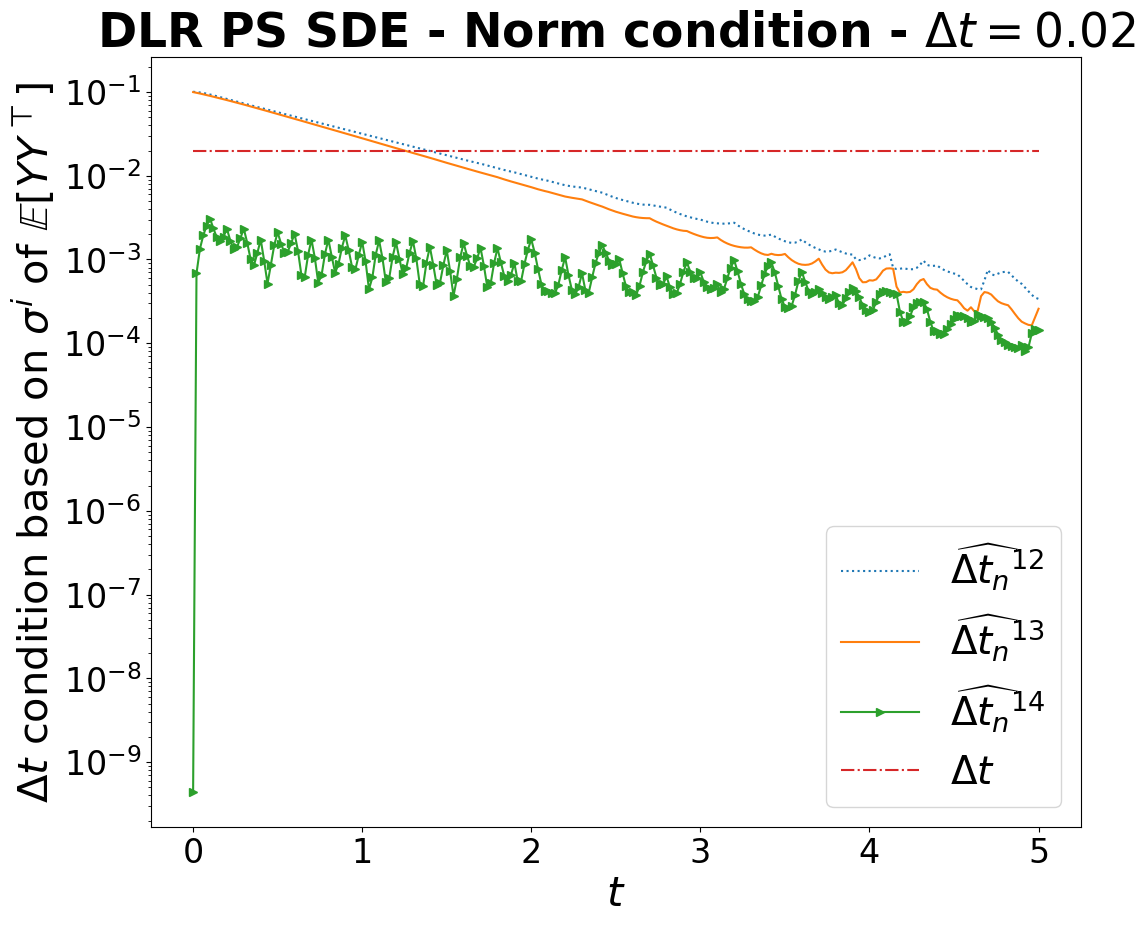}\\
	\includegraphics[scale=0.16]{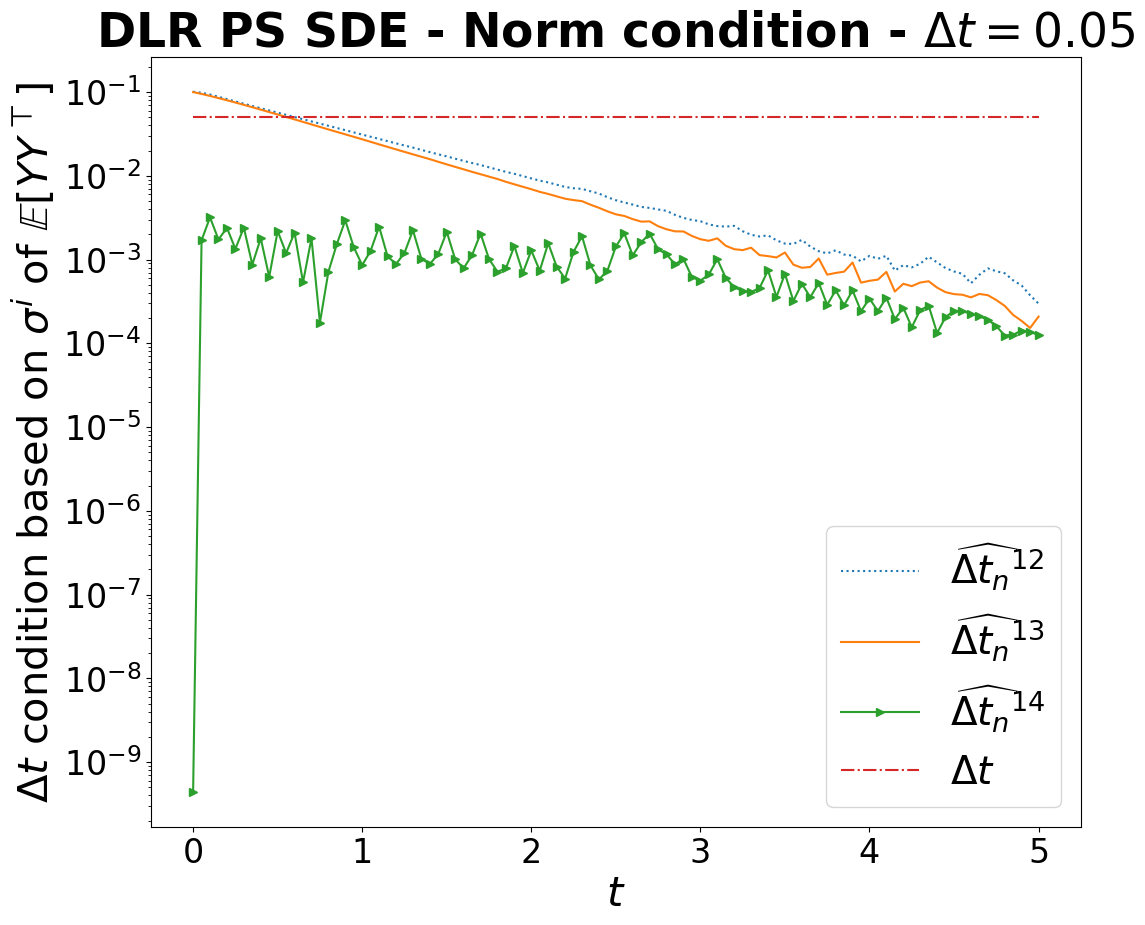}
	\includegraphics[scale=0.16]{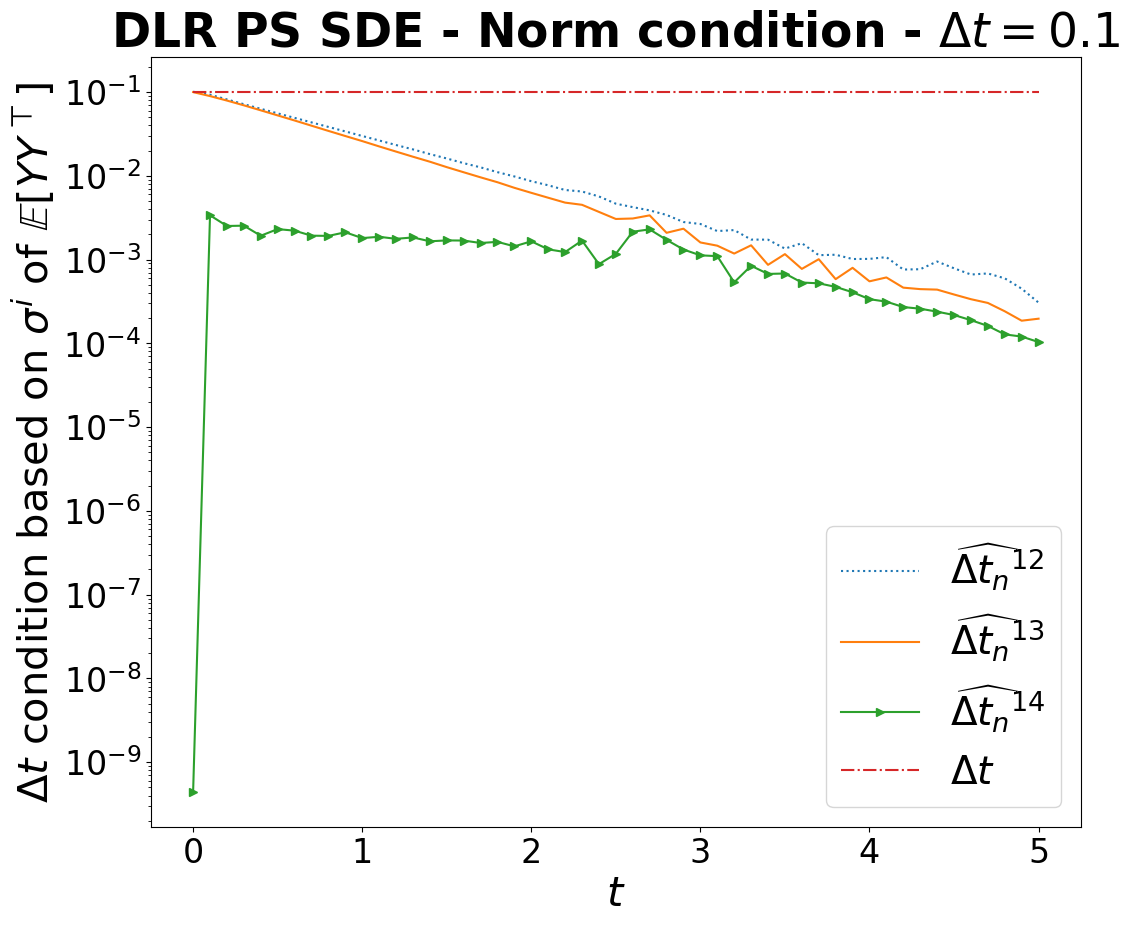}
	\caption{Condition \eqref{eq: delta t stab 2} related to the square root of the singular values $\sigma^i$ of the Gramian $\mathbb{E}[X_nX^{\top}_n]$ for DLR Projector Splitting for SDEs with $\Delta t = 0.01,0.02,0.05,0.1$ for the Problem \eqref{ex: laplacian}.}
	\label{fig: laplacian singular values Stoch Proj}
\end{figure}

\begin{figure}[!h]
	\centering
	\includegraphics[scale=0.165]{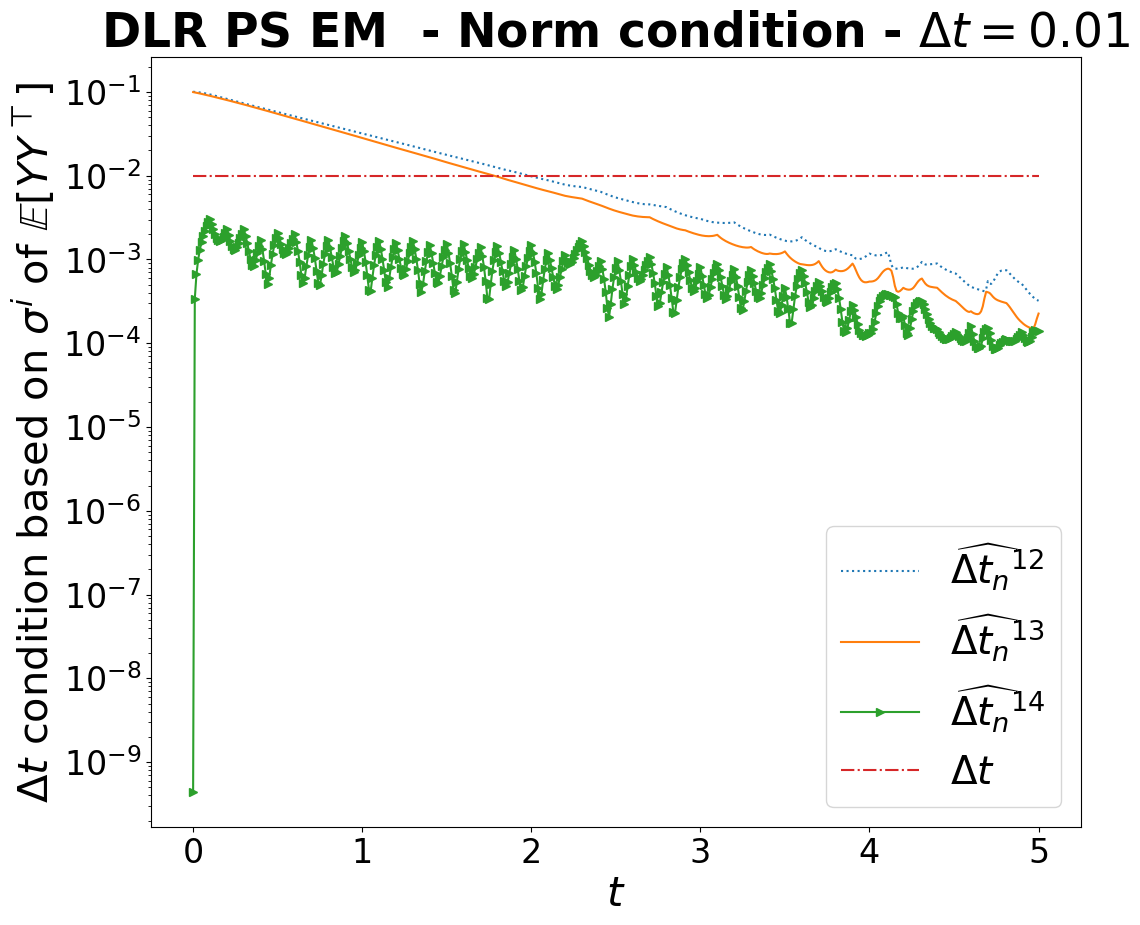}
	\includegraphics[scale=0.165]{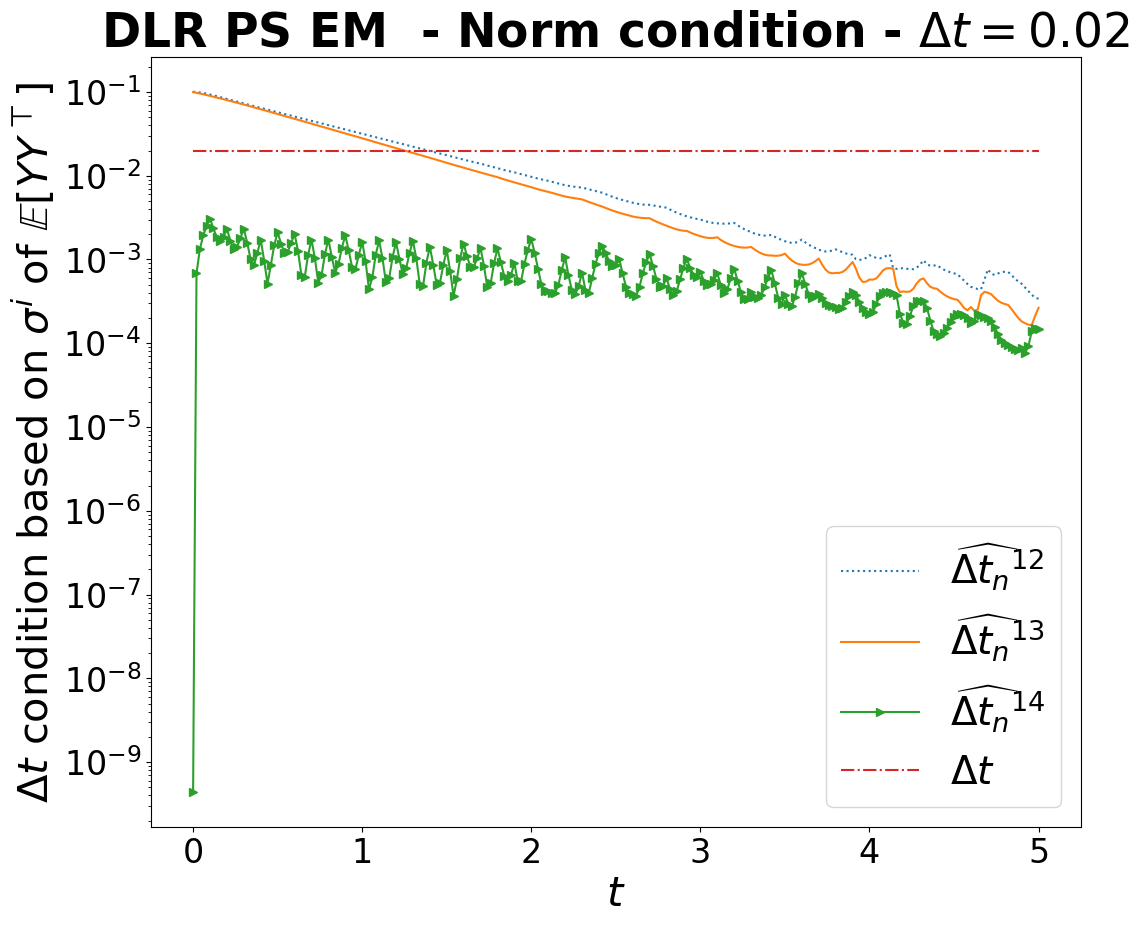}\\
	\includegraphics[scale=0.16]{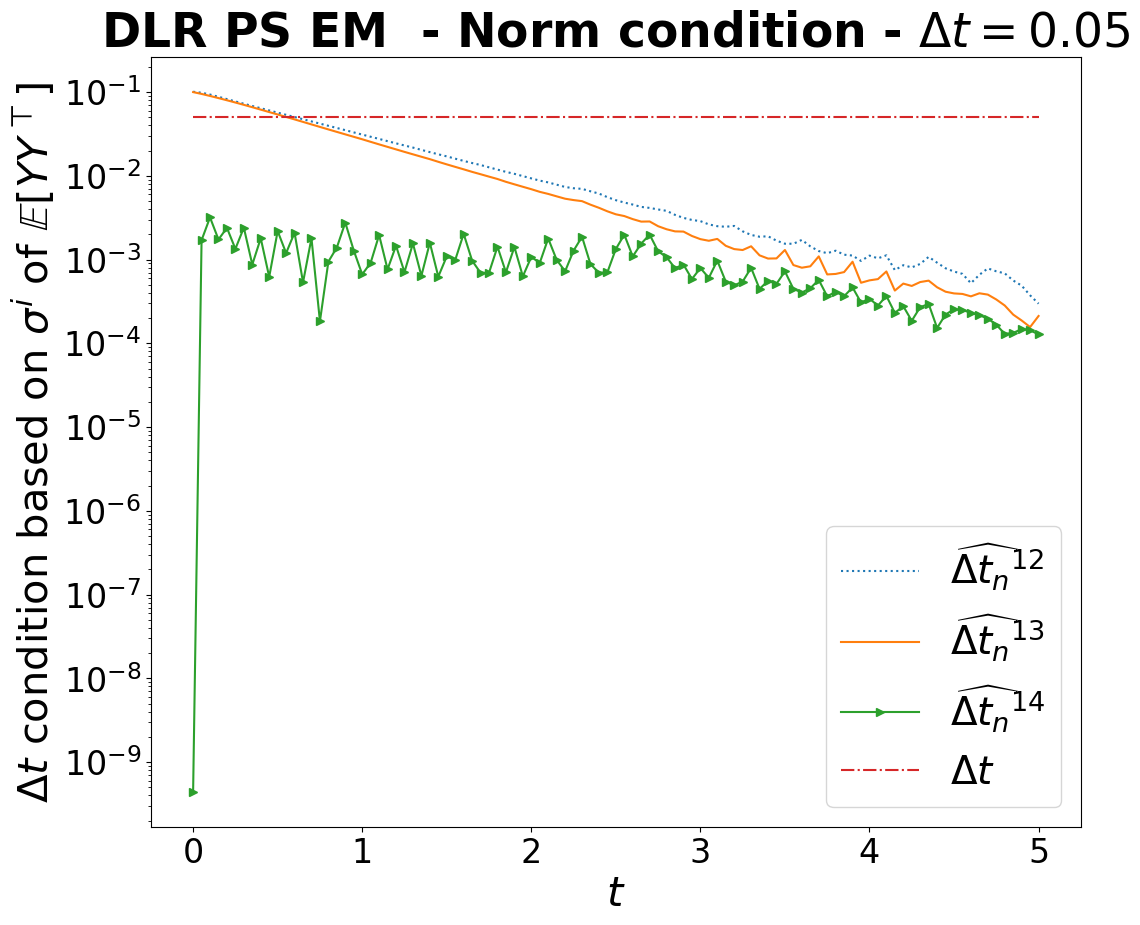}
	\includegraphics[scale=0.16]{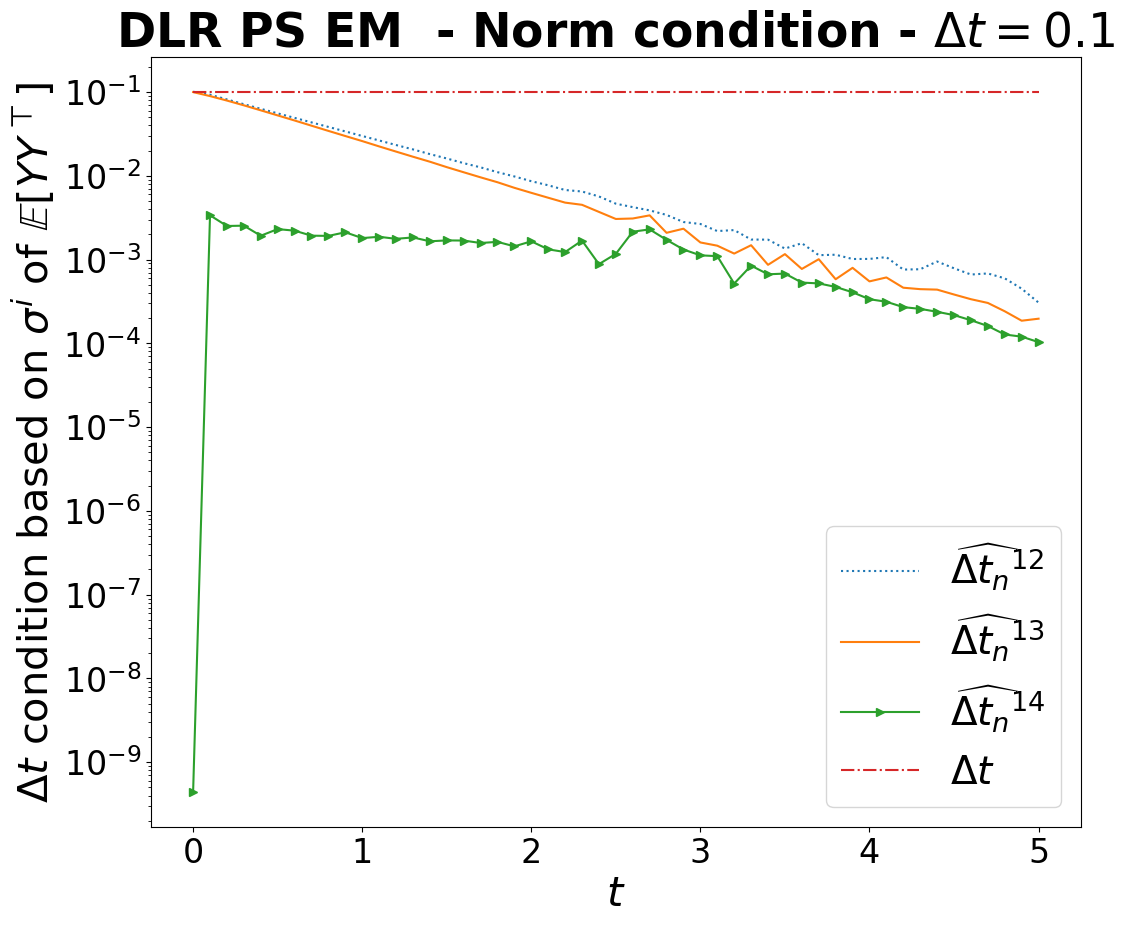}
	\caption{Condition \eqref{eq: delta t stab 2} related to the square root of the singular values $\sigma^i$ of the Gramian $\mathbb{E}[X_nX^{\top}_n]$ for DLR Projector Splitting for EM with $\Delta t = 0.01,0.02,0.05,0.1$ for the Problem \eqref{ex: laplacian}.}
	\label{fig: laplacian singular values DLR KNV}
\end{figure}

Figures \ref{fig: laplacian singular values DLR EM}, \ref{fig: laplacian singular values Stoch Proj}, and \ref{fig: laplacian singular values DLR KNV} illustrate that the time step does not always satisfy \eqref{eq: delta t stab 2} in $[0,T]$. Again, we judge this behavior as a violation of the condition \eqref{eq: dt sup n cond} and we expect that the DLR Euler-Maruyama method does not present error decay over time for large $\Delta t$, unlike the two Projector Splitting methods.
	
Figure \ref{fig:laplacian example errors} presents relative errors between the DLRA algorithms and the continuous DLRA and between the DLRA algorithms and the true solution for different values of $\Delta t$. 
We observe that the projector-splitting methods, Algorithms \ref{alg: Stoch DLR Proj algorithm} and \ref{alg: Eva Proj Splitt SDE algorithm}, are unaffected by condition \eqref{eq: dt sup n cond} and their convergence to the true solution $X^{\mathrm{true}}$ is not influenced by the smallest singular value of the Gramian matrix of the stochastic basis, as stated in Theorems~\ref{thm: Numerical Convergence Stoch DLR Proj - True} and \ref{thm: Numerical Convergence Eva}. On the other hand, the DLR Euler-Maruyama presents a huge constant multipliying the convergence rate in time. This numerical behavior supports Theorem~\ref{thm: convergence of DLR Euler-Maruyama}. 

\begin{figure}[!h]
	\centering
\includegraphics[scale=0.28]{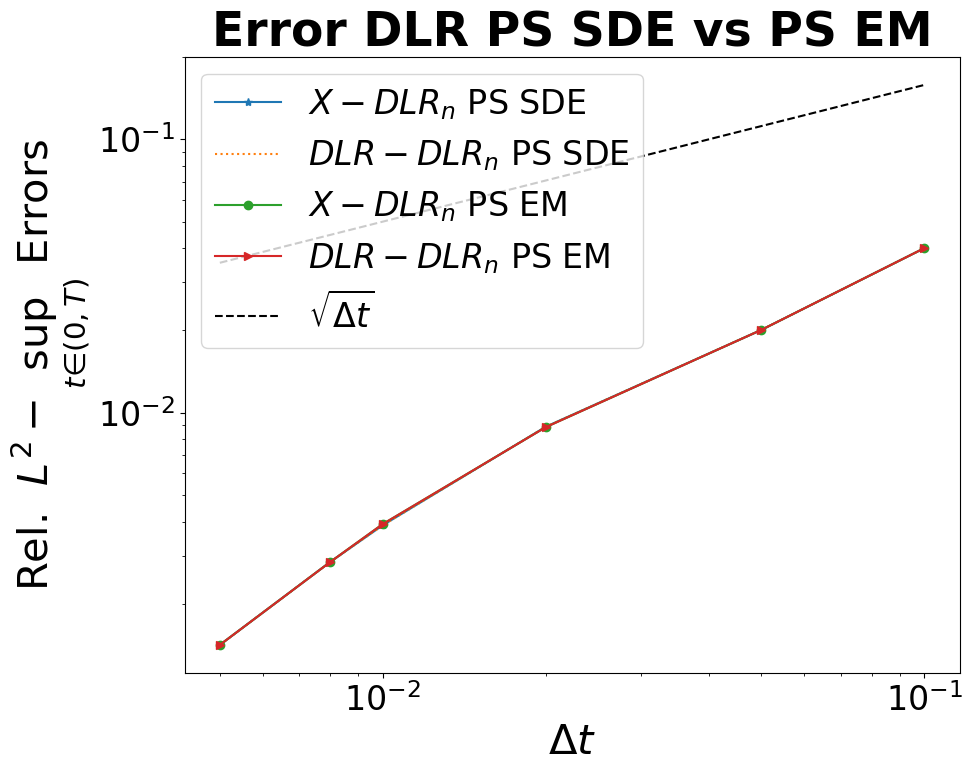}
\includegraphics[scale=0.28]{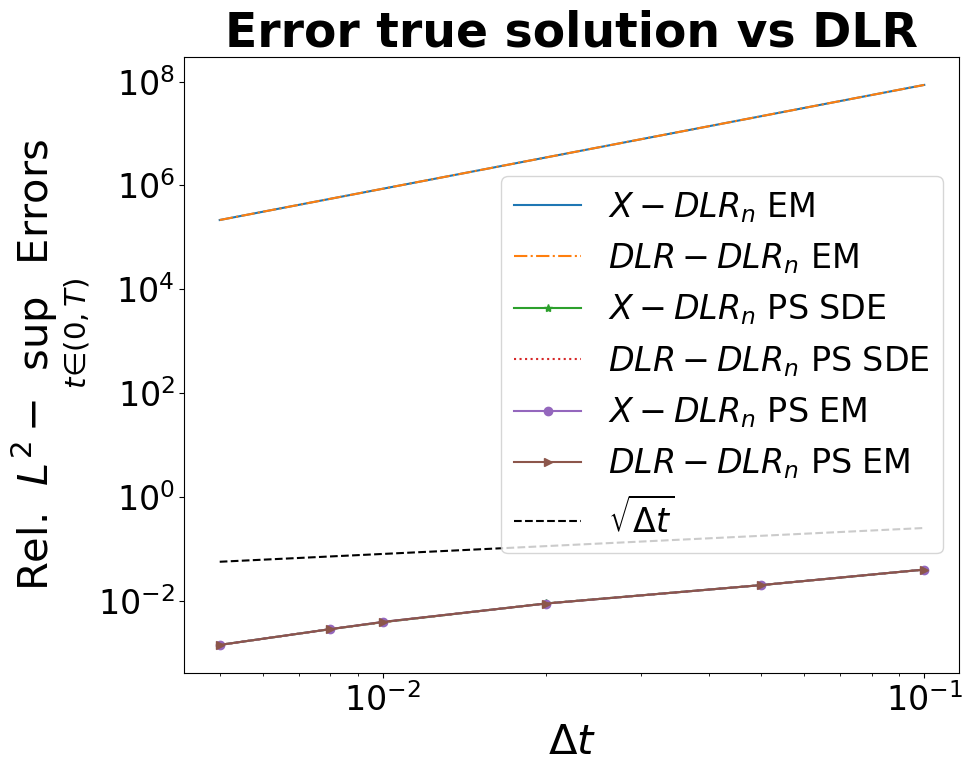}
	\caption{Rel. errors in $L^{2}-\sup\limits_{t \in [0,T]}$ norm for the Problem \eqref{ex: laplacian}, $k=14$, $M=2200$ paths, (\underline{Blue} $(-)$) error between the true solution and the DLRA algorithms and between the continuous DLRA and the DLRA algorithms}
	\label{fig:laplacian example errors}
\end{figure}

\section{Conclusion}
\label{sec: conclusion}
In this work, we have proposed three numerical schemes to discretize DLRA for SDEs. The first scheme was built by discretizing both the deterministic and stochastic modes forward in time, followed by an
 orthonormalization procedure applied to the deterministic modes. 
 Despite this final nonlinear operation, it turned out that this algorithm is convergent. 
 However, its convergence requires a time-step restriction involving the smallest singular value of the Gramian matrix of the stochastic modes.
 If this condition is not met, the numerical solution may not be accurate, as our numerical experiments show.
 Under the same condition, in a bounded interval of time, we proved that the Gramian of the stochastic modes is always lower-bounded in a matrix sense by a strictly positive constant if the diffusion matrix is strictly positive. 
   This result is consistent with the theory of DLRA for SDEs developed in \cite{kazashi2025dynamical}, which provides analagous lower bound on the Gramian under the same condition on the diffusion. 

The second and the third discretization procedures consist in evolving the stochastic modes forward in time via a standard Euler-Maruyama method (as in the previous scheme), whereas the deterministic modes are computed using the new stochastic components before orthonormalization. 
These algorithms differ in the arguments inside the expectation present in the equation of the modes: the former involves only the drift, the latter both drift and diffusion increments. This distinction characterizes the DLRA-type equations that they approximate: the \textit{Projector Splitting method for SDEs} converges to the DLRA for SDEs defined in \cite{kazashi2025dynamical}, whereas the \textit{Projector Splitting method for Euler-Maruyama} converges to the DO equations defined in \cite{cao2018stochastic}.
We have shown that both methods converge to the respective continuous DLR solution without any restriction on the time step, and satisfy the same lower bound on the Gramian as Algorithm \ref{alg: DLR EM SDE algorithm} in case of strictly positive diffusion. In particular, our error estimates for Algorithm \ref{alg: Stoch DLR Proj algorithm} are unaffected by the smallest singular value of the Gramian. This is unfortunately not the case for Algorithm \ref{alg: Eva Proj Splitt SDE algorithm}. However, both algorithms show similar convergence behavior in time in our numerical simulations, evidencing a possible limitation of our theoretical analysis of Algorithm \ref{alg: Eva Proj Splitt SDE algorithm}.

Among possible future directions of research, an in-depth theoretical analysis of the mean-square stability of the DLR EM scheme for Equation \eqref{eq: moltiplic sdes} can be of interest. 
Indeed, in this case the method takes the form
\begin{equation}\label{eq: DLR EM sde stab}
	\begin{aligned}
		X_{n+1} = & X_n +  A(t_n) X_n \Delta t_n + U^{\top}_n U_n \sum_{k=1}^{m} B_k(t_n) X_n \Delta W^{k}_n,  \\
		& + \left(I_{d \times d} - U_n^{\top}U_n \right)A(t_n) U^{\top}_n U_{n}A (t)X_n (\Delta t_n)^2 \\
		&+  \left(I_{d \times d} - U_n^{\top}U_n \right)A (t)U^{\top}_n U_{n} \left(\sum_{k=1}^{m} B_k(t_n) X_n \Delta W^{k}_n\right) \Delta t_n, \\
	\end{aligned}
\end{equation}
which is not affected by the smallest singular value (see \eqref{eq: linear drift}) and whose second moment can be bounded in the following way
\begin{equation}\label{eq: DLR EM stability cond}
	\begin{aligned}
		\mathbb{E}[|X_{n+1}|^2] \leq  \mathbb{E}[|X_n|^2] \Big(& 1+ \lambda_{\max}(A(t_n) + A^{\top}(t_n) ) \Delta t_n + \sigma^2_{\max}(A(t_n)) (\Delta t_n)^2  \\
		& + \lambda_{\max}\Big(A^{\top}(t_n) \left( I - U^{\top}_n U_{n}\right) A(t_n) U^{\top}_n U_{n} A(t_n) \\
		& \ + A^{\top}(t_n) U^{\top}_n U_{n} A^{\top}(t_n) \left( I - U^{\top}_n U_{n}\right) A(t_n) \Big) (\Delta t_n)^3 \\
		& + \sigma_{\max}(A(t_n))^4 (\Delta t_n)^4 + \sigma_{\max}^2(A(t))\sum\limits_{k=1}^{m} |B_k(t_n)|^2(\Delta t_n)^3 +  \sum\limits_{k=1}^{m}  |B_k(t_n)|^2 \Delta t_n \Big).
	\end{aligned}
\end{equation}
However, it is not immediately clear how to interpret \eqref{eq: DLR EM stability cond} and how to compare it with the stability conditions of the other two schemes. In addition, further investigation should be conducted to obtain convergence results of the DLR PS EM scheme independent of the smallest singular value.

To conclude this paper, we note that all the analyses done in this first part always assumes exact expectations for Algorithms \ref{alg: DLR EM SDE algorithm},  \ref{alg: Stoch DLR Proj algorithm}, and \ref{alg: Eva Proj Splitt SDE algorithm}.
	In the second part of this work \cite{kazashi2025dynamicalpartII} these three procedures will be analyzed with a focus on discretizing the probability measure using a Monte Carlo method.

\section*{Acknowledgements}
FN and FZ were supported by the Swiss National Science Foundation under the
Project n. 200518 “Dynamical low rank methods for uncertainty quantification and data assimilation”.

\printbibliography

\appendix

\section{Auxiliary results}\label{app:aux}
In this section, we collect some auxiliary results used throughout the paper, as well as all the proofs deferred from the main article.

\begin{Lemma}[Analyticity of the QR decomposition of U]\label{lem: analyticity of U}
		Consider $U_n \in \mathbb{R}^{k \times d}$ with orthogonal rows and $A \in \mathbb{R}^{k \times d}$. Define
		\begin{equation*}
			\begin{aligned}
				U(s) := U_n + A P_{U_n}^\perp s, \quad \text{where } 	\qquad
				P_{U_n}^\perp = I - U_n^{\top}U_n.
			\end{aligned}
		\end{equation*}
		Let \texttt{QR} denote either the QR decomposition based on Cholesky factorization or Gram-Schmidt orthonormalization \cite{golub2013matrix}. Then the two functions $Q : [0, \infty) \to \mathbb{R}^{d \times k}$ and $R : [0, \infty) \to \mathbb{R}^{k \times k}$ defined as
		\begin{equation}\label{eq: def Q R}
			\begin{aligned}
				Q(0) = U_n^{\top}, \quad R(0) = I_{k \times k}, \quad \texttt{QR}(U^{\top}(s)) = Q(s) R(s), \\
			\end{aligned}
		\end{equation} 
		where $Q(s)$ has orthogonal columns and $R(s)$ is upper triangular, are well-defined and analytic in $s$ on $[0, \infty)$. Moreover $\left. \frac{d R}{d s} \right|_{s = 0} =0$ and for $s\in [0,T]$ it holds that
	\begin{equation}\label{eq: error in R}
		\| R(s) - I_{k \times k} \|_{\mathrm{F}} \leq  \frac{s^2 }{2 } \sqrt{2(1+T^2|A|^2)} \bigl(\|A\|^2_{\mathrm{F}}+2T^2\|A\|_{\mathrm{F}}^4\bigr).
	\end{equation}
	\begin{proof}
		We split the proof into several steps. First we show that $U(s)U(s)^{\top}$ is always of full rank, which implies that the QR is always well defined and unique. Then, we show that the \texttt{QR} based on the Cholesky factorization is made by only analytical operations, as well as the one built on Gram-Schmidt orthonormalization.
	
		By orthogonality of the rows of $U_n$, one has
		\begin{equation}\label{eq: QR}
			\begin{aligned}
				U(s)U(s)^\top 
				&=
				(U_n + A P_{U_n}^\perp s)(U_n + A P_{U_n}^\perp s)^\top  \\
				&=
				U_n U_n^\top 
				+U_n P_{U_n}^\perp A^\top s
				+ A P_{U_n}^\perp U_n^\top s
				+ A P_{U_n}^\perp P_{U_n}^\perp A^\top s^2 \\
				& = I_{k \times k} + A P_{U_n}^\perp A^\top s^2.
			\end{aligned}
		\end{equation}
	  As $A P_{U_n}^\perp A^\top s^2$ is symmetric positive semidefinite, then $U(s) U(s)^\top$ is symmetric positive definite for every $s \in \mathbb{R}$, and therefore $U(s)^{\top}$ has full column rank for every $s \in \mathbb{R}$. Therefore $U(s)$ admits a reduced QR factorization
		\begin{equation*}
			\begin{aligned}
				U(s)^{\top}=Q(s)R(s),
			\end{aligned}
		\end{equation*}
		where $Q(s) :=[q_1(s) \dots q_k(s)]$ with $q_i(s) \in \mathbb{R}^d$ for $i=1,\dots,k$ and $Q(s)^\top Q(s)=I_{k\times k}$, and $R(s)=(r_{ij}(s))_{ij}$ is upper triangular. Moreover, if 
		\begin{equation*}
			\begin{aligned}
				r_{ii}(s)>0
				\qquad
				\text{for all } i=1,\dots,k,
			\end{aligned}
		\end{equation*}
		then the factorization is unique.
		
		Let us first show the analyticity of $Q(s)$ and $R(s)$ via Cholesky QR.
		Define the Gram matrix
		\begin{equation*}
			\begin{aligned}
				G(s):=U(s) U(s)^\top =I_{k \times k}+ A P_{U_n}^\perp A^\top s^2.
			\end{aligned}
		\end{equation*}
		Since $G(s)=(g_{ij}(s))_{ij}$ is symmetric positive definite for all $s$, it admits a unique Cholesky factorization $G(s)=R(s)^\top R(s)$, where $R(s)$ is upper triangular with only positive elements in the diagonal.
	    We now show that the entries of $R(s)$ are analytic. The Cholesky recursion is given by the following relation
		\begin{equation*}
			\begin{aligned}
				r_{ii}(s)&=\sqrt{g_{ii}(s)-\sum\limits_{\ell=1}^{i-1} r_{\ell i}(s)^2},
				\quad r_{ij}(s)=\frac{g_{ij}(s)-\sum\limits_{\ell=1}^{i-1} r_{\ell i}(s)r_{\ell j}(s)}{r_{ii}(s)},
				\ 1 \le i< j \le k,
			\end{aligned}
		\end{equation*}
		and, of course, $r_{ij}=0$ for $i>j$. We prove by induction on $k$ that all entries $r_{ij}(s)$ are analytic.
		
		For $i=1$, $r_{11}(s)=\sqrt{g_{11}(s)}$ and since $g_{11}(s)$ is quadratic in s and strictly positive by $G$ being symmetric positive definite, $r_{11}(s)$ is analytic.
		Assume now that all entries computed before step $i$ are analytic. Then the function
		\begin{equation*}
			\begin{aligned}
				h_i(s):=g_{ii}(s)-\sum_{\ell=1}^{i-1} r_{\ell i}(s)^2
			\end{aligned}
		\end{equation*}
		is analytic, as it is obtained from analytic functions by additions and multiplications. Moreover, by the Cholesky construction, $h_i(s)=r_{ii}(s)^2>0.$ Since the square root is analytic on $(0,\infty)$, it follows that $r_{ii}(s)=\sqrt{h_i(s)}$ is analytic. Then for each $i>j$, the numerator
		\begin{equation*}
			\begin{aligned}
				g_{ij}(s)-\sum_{\ell=1}^{i-1} r_{\ell i}(s)r_{\ell j}(s)
			\end{aligned}
		\end{equation*}
		is analytic, and the denominator $r_{ii}(s)$ is analytic and strictly positive, implying that the QR is unique and $R(s)$ is analytic in $s$.
		
		Now define $Q(s):=U(s)R(s)^{-1}$.
		Since matrix inversion is analytic on the set of invertible matrices and $R(s)$ is invertible for all $s$, it follows that $R(s)^{-1}$ is analytic, hence $Q(s)$ is analytic, too. 
		
		We briefly show the properties of the analyticity of $Q$, $R$ obtained via classical Gram-Schmidt orthonormalization.
		Let $u_1(s),\dots,u_k(s)$ denote the columns of $U(s)^{\top}$. Since $U(s)^{\top}$ depends affinely on $s$, each column $u_i(s)$ is analytic in $s$. The Gram-Schmidt procedure reads as
		\begin{equation*}
			\begin{aligned}
				v_1(s)&:=u_1(s), \qquad	r_{11}(s):=|v_1(s)|, \qquad q_1(s):=\frac{v_1(s)}{r_{11}(s)},
			\end{aligned}
		\end{equation*}
		and for $j \ge 2$,
		\begin{equation*}
			\begin{aligned}
				r_{ij}(s)&:=q_i(s)^\top u_j(s),
				\ i=1,\dots,j-1,
				\quad
				v_j(s):=u_j(s)-\sum_{i=1}^{j-1} q_i(s)r_{i j}(s),\\
				\
				r_{jj}(s)&:=|v_j(s)|,
				\quad
				q_j(s):=\frac{v_j(s)}{r_{jj}(s)},
			\end{aligned}
		\end{equation*}
		and $r_{ij}=0$ for $i>j$. Suppose by contradiction that for some $i$ and some $s$ one had $v_i(s)=0$. Then $u_j(s)$ would belong to the span of $u_1(s),\dots,u_{j-1}(s)$, so the first $j$ columns of $U(s)$ would be linearly dependent. This contradicts the full-column-rank property of $U(s)$. Therefore
		\begin{equation*}
			\begin{aligned}
				v_j(s)\neq 0\qquad \text{for all } j=1,\dots,k \text{ and for all } s\in[0, \infty),
			\end{aligned}
		\end{equation*}
		and hence $r_{kk}(s)=|v_k(s)|>0$, and the QR decomposition is well-defined.
		We now prove analyticity by induction on $i$. For $i=1$, $v_1(s)=u_1(s)$ is analytic, and $r_{11}(s)=\sqrt{v_1(s)^\top v_1(s)}$ is analytic because $v_1(s)^\top v_1(s)$ is analytic and strictly positive. Thus $q_1(s)=\frac{v_1(s)}{r_{11}(s)}$ is analytic.
		
		Assume now that $q_1(s),\dots,q_{j-1}(s)$ are analytic. Then each $r_{ij}(s)=q_i(s)^\top u_j(s)$ is analytic and, hence,
		\begin{equation*}
			\begin{aligned}
				v_j(s)=u_j(s)-\sum_{i=1}^{j-1} q_i(s)r_{i j}(s)
			\end{aligned}
		\end{equation*}
		is analytic. Since $v_j(s)\neq 0$ for all $s$, also $r_{jj}(s)=\sqrt{v_j(s)^\top v_j(s)}$ is analytic, and, hence, $q_j(s)=\frac{v_j(s)}{r_{jj}(s)}$ is analytic, too, implying the thesis by induction. 
		
		To prove the final bound in \eqref{eq: error in R}, we want to find an explicit expression of the derivatives of $R$ and then bound them.
		Notice that from \eqref{eq: QR}, one has
		\begin{equation}\label{eq: R.TR }
			R(s)^\top R(s)=I_{k \times k}+s^2M,
		\end{equation}
		with $M=A P_{U_n}^\perp A^\top$. 
		On the other hand, by analyticity of $s \mapsto R(s)$ and from \eqref{eq: R.TR } one has $R(s)=I_{k \times k} + R_1 s + R_2 s^2 + ...$, for reasonable upper triangular matrices $R_i \in \mathbb{R}^{k \times k}$, implying
		\begin{equation*}
			R(\Delta t)^{\top}R(\Delta t) = I_{k \times k} + \left( R_1 + R_1^{\top} \right) \Delta t + ( R_1^{\top} R_1 +  R_2 + R_2^{\top}) (\Delta t)^2 + O(\Delta t^3).
		\end{equation*}
		These relations and the fact that $R$ is upper triangular by construction imply that $R_1 = 0$, and therefore by analyticity of $R$ we can write		
		\begin{equation}\label{eq: R taylor}
			R(s) = I_{k \times k} + \int_{0}^{s} r \frac{\mathrm{d}^2 R(r)}{\mathrm{d} r^2} \mathrm{d} r.
		\end{equation}
	
		Now, we want to find a uniform bound on the second derivative of $R$. Via differentiating the relation \eqref{eq: R.TR } we find that
		\begin{equation*}
			\frac{\mathrm{d} R(s)^\top}{\mathrm{d} s} R(s)+R(s)^\top \frac{\mathrm{d} R(s)}{\mathrm{d} s}=2sM,
		\end{equation*}
		and
		\begin{equation*}
			\frac{\mathrm{d}^2 R(s)^\top}{\mathrm{d} s^2} R(s)+2\frac{\mathrm{d} R(s)^\top}{\mathrm{d} s} \frac{\mathrm{d} R(s)}{\mathrm{d} s}+R(s)^\top \frac{\mathrm{d} ^2 R(s)}{\mathrm{d} s^2}=2M.
		\end{equation*}
		Now, as $R(s)^{-1}$ is well-defined, then we can define
		\begin{equation*}
		  D(s):=\frac{\mathrm{d} R(s)}{\mathrm{d} s}R(s)^{-1},
			\qquad
			E(s):=\frac{\mathrm{d}^2 R(s)}{\mathrm{d} s^2}R(s)^{-1},
		\end{equation*}
		where $D$ and $E$ are still upper-triangular. Then, the following relations hold
		\begin{equation*}
			D(s)^\top+D(s)=2sR(s)^{-\top}MR(s)^{-1},
		\end{equation*}
		and
		\begin{equation*}
			E(s)^\top+E(s)=2R(s)^{-\top}MR(s)^{-1}-2D(s)^\top D(s).
		\end{equation*}
		By means of the upper-triangular structure, it holds that
		\begin{equation*}
			\left\| D(s) \right\|_{\mathrm{F}} \leq \sqrt{2} \left\| \frac{D(s) + D(s)^{\top}}{2}\right\|_{\mathrm{F}},\quad \left\| E(s) \right\|_{\mathrm{F}} \leq \sqrt{2} \left\| \frac{E(s) + E(s)^{\top}}{2}\right\|_{\mathrm{F}},
		\end{equation*}
		whereas by relation \eqref{eq: R.TR }, one has $R(s)^\top R(s) \succeq I_{k \times k},$ which implies $|R(s)^{-1}| \leq 1$
	and, hence, 
		\begin{equation*}
			\|D(s)\|_{\mathrm{F}}\le \sqrt{2}\,s\,\|M\|_{\mathrm{F}}, \quad 
			\|E(s)\|_{\mathrm{F}}\le \sqrt{2}\bigl(\|M\|_{\mathrm{F}}+2s^2\|M\|_{\mathrm{F}}^2\bigr).
		\end{equation*}
		By means of putting together previous computations and relation \eqref{eq: R.TR }, one obtains
		\begin{equation*}
			\begin{aligned}
			\left\|\frac{\mathrm{d}^2 R(s)}{\mathrm{d} s^2}\right\|_{\mathrm{F}}\le & \|E(s)\|_{\mathrm{F}}|R(s)|
			\le	\sqrt{2(1+s^2|M|)} \bigl(\|M\|_{\mathrm{F}}+2s^2\|M\|_{\mathrm{F}}^2\bigr).
			\end{aligned}
		\end{equation*}
		Finally, from \eqref{eq: R taylor} we can write
			\begin{equation*}
				\begin{aligned}
			\| R(s) - I_{k \times k} \|_{\mathrm{F}} = \left\| \int_{0}^{s} r \frac{\mathrm{d}^2 R(r)}{\mathrm{d} r^2} \mathrm{d} r \right\|_{\mathrm{F}} &\leq  \frac{s^2 }{2 }\max\limits_{s \in [0,T]} 	\left\|\frac{\mathrm{d}^2 R(s)}{\mathrm{d} s^2}\right\|_{\mathrm{F}} \\
			&\leq  \frac{s^2 }{2 } \sqrt{2(1+s^2|M|)} \bigl(\|M\|_{\mathrm{F}}+2s^2\|M\|_{\mathrm{F}}^2\bigr) \\
			&\leq  \frac{s^2 }{2 } \sqrt{2(1+T^2|A|^2)} \bigl(\|A\|^2_{\mathrm{F}}+2T^2\|A\|_{\mathrm{F}}^4\bigr).
			\end{aligned}
		\end{equation*}
\end{proof}
\end{Lemma}

\section{Convergence of the DLR Euler-Maruyama}\label{app:DLR EM}

\begin{proof}[\bfseries Proof of Lemma~\ref{lem: sup n yn}]\label{proof: lem: sup n yn}
	To prove \eqref{eq: sup n Y_n}, we first note that the orthogonality of the rows of $U_n$ implies
	\begin{equation*}
		\mathbb{E}[\max\limits_{0 \leq k \leq n}|Y_{k}|^{2}]= \mathbb{E}[\max\limits_{0 \leq k \leq n}|U_{k}^{\top}Y_{k}|^{2}] = 
		\mathbb{E}[\max\limits_{0 \leq k \leq n}|X_{k}|^2].
	\end{equation*}
	First, we want to retrieve a recursion on $\mathbb{E}[\max\limits_{0 \leq k \leq n}|X_{k}|^2]$ via using \eqref{Euler--Maruyama eq}, then we will bound the terms of this relation in order to use a discrete Gronwall's lemma to prove the statement. 
	
	From \eqref{Euler--Maruyama eq} one has that
	\begin{equation*}
		\begin{aligned}
			X_{n+1}
			=& U_{n+1}^{\top} Y_{n+1} \\
			=& X_n
			+ P_{U_n}^{\perp} P_{Y_n} a_n \Delta t
			+ P_{U_n} a_n \Delta t
			+ P_{U_n} b_n \Delta W_n
			+ P_{U_n}^{\perp} \mathbb{E}[a_n Y_n^{\top}] C_{Y_n}^{-1} U_n
			\left[ a_n \Delta t + b_n \Delta W_n \right] \Delta t \\
			=& X_0
			+ \sum_{j=0}^{n} \left( P_{U_j}^{\perp} P_{Y_j} + P_{U_j} \right) a_j \Delta t
			+ \sum_{j=0}^{n} P_{U_j} b_j \Delta W_j 
			+ \sum_{j=0}^{n} P_{U_j}^{\perp} \mathbb{E}[a_j Y_j^{\top}]
			C_{Y_j}^{-1} U_j a_j (\Delta t)^2 \\
			&
			+ \sum_{j=0}^{n} P_{U_j}^{\perp} \mathbb{E}[a_j Y_j^{\top}]
			C_{Y_j}^{-1} U_j b_j \Delta W_j \Delta t .
		\end{aligned}
	\end{equation*}
	Via passing to square norm, using basic norm inequalities, first passing to the sup, and then applying the expectation to both members, one gets	
	\begin{equation*}
		\begin{aligned}
			\mathbb{E}[\max_{0 \leq k\leq n+1} |X_{k}|^2]
			\le& 5 \mathbb{E}\left[|X_0|^2 \right]
			+ 5\underbrace{\mathbb{E}\left[ \max\limits_{0 \leq k \leq n} \left| \sum_{j=0}^{k}
				\left( P_{U_j}^{\perp} P_{Y_j} + P_{U_j} \right) a_j \Delta t  \right|^2\right] }_{T_1}
			+ 5 \underbrace{\mathbb{E}\left[  \max\limits_{0 \leq k \leq n} \left|\sum_{j=0}^{k}
				P_{U_j} b_j \Delta W_j \right|^2\right]}_{T_2} \\
			&
			+ 5 \underbrace{\mathbb{E}\left[ \max\limits_{0 \leq k \leq n} \left|\sum_{j=0}^{k}
				P_{U_j}^{\perp} \mathbb{E}[a_j Y_j^{\top}]
				C_{Y_j}^{-1} U_j a_j (\Delta t)^2 \right|^2\right]}_{T_3} \\
			&+ 5 \underbrace{\mathbb{E}\left[ \max\limits_{0 \leq k \leq n} \left|  \sum_{j=0}^{k}
				P_{U_j}^{\perp} \mathbb{E}[a_j Y_j^{\top}]
				C_{Y_j}^{-1} U_j b_j \Delta W_j \Delta t  \right|^2\right]}_{T_4}
		\end{aligned}
	\end{equation*}
	
	Now, we bound each of the terms $\mathrm{T}_1$-$\mathrm{T}_4$ in a suitable manner. By means of exploiting the linearity of the expectation, basic norm inequalities, the fact that $P_{X_n}$ is an orthogonal projector in $L^2(\Omega, \mathbb{R}^d)$, the linear-growth bound of the drift, and properties of the supremum, for the first component one has that
	\begin{equation*}
		\begin{aligned}
			\mathrm{T}_1 =	\mathbb{E}\left[
			\max\limits_{0 \leq k \leq n}
			\left| \sum_{j=0}^{k} P_{X_j} a_j \Delta t  \right|^2
			\right]
			&\le
			\mathbb{E}\left[
			\max\limits_{0 \leq k \leq n}
			\left(\sum_{j=0}^{k} |P_{X_j} a_j| \Delta t  \right)^2
			\right] \leq  	\mathbb{E}\left[ \max\limits_{0 \leq k \leq n} (k+1)
			\sum_{j=0}^{k} |P_{X_j} a_j|^2 \Delta t^2 
			\right] \\
			&\le
			\mathbb{E}\left[
			(n+1)\sum_{j=0}^{n}
			|P_{X_j} a_j|^2 \Delta t^2
			\right] \\
			&\le
			T \sum_{j=0}^{n}
			\mathbb{E}\left[ |P_{X_j} a_j|^2 \right] \Delta t \\
			&\le
			T \sum_{j=0}^{n}
			\mathbb{E}\left[ |a_j|^2 \right] \Delta t
			\le
			T C_{\mathrm{lgb}}
			\sum_{j=0}^{n}
			\left( 1 + \mathbb{E}[|X_j|^2] \right) \Delta t \\
			&\le
			T C_{\mathrm{lgb}}
			\sum_{j=0}^{n}
			\left( 1 + \mathbb{E}\left[\max\limits_{0 \leq k \leq j} |X_k|^2 \right] \right)
			\Delta t
		\end{aligned}
	\end{equation*}
	
	We notice that the term $\mathrm{T}_2$ is a discrete martingale, hence through \cite[Theorem 4.4.4]{durrett2019probability} one has that
	\begin{equation*}
		\begin{aligned}
			\mathrm{T}_2 := \mathbb{E}\left[
			\max\limits_{0 \leq k \leq n}
			\left| \sum_{j=0}^{k}
			P_{U_j} b_j \Delta W_j
			\right|^2
			\right]
			&\le
			4\,
			\mathbb{E}\left[
			\left|
			\sum_{j=0}^{n}
			P_{U_j} b_j \Delta W_j
			\right|^2
			\right],
		\end{aligned}
	\end{equation*}
	and, hence, via independence of Brownian increments, properties of orthogonal projectors, linear-growth bound of the diffusion, linearity of the expectation, and properties of the supremum, one finds that		
	\begin{equation*}
		\begin{aligned}
			\mathrm{T}_2 \leq & 4 \sum_{j=0}^{n}
			\mathbb{E}\left[ |P_{U_j} b_j \Delta W_j|^2 \right] \\
			\leq  & 4 \sum_{j=0}^{n}
			\mathbb{E}\left[ |b_j|^2 \right] \Delta t
			\le
			4 \sum_{j=0}^{n}
			C_{\mathrm{lgb}} \left( 1 + \mathbb{E}[|X_j|^2] \right) \Delta t \\
			\leq &
			4 \sum_{j=0}^{n}
			C_{\mathrm{lgb}}
			\left( 1 + \mathbb{E}\left[\max\limits_{0 \leq k \leq j} |X_k|^2 \right] \right)
			\Delta t.
		\end{aligned}
	\end{equation*}
	
	Concerning the term $\mathrm{T}_3$, similarly to the treatment of $\mathrm{T}_1$, one gets
	\begin{equation*}
		\begin{aligned}
			\mathbb{E}\left[
			\max\limits_{0 \leq k \leq n}\left|\sum_{j=0}^{k}P_{U_j}^{\perp}
			\mathbb{E}[a_j Y_j^{\top}]C_{Y_j}^{-1}U_j a_j\right|^2\Delta t^4\right]
			&\le
			\mathbb{E}\left[
			(n+1)
			\sum_{j=0}^{n}
			\left|
			P_{U_j}^{\perp}
			\mathbb{E}[a_j Y_j^{\top}]
			C_{Y_j}^{-1}
			U_j a_j
			\right|^2
			\Delta t^4
			\right] \\
			&\le T
			\mathbb{E}\left[
			\sum_{j=0}^{n}
			\left|
			\mathbb{E}[a_j Y_j^{\top}C_{Y_j}^{-\frac{1}{2}}]
			C_{Y_j}^{-\frac{1}{2}} 
			U_j a_j
			\right|^2
			\Delta t^3
			\right] \\
			&\le 
			\frac{T}{\sigma_{n}^k}
			\sum_{j=0}^{n}
			\mathbb{E}\left[ |a_j|^2 \right]
			\mathbb{E}\left[ |U_j a_j|^2 \right]
			\Delta t^3 \\
			&\le
			\frac{T}{\sigma_{n}^k}
			\sum_{j=0}^{n}
			C_{\mathrm{lgb}}^2
			\left( 1 + \mathbb{E}[|X_j|^2] \right)^2
			\Delta t^3 \\
			&\le
			\frac{T}{\sigma_{n}^k}
			\sum_{j=0}^{n}
			C_{\mathrm{lgb}}^2
			\left( 1 + \mathbb{E}\left[\max\limits_{0 \leq k \leq j} |X_k|^2 \right] \right)^2
			\Delta t^3.
		\end{aligned}
	\end{equation*}
	
	Finally, for the term $\mathrm{T}_4$, similarly to the development for $\mathrm{T}_2$, we can exploit the fact that it is a martingale, and via \cite[Theorem 4.4.4]{durrett2019probability}, independence of the Brownian increments, and similar computations one gets
	\begin{equation*}
		\begin{aligned}
			\mathbb{E}\left[
			\max\limits_{1 \leq k\le n}
			\left|
			\sum_{j=0}^{k}
			P_{U_j}^{\perp}
			\mathbb{E}[a_j Y_j^{\top}]
			C_{Y_j}^{-1}
			U_j b_j
			\Delta W_j \Delta t
			\right|^2
			\right]
			&\le
			4\,\mathbb{E}\left[
			\left|
			\sum_{j=0}^{n}
			P_{U_j}^{\perp}
			\mathbb{E}[a_j Y_j^{\top}]
			C_{Y_j}^{-1}
			U_j b_j
			\Delta W_j
			\right|^2
			\Delta t^2
			\right] \\
			&\le
			4 \Delta t^2  \sum_{j=0}^{n}
			\left|
			P_{U_j}^{\perp}
			\mathbb{E}[a_j Y_j^{\top}C_{Y_j}^{-\frac{1}{2}}]
			C_{Y_j}^{-\frac{1}{2}}
			U_j b_j \right|^2
			\Delta t  \\
			&\le
			4 \sum_{j=0}^{n} \frac{1}{\sigma_{n}^k}
			\mathbb{E}[|a_j|^2] \mathbb{E}[|b_j|^2]
			\Delta t^3 \\
			&\le 4 \sum_{j=0}^{n}	\frac{1}{\sigma_{n}^k}
			C_{\mathrm{lgb}}^2
			\left( 1 + \mathbb{E}[|X_j|^2] \right)^2
			\Delta t^3 \\
			&\le 4 \sum_{j=0}^{n} \frac{1}{\sigma_{n}^k} C_{\mathrm{lgb}}^2
			\left( 1 + \mathbb{E}\left[\max\limits_{0 \leq k \leq j} |X_k|^2 \right] \right)^2
			\Delta t^3.
		\end{aligned}
	\end{equation*}
	
	Putting all together, one gets
	\begin{equation}\label{eq: last X imp calc}
		\begin{aligned}
			\mathbb{E}[\max\limits_{0 \leq k \leq n+1} |X_{k}|^2]
			\le& 5 \mathbb{E}\left[|X_0|^2 \right] + 5(T+4) C_{\mathrm{lgb}} \sum_{j=0}^{n}
			\left( 1 + \mathbb{E}\left[\max\limits_{0 \leq k \leq j} |X_k|^2 \right] \right)
			\Delta t \\
			& + 5(T+4) 	\frac{1}{\sigma_{n}^k} C_{\mathrm{lgb}}^2 \sum_{j=0}^{n}
			\left( 1 + \mathbb{E}\left[\max\limits_{0 \leq k \leq j} |X_k|^2 \right] \right)^2
			\Delta t^3
		\end{aligned}
	\end{equation}
	Now, to obtain the thesis we proceed via induction. For the sake of notation let us write $Z_n = 5(1 + \mathbb{E}[\max\limits_{0 \leq k \leq n} |X_{k}|^2])$ . The base step of \eqref{eq: sup n Y_n} for $n=0$ is trivially true. Assume \eqref{eq: sup n Y_n} is true for $n$ and choose $\Delta t \leq \frac{ \sqrt{ \sigma_{n}^k} }{ \sqrt{ C_{\mathrm{lgb}}} \sqrt{ 5(1+\mathbb{E}\left[|X_0|^2 \right]) \exp\left(  C_{\mathrm{lgb}} (2T^2+8T) \right)} }$. Then, from \eqref{eq: last X imp calc} one gets
	\begin{equation*}
		\begin{aligned}
			Z_{n+1} =	5 \left(1+ \mathbb{E}\left[
			\max_{0 \leq k\le n+1}
			|X_k|^2
			\right]\right)
			\le&
			5(1+ \mathbb{E}\left[|X_0|^2 \right]) + (T+4) C_{\mathrm{lgb}} \sum_{j=0}^{n}
			5\left( 1 + \mathbb{E}\left[\max\limits_{0 \leq k \leq j} |X_k|^2 \right] \right)
			\Delta t \\
			& +  (T+4) 	\frac{1}{\sigma_{n}^k} C_{\mathrm{lgb}}^2 \sum_{j=0}^{n}
			5 \left( 1 + \mathbb{E}\left[\max\limits_{0 \leq k \leq j} |X_k|^2 \right] \right)^2
			\Delta t^3 \\
			=& Z_0 + (T+4) C_{\mathrm{lgb}} \sum_{j=0}^{n}
			Z_j 
			\Delta t \\
			& + (T+4) 	\frac{1}{\sigma_{n}^k} C_{\mathrm{lgb}}^2 \sum_{j=0}^{n}
			 \left( 1 + \mathbb{E}\left[\max\limits_{0 \leq k \leq j} |X_k|^2 \right] \right) Z_j 
			\Delta t^3 \\
			\leq& Z_0 + 2(T+4) C_{\mathrm{lgb}} \sum_{j=0}^{n}
			Z_j
			\Delta t \\
		\end{aligned}
	\end{equation*}
	and the bound on the right-hand side in \eqref{eq: sup n Y_n} follows by Gronwall's lemma. 
	
	To obtain the first bound in \eqref{eq: sup n Y_n} $\mathbb{E}[\max\limits_{0 \leq k \leq n}|\tilde{Y}_{k}|^{2}] \leq	\mathbb{E}[\max\limits_{0 \leq k \leq n}|Y_{k}|^{2}] $, notice that
	\begin{equation*}
		\begin{aligned}
			|Y_{k}|^{2} &= Y_{k}^{\top} Y_{k} \\
			& = \tilde{Y}_{k}^{\top} R_k^{\top} R_k \tilde{Y}_{k} \\
			&= \tilde{Y}_{k}^{\top} R_k^{\top} U_k U_k^{\top} R_k \tilde{Y}_{k} = \tilde{Y}_{k}^{\top} \tilde{U}_k \tilde{U}_k^{\top} \tilde{Y}_{k}= \tilde{Y}_{k}^{\top} \left(I_{k \times k} + C^{\top}C (\Delta t_n)^2 \right) \tilde{Y}_{k} \\
			& \geq \tilde{Y}_{k}^{\top} \tilde{Y}_{k} = |\tilde{Y}_{k}|^2
		\end{aligned}
	\end{equation*}
	where we used relation \eqref{eq: tilde U tilde U.T} and the fact that $C^{\top}C (\Delta t_n)^2$ is a symmetric positive semidefinite matrix to obtain the thesis.
\end{proof}

\section{Convergence of the DLR Projector-Splitting for SDEs}\label{app:DLR PS SDEs}

\begin{Lemma}[Auxiliary result for DLR PS computations]\label{lem: technical PS}
	Let $X_n= U_n^{\top}Y_n$, $n=0,\dots,N$ be the solution produced by
	Algorithm \ref{alg: Stoch DLR Proj algorithm} with an arbitrary sequence $\{\Delta t_n\}_n$ of step-sizes $\Delta t_n$ such that $\sum_{n=0}^{N-1} \Delta t_n=  T$. Then, the following relations hold:
	\begin{equation*}
		\begin{aligned}
			\mathbb{E}[|X_{n+1}|^2] \leq &	\mathbb{E}[|X_{n}|^2] + 2\mathbb{E}[X^{\top}_n a_n] \Delta t_n + \mathbb{E}\left[|a_n|^2\right](\Delta t_n)^2 + \mathbb{E}\left[\|b_n \|^2_{\mathrm{F}}\right] \Delta t \\
		\end{aligned}
	\end{equation*}
	and
	\begin{equation*}
		\begin{aligned}
			\mathbb{E}[|X_{n+1} - X_n|^2] \leq & \mathbb{E}\left[|a_n|^2\right](\Delta t_n)^2 + \mathbb{E}\left[\|b_n \|^2_{\mathrm{F}}\right] \Delta t.\\
		\end{aligned}
	\end{equation*}	
	\begin{proof}
		First, let us recall that $P_{U_n^{\top}\tilde{Y}_{n+1}} : L^2(\Omega, \mathbb{R}^d) \to L^2(\Omega, \mathbb{R}^d)$, defined in \eqref{eq: proj interm}, is an orthogonal projector for all $n$.
		Via expression \eqref{eq: proj interm}, one has that
		\begin{equation*}
			\begin{aligned}
				X_{n+1} =& X_n + P_{{U}^{\top}_n\tilde{Y}_{n+1}}[ a_n ] \Delta t_n+ P_{U_n} b_n \Delta W_n\\
				=&X_n + \left(I_{d \times d} - P_{U_n} \right) \mathbb{E}\left[ a_n \tilde{Y}_{n+1}^{\top}\right]C^{-1}_{\tilde{Y}_{n+1}}\tilde{Y}_{n+1} \Delta t_n \\
				&+ P_{U_n} a_n  \Delta t_n + P_{U_n} b_n \Delta W_n.
			\end{aligned}
		\end{equation*}
		Therefore the second moment of $X_n$ reads as
		\begin{equation*}
			\begin{aligned}
				\mathbb{E}[|X_{n+1}|^2] = & \mathbb{E}\big[\big| X_n + \left(I_{d \times d} - P_{U_n} \right) \mathbb{E}\left[a_n \tilde{Y}_{n+1}^{\top}\right]C^{-1}_{\widetilde{Y}_{n+1}}\widetilde{Y}_{n+1} \Delta t_n  + P_{U_n} a_n  \Delta t_n + P_{U_n} b_n \Delta W_n \big|^{2}\big] \\
				= &  \mathbb{E}[|X_{n}|^2] + 2\mathbb{E}[X^{\top}_n a_n] \Delta t_n + 2\mathbb{E}[X^{\top}_n b_n \Delta W_n]  +\mathbb{E}\big[\big|  P_{U_n^{\top}\widetilde{Y}_{n+1}} \left[ a_n \right]\big|^2 \big]\Delta t_n^2 + \mathbb{E}[ |P_{U_n} b_n \Delta W_n|^2] \\
				= &  \mathbb{E}[|X_{n}|^2] + 2\mathbb{E}[X^{\top}_n a_n] \Delta t_n  +\mathbb{E}\big[\big|  P_{U_n^{\top}\widetilde{Y}_{n+1}} \left[ a_n \right]\big|^2 \big]\Delta t_n^2+ \mathbb{E}[ \|P_{U_n} b_n\|^2_{\mathrm{F}}] \Delta t_n,\\
			\end{aligned}
		\end{equation*}
		where in the second line we used the orthonormality of the columns of $U_n$ and in the last line we exploited the properties of Brownian increments.
		
		Using the property of the orthogonal projector of $P_{U_n^{\top}\tilde{Y}_{n+1}}$ in $L^2(\Omega, \mathbb{R}^d)$ and of $P_{U_n}$, as well as the properties of Brownian increments, finally one has
		\begin{equation*}
			\begin{aligned}
				\mathbb{E}[|X_{n+1}|^2] \leq & \mathbb{E}[|X_{n}|^2] + 2\mathbb{E}[X^{\top}_n a_n] \Delta t_n + \mathbb{E}\left[|a_n|^2\right]\Delta t_n^2 + \mathbb{E}\left[\|b_n\|^2_{\mathrm{F}}\right] \Delta t_n. \\
			\end{aligned}
		\end{equation*}
		
		Concerning the second bound, likewise one has that
		\begin{equation*}
			\begin{aligned}
				\mathbb{E}[|X_{n+1}-X_n|^2]
				=& \mathbb{E}\big[\big|  P_{U_n^{\top}\widetilde{Y}_{n+1}} \left[a_n\right] \Delta t_n + P_{U_n} b_n \Delta W_n \big|^2 \big] \\
				\leq &  \mathbb{E}\left[|a_n|^2\right](\Delta t_n)^2 + \mathbb{E}\left[\|b_n\|^2_{\mathrm{F}}\right] \Delta t_n.\\ 		
			\end{aligned}
		\end{equation*}
	\end{proof}
\end{Lemma}

\begin{proof}[\bfseries Proof of Lemma \ref{lem: bound stoch dlr proj}] \label{proof: lem: bound stoch dlr proj}
	Thanks to Lemma \ref{lem: technical PS} we have that
	\begin{equation*}
		\begin{aligned}
			\mathbb{E}[|X_{n+1}|^2] \leq &	\mathbb{E}[|X_{n}|^2] + 2\mathbb{E}[X^{\top}_n a_n] \Delta t_n + \mathbb{E}\left[|a_n|^2\right](\Delta t_n)^2 + \mathbb{E}\left[\|b_n\|^2_{\mathrm{F}}\right] \Delta t 
		\end{aligned}
	\end{equation*}
	Then, via Assumption \ref{linear-growth-bound}, Young's inequality, Cauchy-Schwarz inequality, properties of orthogonal projectors, and properties of the Brownian increments, one gets
	\begin{equation}\label{eq: inter norm X_n+1 PS}
		\begin{aligned}
			\mathbb{E}[|X_{n+1}|^2] 
			\leq &  \mathbb{E}[|X_{n}|^2] + \mathbb{E}[|X_{n}|^2] \Delta t_n  + C_{\mathrm{lgb}}(1+ \mathbb{E}[|X_{n}|^2] )\Delta t_n + C_{\mathrm{lgb}}(1+ \mathbb{E}[|X_{n}|^2] )\Delta t_n^2 \\
			& + C_{\mathrm{lgb}}(1+ \mathbb{E}[|X_{n}|^2] )\Delta t_n \\
			= &   \mathbb{E}[|X_{n}|^2] + \mathbb{E}[|X_{n}|^2] \Delta t_n  +2C_{\mathrm{lgb}}(1+ \mathbb{E}[|X_{n}|^2] )\Delta t_n + C_{\mathrm{lgb}}(1+ \mathbb{E}[|X_{n}|^2] )\Delta t_n^2 \\
			\leq &   \mathbb{E}[|X_{n}|^2] + \mathbb{E}[|X_{n}|^2] \Delta t_n  + C_{\mathrm{lgb}}(2+ T)(1+ \mathbb{E}[|X_{n}|^2] )\Delta t_n, \\
		\end{aligned}
	\end{equation}
	where in the last line we used the trivial fact that $\Delta t_n \leq T$.
	
	For the ease of notation, define 
	$Z_n =1 + \mathbb{E}[|X_n|^2]$. The result follows by induction;
	we claim that
	\begin{equation}\label{eq: dlr-ps-ind}
		Z_n \leq Z_0 e^{( 1+ C_{\mathrm{lgb}}(2+ T))t_n}, \quad \forall n.
	\end{equation}
	By construction, the induction hypothesis holds for $n=0$. On the other hand, from \eqref{eq: inter norm X_n+1 PS} by adding $1$ on the left-hand side and $1+ \Delta t_n$ on the right-hand side one has that 
	\begin{equation*}
		Z_{n+1} \leq Z_n + ( 1+ C_{\mathrm{lgb}}(2+ T))Z_n \Delta t_n.
	\end{equation*}
	Then, using \eqref{eq: dlr-ps-ind} one gets
	\begin{equation*}
		\begin{aligned}
			Z_{n+1} \leq & Z_n \left(1+ ( 1+ C_{\mathrm{lgb}}(2+ T)) \Delta t_n \right)\\
			\leq &Z_0 e^{( 1+ C_{\mathrm{lgb}}(2+ T))t_n}\left(1+ ( 1+ C_{\mathrm{lgb}}(2+ T)) \Delta t_n \right)\\
			\leq &Z_0 e^{( 1+ C_{\mathrm{lgb}}(2+ T))t_n}e^{( 1+ C_{\mathrm{lgb}}(2+ T)) \Delta t_n }\\
			= & Z_0 e^{( 1+ C_{\mathrm{lgb}}(2+ T))t_{n+1}},
		\end{aligned}
	\end{equation*} 
	which concludes the proof.
\end{proof}

\begin{proof}[\bfseries Proof of Theorem \ref{thm: convergence of Stoch DLR Proj}]\label{proof: thm: convergence of Stoch DLR Proj}
	The proof follows very closely the one of Theorem \ref{thm: convergence of DLR Euler-Maruyama}, hence we will present only the key steps here.
	Fix $n$. Suppose that we have just completed Step 6 of Algorithm \ref{alg: Stoch DLR Proj algorithm}; then we know
	\begin{equation*}
		X_{n+1} = U_{n+1}^{\top} Y_{n+1} = \tilde{U}_{n+1}^{\top} \tilde{Y}_{n+1},
	\end{equation*}
	where $\tilde{U}_{n+1},\tilde{Y}_{n+1}$ are the updates of the deterministic and stochastic modes before the \texttt{QR} decomposition in Step 6 of Algorithm \ref{alg: Stoch DLR Proj algorithm}, respectively. 
	Then, one has
	\begin{equation*}
		\begin{aligned}
			X_{n+1}& = X_{n} + P_{U_n^{\top}\tilde{Y}_{n+1} } \left[ a_n \right]\Delta t  + P_{U_n} b_n \Delta W_n \\
		\end{aligned}
	\end{equation*}
	
	We recall that $n_r = \max\{n = 0,1, \dots, N : t_n \leq r\}$. We estimate the error in the following way
	\begin{equation*}
		\begin{aligned}
			\mathbb{E}\left[\max\limits_{0 \leq n \leq N} |X(t_n)-X_{n}|^2\right]
			\leq & 3\mathbb{E}\left[ \max\limits_{0 \leq n\leq N}   |  \sum_{\ell = 0}^{n-1} \int_{t_\ell}^{t_{\ell+1}} \left(P_{X(r) } \left[a(r,X(r)) \right]  - P_{U_\ell ^{\top}\tilde{Y}_{\ell +1} } \left[a_\ell  \right] \right) \mathrm{d}r|^2\right ] \\
			&+  3 \mathbb{E}\left[  \max\limits_{0 \leq n\leq N}  | \sum_{\ell = 0}^{n-1} \int_{t_\ell}^{t_{\ell+1}} \left( P_{U_r}  [ b(r,X(r))]- P_{U_\ell }   \left[ b_\ell \right]\right) \mathrm{d}W_r |^2\right] \\
			\leq & 3\mathbb{E}\left[   N \sum_{\ell = 0}^{N-1} |   \int_{t_\ell}^{t_{\ell+1}} \left( P_{X(r) } \left[a(r,X(r)) \right]  - P_{U_\ell ^{\top}\tilde{Y}_{\ell +1} } \left[a_\ell  \right]\right) \mathrm{d}r|^2\right ] \\
			&+  12 \mathbb{E}\left[  | \sum_{\ell = 0}^{N-1} \int_{t_\ell}^{t_{\ell+1}} \left( P_{U_r}  [ b(r,X(r))]- P_{U_\ell }   \left[ b_\ell \right] \right) \mathrm{d}W_r |^2\right] \\
			\leq & 3T  \int_{0}^{T} \mathbb{E}\left[ |   P_{X(r) } \left[a(r,X(r)) \right]  - P_{U_{n_r}^{\top}\tilde{Y}_{{n_r}+1} } \left[a_{n_r} \right] |^2 \right ]\mathrm{d}r  \\
			&+  12   \int_{0}^{T}  \mathbb{E}\left[  | P_{U_r}  [ b(r,X(r))]- P_{U_{n_r}}   \left[ b_{n_r}\right]  |^2\right] \mathrm{d}r \\
		\end{aligned}
	\end{equation*}	
	where in the second line we use the Jensen's inequality and \cite[Theorem 4.4.4]{durrett2019probability}, in the third line we use the Itô's isometry. \\			
	For the first term on the right-hand side, using relations \eqref{eq:lip}, \eqref{eq:lin-growth}, \eqref{polynomial growth in t} \cite[Proposition A.3]{kazashi2025dynamical}, properties of Itô's integral, and Lemma \ref{lem: bound stoch dlr proj}, one can easily get that
	\begin{equation}\label{eq: inter thm conv do}
		\begin{aligned}
			&\mathbb{E}\left[  |P_{X(r) } \left[a(r,X(r)) \right]  - P_{U_{n_r}^{\top}\tilde{Y}_{{n_r}+1} } \left[a_{n_r} \right] |^2 \right]\\
			= &   \mathbb{E}\Big[  |P_{X(r) } \left[a(r,X(r)) \right] - P_{X(r) } \left[a(t_{n_r},X(r)) \right] + P_{X(r) } \left[a(t_{n_r},X(r)) \right]   -P_{X(r) } \left[a_{n_r} \right] \\
			&+ P_{X(r) } \left[a_{n_r} \right]  - P_{U_{n_r}^{\top}\tilde{Y}_{{n_r}+1} } \left[a_{n_r} \right] |^2 \Big]\\
			\leq &  3 C_{\mathrm{na}} (1+K_2(T)) (\Delta t)^{2\alpha}  +3 \mathbb{E}\left[  |P_{X(r) }\left[a(t_{n_r},X(r))-a_n \right]  |^2 \right] +3\mathbb{E}\left[    |(P_{X(r) }  - P_{U_{n_r}^{\top}\tilde{Y}_{{n_r}+1} } ) \left[a_{n_r} \right] |^2 \right] \\
			\leq & 3 C_{\mathrm{na}} (1+K_2(T)) (\Delta t)^{2\alpha} + 3C_{\mathrm{Lip}} \mathbb{E}\left[  |X(r) -X_{n_r}   |^2 \right] + 6C_{\mathrm{lgb}} (1+K_1(T)) \frac{1}{\gamma} \mathbb{E} \left[  |X(r) -U_{n_r}^{\top}\tilde{Y}_{{n_r}+1}   |^2 \right]\\
			\leq & 3 C_{\mathrm{na}} (1+K_2(T)) (\Delta t)^{2\alpha} + 3C_{\mathrm{Lip}} \mathbb{E}\left[  |X(r) -X_{n_r}   |^2 \right] \\
			&+ 6C_{\mathrm{lgb}} (1+K_3(T)) \frac{1}{\gamma} \mathbb{E} \left[  |X(r) -X_{n_r} - U_{n_r}^{\top}U_{n_r} a_{n_r} \Delta t - U_{n_r}^{\top}U_{n_r}  b_{n_r} \Delta W_{n_r}   |^2 \right]\\
			\leq & 3 C_{\mathrm{na}} (1+K_2(T)) (\Delta t)^{2\alpha} + 3C_{\mathrm{Lip}} \mathbb{E}\left[  |X(t_{n_r}) + \int_{t_{n_r}}^{r} P_{X(r)}a(r)\mathrm{d} r +\int_{t_{n_r}}^{r} P_{U(r)}b(r)\mathrm{d} W_r -X_{n_r}   |^2 \right] \\
			&+ 6C_{\mathrm{lgb}} (1+K_2(T)) \frac{1}{\gamma} \cdot \\
			& \quad \mathbb{E} \left[  |X(t_{n_r}) -X_{n_r} + \int_{t_{n_r}}^{r} P_{X(r)}a(r)\mathrm{d} r +\int_{t_{n_r}}^{r} P_{U(r)}b(r)\mathrm{d} W_r - U_{n_r}^{\top}U_{n_r} a_{n_r} \Delta t - U_{n_r}^{\top}U_{n_r}  b_{n_r} \Delta W_{n_r}   |^2 \right]\\
			\leq & 3 C_{\mathrm{na}} (1+K_2(T)) (\Delta t)^{2\alpha} + 6C_{\mathrm{Lip}} \mathbb{E}\left[  |X(t_{n_r}) -X_{n_r}   |^2 \right] + 6C_{\mathrm{Lip}}  C_{\mathrm{lgb}}(1 + K_2(T)) ((\Delta t)^2 + \Delta t)\\
			&+ 18C_{\mathrm{lgb}} (1+K_2(T)) \frac{1}{\gamma} \cdot \left( \mathbb{E} \left[  |X(t_{n_r}) -X_{n_r} |^2 \right] + C_{\mathrm{lgb}}(2 + K_3(T) + K_2(T)) ((\Delta t)^2 + \Delta t) \right)\\
			\leq & 3 C_{\mathrm{na}} (1+K_2(T)) (\Delta t)^{2\alpha} + \left(6C_{\mathrm{Lip}}  + 18C_{\mathrm{lgb}} (1+K_2(T)) \frac{1}{\gamma}\right)\mathbb{E}\left[  \max_{0 \le n_p \le n_r}|X(t_{n_p}) -X_{n_p}   |^2 \right] \\
			&+ \left[6C_{\mathrm{Lip}}  C_{\mathrm{lgb}}(1 + K_2(T)) + 18C_{\mathrm{Lip}}  C_{\mathrm{lgb}}(2 + K_3(T) + K_2(T)) \right] ((\Delta t)^2 + \Delta t).\\
		\end{aligned}
	\end{equation}	
	Similarly, after having applied Jensen's and Doob's inequalities, and Itô's isometry, via using the linear-growth bound and Lemma \ref{lem: bound stoch dlr proj} one can get that
	\begin{equation*}
		\begin{aligned}
			&\mathbb{E}\left[  |P_{U_r}  [ b(r,X(r))]- P_{U_{n_r}}   \left[ b_{n_r}\right]|^2 \right]\\
			\leq & 3 C_{\mathrm{na}} (1+K_2(T)) (\Delta t)^{2\alpha} + \left(3 C_{\mathrm{Lip}} + 18 C_{\mathrm{lgb}} (1+K_3(T)) \frac{1}{\gamma} \right) \mathbb{E}\left[  \max_{0 \le n_p \le n_r}|X(t_{n_p}) -X_{n_p}   |^2 \right]  \\
			& +  \left(6 C_{\mathrm{Lip}} + 18 C_{\mathrm{lgb}} (1+K_3(T)) \frac{1}{\gamma} \right)(1+K_2(T)) ((\Delta t)^2 + \Delta t),\\
		\end{aligned}
	\end{equation*}	
 where we recall that $K_2(T)$ is a positive constant bounding on the second moment of $Y(t)$ \cite[Lemma 2.7]{kazashi2025dynamical}.
	Using these relations in \eqref{eq: inter thm conv do} and Gronwall's lemma, we finally get the thesis.
\end{proof}

\begin{proof}[\bfseries Proof of Proposition \ref{prop: stab DLRA}]\label{proof: prop: stab DLRA}
	The DO equations for \eqref{eq: moltiplic sdes} read as follows
	\begin{equation*}
		\begin{aligned}
			\dot{U}(t) &  = U(t)A^{\top}(t)\left(I_{d \times d} - P_{U(t)} \right), \\
			\mathrm{d}Y(t) & = U(t)A(t)U(t)^{\top}Y(t)\mathrm{d}t + U(t)\sum_{k=1}^{m} B_k(t) U(t)^{\top}Y(t)\mathrm{d}W^{k}_t,
		\end{aligned}
	\end{equation*}
	and, hence, our DLR approximation is expressed via
	\begin{equation} \label{eq: DLR sde stab}
		\begin{aligned}
			\mathrm{d}X(t) & = A(t)U(t)^{\top}Y(t)\mathrm{d}t+U^{\top}(t)U(t)\sum_{k=1}^{m} B_k(t) U(t)^{\top}Y(t)\mathrm{d}W^{k}_t.
		\end{aligned}
	\end{equation}
	We compute an upper bound of the second moment of the solution $X(t)$ via It\^o's formula:
	\begin{equation*}
		\begin{aligned}
			\mathrm{d}\mathbb{E}[X(t)^{\top}X(t)]
			\leq & \lambda_{\max}(A(t)+ A^{\top}(t))\mathbb{E}[|X(t)|^2 ] \mathrm{d}t +  \mathbb{E}[|X(t)|^2]\sum\limits_{k=1}^{m}|U(t)B_k(t)|^2\mathrm{d}t.
		\end{aligned}
	\end{equation*}
	Therefore, one finally gets
	\begin{equation}\label{eq: stab2}
		\frac{\mathrm{d}\mathbb{E}[|X(t)|^2]}{\mathrm{d}t} \leq \left(\lambda_{\max}(A(t)+ A^{\top}(t)) + \sum\limits_{k=1}^{m}  |B_k(t)|^2 \right) \mathbb{E}[|X(t)|^2],
	\end{equation}
	and, hence, 
	\begin{equation*}
		\lim\limits_{t \to \infty} \mathbb{E}[|X(t)|^2] = 0,
	\end{equation*}
	under \eqref{eq: cond stab sde}.
\end{proof}

\end{document}